\documentclass[11pt]{article}
\usepackage{amsthm,amssymb,amsmath}
\usepackage{graphicx,algorithmic,algorithm}%
\usepackage{enumerate}
\usepackage{tikz,tikz-cd}
\usepackage{tabularx}
\usepackage{multirow}
\usepackage{makecell}
\usepackage{subcaption}
\tikzcdset{scale cd/.style={every label/.append style={scale=#1},
    cells={nodes={scale=#1}}}}
\usetikzlibrary{shapes}
\usepackage[letterpaper, margin=1in]{geometry}
\usepackage{hyperref}
\usepackage[nameinlink,capitalize]{cleveref}
\usepackage{todonotes}
\usepackage[normalem]{ulem}
\hypersetup{
	pdftitle=   {Chi-bounded of P7C4C6-free graphs},
	pdfauthor=  {Yeonsu Chang, Shenwei Huang, O-joung Kwon, Myounghwan Lee, and Changqing Xi}
}
\usepackage{authblk}
\usepackage{mathabx}
\usepackage{caption}
\usepackage{subcaption}
\usepackage{commath}

\usepackage{comment}

\newtheorem{theorem}{Theorem}[section]
\newtheorem{lemma}[theorem]{Lemma}
\newtheorem{proposition}[theorem]{Proposition}

\newtheorem{conjecture}[theorem]{Conjecture}

\newtheorem{definition}[theorem]{Definition}
\newtheorem*{thm:main1}{Theorem~\ref{thm:main1}}
\newtheorem*{thm:main2}{Theorem~\ref{thm:main2}}

\newcommand\extrafootertext[1]{%
    \bgroup
    \renewcommand\thefootnote{\fnsymbol{footnote}}%
    \renewcommand\thempfootnote{\fnsymbol{mpfootnote}}%
    \footnotetext[0]{#1}%
    \egroup
}

\newcommand{\supp}{\operatorname{supp}}

\newcommand\ceil[1]{\left\lceil #1 \right\rceil}

\title{The optimal chromatic bound for even-hole-free graphs without induced seven-vertex paths}
	\author[1]{Shenwei Huang}
    \author[2]{Yidong Zhou}
	\author[3]{Yeonsu Chang}

	\affil[1] {School of Mathematical Sciences and LPMC, Nankai University, Tianjin 300071, P.R.~China. Email: {\tt shenweihuang@nankai.edu.cn}.}
    \affil[2] {College of Computer Science, Nankai University, Tianjin 300350, China. Email: {\tt 1120240329@mail.nankai.edu.cn}.}
	\affil[3]{Department of Mathematics, Hanyang University, Seoul, South Korea. Email: {\tt yeonsu@hanyang.ac.kr}.}

	\date\today

\begin{document}

\maketitle

\begin{abstract}
The class of even-hole-free graphs has been extensively studied on its own and on its relation to perfect graphs.
In this paper, we study the $\chi$-boundedness of even-hole-free graphs which itself is an important topic in graph theory.
In particular, we prove that every even-hole-free graph $G$ without induced 7-vertex paths satisfies $\chi(G)\le \ceil{\frac{5}{4}\omega(G)}$, where $\chi(G)$ and $\omega(G)$ denote the chromatic number and clique number of $G$, respectively. This bound is optimal. 
Our result strictly extends the result of Karthick and Maffary \cite{KM19} on even-hole-free graphs without induced 6-vertex paths, and implies that even-hole-free graphs without induced 7-vertex paths satisfy Reed's Conjecture. 
Our proof relies on a heavy structural analysis on a maximal substructure called a nice blowup of a five-cycle and can be viewed for graphs in which all holes are of length five (graphs with all holes having the same length gain increasing interest in recent years \cite{COOK202496}). Our result gives a partial answer to a conjecture of Wang and Wu \cite{WW25} on graphs in which all holes are of length 5.
One of the key technical ingredients is a technical lemma proved via clique cutset argument combined with the idea of Infinite Descent Method (often used in number theory).  
\end{abstract}

\tableofcontents

\section{Introduction}

For a list of graphs $(H_1,\ldots,H_k)$, we say that a graph $G$ is {\em $(H_1,\ldots,H_k)$-free} if $G$ has no graph $H_i$ as an induced subgraph for each $1\le i\le k$.
Let $\omega(G)$ denote the clique number of a graph $G$. 
For a positive integer $k$, a \textit{proper $k$-coloring} of a graph $G$ is a function $f:V(G)\rightarrow [k]$ such that $f(u)\neq f(v)$ for every edge $uv\in E(G)$. In such a case, we say that $G$ is \textit{$k$-colorable}.
The \textit{chromatic number} of $G$, denoted by $\chi(G)$, is the minimum positive integer $k$ such that $G$ is $k$-colorable.
For a graph class $\mathcal{G}$, we say that $\mathcal{G}$ is {\em $\chi$-bounded} if there exists a function $f:\mathbb{N}\rightarrow \mathbb{N}$ such that every graph $G\in \mathcal{G}$ has $\chi(G)\le f(\omega(G))$. The function $f$ is called a {\em $\chi$-bounding function} for $\mathcal{G}$.

The study of $\chi$-bounding functions for graph classes is an important area in graph theory in the past two decades, see for example the survey of Scott and Seymour \cite{SS20} and Scott's survey \cite{Scott23} in Proceedings of International Congress Mathematics 2022. The study can be roughly divided into two lines: graphs with a single forbidden induced subgraph and graphs with forbidden holes. The research on $\chi$-boundedness of graphs with a single forbidden induced subgraph is largely surrounded by Gy\'arf\'as-Sumner Conjecture. The conjecture is notoriously hard despite some partial progress (see \cite{SS20XIII} for instance).

In this paper, we focus on graphs with forbidden holes. 
A \textit{hole} in $G$ is an induced cycle of order at least four.
Scott and Seymour \cite{SS16} proved that the class of {\em odd-hole-free} graphs, graphs with no holes of odd length, is $\chi$-bounded with $f=2^{2^n}$ being a $\chi$-bounding function. It is still open whether this double-exponential bound can be improved to a single-exponential bound. In sharp contrast, the class of {\em even-hole-free} graphs, graphs with no holes of even length, admits a linear $\chi$-bounding function $f(n)=2n-1$. This follows from a stronger result by Chudnovsky and Seymour \cite{CS23} that every even-hole-free graph has a vertex whose neighborhood can be covered by two cliques. It is not known whether this bound is optimal.

\medskip
In fact, several known results on subclasses of even-hole-free graphs suggest that this bound may not be optimal.

\begin{itemize}
    \item[(Pan)] A {\em pan} is a hole with a pendent edge. It is shown in \cite{CCH18} that every (pan, even-hole)-free graph $G$ has $\chi(G)\le \ceil{\frac{3}{2}\omega(G)}$.
    \item[(Cap)] A {\em cap} is a hole with an additional vertex adjacent to an edge on the hole. It is shown in \cite{cap} that every (cap, even-hole)-free graph $G$ has $\chi(G)\le \ceil{\frac{5}{4}\omega(G)}$.
    \item[($P_6$)] The path on $t$ vertices is denoted by $P_t$. It is shown in \cite{KM19} that every ($P_6$, even-hole)-free graph $G$ has $\chi(G)\le \ceil{\frac{5}{4}\omega(G)}$. 
    \item[(Diamond)] A {\em diamond} is obtained from the complete graph on four vertices by removing an edge. It is shown in \cite{KLOKS2009733} that every (diamond, even-hole)-free graph $G$ has $\chi(G)\le \omega(G)+1$.

\end{itemize}

For more background and results on even-hole-free graphs, we refer the readers to a comprehensive survey by Vu\v{s}kovi\'c \cite{Vus10}.
The main result of this paper is the following optimal $\chi$-bounding function for the class of even-hole-free graphs without an induced 7-vertex path.

\begin{theorem}\label{thm:main}
    For every ($P_7$, even-hole)-free graph $G$, $\chi(G)\le \ceil{\frac{5}{4}\omega(G)}$.
\end{theorem}

\noindent {\bf Remark.} This bound is optimal as seen by equal-size blowups of a five-cycle (the precise definition of equal-size blowups is given in Section \ref{sec:pre}).
Reed's Conjecture \cite{Reed98} states that every graph $G$ satisfies $\chi(G)\le \ceil{\frac{\Delta(G)+\omega(G)+1}{2}}$ where $\Delta(G)$ is the maximum degree of $G$. 
Our result implies by a result in \cite{KM18} that the class of ($P_7$, even-hole)-free graphs satisfies Reed's Conjecture. 
For an integer $\ell\ge 4$, a graph is {\em $\ell$-holed} if all holes in this graph have length $\ell$. 
Such graphs were studied \cite{Woodroofe09} in the context of algebraic combinatorics and commutative algebra.
A structure theorem for $\ell$-holed graphs 
when $\ell \ge 7$ is proved in \cite{COOK202496}. Using this structure theorem, Wang and Wu \cite{WW25} showed that $f(n)=\ceil{\frac{\ell}{\ell-1}n}$ is the optimal $\chi$-bounding function for the class of $\ell$-holed graphs when $\ell\ge 7$ is odd. They further conjecture that this bound also holds for $\ell=5$. 
Our result is essentially for the class of 5-holed graphs without an induced $P_7$, and proves their conjecture for this graph class. 
One difficulty in our proof is that there are no structure theorems for 5-holed graphs at hand and we need to do heavy structural analysis along the way.

\medskip
\noindent {\bf Proof sketch.} Since the proof is long, let us sketch the proof here. For integer $s\ge 3$, we denote by $C_s$ the cycle on $s$ vertices.
Let $G$ be a $(P_7,C_4,C_6)$-free (which is the same as ($P_7$, even-hole)-free) graph. 
The overall strategy is to consider two cases depending on whether $G$ contains an induced $C_7$. In the case that $G$ contains an induced $C_7$, we show, based on a structural result of Penev \cite{Penev25}, that Theorem \ref{thm:main} holds (this case is easy). 
Afterwards, we can assume that $G$ is $C_7$-free. In this case, we can assume that $G$ contains an induced $C_5$ (this case is difficult and most of our proofs are devoted to this case) using a classical result on chordal graphs by Dirac (Lemma \ref{lem:chordal has simplicial vertices}). Then we define and analyze the so-called nice blowup $H$ of an induced $C_5$ that mimics the structure of an induced $C_5$. We then partition $V(G)\setminus V(H)$ into subsets $A_0,A_1,A_2,A_3,A_5$ in terms of their neighborhoods on $H$.
After proving structural properties of those subsets, we show that Theorem \ref{thm:main} holds when $A_1\neq \emptyset$ (Section \ref{sec:A_1 is not empty}), $A_1=\emptyset$ and $A_2\neq \emptyset$ (Section \ref{sec:A_2 is not emptyset}), and $A_1=A_2=\emptyset$ but $A_3\neq \emptyset$ (Section \ref{sec:only A3}). 
These three sections together complete the proof of Theorem \ref{thm:main}.

\medskip
\noindent {\bf Technical highlights.} To carry out this strategy, we use technical arguments based on several reducible structures in the study of $\chi$-bounding functions-clique cutsets, small vertices (Lemma \ref{lem:reducible structure}) and $(p,q)$-good subgraphs (Lemma \ref{lem:goodsubgraph}). It is easy to show that it suffices to prove Theorem \ref{thm:main} for graphs $G$ without those structures (Lemma \ref{lem:basic graphs}).
The first argument is called {\em clique cutset argument}, which aims to show that the neighborhood of certain subsets is a clique. In some cases, a direct analysis does not work and we need to combine new technical ideas such as nice partitions defined in Section \ref{sec:a general lemma}. The idea is similar to Infinite Descent Method in number theory: if the desired property does not hold, one can first show the existence of a nice partition and then argue that a given nice partition can be made ``smaller" in the sense defined in Section \ref{sec:a general lemma}, as long as $G$ has no clique cutsets.  
Since the process cannot continue forever, $G$ contains a clique cutset, which is a contradiction.
The second argument is called {\em small vertex argument}. 
The idea is to show that the neighborhood of a vertex $v$ in $V(G)\setminus K$ is a clique where $K$ is a clique of $G$. 
This implies that $K$ is ``large", meaning that $|K| > \ceil{\frac{\omega(G)}{4}}$. Applying this argument to (minimal) simplicial vertices allows us to conclude in many cases that there are four pairwise disjoint ``large" cliques whose union is still a clique. This gives a clique of size larger than $\omega(G)$, which is a contradiction. The {\em good subgraph argument} is to find a $(p,q)$-good subgraph. The $(p,q)$-good subgraphs were already implicit in \cite{KM19}. We explicitly formulate this notion and use $(4,5)$-good subgraphs to prove our main theorem.
This argument is technically involved and used in Sections \ref{sec:C_7} and \ref{sec:only A3}. The final argument is called {\em pre-coloring argument}. For graphs $G$ that cannot be handled by the previous three arguments, we need to give an explicit proper coloring using $\ceil{\frac{5}{4}\omega(G)}$ colors. The idea is to color a subgraph of $G$ first and then color the rest of the graph by reducing the problem (based on structural insight) to color hyperholes (Lemma \ref{lem:hyperhole}). To be able to apply Lemma \ref{lem:hyperhole}, we need to carefully choose the subgraph to be pre-colored so that the calculations can go through.

\medskip
\noindent {\bf Paper organization.} The paper is organized as follows. In Section \ref{sec:pre}, we define terminology and notations used in this paper and prove some general lemmas. In Section \ref{sec:C_7}, we handle the case that $G$ contains an induced $C_7$. In Section \ref{sec:nice blowup of C_5}, we define nice blowups of $C_5$ and their attachments $A_0,A_1,A_2,A_3,A_5$ and prove some basic properties about those sets.
In Section \ref{sec:A_0}, we show that $A_0=\emptyset$. In Sections \ref{sec:A_1}-\ref{sec:A_3}, we obtain further properties of $A_1,A_2, A_3$. In Section \ref{sec:non-neighbors of C5}, we prove a useful lemma about components of non-neighbors of a $C_5$. In Section \ref{sec:a general lemma}, we prove a key technical lemma. In Section \ref{sec:A_1 is not empty}, we prove Theorem \ref{thm:main} when $A_1$ is not empty. In Section \ref{sec:A_2 is not emptyset}, we prove Theorem \ref{thm:main} when $A_1$ is empty and $A_2$ is not empty. In Section \ref{sec:only A3}, we prove Theorem \ref{thm:main} when $A_1, A_2$ are empty and $A_3$ is not empty. In Section \ref{sec:conclue}, we give some concluding remarks.

\section{Preliminaries}\label{sec:pre}

All graphs in this paper are simple and finite.
Let $G$ be a graph. We denote by $V(G)$ and $E(G)$ the {\em vertex set} and the {\em edge set} of $G$, respectively. For an edge $\{u,v\}$ of $G$, we shortly say $uv$.  For $S\subseteq V(G)$, let $G[S]$ denote the subgraph of $G$ {\em induced by} $S$, and let $G-S=G[V(G)\setminus S]$. For convenience, we identify a subset $S\subseteq V(G)$ as $G[S]$.
A graph $H$ is an {\em induced subgraph} of $G$ if there exists a subset $S\subseteq V(G)$ such that $G[S]$ is isomorphic to $H$.
For two disjoint vertex sets $S_1$ and $S_2$ in $G$, $S_1$ is \textit{complete} to $S_2$ if every vertex in $S_1$ is adjacent to every vertex in $S_2$, and $S_1$ is \textit{anticomplete} to $S_2$ if there are no edges between $S_1$ and $S_2$. If $S_1=\{v\}$, we simply say that $v$ is complete or anticomplete to $S_2$ instead of $\{v\}$ is complete or anticomplete to $S_2$. A vertex $v$ {\em mixes on} a subset $S$ if $v$ is neither complete nor anticomplete to $S$. We say $v$ mixes on an edge $xy$ if $v$ mixes on $\{x,y\}$.

Let $v$ be a vertex of $G$ and $H$ be a subgraph of $V(G)$. The neighbors of $v$ in $G$ that are contained in $H$ is denoted by $N_{H}(v)$. So $N_G(v)$ is the set of all neighbors of $v$ in $G$. The \textit{degree} of $v$ in $G$, denoted by $d_G(v)$, is the number of neighbors of $v$ in $G$.  
Let $N_G[v]=N_G(v)\cup \{v\}$.
For a non-empty vertex subset $S$ of $G$, we define $N_H(S)=\left(\bigcup_{v\in S}N_H(v)\right)\setminus S$.
If the context is clear, we shall omit the subscript $G$, writing $N_G(v)$ as $N(v)$ and so on. We say that $v\in V(H)$ is {\em universal} in $H$ if $N_H[v]=V(H)$. A set $S\subseteq V(H)$ is universal in $H$ if each vertex in $S$ is universal in $H$.

For a positive integer $n$, let $[n]$ denote the set of positive integers at most $n$. For two integers $i,j\in [n]$, the operations $+$ and $-$ are the same as in the finite cyclic group $\mathbb{Z}_n$.
For a positive integer $n$, let $P_n$ denote the path on $n$ vertices where $V(P_n)=\{v_1,\ldots, v_n\}$ and $E(P_n)=\{v_i v_{i+1}:i\in [n-1]\}$. The {\em internal} of $P$ is $V(P)\setminus \{v_1,v_n\}$.
Let $C_n$ denote the cycle on $n$ vertices where $V(C_n)=\{v_1,\ldots, v_n\}$ and $E(C_n)=\{v_i v_{i+1}:i\in [n]\}$.
For clarity, we write $P_n$ as $v_1-v_2-\cdots-v_n$ and write $C_n$ as $v_1-v_2-\cdots-v_n-v_1$. For $v_i$ and $v_j$ in $P_n$ with $i<j$, we denote by $v_i-P-v_j$ the {\em subpath} of $P$ from $v_i$ to $v_j$.
We say that a graph $G$ is \textit{connected} if there is a path between $v$ to $w$ for every vertices $v,w\in V(G)$.
A \textit{component} of $G$ is a maximal connected subgraph of $G$.
The complete graph of order $n$ is denoted by $K_n$.
The \textit{clique number} of a graph $G$, denoted by $\omega(G)$, is the size of a largest clique in $G$. The \textit{independence number} of a graph $G$, denoted by $\alpha(G)$, is the size of a largest stable set in $G$.
A vertex subset $S$ of a connected graph $G$ is said to be a \textit{clique cutset} if $S$ is a clique and $G-S$ is not connected. If $H_1,\ldots,H_t$ are components of $G-S$ where $S$ is a clique cutset, then $G[V(H_i)\cup S]$ ($i\in [t]$) is called an {\em $S$-component}. When $G$ is disconnected, any component of $G$ is simply an $S$-component for $S=\emptyset$.

A graph $H$ is said to be a \textit{blowup} of a graph $G$ if $H$ is obtained from $G$ by replacing each vertex $v$ of $G$ by some (possibly empty) clique $C_v$ and adding all edges between $C_v$ and $C_w$ when $vw$ is an edge of $G$. If $C_v=\emptyset$ for some $v$, it means to remove $v$ from $G$. If $C_v\neq \emptyset$ for each $v\in V(G)$, then $H$ is called a {\em non-empty} blowup of $G$.
If $C_v$ is of the same size for all $v$, then $H$ is called an {\em equal-size blowup} of $G$.
For an integer $k\ge 4$, a {\em $k$-hyperhole} is a non-empty blowup of $C_k=v_1-v_2-\cdots-v_k-v_1$. We denote $C_k[s_1,\ldots,s_k]$ by a $k$-hyperhole where the clique $C_{v_i}$ substituting for $v_i$ has size $s_i$. 

\begin{lemma}[\cite{NS01}]\label{lem:hyperhole}
    For any $k$-hyperhole $H$, $\chi (H)=\max \left \{\omega(H), \ceil{\frac{|V(H)|}{\alpha(H)}} \right \}$.
\end{lemma}

\medskip
We first present two simple properties of $C_4$-free graphs.

\begin{lemma}\label{lem:nbd containment for an edge}
    Let $G$ be a $C_4$-free graph. For any clique $K$ of $G$ and any edge $vw\in E(G)$ with $v,w\notin K$, $N_K(v)\subseteq N_K(w)$ or $N_K(w)\subseteq N_K(v)$. 
\end{lemma}

\begin{proof}
    Suppose not. Then there exist vertices $a\in N_K(v)\setminus N_K(w)$ and $b\in N_K(w)\setminus N_K(v)$. It follows that $a-v-w-b-a$ is an induced $C_4$. 
\end{proof}

\begin{lemma}\label{lem:P3 to P4}
    Let $G$ be a $C_4$-free graph and $a-b-c$ be an induced $P_3$ in $G$. If $c$ has a neighbor $d$ not adjacent to $b$, then $d$ is not adjacent to $a$.
\end{lemma}

\begin{proof}
    If not, then $a-b-c-d-a$ is an induced $C_4$.
\end{proof}

Throughout the paper, we will use Lemmas \ref{lem:nbd containment for an edge} and \ref{lem:P3 to P4} without referring to them.
A graph is {\em chordal} if it has no induced cycle of length at least 4. A vertex $v$ in a graph $G$ is {\em simplicial in} $G$ if $N_G[v]$ is a clique. In this case, we also say that $v$ is a simplicial vertex of $G$. 
The following is a well-known result for chordal graphs.

\begin{lemma}[\cite{Dirac61}]\label{lem:chordal has simplicial vertices}
    For every chordal graph $G$ that is not a complete graph, there are two non-adjacent simplicial vertices of $G$.
\end{lemma}

Lemma \ref{lem:chordal has simplicial vertices} implies that every chordal graph $G$ has $\chi(G)=\omega(G)$. Next, we define the notion of minimal simplicial vertices.

\begin{definition}[Minimal simplicial vertices]
    {\em Let $G$ be a graph, and $X$ and $Y$ be two disjoint subsets of $V(G)$.
    For a simplicial vertex $s'$ of $X$, we say that $s$ is a {\em $(X,Y)$-minimal simplicial vertex arising from} $s'$ if $s\in N_X[s']$ is simplicial in $X$ and $N_Y(s)$ is minimal (under set inclusion) over all vertices in $N_X[s']$ simplicial in $X$. We say $s$ is {\em $(X,Y)$-minimal simplicial vertex} if it is a $(X,Y)$-minimal simplicial vertex arising from some simplicial vertex of $X$.}
\end{definition}

The next lemma says that two non-adjacent simplicial vertices give rise to two non-adjacent minimal simplicial vertices.

\begin{lemma}[Two non-adjacent minimal simplicail vertices]\label{lem:two non-adjacent simplicail verteices}
    If $s_j$ is a $(X,Y)$-minimal simplicial vertex arising from a simplicial vertex $s'_j$ of $X$ for $j=1,2$ with $s'_1s'_2\notin E(G)$, then $s_1\neq s_2$ and $s_1s_2\notin E(G)$.
\end{lemma}

\begin{proof}
    Note that $s_1$ may be $s'_1$ and $s_2$ may be $s'_2$. If $s_1=s'_1$ and $s_2=s'_2$, then we are done. So we may assume that $s_1\neq s'_1$. As $s'_1$ mixes on $\{s_1,s'_2\}$, $s_1\neq s'_2$. Since $s_1$ is simplicial in $X$ and $s'_1s'_2\notin E(G)$, $s_1$ is not adjacent to $s'_2$ and thus $s_1\neq s_2$. If $s_2=s'_2$, we are done. So $s_2\neq s'_2$. If $s_1s_2\in E(G)$, then $s'_1,s_2$ are two neighbors of $s_2$ but they are not adjacent. This contradicts that $s_2$ is a simplicial vertex of $X$. 
\end{proof}

The next lemma is a useful property of minimal simplicial vertex which will be used frequently.

\begin{lemma}\label{lem:minimal simplicial vertex}
    Let $G$ be a $C_4$-free graph and let $X,Y\subseteq V(G)$ be disjoint subsets of $V(G)$ with $Y$ being a clique. For any $(X,Y)$-minimal simplicial vertex $s$, $N_{X\cup Y}(s)$ is a clique or there is an induced $P_4=a-s-c-d$ with $a\in Y$ and $c,d\in X$.
\end{lemma}

\begin{proof}
    Suppose that $s$ is a $(X,Y)$-minimal simplicial vertex arising from $s'\in X$.
    Suppose that $N_{X\cup Y}(s)$ is not clique. 
    Let $v,v'\in N_{X\cup Y}(s)$ be non-adjacent. 
    Since $s$ is simplicial in $X$ and $Y$ is a clique, we may assume that $v\in X$ and $v'\in Y$. 
    Suppose first that $v$ is simplicial in $X$. 
    Then $v\in N_X[v]=N_X[s]=N_X[s']$. By the choice of $s$, $v$ has a neighbor $w\in Y$ which is not adjacent to $s$. 
    Then $v'-s-v-w-v'$ is an induced $C_4$. So $v$ is not simplicial in $X$ and thus has a neighbor $w\in X\setminus N[s]$. 
    Now $v'-s-v-w$ is a desired induced $P_4$.
\end{proof}

The next lemma allows to conclude that certain subgraphs are $P_4$-free.

\begin{lemma}[The $P_4$-free lemma]\label{lem:P_4-free lemma}
    Let $G$ be $(C_4,C_6)$-free graph and $u\in V(G)$. If $X$ and $Y$ are disjoint subsets of $V(G)\setminus \{u\}$ satisfies 
    \begin{itemize}
        \item[$\bullet$] $u$ is complete to $X$ and anticomplete to $Y$, 
        \item[$\bullet$] every two non-adjacent vertices in $X$ have their non-empty neighborhoods in $Y$ complete to each other, and
        \item[$\bullet$] there is no induced $P_5=a-b-c-d-e$ with $e\in Y$ and $a,b,c,d\in X$,
    \end{itemize}
    then $X$ is $P_4$-free.
\end{lemma}

\begin{proof}
    Suppose that $a-b-c-d$ is an induced $P_4$ in $X$. Let $a'\in X$ be a neighbor of $a$ and $d'\in X$ be a neighbor of $d$. Then $a'd'\in E(G)$. Since $G$ is $C_4$-free, $a'c,a'd,d'a,d'b\notin E(G)$. If $a'b,d'c\in E(G)$, then $a'-b-c-d'-a'$ is an induced $C_4$. If $a'b,d'c\notin E(G)$, then $a-b-c-d-d'-a'-a$ is an induced $C_6$. So we may assume that $a'b\in E(G)$ and $d'c\notin E(G)$. Then $a-b-c-d-d'$ is an induced $P_5$ violating the last condition of $X$ and $Y$. 
\end{proof}

Next we define certain special subgraphs that are useful for induction.

\begin{definition}[$(p,q)$-good subgraphs]
{\em 
    For two positive integers $p$ and $q$, an induced subgraph $H$ of a graph $G$ is said to be a \emph{$(p,q)$-good subgraph} if
\begin{enumerate}[(1)]
    \item for every maximal clique $K$ of $G$, $K$ contains at least $p - (\omega(G) -\abs{K})$ vertices from $H$ and
    \item $H$ is $q$-colorable.
\end{enumerate}
}
\end{definition}

Note that a $(1,1)$-good subgraph is a stable set whose removal decreases the clique number by at least one. This kind of stable set is called {\em good stable set} in the literature. In particular, a universal vertex is a good stable set and so is a $(1,1)$-good subgraph.

\begin{lemma}\label{lem:goodsubgraph}
Let $p$ and $q$ be positive integers.
If a graph $G$ contains a $(p,q)$-good subgraph $H$ and $\chi(G-V(H))\le \ceil{\frac{q}{p}\omega(G-V(H))}$, then $\chi(G)\le \ceil{\frac{q}{p}\omega(G)}$.
\end{lemma}

\begin{proof}

We first show that $\omega(G-V(H))\le \omega(G)-p$.
Let $K^*$ be a maximum clique in $G-V(H)$ and $K$ be a maximal clique of $G$ containing $K^*$. Since $H$ is $(p,q)$-good, $|K\cap H|\ge p-(\omega(G)-|K|)$. So 
 \begin{align*} 
    |K^*| & = |K| - |K\cap H| \quad \quad \quad \quad \quad \quad (K^* \text{ is a maximum clique in } G-V(H))
    \\[3pt]
     & \le    |K| - (p-(\omega(G)-|K|))
    \\[3pt] 
     & = \omega(G) - p.
\end{align*}
It follows that
 \begin{align*} 
    \chi(G)& \le \chi(G-V(H)) + \chi (H)
    \\[3pt]
     & \le  \ceil{\frac{q}{p}\omega(G-V(H))} + q \quad \quad (\text{assumption})
    \\[3pt] 
     & \le \ceil{\frac{q}{p}(\omega(G)-p)} + q
     \quad \quad \quad (\omega(G-V(H))\le \omega(G)-p))
    \\[3pt]
    & =\ceil{\frac{q}{p}\omega(G)}.
\end{align*}
This completes the proof.
\end{proof}

\begin{lemma}[Reducible structures]\label{lem:reducible structure}
    Let $G$ be a connected graph and $f:\mathbb{N}\rightarrow \mathbb{N}$ be a function with $f(n+1)\ge f(n)+1$ for every $n\in \mathbb{N}$. 
    Then the following holds.
    \begin{itemize}
        \item[$(1)$]   If $G$ has a universal vertex $u$ and $\chi(G-u)\le f(\omega(G-u))$, then $\chi(G)\le f(\omega(G))$.
        \item[$(2)$] If $G$ has a clique cutset $K$ with $G_1,\ldots,G_t$ being $K$-components such that $\chi(G_i)\le f(\omega(G_i))$ for $i\in [t]$, then $\chi(G)\le f(\omega(G))$.
        \item[$(3)$]   If $G$ has a vertex $u$ with $d(u)\le f(\omega(G))-1$ and $\chi(G-u)\le f(\omega(G-u))$, then $\chi(G)\le f(\omega(G))$.  
    \end{itemize}   
\end{lemma}

\begin{proof}
    We prove one by one.

    \medskip
    \noindent (1) It is clear that $\chi(G)\le \chi(G-u)+1$ and $\omega(G-u)=\omega(G)-1$. So
     \begin{align*} 
     \chi(G) & \le f(\omega(G-u))+1 
    \\[3pt]
      & =     f(\omega(G)-1)+1
    \\[3pt] 
     & \le (f(\omega(G))-1) + 1 \quad \quad \quad (f(n+1)\ge f(n)+1)
         \\[3pt] 
         & =f(\omega(G)).
    \end{align*}

    \medskip
    \noindent (2) It is known \cite{Ta85} that $\chi(G)=\max\limits_{i\in [t]}\chi(G_i)$. So
    \begin{align*} 
     \chi(G) & = \max_{i\in [t]}\chi(G_i)
    \\[3pt]
      &  \le \max_{i\in [t]} f(\omega(G_i)) 
    \\[3pt] 
     &  \le f(\omega(G)). \quad \quad \quad \quad \quad \quad (f(n+1)\ge f(n)+1)
         \\[3pt] 
    \end{align*}

    \noindent (3) It follows that $\chi(G-u)\le f(\omega(G-u))\le f(\omega(G))$ since $f$ is increasing. So $G-u$ has a $f(\omega(G))$-coloring. 
    Since $u$ has degree at most $f(\omega(G))-1$, $N(u)$ are colored with at most $f(\omega(G))-1$ colors. So we can assign a color from $[f(\omega(G))]$ to $u$. This shows that $\chi(G)\le f(\omega(G))$.
\end{proof}

Let $f:\mathbb{N}\rightarrow \mathbb{N}$ be a function with $f(n+1)\ge f(n)+1$ for every $n\in \mathbb{N}$. A vertex $v$ of $G$ is said to be \textit{$f$-small} if $d_G(v)\le f(\omega(G))-1$. We say $v$ is {\em small} if it is $f$-small for $f(n)=\ceil{\frac{5}{4}n}$.
In particular, any simplicial vertex is small.
We say that a graph $G$ is {\em basic} if it is connected and has no $(4,5)$-good subgraphs, no $(1,1)$-good subgraphs, no clique cutsets and no small vertices. The next lemma follows immediately from Lemmas \ref{lem:goodsubgraph} and \ref{lem:reducible structure}.

\begin{lemma}\label{lem:basic graphs}
    Let $\mathcal{G}$ be a graph class. If every basic graph $G'$ in $\mathcal{G}$ has $\chi(G')\le \ceil{\frac{5\omega(G')}{4}}$, then every graph $G$ in $\mathcal{G}$ has $\chi(G)\le \ceil{\frac{5\omega(G)}{4}}$.
\end{lemma}

In the following, we can assume that $G$ is basic unless otherwise stated. For convenience, we set $\omega=\omega(G)$.

\section{Containing an induced $C_7$}\label{sec:C_7}

In this section, we assume that a $(P_7, C_4, C_6)$-free graph $G$ contains an induced $C=C_7=v_1-v_2-\cdots-v_7-v_1$. 
Our main goal of this section is to show that $\chi(G)\le \ceil{\frac{5}{4}\omega(G)}$ for such a graph $G$. 
We need to use a known result of Penev \cite{Penev25}.
To introduce the result, we first define three graphs $M$, $M_1$ and $M_2$. The graph $M$ has $V(M)=\{v_1,\dots,v_{12}\}$ and $E(M)$ as follows.
\begin{itemize}
    \item $v_1-v_2-\cdots-v_7-v_1$ is an induced $C_7$.
    \item $\{v_8,v_9,v_{10},v_{11},v_{12}\}$ induces a clique.
    \item $v_8$ is complete to $\{v_6,v_7,v_1,v_2,v_3\}$ and anticomplete to $\{v_4,v_5\}$.
    \item $v_9$ is complete to $\{v_7,v_1,v_2,v_3,v_4\}$ and anticomplete to $\{v_5,v_6\}$.
    \item $v_{10}$ is complete to $\{v_3,v_4,v_5,v_6,v_7\}$ and anticomplete to $\{v_1,v_2\}$.
    \item $v_{11}$ is complete to $\{v_3,v_6,v_7\}$ and anticomplete to $\{v_1,v_2,v_4,v_5\}$.
    \item $v_{12}$ is complete to $\{v_3,v_4,v_7\}$ and anticomplete to $\{v_1,v_2,v_5,v_6\}$.            
\end{itemize}
The graph $M_1$ has $V(M_1)=\{v_1,\dots,v_{9},v^*\}$ and $E(M_1)$ as follows.
\begin{itemize}
    \item $v_1-v_2-\cdots-v_7-v_1$ is an induced $C_7$.
    \item $v_8-v_9-v^*$ is an induced $P_3$.
    \item $v_8$ is complete to $\{v_6,v_7,v_1,v_2,v_3\}$ and anticomplete to $\{v_4,v_5\}$.
    \item $v_9$ is complete to $\{v_7,v_1,v_2,v_3,v_4\}$ and anticomplete to $\{v_5,v_6\}$.
    \item $v^*$ is complete to $\{v_1,v_4,v_5\}$ and anticomplete to $\{v_2,v_3,v_6,v_7\}$.       
\end{itemize}
The graph $M_2=M_1-v_9$ (see Figure \ref{fig:M} for $M,M_1,M_2$).
A blowup of $M$ is {\em special} if the cliques substituting $v_i\in V(M)$ is not empty for each $i\in [7]$.

\begin{lemma}[Theorem 3.11 in \cite{Penev25}]\label{blowup-G-of-H}
$G$ is a special blowup of $M$  or a non-empty blowup of $M_1$ or $M_2$.
\end{lemma}

\begin{figure}[htbp]
    \centering
    \begin{subfigure}{0.45\linewidth}
        \centering
        \begin{tikzpicture}[scale=0.8]
        \tikzstyle{v}=[circle, draw, solid, fill=black, inner sep=0pt, minimum width=3pt]

        \draw[thick] (0, 3.5)--(3, 2);
        \draw[thick] (3, 2)--(4, -1);
        \draw[thick] (4, -1)--(1.8, -3.5);
        \draw[thick] (1.8, -3.5)--(-1.8, -3.5);
        \draw[thick] (-1.8, -3.5)--(-4, -1);
        \draw[thick] (-4, -1)--(-3, 2);
        \draw[thick] (-3, 2)--(0, 3.5);

        \draw[thick] (0, 1.2)--(1.5, 0);
        \draw[thick] (0, 1.2)--(-1.5, 0);
        \draw[thick] (0, 1.2)--(-0.9, -1.5);
        \draw[thick] (0, 1.2)--(0.9, -1.5);
        \draw[thick] (1.5, 0)--(-1.5, 0);
        \draw[thick] (1.5, 0)--(0.9, -1.5);
        \draw[thick] (1.5, 0)--(-0.9, -1.5);
        \draw[thick] (-1.5, 0)--(0.9, -1.5);
        \draw[thick] (-1.5, 0)--(-0.9, -1.5);
        \draw[thick] (0.9, -1.5)--(-0.9, -1.5);

        \draw[red, thick] (0, 1.2)..controls (2.2, 1)..(4, -1);
        \draw[red, thick] (0, 1.2)..controls (-2.2, 1)..(-4, -1);
        \draw[red, thick] (0, 1.2)--(3, 2);
        \draw[red, thick] (0, 1.2)--(-3, 2);
        \draw[red, thick] (0, 1.2)--(0, 3.5);

        \draw[orange, thick] (1.5, 0)..controls (-0.5, 0.7)..(-3, 2);
        \draw[orange, thick] (1.5, 0)--(0, 3.5);
        \draw[orange, thick] (1.5, 0)--(3, 2);
        \draw[orange, thick] (1.5, 0)--(4, -1);
        \draw[orange, thick] (1.5, 0)--(1.8, -3.5);

        \draw[purple, thick] (-0.9, -1.5)..controls (1.5, -2.5)..(4, -1);
        \draw[purple, thick] (-0.9, -1.5)--(1.8, -3.5);
        \draw[purple, thick] (-0.9, -1.5)--(-1.8, -3.5);
        \draw[purple, thick] (-0.9, -1.5)--(-4, -1);
        \draw[purple, thick] (-0.9, -1.5)..controls (-2.6, -0.5)..(-3, 2);

        \draw[blue, thick] (-1.5, 0)--(4, -1);
        \draw[blue, thick] (-1.5, 0)--(-4, -1);
        \draw[blue, thick] (-1.5, 0)--(-3, 2);

        \draw[green, thick] (0.9, -1.5)--(4, -1);
        \draw[green, thick] (0.9, -1.5)--(1.8, -3.5);
        \draw[green, thick] (0.9, -1.5)..controls (-1.5, 1.2)..(-3, 2);

        \filldraw[color=black, fill=white!100, thick] (0, 3.5) circle (0.7);
        \node [label=$v_1$] (v) at (0, 3.0){};

        \filldraw[color=black, fill=white!100, thick] (3, 2) circle (0.7);
        \node [label=$v_2$] (v) at (3, 1.5){};

        \filldraw[color=black, fill=white!100, thick] (4, -1) circle (0.7);
        \node [label=$v_3$] (v) at (4, -1.5){};

        \filldraw[color=black, fill=white!100, thick] (1.8, -3.5) circle (0.7);
        \node [label=$v_4$] (v) at (1.8, -4){};

        \filldraw[color=black, fill=white!100, thick] (-1.8, -3.5) circle (0.7);
        \node [label=$v_5$] (v) at (-1.8, -4){};

        \filldraw[color=black, fill=white!100, thick] (-4, -1) circle (0.7);
        \node [label=$v_6$] (v) at (-4, -1.5){};

        \filldraw[color=black, fill=white!100, thick] (-3, 2) circle (0.7);
        \node [label=$v_7$] (v) at (-3, 1.5){};

        \filldraw[color=red, fill=white!100, thick] (0, 1.2) circle (0.5);
        \node [label=$v_8$] (v) at (0, 0.7){};
        
        \filldraw[color=orange, fill=white!100, thick] (1.5, 0) circle (0.5);
        \node [label=$v_{9}$] (v) at (1.5, -0.5){};
        
        \filldraw[color=purple, fill=white!100, thick] (-0.9, -1.5) circle (0.5);
        \node [label=$v_{10}$] (v) at (-0.9, -2){};

        \filldraw[color=blue, fill=white!100, thick] (-1.5, 0) circle (0.5);
        \node [label=$v_{11}$] (v) at (-1.5, -0.5){};
        
        \filldraw[color=green, fill=white!100, thick] (0.9, -1.5) circle (0.5);
        \node [label=$v_{12}$] (v) at (0.9, -2){};
        
    \end{tikzpicture}
    \subcaption{The graph $M$.}
    \end{subfigure}
    \medskip

    \begin{subfigure}{0.45\linewidth}
        \centering
        \begin{tikzpicture}[scale=0.8]
        \tikzstyle{v}=[circle, draw, solid, fill=black, inner sep=0pt, minimum width=3pt]

        \draw[thick] (0, 3.5)--(3, 2);
        \draw[thick] (3, 2)--(4, -1);
        \draw[thick] (4, -1)--(1.8, -3.5);
        \draw[thick] (1.8, -3.5)--(-1.8, -3.5);
        \draw[thick] (-1.8, -3.5)--(-4, -1);
        \draw[thick] (-4, -1)--(-3, 2);
        \draw[thick] (-3, 2)--(0, 3.5);

        \draw[thick] (0.25, 1.2)--(1.5, -0.2);
        \draw[thick] (1.5, -0.2)--(-0.9, -1.5);

        \draw[red, thick] (0.25, 1.2)--(4, -1);
        \draw[red, thick] (0.25, 1.2)--(-4, -1);
        \draw[red, thick] (0.25, 1.2)--(3, 2);
        \draw[red, thick] (0.25, 1.2)--(-3, 2);
        \draw[red, thick] (0.25, 1.2)--(0, 3.5);

        \draw[orange, thick] (1.5, -0.2)..controls (-0.5, 0.7)..(-3, 2);
        \draw[orange, thick] (1.5, -0.2)--(0, 3.5);
        \draw[orange, thick] (1.5, -0.2)--(3, 2);
        \draw[orange, thick] (1.5, -0.2)--(4, -1);
        \draw[orange, thick] (1.5, -0.2)--(1.8, -3.5);

        \draw[purple, thick] (-0.9, -1.5)--(0, 3.5);
        \draw[purple, thick] (-0.9, -1.5)--(1.8, -3.5);
        \draw[purple, thick] (-0.9, -1.5)--(-1.8, -3.5);

        \filldraw[color=black, fill=white!100, thick] (0, 3.5) circle (0.7);
        \node [label=$v_1$] (v) at (0, 3.0){};

        \filldraw[color=black, fill=white!100, thick] (3, 2) circle (0.7);
        \node [label=$v_2$] (v) at (3, 1.5){};

        \filldraw[color=black, fill=white!100, thick] (4, -1) circle (0.7);
        \node [label=$v_3$] (v) at (4, -1.5){};

        \filldraw[color=black, fill=white!100, thick] (1.8, -3.5) circle (0.7);
        \node [label=$v_4$] (v) at (1.8, -4){};

        \filldraw[color=black, fill=white!100, thick] (-1.8, -3.5) circle (0.7);
        \node [label=$v_5$] (v) at (-1.8, -4){};

        \filldraw[color=black, fill=white!100, thick] (-4, -1) circle (0.7);
        \node [label=$v_6$] (v) at (-4, -1.5){};

        \filldraw[color=black, fill=white!100, thick] (-3, 2) circle (0.7);
        \node [label=$v_7$] (v) at (-3, 1.5){};

        \filldraw[color=red, fill=white!100, thick] (0.25, 1.2) circle (0.5);
        \node [label=$v_8$] (v) at (0.25, 0.7){};
        
        \filldraw[color=orange, fill=white!100, thick] (1.5, -0.2) circle (0.5);
        \node [label=$v_{9}$] (v) at (1.5, -0.7){};
        
        \filldraw[color=purple, fill=white!100, thick] (-0.9, -1.5) circle (0.5);
        \node [label=$v^*$] (v) at (-0.9, -2){};
        
    \end{tikzpicture}   
    \subcaption{The graph $M_1$.}
    \end{subfigure}
    \hfill
    \begin{subfigure}{0.45\linewidth}
        \centering
        \begin{tikzpicture}[scale=0.8]
        \tikzstyle{v}=[circle, draw, solid, fill=black, inner sep=0pt, minimum width=3pt]

        \draw[thick] (0, 3.5)--(3, 2);
        \draw[thick] (3, 2)--(4, -1);
        \draw[thick] (4, -1)--(1.8, -3.5);
        \draw[thick] (1.8, -3.5)--(-1.8, -3.5);
        \draw[thick] (-1.8, -3.5)--(-4, -1);
        \draw[thick] (-4, -1)--(-3, 2);
        \draw[thick] (-3, 2)--(0, 3.5);

        \draw[red, thick] (0.5, 1.2)--(4, -1);
        \draw[red, thick] (0.5, 1.2)--(-4, -1);
        \draw[red, thick] (0.5, 1.2)--(3, 2);
        \draw[red, thick] (0.5, 1.2)--(-3, 2);
        \draw[red, thick] (0.5, 1.2)--(0, 3.5);

        \draw[purple, thick] (-0.9, -1.5)--(0, 3.5);
        \draw[purple, thick] (-0.9, -1.5)--(1.8, -3.5);
        \draw[purple, thick] (-0.9, -1.5)--(-1.8, -3.5);

        \filldraw[color=black, fill=white!100, thick] (0, 3.5) circle (0.7);
        \node [label=$v_1$] (v) at (0, 3.0){};

        \filldraw[color=black, fill=white!100, thick] (3, 2) circle (0.7);
        \node [label=$v_2$] (v) at (3, 1.5){};

        \filldraw[color=black, fill=white!100, thick] (4, -1) circle (0.7);
        \node [label=$v_3$] (v) at (4, -1.5){};

        \filldraw[color=black, fill=white!100, thick] (1.8, -3.5) circle (0.7);
        \node [label=$v_4$] (v) at (1.8, -4){};

        \filldraw[color=black, fill=white!100, thick] (-1.8, -3.5) circle (0.7);
        \node [label=$v_5$] (v) at (-1.8, -4){};

        \filldraw[color=black, fill=white!100, thick] (-4, -1) circle (0.7);
        \node [label=$v_6$] (v) at (-4, -1.5){};

        \filldraw[color=black, fill=white!100, thick] (-3, 2) circle (0.7);
        \node [label=$v_7$] (v) at (-3, 1.5){};

        \filldraw[color=red, fill=white!100, thick] (0.5, 1.2) circle (0.5);
        \node [label=$v_8$] (v) at (0.5, 0.7){};
        
        \filldraw[color=purple, fill=white!100, thick] (-0.9, -1.5) circle (0.5);
        \node [label=$v^*$] (v) at (-0.9, -2){};

    \end{tikzpicture}  
    \subcaption{The graph $M_2$.}
    \end{subfigure}
    \caption{Three special graphs $M$, $M_1$ and $M_2$.}
    \label{fig:M}
\end{figure}
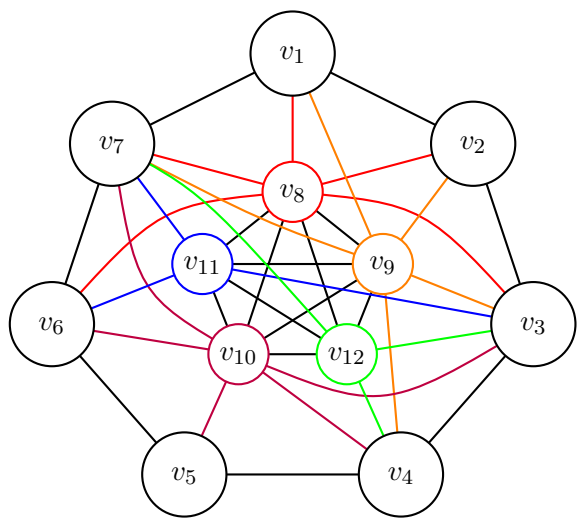
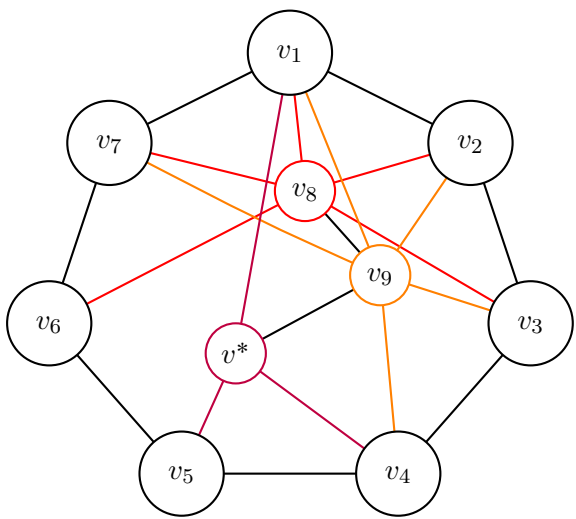
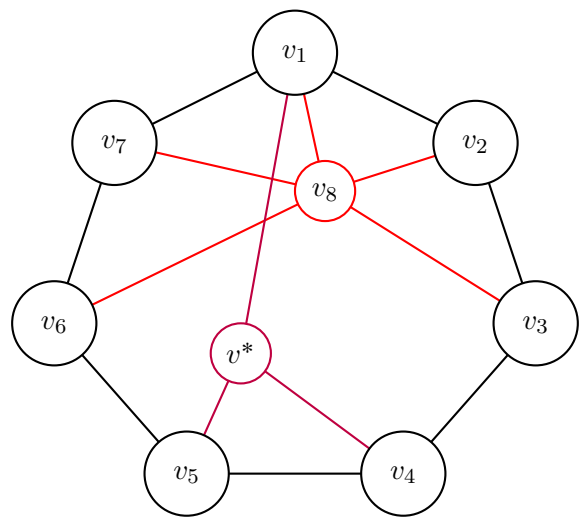

We first deal with special blowups of $M$.
Let $G=L_1\cup \cdots\cup L_{12}$ be a blowup of $M$ where $L_i$ is the clique substituting $v_i$ for $i\in [12]$ with $L_j\neq \emptyset$ for each $j\in [7]$.
By Lemma \ref{lem:hyperhole}, we may assume that $L_j\neq \emptyset$ for some $j\in \{8,9,10,11,12\}$.

\begin{lemma}\label{lem:L_i is large}
    For each $i\in [7]$, $|L_i|>\ceil{\frac{\omega}{4}}$ and  $|L_i|\ge 3$.
\end{lemma}

\begin{proof}
    By symmetry, it suffices to prove the lemma for $i=2,3,4,5$. 
\begin{itemize}
    \item[$(L_2)$]  Let $v\in L_1$. Then $N(v)\setminus L_2\subseteq L_1 \cup L_7\cup L_8\cup L_9$ is a clique. Since $G$ has no small vertices, $|L_2|\ge |N(v)\cap L_2|>\ceil{\frac{\omega}{4}}$. 
    \item[$(L_3)$] Let $v\in L_2$. Then $N(v)\setminus L_3\subseteq L_1 \cup L_8\cup L_9$ is a clique. Since $G$ has no small vertices, $|L_3|\ge |N(v)\cap L_3|>\ceil{\frac{\omega}{4}}$. 
    \item[$(L_4)$]  Let $v\in L_5$. Then $N(v)\setminus L_4\subseteq L_6 \cup L_{10}$ is a clique. Since $G$ has no small vertices, $|L_4|\ge |N(v)\cap L_4|>\ceil{\frac{\omega}{4}}$.
    \item[$(L_5)$]  Let $v\in L_4$. Then $N(v)\setminus L_5\subseteq L_3 \cup L_9\cup L_{10}\cup L_{12}$ is a clique. Since $G$ has no small vertices, $|L_5|\ge |N(v)\cap L_5|>\ceil{\frac{\omega}{4}}$
\end{itemize}
Since $\omega\ge 2$, $|L_i|\ge 2$ for $i\in [7]$. Since $L_j\neq \emptyset$ for some $j\in \{8,9,10,11,12\}$, $\omega\ge 5$ and so $|L_i|\ge 3$ for $i\in [7]$.   
\end{proof}

We show that $G$ is not basic by finding a $(4,5)$-good subgraph, which contradicts our assumption on $G$. By Lemma \ref{lem:L_i is large}, one can choose three vertices from each $L_i$ for $i\in [7]$ when constructing $(4,5)$-good subgraphs. We first give two lemmas (Lemmas \ref{lem:general 1} and \ref{lem:general 2}) and then consider all possible cases depending on the number of non-empty $L_j$s for $j\in \{8,9,10\}$ (Lemmas \ref{lem:L_8 L_9 L_{10} are not empty}-\ref{lem:only one S_Y}).

\begin{lemma}\label{lem:general 1}
    If $L_{10}=L_{11}=L_{12}=\emptyset$, then $G$ has a $(4,5)$-good subgraph.
\end{lemma}

\begin{proof}
    Suppose not. Choose two vertices in $L_i$ for each $i\in [7]$ and let $H$ be the union of all those chosen vertices. By Lemma \ref{lem:hyperhole}, $\chi(H)\le 5$. 
    It is easy to see that every maximal clique of $G$ contains 4 vertices of $H$. Therefore, $H$ is a $(4,5)$-good subgraph. This contradicts that $G$ is basic. 
\end{proof}

\begin{lemma}\label{lem:general 2}
    If $L_{10}=\emptyset$ and $L_{11},L_{12}\neq \emptyset$, $G$ has a $(4,5)$-good subgraph.    
\end{lemma}

\begin{proof}
    If not, we give a $(4,5)$-good subgraph with a 5-coloring as follows (see Figure \ref{fig:1}).

    \begin{figure}[htbp]
        \centering
        \begin{tikzpicture}[scale=0.8]
        \tikzstyle{v}=[circle, draw, solid, fill=black, inner sep=0pt, minimum width=3pt]

        \draw[thick] (0, 3.5)--(3, 2);
        \draw[thick] (3, 2)--(4, -1);
        \draw[thick] (4, -1)--(1.8, -3.5);
        \draw[thick] (1.8, -3.5)--(-1.8, -3.5);
        \draw[thick] (-1.8, -3.5)--(-4, -1);
        \draw[thick] (-4, -1)--(-3, 2);
        \draw[thick] (-3, 2)--(0, 3.5);

        \draw[thick] (0, 1.2)--(1.5, 0);
        \draw[thick] (0, 1.2)--(-1.5, 0);
        \draw[thick] (0, 1.2)--(0.9, -1.5);
        \draw[thick] (1.5, 0)--(-1.5, 0);
        \draw[thick] (1.5, 0)--(0.9, -1.5);
        \draw[thick] (-1.5, 0)--(0.9, -1.5);

        \draw[red, thick] (0, 1.2)..controls (2.2, 1)..(4, -1);
        \draw[red, thick] (0, 1.2)..controls (-2.2, 1)..(-4, -1);
        \draw[red, thick] (0, 1.2)--(3, 2);
        \draw[red, thick] (0, 1.2)--(-3, 2);
        \draw[red, thick] (0, 1.2)--(0, 3.5);

        \draw[orange, thick] (1.5, 0)..controls (-0.5, 0.7)..(-3, 2);
        \draw[orange, thick] (1.5, 0)--(0, 3.5);
        \draw[orange, thick] (1.5, 0)--(3, 2);
        \draw[orange, thick] (1.5, 0)--(4, -1);
        \draw[orange, thick] (1.5, 0)--(1.8, -3.5);

        \draw[blue, thick] (-1.5, 0)--(4, -1);
        \draw[blue, thick] (-1.5, 0)--(-4, -1);
        \draw[blue, thick] (-1.5, 0)--(-3, 2);

        \draw[green, thick] (0.9, -1.5)--(4, -1);
        \draw[green, thick] (0.9, -1.5)--(1.8, -3.5);
        \draw[green, thick] (0.9, -1.5)..controls (-1.5, 1.2)..(-3, 2);

        \filldraw[color=black, fill=white!100, thick] (0, 3.5) circle (0.7);
        \node [label={$1, 2$}] (v) at (0, 2.95){};
        \node [label=$L_1$] (v) at (0, 4){};

        \filldraw[color=black, fill=white!100, thick] (3, 2) circle (0.7);
        \node [label={$3, 4$}] (v) at (3, 1.45){};
        \node [label=$L_2$] (v) at (3, 2.5){};

        \filldraw[color=black, fill=white!100, thick] (4, -1) circle (0.7);
        \node [label={$1, 5$}] (v) at (4, -1.55){};
        \node [label=$L_3$] (v) at (4.2, -0.5){};

        \filldraw[color=black, fill=white!100, thick] (1.8, -3.5) circle (0.7);
        \node [label={$3$}] (v) at (1.8, -4){};
        \node [label=$L_4$] (v) at (2.9, -4){};

        \filldraw[color=black, fill=white!100, thick] (-1.8, -3.5) circle (0.7);
        \node [label={$2, 4, 5$}] (v) at (-1.8, -4.05){};
        \node [label=$L_5$] (v) at (-2.9, -4){};

        \filldraw[color=black, fill=white!100, thick] (-4, -1) circle (0.7);
        \node [label={$1$}] (v) at (-4, -1.5){};
        \node [label=$L_6$] (v) at (-4.2, -0.5){};

        \filldraw[color=black, fill=white!100, thick] (-3, 2) circle (0.7);
        \node [label={$3, 5$}] (v) at (-3, 1.45){};
        \node [label=$L_7$] (v) at (-3, 2.5){};

        \filldraw[color=red, fill=white!100, thick] (0, 1.2) circle (0.5);
        \node [label=$L_8$] (v) at (-0.5, 1.4){};
        
        \filldraw[color=orange, fill=white!100, thick] (1.5, 0) circle (0.5);
        \node [label=$L_{9}$] (v) at (2.35, -0.3){};

        \filldraw[color=blue, fill=white!100, thick] (-1.5, 0) circle (0.5);
        \node [label={$4$}] (v) at (-1.5, -0.5){};
        \node [label=$L_{11}$] (v) at (-1.5, -1.4){};
        
        \filldraw[color=green, fill=white!100, thick] (0.9, -1.5) circle (0.5);
        \node [label={$2$}] (v) at (0.9, -2){};
        \node [label=$L_{12}$] (v) at (0.5, -2.7){};
        
    \end{tikzpicture}
        \caption{A $(4,5)$-good subgraph in Lemma \ref{lem:general 2}.}
        \label{fig:1}
    \end{figure}
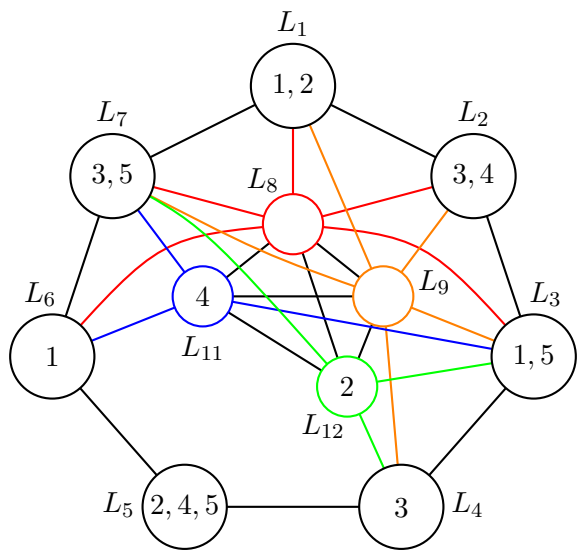
    
    Choose two vertices in $L_1$ and color them with colors 1 and 2, two vertices in $L_2$ and color them with colors 3 and 4, two vertices in $L_3$ and color them with colors 1 and 5, one vertex in $L_4$ and color them with color 3, three vertices in $L_5$ and color them with colors 2, 4 and 5, one vertex in $L_6$ and color them with color 1, two vertices in $L_7$ and color them with colors 3 and 5, one vertex in $L_{11}$ and color it with color 4, one vertex in $L_{12}$ and color it with color 2.
\end{proof}

\begin{lemma}\label{lem:L_8 L_9 L_{10} are not empty}
    If $L_8,L_9,L_{10}\neq \emptyset$, then $G$ has a $(4,5)$-good subgraph.
\end{lemma}

\begin{proof}
    We give a $(4,5)$-good subgraph with a 5-coloring as follows (see Figure \ref{fig:extremesubgraph}).
    
    \begin{figure}[htbp]
        \centering
        \begin{tikzpicture}[scale=0.8]
        \tikzstyle{v}=[circle, draw, solid, fill=black, inner sep=0pt, minimum width=3pt]

        \draw[thick] (0, 3.5)--(3, 2);
        \draw[thick] (3, 2)--(4, -1);
        \draw[thick] (4, -1)--(1.8, -3.5);
        \draw[thick] (1.8, -3.5)--(-1.8, -3.5);
        \draw[thick] (-1.8, -3.5)--(-4, -1);
        \draw[thick] (-4, -1)--(-3, 2);
        \draw[thick] (-3, 2)--(0, 3.5);

        \draw[thick] (0, 1.2)--(1.5, 0);
        \draw[thick] (0, 1.2)--(-1.5, 0);
        \draw[thick] (0, 1.2)--(-0.9, -1.5);
        \draw[thick] (0, 1.2)--(0.9, -1.5);
        \draw[thick] (1.5, 0)--(-1.5, 0);
        \draw[thick] (1.5, 0)--(0.9, -1.5);
        \draw[thick] (1.5, 0)--(-0.9, -1.5);
        \draw[thick] (-1.5, 0)--(0.9, -1.5);
        \draw[thick] (-1.5, 0)--(-0.9, -1.5);
        \draw[thick] (0.9, -1.5)--(-0.9, -1.5);

        \draw[red, thick] (0, 1.2)..controls (2.2, 1)..(4, -1);
        \draw[red, thick] (0, 1.2)..controls (-2.2, 1)..(-4, -1);
        \draw[red, thick] (0, 1.2)--(3, 2);
        \draw[red, thick] (0, 1.2)--(-3, 2);
        \draw[red, thick] (0, 1.2)--(0, 3.5);

        \draw[orange, thick] (1.5, 0)..controls (-0.5, 0.7)..(-3, 2);
        \draw[orange, thick] (1.5, 0)--(0, 3.5);
        \draw[orange, thick] (1.5, 0)--(3, 2);
        \draw[orange, thick] (1.5, 0)--(4, -1);
        \draw[orange, thick] (1.5, 0)..controls (2.5, -1.5)..(1.8, -3.5);

        \draw[purple, thick] (-0.9, -1.5)..controls (1.5, -2.5)..(4, -1);
        \draw[purple, thick] (-0.9, -1.5)--(1.8, -3.5);
        \draw[purple, thick] (-0.9, -1.5)--(-1.8, -3.5);
        \draw[purple, thick] (-0.9, -1.5)--(-4, -1);
        \draw[purple, thick] (-0.9, -1.5)..controls (-2.7, -0.6)..(-3, 2);

        \draw[blue, thick] (-1.5, 0)--(4, -1);
        \draw[blue, thick] (-1.5, 0)--(-4, -1);
        \draw[blue, thick] (-1.5, 0)--(-3, 2);

        \draw[green, thick] (0.9, -1.5)--(4, -1);
        \draw[green, thick] (0.9, -1.5)--(1.8, -3.5);
        \draw[green, thick] (0.9, -1.5)..controls (-1.5, 1.2)..(-3, 2);

        \filldraw[color=black, fill=white!100, thick] (0, 3.5) circle (0.7);
        \node [label={$1$}] (v) at (0, 3){};
        \node [label=$L_1$] (v) at (0, 4){};

        \filldraw[color=black, fill=white!100, thick] (3, 2) circle (0.7);
        \node [label={$2$}] (v) at (3, 1.5){};
        \node [label=$L_2$] (v) at (3, 2.5){};

        \filldraw[color=black, fill=white!100, thick] (4, -1) circle (0.7);
        \node [label={$3$}] (v) at (4, -1.5){};
        \node [label=$L_3$] (v) at (4.2, -0.5){};

        \filldraw[color=black, fill=white!100, thick] (1.8, -3.5) circle (0.7);
        \node [label={$4$}] (v) at (1.8, -4){};
        \node [label=$L_4$] (v) at (2.9, -4){};

        \filldraw[color=black, fill=white!100, thick] (-1.8, -3.5) circle (0.7);
        \node [label={$2, 3$}] (v) at (-1.8, -4.05){};
        \node [label=$L_5$] (v) at (-2.9, -4){};

        \filldraw[color=black, fill=white!100, thick] (-4, -1) circle (0.7);
        \node [label={$5$}] (v) at (-4, -1.5){};
        \node [label=$L_6$] (v) at (-4.2, -0.5){};

        \filldraw[color=black, fill=white!100, thick] (-3, 2) circle (0.7);
        \node [label={$2$}] (v) at (-3, 1.5){};
        \node [label=$L_7$] (v) at (-3, 2.5){};

        \filldraw[color=red, fill=white!100, thick] (0, 1.2) circle (0.5);
        \node [label=$4$] (v) at (0, 0.7){};
        \node [label=$L_8$] (v) at (-0.5, 1.4){};
        
        \filldraw[color=orange, fill=white!100, thick] (1.5, 0) circle (0.5);
        \node [label=$5$] (v) at (1.5, -0.5){};
        \node [label=$L_{9}$] (v) at (2.35, -0.3){};

        \filldraw[color=purple, fill=white!100, thick] (-0.9, -1.5) circle (0.5);
        \node [label=$1$] (v) at (-0.9, -2){};
        \node [label=$L_{10}$] (v) at (-1.7, -2.4){};

        \filldraw[color=blue, fill=white!100, thick] (-1.5, 0) circle (0.5);
        \node [label={}] (v) at (-1.5, -0.5){};
        \node [label=$L_{11}$] (v) at (-1.7, -1.2){};
        
        \filldraw[color=green, fill=white!100, thick] (0.9, -1.5) circle (0.5);
        \node [label={}] (v) at (0.9, -2){};
        \node [label=$L_{12}$] (v) at (1.8, -2.3){};
        
    \end{tikzpicture}
        \caption{A $(4,5)$-good subgraph in Lemma \ref{lem:L_8 L_9 L_{10} are not empty}.}
        \label{fig:extremesubgraph}
    \end{figure}
    
    Choose 
    one vertex in $L_1$ and color it with color 1, 
    one vertex in $L_2$ and color it with color 2, 
    one vertex in $L_3$ and color it with color 3, 
    one vertex in $L_4$ and color it with color 4, 
    two vertices in $L_5$ and color them with colors 2 and 3, 
    one vertex in $L_6$ and color it with color 5,
    one vertex in $L_7$ and color it with color 2,
    one vertex in $L_8$ and color it with color 4,
    one vertex in $L_9$ and color it with color 5,
    one vertex in $L_{10}$ and color it with color 1. 
\end{proof}

\begin{lemma}\label{lem:exactly two S_5}
    If exactly two of $L_8$, $L_9$ and $L_{10}$ are not empty, then $G$ has a $(4,5)$-good subgraph.
\end{lemma}

\begin{proof}
    Suppose first that $L_8,L_{10}\neq \emptyset$ (the case $L_9,L_{10}\neq \emptyset$ is symmetric). We give a $(4,5)$-good subgraph with a 5-coloring as follows (see Figure \ref{fig:8 and 10}).
    Choose 
    two vertices in $L_1$ and color them with colors 1 and 2, 
    one vertex in $L_2$ and color it with color 3, 
    two vertices in $L_3$ and color them with colors 2 and 5,  
    one vertex in $L_4$ and color it with color 4, 
    two vertices in $L_5$ and color them with colors 3 and 5, 
    one vertex in $L_6$ and color it with color 2,
    two vertices in $L_7$ and color them with colors 3 and 5, 
    one vertex in $L_8$ and color it with color 4,
    one vertex in $L_{10}$ and color it with color 1. 

    Suppose now that $L_8,L_{9}\neq \emptyset$ and so $L_{10}=\emptyset$. By Lemma \ref{lem:general 1}, we may assume by symmetry that $L_{12}\neq \emptyset$. 
    We give a $(4,5)$-good subgraph with a 5-coloring as follows (see Figure \ref{fig:8 9 and 12}.
    Choose 
    one vertex in $L_1$ and color it with color 1, 
    one vertex in $L_2$ and color it with color 2, 
    one vertex in $L_3$ and color it with color 3, 
    one vertex in $L_4$ and color it with color 4, 
    three vertices in $L_5$ and color them with colors 1, 2 and 3, 
    one vertex in $L_6$ and color it with color 5,
    two vertices in $L_7$ and color them with colors 2 and 3,
    one vertex in $L_8$ and color it with color 4,
    one vertex in $L_9$ and color it with color 5,
    one vertex in $L_{12}$ and color it with color 1. 
\end{proof}

\begin{lemma}\label{lem:exactly one S_5}
    If exactly one of $L_8$, $L_9$ and $L_{10}$ is not empty, then $G$ has a $(4,5)$-good subgraph.
\end{lemma}

\begin{proof}
    Suppose first that $L_9\neq \emptyset$ (the case that $L_8\neq \emptyset$ is symmetric). By Lemmas \ref{lem:general 1} and \ref{lem:general 2}, exactly one of $L_{11}$ and $L_{12}$ is not empty. 
    If $L_{11}\neq \emptyset$, 
    we give a $(4,5)$-good subgraph with a 5-coloring as follows (see Figure \ref{fig:9 and 11}).
    Choose 
    two vertices in $L_1$ and color them with colors 1 and 2, 
    one vertex in $L_2$ and color it with color 3, 
    two vertices in $L_3$ and color them with colors 2 and 5,
    one vertex in $L_4$ and color it with color 3, 
    three vertices in $L_5$ and color them with colors 1, 2 and 5, 
    one vertex in $L_6$ and color it with color 4,
    two vertices in $L_7$ and color them with colors 3 and 5,
    one vertex in $L_9$ and color it with color 4,
    one vertex in $L_{11}$ and color it with color 1.

    If $L_{12}\neq \emptyset$, 
    we give a $(4,5)$-good subgraph with a 5-coloring as follows (see Figure \ref{fig:9 and 12}).
    Choose 
    one vertex in $L_1$ and color it with color 1,
    two vertices in $L_2$ and color them with colors 2 and 3,
    one vertex in $L_3$ and color it with color 5,
    one vertex in $L_4$ and color it with color 2, 
    three vertices in $L_5$ and color them with colors 1, 3 and 5, 
    two vertices in $L_6$ and color them with colors 2 and 4,
    two vertices in $L_7$ and color them with colors 3 and 5,
    one vertex in $L_9$ and color it with color 4,
    one vertex in $L_{12}$ and color it with color 1.

    Suppose now that $L_{10}\neq \emptyset$. By Lemma \ref{lem:general 1} and by symmetry, $L_{11}\neq \emptyset$ or $L_{12}\neq \emptyset$. By symmetry, we assume that $L_{11}\neq \emptyset$.
    We give a $(4,5)$-good subgraph with a 5-coloring as follows (see Figure \ref{fig:10 and 11}). 
    Choose 
    two vertices in $L_1$ and color them with colors 1 and 2,
    two vertices in $L_2$ and color them with colors 3 and 5,
    two vertices in $L_3$ and color them with colors 1 and 5,
    one vertex in $L_4$ and color it with color 4, 
    three vertices in $L_5$ and color them with colors 1, 2 and 5, 
    two vertices in $L_7$ and color them with colors 4 and 5,
    one vertex in $L_{10}$ and color it with color 3,
    one vertex in $L_{11}$ and color it with color 2.
\end{proof}

\begin{lemma}\label{lem:only one S_Y}
    If $L_8=L_9=L_{10}=\emptyset$, then $G$ has a $(4,5)$-good subgraph.
\end{lemma}

\begin{proof}
    If $L_{11}$ and $L_{12}$ are not empty, we are done by Lemma \ref{lem:general 2}. By symmetry, we assume that $L_{12}\neq \emptyset$ and $L_{11}=\emptyset$. 
    We give a $(4,5)$-good subgraph with a 5-coloring as follows (see Figure \ref{fig:one S_Y}).
    Choose 
    one vertex in $L_1$ and color it with color 4, 
    three vertices in $L_2$ and color them with colors 1, 2 and 5, 
    one vertex in $L_3$ and color it with color 3, 
    two vertices in $L_4$ and color them with colors 2 and 4,
    three vertices in $L_5$ and color them with colors 1, 3 and 5, 
    one vertex in $L_6$ and color it with color 4,
    three vertices in $L_7$ and color them with colors 1, 2 and 3,
    one vertex in $L_{12}$ and color it with color 5.
\end{proof}

This completes the proof for special blowups of $M$. Non-empty blowups of $M_1$ and $M_2$ can be handled in a similar way. We leave the routine work to the readers. 

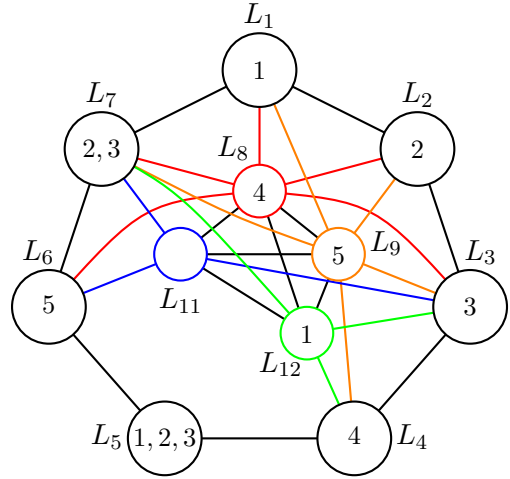
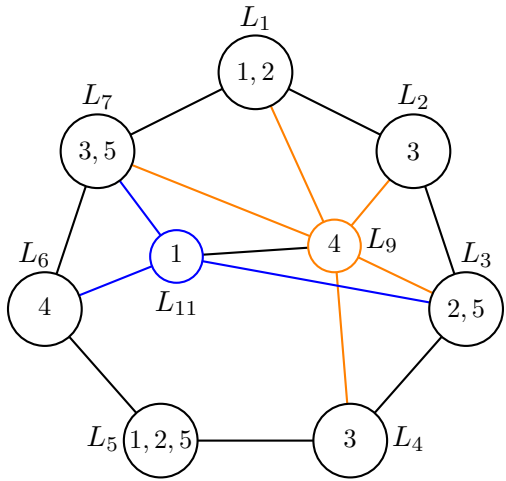
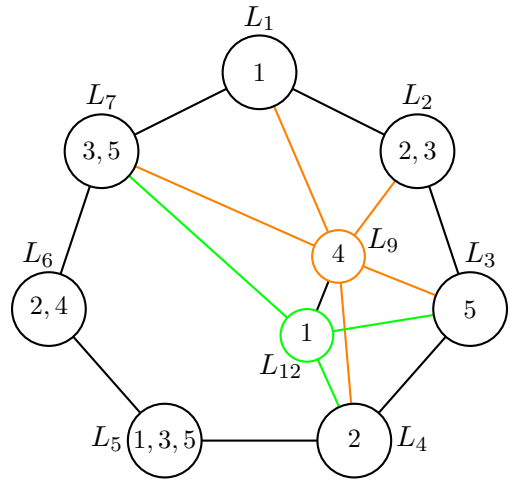
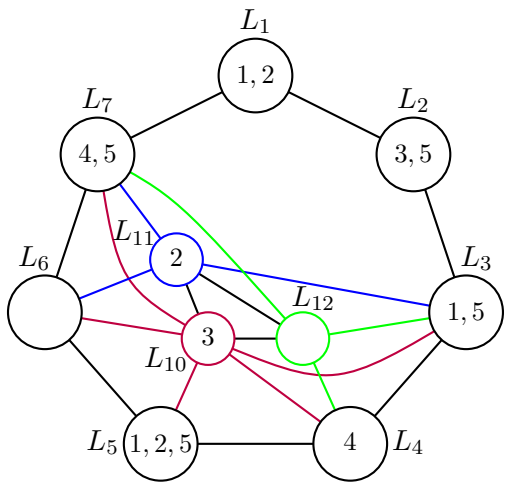
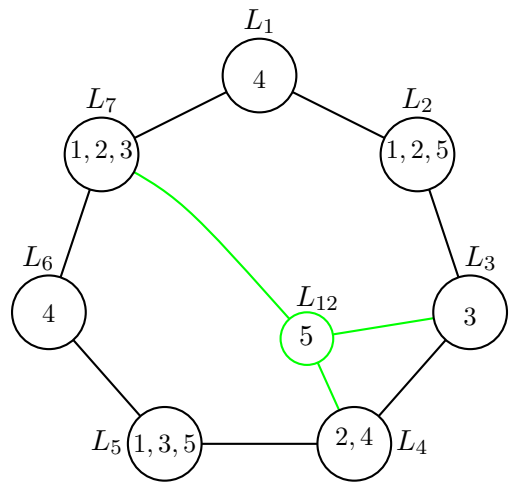
\begin{figure}
    \centering
    \begin{subfigure}{0.45\linewidth}
        \centering
        \begin{tikzpicture}[scale=0.7]
        \tikzstyle{v}=[circle, draw, solid, fill=black, inner sep=0pt, minimum width=3pt]

        \draw[thick] (0, 3.5)--(3, 2);
        \draw[thick] (3, 2)--(4, -1);
        \draw[thick] (4, -1)--(1.8, -3.5);
        \draw[thick] (1.8, -3.5)--(-1.8, -3.5);
        \draw[thick] (-1.8, -3.5)--(-4, -1);
        \draw[thick] (-4, -1)--(-3, 2);
        \draw[thick] (-3, 2)--(0, 3.5);

        \draw[thick] (0, 1.2)--(-1.5, 0);
        \draw[thick] (0, 1.2)--(-0.9, -1.5);
        \draw[thick] (0, 1.2)--(0.9, -1.5);
        \draw[thick] (-1.5, 0)--(0.9, -1.5);
        \draw[thick] (-1.5, 0)--(-0.9, -1.5);
        \draw[thick] (0.9, -1.5)--(-0.9, -1.5);

        \draw[red, thick] (0, 1.2)--(4, -1);
        \draw[red, thick] (0, 1.2)..controls (-2.6, 1.4)..(-4, -1);
        \draw[red, thick] (0, 1.2)--(3, 2);
        \draw[red, thick] (0, 1.2)--(-3, 2);
        \draw[red, thick] (0, 1.2)--(0, 3.5);

        \draw[purple, thick] (-0.9, -1.5)..controls (1.5, -2.5)..(4, -1);
        \draw[purple, thick] (-0.9, -1.5)--(1.8, -3.5);
        \draw[purple, thick] (-0.9, -1.5)--(-1.8, -3.5);
        \draw[purple, thick] (-0.9, -1.5)--(-4, -1);
        \draw[purple, thick] (-0.9, -1.5)..controls (-2.6, -0.5)..(-3, 2);

        \draw[blue, thick] (-1.5, 0)--(4, -1);
        \draw[blue, thick] (-1.5, 0)--(-4, -1);
        \draw[blue, thick] (-1.5, 0)--(-3, 2);

        \draw[green, thick] (0.9, -1.5)--(4, -1);
        \draw[green, thick] (0.9, -1.5)--(1.8, -3.5);
        \draw[green, thick] (0.9, -1.5)..controls (-1.5, 1.2)..(-3, 2);

        \filldraw[color=black, fill=white!100, thick] (0, 3.5) circle (0.7);
        \node [label={\small $1, 2$}] (v) at (0, 2.95 - 0.08){};
        \node [label=$L_1$] (v) at (0, 4 - 0.08){};

        \filldraw[color=black, fill=white!100, thick] (3, 2) circle (0.7);
        \node [label={\small $3$}] (v) at (3, 1.45){};
        \node [label=$L_2$] (v) at (3, 2.5 - 0.08){};

        \filldraw[color=black, fill=white!100, thick] (4, -1) circle (0.7);
        \node [label={\small $2, 5$}] (v) at (4, -1.55 - 0.08){};
        \node [label=$L_3$] (v) at (4.2, -0.5 - 0.08){};

        \filldraw[color=black, fill=white!100, thick] (1.8, -3.5) circle (0.7);
        \node [label={\small $4$}] (v) at (1.8, -4){};
        \node [label=$L_4$] (v) at (2.9, -4 - 0.08){};

        \filldraw[color=black, fill=white!100, thick] (-1.8, -3.5) circle (0.7);
        \node [label={\small $3, 5$}] (v) at (-1.8, -4.05 - 0.08){};
        \node [label=$L_5$] (v) at (-2.9, -4 - 0.08){};

        \filldraw[color=black, fill=white!100, thick] (-4, -1) circle (0.7);
        \node [label={\small $2$}] (v) at (-4, -1.5){};
        \node [label=$L_6$] (v) at (-4.2, -0.5 - 0.08){};

        \filldraw[color=black, fill=white!100, thick] (-3, 2) circle (0.7);
        \node [label={\small $3, 5$}] (v) at (-3, 1.45 - 0.08){};
        \node [label=$L_7$] (v) at (-3, 2.5 - 0.08){};

        \filldraw[color=red, fill=white!100, thick] (0, 1.2) circle (0.5);
        \node [label={\small $4$}] (v) at (0, 0.7 - 0.08){};
        \node [label=$L_8$] (v) at (-0.5, 1.4){};
        
        \filldraw[color=purple, fill=white!100, thick] (-0.9, -1.5) circle (0.5);
        \node [label={\small $1$}] (v) at (-0.9, -2 - 0.08){};
        \node [label=$L_{10}$] (v) at (-1.7, -2.5){};

        \filldraw[color=blue, fill=white!100, thick] (-1.5, 0) circle (0.5);
        \node [label={$L_{11}$}] (v) at (-2.3, -0.1){};
        
        \filldraw[color=green, fill=white!100, thick] (0.9, -1.5) circle (0.5);
        \node [label=$L_{12}$] (v) at (1.1, -1.35){};
        
    \end{tikzpicture}
    \subcaption{A $(4,5)$-good subgraph in  Lemma \ref{lem:exactly two S_5}.}
    \label{fig:8 and 10}
    \end{subfigure}
    \hfill
    \begin{subfigure}{0.45\linewidth}
        \centering
        \begin{tikzpicture}[scale=0.7]
        \tikzstyle{v}=[circle, draw, solid, fill=black, inner sep=0pt, minimum width=3pt]

        \draw[thick] (0, 3.5)--(3, 2);
        \draw[thick] (3, 2)--(4, -1);
        \draw[thick] (4, -1)--(1.8, -3.5);
        \draw[thick] (1.8, -3.5)--(-1.8, -3.5);
        \draw[thick] (-1.8, -3.5)--(-4, -1);
        \draw[thick] (-4, -1)--(-3, 2);
        \draw[thick] (-3, 2)--(0, 3.5);

        \draw[thick] (0, 1.2)--(1.5, 0);
        \draw[thick] (0, 1.2)--(-1.5, 0);
        \draw[thick] (0, 1.2)--(0.9, -1.5);
        \draw[thick] (1.5, 0)--(-1.5, 0);
        \draw[thick] (1.5, 0)--(0.9, -1.5);
        \draw[thick] (-1.5, 0)--(0.9, -1.5);

        \draw[red, thick] (0, 1.2)..controls (2.2, 1)..(4, -1);
        \draw[red, thick] (0, 1.2)..controls (-2.2, 1)..(-4, -1);
        \draw[red, thick] (0, 1.2)--(3, 2);
        \draw[red, thick] (0, 1.2)--(-3, 2);
        \draw[red, thick] (0, 1.2)--(0, 3.5);

        \draw[orange, thick] (1.5, 0)..controls (-0.5, 0.7)..(-3, 2);
        \draw[orange, thick] (1.5, 0)--(0, 3.5);
        \draw[orange, thick] (1.5, 0)--(3, 2);
        \draw[orange, thick] (1.5, 0)--(4, -1);
        \draw[orange, thick] (1.5, 0)--(1.8, -3.5);

        \draw[blue, thick] (-1.5, 0)--(4, -1);
        \draw[blue, thick] (-1.5, 0)--(-4, -1);
        \draw[blue, thick] (-1.5, 0)--(-3, 2);

        \draw[green, thick] (0.9, -1.5)--(4, -1);
        \draw[green, thick] (0.9, -1.5)--(1.8, -3.5);
        \draw[green, thick] (0.9, -1.5)..controls (-1.5, 1.2)..(-3, 2);

        \filldraw[color=black, fill=white!100, thick] (0, 3.5) circle (0.7);
        \node [label={\small $1$}] (v) at (0, 2.95){};
        \node [label=$L_1$] (v) at (0, 4 - 0.08){};

        \filldraw[color=black, fill=white!100, thick] (3, 2) circle (0.7);
        \node [label={\small $2$}] (v) at (3, 1.45){};
        \node [label=$L_2$] (v) at (3, 2.5 - 0.08){};

        \filldraw[color=black, fill=white!100, thick] (4, -1) circle (0.7);
        \node [label={\small $3$}] (v) at (4, -1.55){};
        \node [label=$L_3$] (v) at (4.2, -0.5 - 0.08){};

        \filldraw[color=black, fill=white!100, thick] (1.8, -3.5) circle (0.7);
        \node [label={\small $4$}] (v) at (1.8, -4){};
        \node [label=$L_4$] (v) at (2.9, -4 - 0.08){};

        \filldraw[color=black, fill=white!100, thick] (-1.8, -3.5) circle (0.7);
        \node [label={\small $1, 2, 3$}] (v) at (-1.8, -4.05 - 0.08){};
        \node [label=$L_5$] (v) at (-2.9, -4 - 0.08){};

        \filldraw[color=black, fill=white!100, thick] (-4, -1) circle (0.7);
        \node [label={\small $5$}] (v) at (-4, -1.5){};
        \node [label=$L_6$] (v) at (-4.2, -0.5 - 0.08){};

        \filldraw[color=black, fill=white!100, thick] (-3, 2) circle (0.7);
        \node [label={\small $2, 3$}] (v) at (-3, 1.45 - 0.08){};
        \node [label=$L_7$] (v) at (-3, 2.5 - 0.08){};

        \filldraw[color=red, fill=white!100, thick] (0, 1.2) circle (0.5);
        \node [label={\small $4$}] (v) at (0, 0.7 - 0.08){};
        \node [label=$L_8$] (v) at (-0.5, 1.4){};
        
        \filldraw[color=orange, fill=white!100, thick] (1.5, 0) circle (0.5);
        \node [label={\small $5$}] (v) at (1.5, -0.5 - 0.08){};
        \node [label=$L_{9}$] (v) at (2.4, -0.4){};
        
        \filldraw[color=blue, fill=white!100, thick] (-1.5, 0) circle (0.5);
        \node [label=$L_{11}$] (v) at (-1.5, -1.5){};
        
        \filldraw[color=green, fill=white!100, thick] (0.9, -1.5) circle (0.5);
        \node [label={\small $1$}] (v) at (0.9, -2 - 0.08){};
        \node [label=$L_{12}$] (v) at (0.4, -2.7){};
        
    \end{tikzpicture}
    \subcaption{A $(4,5)$-good subgraph in Lemma \ref{lem:exactly two S_5}.}
    \label{fig:8 9 and 12}
    \end{subfigure}

    \medskip

    \begin{subfigure}{0.45\linewidth}
        \centering
        \begin{tikzpicture}[scale=0.7]
        \tikzstyle{v}=[circle, draw, solid, fill=black, inner sep=0pt, minimum width=3pt]

        \draw[thick] (0, 3.5)--(3, 2);
        \draw[thick] (3, 2)--(4, -1);
        \draw[thick] (4, -1)--(1.8, -3.5);
        \draw[thick] (1.8, -3.5)--(-1.8, -3.5);
        \draw[thick] (-1.8, -3.5)--(-4, -1);
        \draw[thick] (-4, -1)--(-3, 2);
        \draw[thick] (-3, 2)--(0, 3.5);

        \draw[thick] (1.5, 0.2)--(-1.5, 0);

        \draw[orange, thick] (1.5, 0.2)--(-3, 2);
        \draw[orange, thick] (1.5, 0.2)--(0, 3.5);
        \draw[orange, thick] (1.5, 0.2)--(3, 2);
        \draw[orange, thick] (1.5, 0.2)--(4, -1);
        \draw[orange, thick] (1.5, 0.2)--(1.8, -3.5);

        \draw[blue, thick] (-1.5, 0)--(4, -1);
        \draw[blue, thick] (-1.5, 0)--(-4, -1);
        \draw[blue, thick] (-1.5, 0)--(-3, 2);

        \filldraw[color=black, fill=white!100, thick] (0, 3.5) circle (0.7);
        \node [label={\small $1, 2$}] (v) at (0, 2.95 - 0.08){};
        \node [label=$L_1$] (v) at (0, 4 - 0.08){};

        \filldraw[color=black, fill=white!100, thick] (3, 2) circle (0.7);
        \node [label={\small $3$}] (v) at (3, 1.45){};
        \node [label=$L_2$] (v) at (3, 2.5 - 0.08){};

        \filldraw[color=black, fill=white!100, thick] (4, -1) circle (0.7);
        \node [label={\small $2, 5$}] (v) at (4, -1.55 - 0.08){};
        \node [label=$L_3$] (v) at (4.2, -0.5 - 0.08){};

        \filldraw[color=black, fill=white!100, thick] (1.8, -3.5) circle (0.7);
        \node [label={\small $3$}] (v) at (1.8, -4){};
        \node [label=$L_4$] (v) at (2.9, -4 - 0.08){};

        \filldraw[color=black, fill=white!100, thick] (-1.8, -3.5) circle (0.7);
        \node [label={\small $1, 2, 5$}] (v) at (-1.8, -4.05 - 0.08){};
        \node [label=$L_5$] (v) at (-2.9, -4 - 0.08){};

        \filldraw[color=black, fill=white!100, thick] (-4, -1) circle (0.7);
        \node [label={\small $4$}] (v) at (-4, -1.5){};
        \node [label=$L_6$] (v) at (-4.2, -0.5 - 0.08){};

        \filldraw[color=black, fill=white!100, thick] (-3, 2) circle (0.7);
        \node [label={\small $3, 5$}] (v) at (-3, 1.45 - 0.08){};
        \node [label=$L_7$] (v) at (-3, 2.5 - 0.08){};

        \filldraw[color=orange, fill=white!100, thick] (1.5, 0.2) circle (0.5);
        \node [label={\small $4$}] (v) at (1.5, -0.3){};
        \node [label=$L_{9}$] (v) at (2.4, -0.3){};
        
        \filldraw[color=blue, fill=white!100, thick] (-1.5, 0) circle (0.5);
        \node [label={\small $1$}] (v) at (-1.5, -0.5){};
        \node [label=$L_{11}$] (v) at (-1.5, -1.5){};
                
    \end{tikzpicture}
    \subcaption{A $(4,5)$-good subgraph in Lemma \ref{lem:exactly one S_5}.}
    \label{fig:9 and 11}
    \end{subfigure}
    \hfill
    \begin{subfigure}{0.45\linewidth}
        \centering
        \begin{tikzpicture}[scale=0.7]
        \tikzstyle{v}=[circle, draw, solid, fill=black, inner sep=0pt, minimum width=3pt]

        \draw[thick] (0, 3.5)--(3, 2);
        \draw[thick] (3, 2)--(4, -1);
        \draw[thick] (4, -1)--(1.8, -3.5);
        \draw[thick] (1.8, -3.5)--(-1.8, -3.5);
        \draw[thick] (-1.8, -3.5)--(-4, -1);
        \draw[thick] (-4, -1)--(-3, 2);
        \draw[thick] (-3, 2)--(0, 3.5);

        \draw[thick] (1.5, 0)--(0.9, -1.5);

        \draw[orange, thick] (1.5, 0)--(-3, 2);
        \draw[orange, thick] (1.5, 0)--(0, 3.5);
        \draw[orange, thick] (1.5, 0)--(3, 2);
        \draw[orange, thick] (1.5, 0)--(4, -1);
        \draw[orange, thick] (1.5, 0)--(1.8, -3.5);

        \draw[green, thick] (0.9, -1.5)--(4, -1);
        \draw[green, thick] (0.9, -1.5)--(1.8, -3.5);
        \draw[green, thick] (0.9, -1.5)--(-3, 2);

        \filldraw[color=black, fill=white!100, thick] (0, 3.5) circle (0.7);
        \node [label={\small $1$}] (v) at (0, 2.95){};
        \node [label=$L_1$] (v) at (0, 4 - 0.08){};

        \filldraw[color=black, fill=white!100, thick] (3, 2) circle (0.7);
        \node [label={\small $2, 3$}] (v) at (3, 1.45 - 0.08){};
        \node [label=$L_2$] (v) at (3, 2.5 - 0.08){};

        \filldraw[color=black, fill=white!100, thick] (4, -1) circle (0.7);
        \node [label={\small $5$}] (v) at (4, -1.55){};
        \node [label=$L_3$] (v) at (4.2, -0.5 - 0.08){};

        \filldraw[color=black, fill=white!100, thick] (1.8, -3.5) circle (0.7);
        \node [label={\small $2$}] (v) at (1.8, -4){};
        \node [label=$L_4$] (v) at (2.9, -4 - 0.08){};

        \filldraw[color=black, fill=white!100, thick] (-1.8, -3.5) circle (0.7);
        \node [label={\small $1, 3, 5$}] (v) at (-1.8, -4.05 - 0.08){};
        \node [label=$L_5$] (v) at (-2.9, -4 - 0.08){};

        \filldraw[color=black, fill=white!100, thick] (-4, -1) circle (0.7);
        \node [label={\small $2, 4$}] (v) at (-4, -1.5 - 0.08){};
        \node [label=$L_6$] (v) at (-4.2, -0.5 - 0.08){};

        \filldraw[color=black, fill=white!100, thick] (-3, 2) circle (0.7);
        \node [label={\small $3, 5$}] (v) at (-3, 1.45 - 0.08){};
        \node [label=$L_7$] (v) at (-3, 2.5 - 0.08){};
        
        \filldraw[color=orange, fill=white!100, thick] (1.5, 0) circle (0.5);
        \node [label={\small $4$}] (v) at (1.5, -0.5){};        
        \node [label=$L_{9}$] (v) at (2.35, -0.3){};
        
        \filldraw[color=green, fill=white!100, thick] (0.9, -1.5) circle (0.5);
        \node [label={\small $1$}] (v) at (0.9, -2){};
        \node [label=$L_{12}$] (v) at (0.4, -2.7){};
        
    \end{tikzpicture}
    \subcaption{A $(4,5)$-good subgraph in Lemma \ref{lem:exactly one S_5}.}
    \label{fig:9 and 12}
    \end{subfigure}

    \medskip

    \begin{subfigure}{0.45\linewidth}
        \centering
        \begin{tikzpicture}[scale=0.7]
        \tikzstyle{v}=[circle, draw, solid, fill=black, inner sep=0pt, minimum width=3pt]

        \draw[thick] (0, 3.5)--(3, 2);
        \draw[thick] (3, 2)--(4, -1);
        \draw[thick] (4, -1)--(1.8, -3.5);
        \draw[thick] (1.8, -3.5)--(-1.8, -3.5);
        \draw[thick] (-1.8, -3.5)--(-4, -1);
        \draw[thick] (-4, -1)--(-3, 2);
        \draw[thick] (-3, 2)--(0, 3.5);

        \draw[thick] (-1.5, 0)--(0.9, -1.5);
        \draw[thick] (-1.5, 0)--(-0.9, -1.5);
        \draw[thick] (0.9, -1.5)--(-0.9, -1.5);

        \draw[purple, thick] (-0.9, -1.5)..controls (1.5, -2.5)..(4, -1);
        \draw[purple, thick] (-0.9, -1.5)--(1.8, -3.5);
        \draw[purple, thick] (-0.9, -1.5)--(-1.8, -3.5);
        \draw[purple, thick] (-0.9, -1.5)--(-4, -1);
        \draw[purple, thick] (-0.9, -1.5)..controls (-2.6, -0.5)..(-3, 2);

        \draw[blue, thick] (-1.5, 0)--(4, -1);
        \draw[blue, thick] (-1.5, 0)--(-4, -1);
        \draw[blue, thick] (-1.5, 0)--(-3, 2);

        \draw[green, thick] (0.9, -1.5)--(4, -1);
        \draw[green, thick] (0.9, -1.5)--(1.8, -3.5);
        \draw[green, thick] (0.9, -1.5)..controls (-1.5, 1.2)..(-3, 2);

        \filldraw[color=black, fill=white!100, thick] (0, 3.5) circle (0.7);
        \node [label={\small $1, 2$}] (v) at (0, 2.95 - 0.08){};
        \node [label=$L_1$] (v) at (0, 4 - 0.08){};

        \filldraw[color=black, fill=white!100, thick] (3, 2) circle (0.7);
        \node [label={\small $3, 5$}] (v) at (3, 1.45 - 0.08){};
        \node [label=$L_2$] (v) at (3, 2.5 - 0.08){};

        \filldraw[color=black, fill=white!100, thick] (4, -1) circle (0.7);
        \node [label={\small $1, 5$}] (v) at (4, -1.55 - 0.08){};
        \node [label=$L_3$] (v) at (4.2, -0.5 - 0.08){};

        \filldraw[color=black, fill=white!100, thick] (1.8, -3.5) circle (0.7);
        \node [label={\small $4$}] (v) at (1.8, -4){};
        \node [label=$L_4$] (v) at (2.9, -4 - 0.08){};

        \filldraw[color=black, fill=white!100, thick] (-1.8, -3.5) circle (0.7);
        \node [label={\small $1, 2, 5$}] (v) at (-1.8, -4.05 - 0.08){};
        \node [label=$L_5$] (v) at (-2.9, -4 - 0.08){};

        \filldraw[color=black, fill=white!100, thick] (-4, -1) circle (0.7);
        \node [label={\small}] (v) at (-4, -1.5 - 0.08){};
        \node [label=$L_6$] (v) at (-4.2, -0.5 - 0.08){};

        \filldraw[color=black, fill=white!100, thick] (-3, 2) circle (0.7);
        \node [label={\small $4, 5$}] (v) at (-3, 1.45 - 0.08){};
        \node [label=$L_7$] (v) at (-3, 2.5 - 0.08){};

        \filldraw[color=purple, fill=white!100, thick] (-0.9, -1.5) circle (0.5);
        \node [label={\small $3$}] (v) at (-0.9, -2){};
        \node [label=$L_{10}$] (v) at (-1.7, -2.5){};

        \filldraw[color=blue, fill=white!100, thick] (-1.5, 0) circle (0.5);
        \node [label={\small $2$}] (v) at (-1.5, -0.5){};
        \node [label=$L_{11}$] (v) at (-2.3, -0.1){};
        
        \filldraw[color=green, fill=white!100, thick] (0.9, -1.5) circle (0.5);
        \node [label={\small}] (v) at (0.9, -2 - 0.08){};
        \node [label=$L_{12}$] (v) at (1.1, -1.35){};
        
    \end{tikzpicture}
    \subcaption{A $(4,5)$-good subgraph in Lemma \ref{lem:exactly one S_5}.}\label{fig:10 and 11}
    \end{subfigure}
    \hfill
    \begin{subfigure}{0.45\linewidth}
        \centering
        \begin{tikzpicture}[scale=0.7]
        \tikzstyle{v}=[circle, draw, solid, fill=black, inner sep=0pt, minimum width=3pt]

        \draw[thick] (0, 3.5)--(3, 2);
        \draw[thick] (3, 2)--(4, -1);
        \draw[thick] (4, -1)--(1.8, -3.5);
        \draw[thick] (1.8, -3.5)--(-1.8, -3.5);
        \draw[thick] (-1.8, -3.5)--(-4, -1);
        \draw[thick] (-4, -1)--(-3, 2);
        \draw[thick] (-3, 2)--(0, 3.5);

        \draw[green, thick] (0.9, -1.5)--(4, -1);
        \draw[green, thick] (0.9, -1.5)--(1.8, -3.5);
        \draw[green, thick] (0.9, -1.5)..controls (-1.5, 1.2)..(-3, 2);

        \filldraw[color=black, fill=white!100, thick] (0, 3.5) circle (0.7);
        \node [label={\small $4$}] (v) at (0, 2.95 - 0.08){};
        \node [label=$L_1$] (v) at (0, 4 - 0.08){};

        \filldraw[color=black, fill=white!100, thick] (3, 2) circle (0.7);
        \node [label={\small $1,2,5$}] (v) at (3, 1.45){};
        \node [label=$L_2$] (v) at (3, 2.5 - 0.08){};

        \filldraw[color=black, fill=white!100, thick] (4, -1) circle (0.7);
        \node [label={\small $3$}] (v) at (4, -1.55 - 0.08){};
        \node [label=$L_3$] (v) at (4.2, -0.5 - 0.08){};

        \filldraw[color=black, fill=white!100, thick] (1.8, -3.5) circle (0.7);
        \node [label={\small $2,4$}] (v) at (1.8, -4){};
        \node [label=$L_4$] (v) at (2.9, -4 - 0.08){};

        \filldraw[color=black, fill=white!100, thick] (-1.8, -3.5) circle (0.7);
        \node [label={\small $1,3,5$}] (v) at (-1.8, -4.05 - 0.08){};
        \node [label=$L_5$] (v) at (-2.9, -4 - 0.08){};

        \filldraw[color=black, fill=white!100, thick] (-4, -1) circle (0.7);
        \node [label={\small $4$}] (v) at (-4, -1.5 - 0.08){};
        \node [label=$L_6$] (v) at (-4.2, -0.5 - 0.08){};

        \filldraw[color=black, fill=white!100, thick] (-3, 2) circle (0.7);
        \node [label={\small $1,2,3$}] (v) at (-3, 1.45){};
        \node [label=$L_7$] (v) at (-3, 2.5 - 0.08){};

        \filldraw[color=green, fill=white!100, thick] (0.9, -1.5) circle (0.5);
        \node [label={\small} $5$] (v) at (0.8, -2){};
        \node [label=$L_{12}$] (v) at (1.1, -1.35){};

    \end{tikzpicture}
    \subcaption{A $(4,5)$-good subgraph in Lemma \ref{lem:only one S_Y}.}
    \label{fig:one S_Y}
    \end{subfigure}
    \caption{$(4,5)$-good subgraphs in Lemmas \ref{lem:exactly two S_5}-\ref{lem:only one S_Y}.}
\end{figure}

\newpage

\section{Nice blowups of $C_5$}\label{sec:nice blowup of C_5}

From now on, we may assume that our graph $G$ is $(P_7, C_4, C_6, C_7)$-free. If $G$ is $C_5$-free, then $G$ is chordal and so $\chi(G)=\omega(G)$. So we assume that $G$ contains an induced $C_5$. 
The indices below are modulo 5.

A tuple $(B_1, B_2, B_3, B_4, B_5)$ is called a \emph{nice blowup of $C_5$} in a graph $G$  if
\begin{itemize}
    \item each $B_i$ is a clique in $G$, and distinct $B_i$ and $B_j$ are disjoint,  
    \item for every $i\in [5]$ and $v\in B_i$, $v$ has a neighbor in $B_{i-1}$ and a neighbor in $B_{i+1}$,
    \item for every $i\in [5]$, $B_i$ is anticomplete to $B_{i+2}$, 
    \item for every $i\in [5]$, $a\in B_i$, $b,c\in B_{i+1}$, $d\in B_{i+2}$ with $b\neq c$, $G[\{a,b,c,d\}]$ is not isomorphic to $P_4$.
\end{itemize}
We call the last condition of nice blowups the {\em $P_4$-structure} condition. Observe that a nice blowup of $C_5$ is $(P_7, C_4, C_6)$-free. Next we give some basic properties of nice blowups of $C_5$.

\begin{lemma}\label{lem:blowup-basicproperty}
    Let $G$ be a $(P_7, C_4, C_6)$-free graph, and 
    let $H=(B_1, B_2, B_3, B_4, B_5)$ be a nice blowup of $C_5$ in $G$. 
    \begin{itemize}
        \item[$(1)$]  One can order vertices in $B_i$ as $b_i^1,b_i^2,\ldots, b_i^{|B_i|}$ so that $N_{B_{i-1}\cup B_{i+1}}(b_i^1)\subseteq N_{B_{i-1}\cup B_{i+1}}(b_i^2)\subseteq \cdots \subseteq N_{B_{i-1}\cup B_{i+1}}(b_i^{|B_i|})$.
        \item[$(2)$] For each $i\in [5]$, there is a vertex in $B_i$ complete to $B_{i-1}\cup B_{i+1}$. In particular, each vertex in $B_{i-1}$ and each vertex in $B_{i+1}$ have a common neighbor in $B_{i}$.
    \end{itemize}
\end{lemma}
\begin{proof}
    We prove the properties one by one.

    \medskip
    \noindent (1)  It suffices to show that for every two distinct vertices $x,y\in B_i$, we have $N_{B_{i-1}\cup B_{i+1}}(x)\subseteq N_{B_{i-1}\cup B_{i+1}}(y)$ or $N_{B_{i-1}\cup B_{i+1}}(y)\subseteq N_{B_{i-1}\cup B_{i+1}}(x)$. By Lemma \ref{lem:nbd containment for an edge}, we may assume that $N_{B_{i-1}}(x)\subseteq N_{B_{i-1}}(y)$. If $N_{B_{i-1}}(x)= N_{B_{i-1}}(y)$, we are done by Lemma \ref{lem:nbd containment for an edge}. So we may assume that exists a vertex $z\in B_{i-1}$ that is adjacent to $y$ but not adjacent to $x$. Then the $P_4$-structure condition implies that $N_{B_{i+1}}(x)\subseteq N_{B_{i+1}}(y)$. So $N_{B_{i-1}\cup B_{i+1}}(x)\subseteq N_{B_{i-1}\cup B_{i+1}}(y)$.

    \medskip
    \noindent (2) If $b_i^{|B_i|}$ is not adjacent to a vertex $v\in B_j$ for some $j\in \{i-1,i+1\}$, then $v$ is anticomplete to $B_i$ by (2). This contradicts the definition of nice blowup of $C_5$. So $b_i^{|B_i|}$ is a desired vertex.
\end{proof}

Let $H=(B_1, B_2, B_3, B_4, B_5)$ be a nice blowup of $C_5$ in a graph $G$. For each $i\in [5]$, we fix a vertex  $q_i\in B_i$ that is complete to $B_{i-1}\cup B_{i+1}$. We point out that nice blowups of $C_5$ coincide with {\em 5-rings} defined in \cite{MIV20}.
It was proved in  \cite{MIV20} (Lemma 6.7)
that any nice blowup $H$ of $C_5$ has $\chi(H)\le \ceil{\frac{5}{4}\omega(H)}$. So we can assume that $V(G)\setminus V(H)$ is not empty.

For every vertex $v\in V(G)\setminus V(H)$, we write $\supp(v)=\{i\in [5]:B_i\cap N_G(v)\neq\emptyset\}$. 
Next, we investigate how the vertices outside $H$ attach to $H$.

\begin{lemma}\label{lem:blowup-consecutive}
    For $v\in V(G)\setminus V(H)$, $\supp(v)$ consists of consecutive integers modulo $5$ that has size in $\{0,1,2,3,5\}$.
\end{lemma}
\begin{proof}
    Suppose that $\supp(v)$ does not consist of consecutive integers modulo $5$.
    Then there exists $i\in [5]$ where $N_{B_i}(v)$ and $N_{B_{i+2}}(v)$ are non-empty and $N_{B_{i+1}}(v)$ is empty. Let $u\in N_{B_i}(v)$ and $w\in N_{B_{i+2}}(v)$. Then $v-u-q_{i+1}-w-v$ is an induced $C_4$, a contradiction. By the same argument, $\supp(v)$ cannot have size $4$.
\end{proof}

\begin{definition}[Attachment of $H$]
{\em We divide the vertices in $V(G)\setminus V(H)$ as follows.  Let $i\in [5]$. 
\begin{itemize}
    \item $A_0$ is the set of vertices $v$ in $V(G)\setminus V(H)$ where $\supp(v)=\emptyset$.
    \item $A_1(i)$ is the set of vertices $v$ in $V(G)\setminus V(H)$ where $\supp(v)=\{i\}$.
    \item $A_2(i)$ is the set of vertices $v$ in $V(G)\setminus V(H)$ where $\supp(v)=\{i,i+1\}$.
    \item $A_3(i)$ is the set of vertices $v$ in $V(G)\setminus V(H)$ where $\supp(v)=\{i-1,i,i+1\}$.
    \item $A_5$ is the set of vertices $v$ in $V(G)\setminus V(H)$ where $\supp(v)=[5]$.
\end{itemize} 
For each $j\in \{1,2,3\}$, let $A_j=\bigcup_{i\in [5]}A_j(i)$.}
\end{definition}

By Lemma~\ref{lem:blowup-consecutive}, the union of above sets is $V(G)\setminus V(H)$. For convenience, we write $A_x(i\pm y)$ for the union of $A_x(i+y)$ and $A_x(i-y)$.   
Let $A'_3(i)=A_3(i)\cup B_i$ for $i\in [5]$. 

Let $G$ be a graph and $P=v_1-v_2-\cdots-v_7$ be a 7-vertex path (not necessarily induced) in $G$. We say that $P$ is {\em bad} if $v_1-v_2-\cdots-v_6$ is an induced $P_6$ and $v_3-v_4-\cdots -v_7$ is an induced $P_5$. The following observation encodes $(P_7,C_6,C_7)$-freeness by bad $P_7$s.

\begin{lemma}\label{lem:badP7}
    A $(P_7,C_6,C_7)$-free graph does not have a bad $P_7$.
\end{lemma}
\begin{proof}
    Suppose that $P=v_1-v_2-\cdots-v_7$ is a bad $P_7$. If $v_7v_2\in E(G)$, then $v_2-v_3-v_4-v_5-v_6-v_7-v_2$ is an induced $C_6$. So we may assume that $v_7v_2\notin E(G)$. Now $V(P)$ induces either a $C_7$ or a $P_7$ depending on whether $v_7v_1\in E(G)$.
\end{proof}

Next we define the notions of induced sequences and induced cyclic sequences which will be useful for finding forbidden induced paths and cycles.

\begin{definition}[Induced sequences and induced cyclic sequences]
{\em    Let $A_1, \ldots, A_m$ be disjoint sets of vertices in $G$ where $m\in \mathbb{N}$. We say that a sequence $A_1-A_2- \cdots -A_m$ is an \emph{induced sequence} if 
\begin{itemize}
    \item for all $i\in [m]$, $A_i$ is connected,
    \item for all $i\in [m-1]$, there is an edge between $A_i$ and $A_{i+1}$, and
    \item for $i, j\in [m]$ with $i<j$, there is no edge between $A_i$ and $A_j$.
\end{itemize}  
Similarly, a sequence $A_1-A_2-\cdots -A_m-A_1$ is an \emph{induced cyclic sequence} if 
\begin{itemize}
    \item for all $i\in [m]$, $A_i$ is connected,
    \item for all $i\in [m-1]$, there is an edge between $A_i$ and $A_{i+1}$, and
    \item for $i, j\in [m]$ with $|j-i|\not\equiv 0,1\pmod m$, there is no edge between $A_i$ and $A_j$.
\end{itemize}
}
\end{definition}

If $A_i$ is a singleton $\{v\}$, then we may replace $A_i$ with $v$ in the notation. So induced sequences and induced cyclic sequences generalize induced paths and induced cycles, respectively.

\begin{lemma}\label{lem:linearsequence}
   If $A_1-A_2-\cdots-A_m$ is an induced sequence in $G$, then $G$ contains an induced path with at least $m$ vertices between any vertex in $A_1$ and any vertex in $A_m$.
\end{lemma}

\begin{proof}
    Note that $A=\bigcup_{i\in [m]}A_i$ is connected by the first two conditions of induced sequences. 
    Let $P\subseteq A$ be a shortest path between a vertex in $A_1$ and a vertex in $A_m$.
    By the third condition of induced sequences, 
    $P$ contains at least one vertex in each $A_i$ and so $P$ is a desired path.
\end{proof}

\begin{lemma}\label{lem:cyclicsequence}
    If $A_1-A_2-\cdots-A_m-A_1$ is an induced cyclic sequence in $G$ with $m\ge 3$, then $G$ contains an induced cycle with at least $m$ vertices.
\end{lemma}

\begin{proof}
    Let $v_j\in A_j$ be a vertex having a neighbor in $A_m$ for $j\in \{1,m-1\}$. 
    By Lemma \ref{lem:linearsequence}, there exists an induced path $P$ contained in $A_1\cup \cdots\cup A_{m-1}$ between $v_1$ and $v_{m-1}$ with at least $m-1$ vertices. 
    Let $v'_1\in P$ be the last vertex (from $v_1$ to $v_{m-1}$) in $A_1$ having a neighbor in $A_m$ and $v'_{m-1}\in P$ be the last vertex (from $v_{m-1}$ to $v_1$) in $A_{m-1}$ having a neighbor in $A_m$. 
    Let $Q$ be a shortest path between $v'_1$ and $v'_{m-1}$ whose internal is contained in $A_m$. 
    Then $v'_1-P-v'_{m-1}-Q-v'_1$ is a desired induced cycle. 
\end{proof}

\subsection{Basic properties}

In this subsection, we prove some basic properties that are used throughout the paper.

\begin{lemma}\label{lem:blowup-twoneighbor}
   For a connected subgraph $K$ of $A_2(i)$, $N_{A'_3(i)}(K)$ is complete to $N_{A'_3(i+1)}(K)$.
\end{lemma}
\begin{proof}
    Suppose for a contradiction that there exist $v\in N_{A'_3(i)}(K)$ and $w\in N_{A'_3(i+1)}(K)$ that are not adjacent. Then $v-K-w-N_{B_{i+2}}(w)-q_{i-2}-N_{B_{i-1}}(v)-v$ is an induced cyclic sequence, a contradiction.
\end{proof}

\begin{lemma}\label{lem:opposite A_1(i) and A_2(j)}
    Let $u\in A_1(i)$ and $v\in A_2(i+2)$ such that $uv\in E(G)$.
    Then $v$ is complete to $B_{i+2}\cup B_{i-2}$.
\end{lemma}
\begin{proof}
    Suppose that $v$ has a non-neighbor $a$ in $B_{i+2}\cup B_{i-2}$. By symmetry, we may assume that $a\in B_{i-2}$. 
    Let $b \in N_{B_i}(u)$.
    If $a$ has a neighbor $x$ in $N_{B_{i+2}}(v)$, then $v-u-b-q_{i-1}-a-x-v$ is an induced $C_6$. So assume that $a$ has no neighbor in $N_{B_{i+2}}(v)$. Let $y\in N_{B_{i+2}}(v)$ and $z \in N_{B_{i+2}}(a)$. Then $v-u-b-q_{i-1}-a-z-y-v$ is an induced $C_7$, a contradiction.
\end{proof}

A nice blowup $H$ of $C_5$ is {\em maximal} if there is no nice blowup $H'$ of $C_5$ such that $V(H)\subset V(H')$.

\begin{lemma}\label{lem:blowup-threeneighbor}
   Let $H=(B_1,B_2,B_3,B_4,B_5)$ be a maximal nice blowup of $C_5$ and $v\in A_3(i)$ for some $i\in [5]$. The following properties hold. 
    \begin{itemize}
        \item[$(1)$] $B_i\setminus N_G(v)\neq \emptyset$.
        \item[$(2)$] Either $B_{i-1}\setminus N_G(v)\neq \emptyset$ or $B_{i+1}\setminus N_G(v)\neq \emptyset$.
        \item[$(3)$] $B_i\setminus N_G(v)$ is anticomplete to either $N_{B_{i-1}}(v)$ or $N_{B_{i+1}}(v)$.
    \end{itemize}
\end{lemma}
\begin{proof} We prove the properties one by one.

    \medskip
    \noindent (1) Suppose for contradiction that $N_{B_i}(v)=B_i$. 
    For every $j\neq i$, let $B_j'=B_j$ and let $B_i'=B_i\cup \{v\}$.
    We claim that $(B_1', B_2', B_3', B_4', B_5')$ is also a nice blowup of $C_5$, which leads a contradiction to the assumption that $H$ is maximal.

    The first three conditions of a nice blowup are straightforward to verify. 
    Suppose that for some $p\in [5]$, there exist four vertices $a\in B_p'$, $b,c\in B_{p+1}'$, $d\in B_{p+2}'$ with $b\neq c$  where $G[\{a,b,c,d\}]$ is isomorphic to $P_4$. 
    As $d$ has a neighbor in $e\in B_{p-2}'$, it follows that $a-b-c-d-e-q_{i-1}-a$ is an induced $C_6$, a contradiction.

    \medskip
    \noindent (2) Suppose that $B_{i-1}\setminus N_G(v)= \emptyset$ and $B_{i+1}\setminus N_G(v)= \emptyset$. By (1), there is a vertex in $B_i\setminus N_G(v)$, say $u$. By definition, $u$ has a neighbor $w$ in $B_{i-1}$ and a neighbor $z$ in $B_{i+1}$. Then $v-w-u-z-v$ is an induced $C_4$, a contradiction.

    \medskip
    \noindent (3) Suppose that there is an edge between $a\in B_i\setminus N_G(v)$ and $b\in N_{B_{i-1}}(v)$, and there is an edge between $c\in B_i\setminus N_G(v)$ and $d\in N_{B_{i+1}}(v)$.
    If $a=c$, then $a-b-v-d-a$ is an induced $C_4$. So we may assume that $a\neq c$. This implies that $c$ is not adjacent to $b$, and $a$ is not adjacent to $d$. Then $b-a-c-d$ is an induced $P_4$, a contradiction.
\end{proof}

The next lemma generalizes the $P_4$-structure condition.

\begin{lemma}\label{lem:generalized P4-structure}
    There is no induced $P_4=a-b-c-d$ with $a\in A'_3(i-2)$, $b,c\in A'_3(i-1)$ and $d\in A'_3(i)$. 
\end{lemma}

\begin{proof}
    If such a $P_4$ exists, then $a-b-c-d-B_{i+1}-B_{i+2}-a$ is an induced cyclic sequence. This contradicts Lemma \ref{lem:cyclicsequence}.
\end{proof}

For a subset $S$ and two vertices $x,y\notin S$, we say that $x$ and $y$ are {\em comparable} or have {\em neighborhood containment in $S$} if $N_S(x)\subseteq N_S(y)$ or $N_S(y)\subseteq N_S(x)$, and have {\em disjoint neighborhoods in $S$} if $N_S(x)\cap N_S(y)=\emptyset$.

\begin{lemma}\label{lem:nonedge in A_3(i)}
    For $x,y\in A_3'(i)$ with $xy\notin E(G)$, $x$ and $y$ have neighborhood containment in one of $B_{i-1}$ and $B_{i+1}$ and have disjoint neighborhoods in the other.
\end{lemma}

\begin{proof}
    If $x$ and $y$ do not have neighborhood containment in either $B_{i-1}$ or $B_{i+1}$, then there is an induced $C_6$. So we may assume that $x$ and $y$ have neighborhood containment in $B_{i-1}$. In particular, $x$ and $y$ have a common neighbor in $B_{i-1}$. Since $G$ is $C_4$-free, $x$ and $y$ cannot have a common neighbor in $B_{i+1}$. So $x$ and $y$ have disjoint neighborhoods in $B_{i+1}$.
\end{proof}

\begin{lemma}\label{lem:comparable vertices in A_3(i)}
    Each pair of vertices in $A_3(i)$ have comparable neighborhoods in $B_{i}$. This implies that $B_i$ contains a vertex that is anticomplete to $A_3(i)$ \textup{(}a non-neighbor of a largest vertex in $A_3(i)$\textup{)}.
\end{lemma}

\begin{proof}
    Suppose not. 
    Assume that there are $u,v\in A_3(i)$ and $s,s'\in B_{i}$ such that $us,vs'\in E(G)$ but $us',vs\notin E(G)$. 
    Since $u-s-s'-v-u$ is not an induced $C_4$, $uv\notin E(G)$. 
    Then $u,v$ have disjoint neighborhoods either in $B_{i+1}$ or in $B_{i-1}$. 
    By symmetry, we may assume that $u,v$ have disjoint neighborhoods in $B_{i-1}$. 
    Let $u',v'\in B_{i-1}$ such that $uu',vv'\in E(G)$. 
    Since $u-s-s'-v-v'-q_{i-2}-q_{i+2}$ is not an induced $P_7$, either $v's\in E(G)$ or $v's'\in E(G)$. 
    If $v's'\notin E(G)$, then $v's\notin E(G)$. 
    It follows that $v's'\in E(G)$. 
    By symmetry, $u's\in E(G)$. 
    Since $s-s'-v'-u'-s$ is not an induced $C_4$, either $sv'\in E(G)$ or $s'u'\in E(G)$. 
    By symmetry, we may assume that $sv'\in E(G)$. 
    Then $v'$ is a common neighbor of $v$ and $s$. 
    Thus, $v$ and $s$ have disjoint neighborhoods in $B_{i+1}$. 
    Let $w$ be a common neighbor of $u,v$ in $B_{i+1}$. 
    Thus, $w$ is not adjacent to $s$. 
    Then $w-q_{i+2}-q_{i-2}-v'-s-u-w$ is an induced $C_6$. 
\end{proof}

In the following, we often say that $a-B-c-\cdots -d$ is an induced path where $B$ is a subset of vertices. By that we mean any vertex in $B$ can be taken on the path so that the resulting sequence is an induced path. The same notation applies for induced cycles.

\begin{lemma}\label{lem:blowup-fiveneighbor1}
The following properties hold for $A_5$.

\begin{itemize}
    \item[$(1)$]  $A_5$ is complete to $V(H)$.
    \item[$(2)$]  $A_5$ is a clique complete to $A_3$.
\end{itemize}
    
\end{lemma}

\begin{proof}  Let $v\in A_5$.

    \medskip
    \noindent (1) We first claim that $v$ is complete to $q_i$. Suppose that $v$ is not adjacent to $q_i$ for some $i\in [5]$. As $q_i$ is complete to $B_{i-1}\cup B_{i+1}$, it follows that $v - N_{B_{i-1}}(v) - q_i - N_{B_{i+1}}(v) - v$ form an induced $C_4$, a contradiction. Now we show that $v$ is complete to $V(H)$. Suppose that $v$ is not adjacent to $w\in B_i$ for some $i\in [5]$. Since $v$ is adjacent to $q_{i-1}$ and $q_{i+1}$, $v-q_{i-1}-w-q_{i+1}-v$ is an induced $C_4$, a contradiction.

    \medskip
    \noindent (2) Suppose that $v$ is not adjacent to $w\in A_3(i)\cup A_5$. Let $a\in N_{B_{i-1}}(w)$ and $b\in N_{B_{i+1}}(w)$. By (1), we obtain $va,vb\in E(G)$ and so $v-a-w-b-v$ is an induced $C_4$.
\end{proof}

\subsection{Anticomplete pairs}

In this subsection, we prove that certain sets are anticomplete to each other.

\begin{lemma}\label{A1A2}
The following holds. 
    \begin{itemize}
        \item[$(1)$] Let $u\in A_1$ and $v\in A_1\cup A_2$ such that $\supp(u)\cap \supp(v)=\emptyset$ and $\{u,v\}$ is anticomplete to at least two consecutive sets $B_{j}$ and $B_{j+1}$. Then $uv\notin E(G)$.
        \item[$(2)$] $A_1(i)$ is anticomplete to $A_{1}(j)$ for $i\neq j$, and $A_{1}(i)$ is anticomplete to $A_{2}(i+1)\cup A_2(i-2)$.
    \end{itemize}
\end{lemma}

\begin{proof} 
    The second statement follows immediately from the first statement. So it remains to prove (1). We may assume that $u\in A_1(i)$ for some $i\in [5]$ and $\{u,v\}$ is anticomplete to $B_{i-2}\cup B_{i-1}$. Let $a \in N_{B_i}(u)$. Suppose to the contrary that $uv\in E(G)$.
    If $v$ has a neighbor $b$ in $B_{i+2}$, then $v - u - a - q_{i-1} - q_{i-2} - b - v$ is an induced $C_6$, a contradiction. So $v$ is anticomplete to $B_{i+2}$. It follows that $v$ has a neighbor in $B_{i+1}$ and let $b \in N_{B_{i+1}}(v)$.
    If $ab\in E(G)$, then $u - v - b - a - u$ is an induced $C_4$. If $ab\notin E(G)$, then $v - u - a - q_{i-1} - q_{i-2} - q_{i+2} - b - v$ is an induced $C_7$, a contradiction.
\end{proof}

\begin{lemma}\label{A1A3}
    $A_1(i)$ is anticomplete to $A_3(i-2) \cup A_3(i+2)$.
\end{lemma}

\begin{proof}
    Let $u\in A_1(i)$ and $v\in A_3(i-2) \cup A_3(i+2)$. Suppose that $uv\in E(G)$. By symmetry, we may assume $v\in A_3(i+2)$. Let $a \in N_{B_i}(u)$. Let $b \in N_{B_{i+1}}(v)$ and $c \in N_{B_{i-2}}(v)$. If $a$ has a neighbor $x$ in $N_{B_{i+1}}(v)$, then $a - x - v - u - a$ is an induced $C_4$. So, $a$ is anticomplete to $N_{B_{i+1}}(v)$ and this implies that $a$ has a neighbor $y$ in $B_{i+1}\setminus N_G(v)$. Then $v - b - y - a - q_{i-1} - c - v$ is an induced $C_6$, a contradiction.
\end{proof}

\begin{lemma}\label{A2A2}
    $A_2(i)$ is anticomplete to $A_2(i-1) \cup A_2(i+1)$.
\end{lemma}

\begin{proof}
   Let $u\in A_2(i)$ and $v\in A_2(i-1) \cup A_2(i+1)$. Suppose that $uv\in E(G)$. By symmetry, we may assume that $v\in A_2(i+1)$. Let $a \in N_{B_i}(u)$ and $b \in N_{B_{i+2}}(v)$. Then $u - v - b - q_{i-2} - q_{i-1} - a - u$ is an induced $C_6$, a contradiction.
\end{proof}

\begin{lemma}\label{lem:anticompleteA3-1}
    Let $u,v\in V(G)\setminus V(H)$ such that 
    \begin{itemize}
        \item $u$ has a neighbor in each of $B_{i-1}$ and $B_i$, and is anticomplete to $B_{i-2}\cup B_{i+2}$, 
        \item $v$ has a neighbor in each of $B_{i-2}$ and $B_{i+2}$, and is anticomplete to $B_{i-1}\cup B_i$.
    \end{itemize}
    Then $uv\notin E(G)$.
    This implies the following.
    \begin{itemize}
        \item $A_3(i)$ is anticomplete to $A_3(i- 2) \cup A_3(i+2)$.
        \item $A_3(i)$ is anticomplete to $A_2(i+2)$.
        \item $A_2(i)$ is anticomplete to $A_2(i-2) \cup A_2(i+2)$. 
    \end{itemize}
\end{lemma}
\begin{proof}
    Suppose that $uv\in E(G)$. Since $G$ is $C_4$-free, $N_{B_{i-1}}(u)$ is anticomplete to $N_{B_{i-2}}(v)$. Let $a \in N_{B_{i-1}}(u)$ and $b \in N_{B_{i-2}}(v)$. Then $a$ has a neighbor $c$ in $B_{i-2}\setminus N_G(v)$ and $b$ has a neighbor $d$ in $B_{i-1}\setminus N_G(u)$. By Lemma \ref{lem:blowup-basicproperty} (1), $cd \in E(G)$. If $c$ has a neighbor $c'$ in $N_{B_{i+2}}(v)$ and $d$ has a neighbor $d'$ in $N_{B_i}(u)$, then $u - v - c' - c - d - d' - u$ is an induced $C_6$. So we may assume that either $c$ is anticomplete to $N_{B_{i+2}}(v)$ or $d$ is anticomplete to $N_{B_i}(u)$. By symmetry, we assume that $c$ is anticomplete to $N_{B_{i+2}}(v)$. Then $c$ has some neighbor $x$ in $B_{i+2}\setminus N_G(v)$ and thus $u - a - c - x - N_{B_{i+2}}(v) - v - u$ is an induced $C_6$, a contradiction.
\end{proof}

\begin{lemma}\label{lem:anticompleteA3}
    $A_3(i)$ is anticomplete to $A_3(i-1) \cup A_3(i+1)$.
\end{lemma}
\begin{proof}
    Let $u \in A_3(i)$ and $v\in A_3(i+1)\cup A_3(i-1)$. Suppose that $uv \in E(G)$. By symmetry, we may assume that $v\in A_3(i+1)$. By Lemma~\ref{lem:nbd containment for an edge}, $N_G(u)$ and $N_G(v)$ are comparable within each of $B_i$ and $B_{i+1}$. By Lemma~\ref{lem:blowup-threeneighbor}, $B_i\setminus N_G(u)$ and $B_{i+1}\setminus N_G(v)$ are not empty.

    \medskip
    \noindent {\bf Case 1: $B_i\setminus N_G(u)$ is not anticomplete to $N_{B_{i-1}}(u)$.} 
    \medskip 
    
    Let $x\in B_i\setminus N_G(u)$ and $y\in N_{B_{i-1}}(u)$ with $xy\in E(G)$. If $x$ is adjacent to $v$, then $x - y - u - v - x$ is an induced $C_4$. So $xv \notin E(G)$.

    Assume that $x$ has a neighbor $x'\in B_{i+1}\setminus N_G(v)$. Then $vx' \notin E(G)$ and it follows that $x'-x-y-u-v-B_{i+2}-x'$ contains an induced $C_6$ or $C_7$ depending on whether $v$ and $x'$ have a common neighbor in $B_{i+2}$. So $x$ is anticomplete to $B_{i+1}\setminus N_G(v)$. By Lemma~\ref{lem:blowup-threeneighbor}, $x$ is anticomplete to $N_{B_{i+1}}(u)$. Therefore, $x$ has a neighbor $a\in N_{B_{i+1}}(v)\setminus N_{B_{i+1}}(u)$.

    Let $z\in B_{i+1}\setminus N_G(v)$. It follows that $xz\notin E(G)$. Assume that $z$ has a neighbor $z'$ in $B_i\setminus N_G(u)$. By Lemma \ref{lem:nonedge in A_3(i)}, $u$ and $z'$ have some common neighbor $z''$ in $B_{i-1}$. Then $v - u - z'' - z' - z - B_{i+2} - v$ is an induced $C_6$ or $C_7$ depending on whether $v$ and $z$ have a common neighbor in $B_{i+2}$. Thus, $z$ is anticomplete to $B_i \setminus N_G(u)$.
 
    Assume that $z$ has a neighbor $b \in N_{B_i}(u)\setminus N_{B_i}(v)$. If $ab\in E(G)$, then $u - b - a - v - u$ is an induced $C_4$. If $ab\notin E(G)$, then $a - x - b - z - a$ is an induced $C_4$. Therefore, $z$ has a neighbor in $N_{B_i}(u)\cap N_{B_i}(v)$. By Lemma \ref{lem:nonedge in A_3(i)}, $z$ and $v$ have disjoint neighborhoods in $B_{i+2}$. Let $c \in N_{B_{i+2}}(v)$ and $d \in N_{B_{i+2}}(z)$. Note that $uz\notin E(G)$ for otherwise $u - z - a - v - u$ is an induced $C_4$. Therefore, $z - d - c - v - u - y - x$ is an induced $P_7$, a contradiction.

    \medskip
    \noindent {\bf Case 2: $B_i\setminus N_G(u)$ is anticomplete to $N_{B_{i-1}}(u)$.} 
    
    \medskip 
    By symmetry, we may assume that $B_{i+1}\setminus N_G(v)$ is anticomplete to $N_{B_{i+2}}(v)$. Let $x\in B_i\setminus N_G(u)$ and $z\in B_{i+1}\setminus N_G(v)$. Let $a \in N_{B_{i-1}}(u)$ and $b \in N_{B_{i-1}}(x)$, and let $c \in N_{B_{i+2}}(v)$ and $d \in N_{B_{i+2}}(z)$. If $uz \notin E(G)$, then $z - d - c - v - u - a - b$ is an induced $P_7$. If $vx \notin E(G)$, then $x - b - a - u - v - c - d$ is an induced $P_7$. Thus, we may assume that $uz, vx \in E(G)$. If $xz \notin E(G)$, then $b - x - v - u - z - d - q_{i-2} - b$ is an induced $C_7$. Hence, we obtain $xz \in E(G)$, and then $u - z - x - v - u$ is an induced $C_4$, a contradiction.
\end{proof}

\begin{lemma}\label{anti-A2A3}
    $A_2(i)$ is anticomplete to $A_3(i+2)\cup A_3(i-1)$.
\end{lemma}
\begin{proof}
    Let $u \in A_2(i)$ and $v\in A_3(i-1)\cup A_3(i+2)$. Suppose that $uv \in E(G)$. By symmetry, we may assume that $v\in A_3(i+2)$. By Lemma~\ref{lem:blowup-threeneighbor}, $B_{i+2}\setminus N_G(v) \neq \emptyset$.

    \medskip 
    \noindent {\bf Case 1:} $B_{i+2}\setminus N_G(v)$ is not anticomplete to $N_{B_{i-2}}(v)$.

     \medskip 
    Let $x\in B_{i+2}\setminus N_G(v)$ have a neighbor $a\in N_{B_{i-2}}(v)$. By Lemma~\ref{lem:blowup-threeneighbor} (3), $x$ is anticomplete to $N_{B_{i+1}}(v)$. Let $y \in B_{i+1}(x)$. If $uy \notin E(G)$, then $u - B_i - y - x - a - v - u$ is an induced $C_6$ or $C_7$ depending on whether $u$ and $y$ have a common neighbor in $B_i$. Thus, $uy \in E(G)$. By Lemma \ref{lem:blowup-threeneighbor} (3), we obtain $q_{i+2} \in N_{B_{i+2}}(v)$ and it follows that $u - y - q_{i+2} - v - u$ is an induced $C_4$, a contradiction.  

    \medskip 
    \noindent {\bf Case 2:}  $B_{i+2}\setminus N_G(v)$ is anticomplete to $N_{B_{i-2}}(v)$.

    \medskip 
Let $x\in B_{i+2}\setminus N_G(v)$. Recall that $q_{i+2} \in N_{B_{i+2}}(v)$. Then $u - v - q_{i+2} - N_{B_{i-2}}(x) - q_{i-1} - N_{B_i}(u) - u$ is an induced $C_6$, a contradiction.
\end{proof}

\section{$A_0$ is empty}\label{sec:A_0}

In this section, we show that $A_0=\emptyset$.

\begin{lemma}\label{lem:homogeneoustoA0}
    Let $K$ be a component of $G[A_0]$.
    \begin{itemize}
        \item[$(1)$] If $v\in A_1\cup A_2$, then $v$ is complete or anticomplete to $K$.
        \item[$(2)$] If $v\in A_3$, then $v$ is complete or anticomplete to $K$.
    \end{itemize}
\end{lemma}

\begin{proof} We prove the properties one by one.

\medskip
\noindent  (1) Suppose that $v$ has a neighbor in $B_i$ and has no neighbors in $B_{i+2}\cup B_{i-2}\cup B_{i-1}$. Let $c\in B_i$ be a neighbor of $v$. 
If $v$ is adjacent to $a$ but not adjacent to $b$ with $ab\in E(K)$, then $b-a-v-c-q_{i-1}-q_{i-2}-q_{i+2}$ is an induced $P_7$.

\medskip 
\noindent (2) Let $v\in A_3(i)$ for some $i\in [5]$. Suppose that $v$ is adjacent to $a$ but not adjacent to $b$ with $ab\in E(K)$. By Lemma~\ref{lem:blowup-threeneighbor} (2), $B_{i-1}\setminus N_G(v)\neq \emptyset$ or $B_{i+1}\setminus N_G(v)\neq \emptyset$. By symmetry, we assume that $B_{i-1}\setminus N_G(v)$ contains a vertex $c$. Let $d\in N(v)\cap B_{i+1}$. Then $a-b-v-d-q_{i+2}-q_{i-2}-c$ is an induced $P_7$.
\end{proof}

\begin{lemma}\label{lem:emptyA0-1}
    Let $K$ be a component of $G[A_0]$.
    Let $v\in A_1$ and $w\in A_1\cup A_2\cup A_3\cup A_5$ such that $v,w$ have neighbors in $K$. Then $vw\in E(G)$.
\end{lemma}

\begin{proof}
    Suppose that $v\in A_1(i)$ for some $i\in [5]$ and $vw\notin E(G)$. By Lemma~\ref{lem:homogeneoustoA0}, $v$ is complete to $K$.
    Let $z$ be a common neighbor of $v$ and $w$ in $K$. Since $G$ is $C_4$-free, $v$ and $w$ have no 
    common neighbors in $H$. So $w\notin A_5$ by Lemma \ref{lem:blowup-fiveneighbor1}.
    Let $a$ be a neighbor of $v$ in $B_i$. We divide cases depending on the place of $w$.

\medskip 
\noindent {\bf Case 1:} $w\in A_1(i)\cup A_1(i\pm 1)\cup A_2(i-1)\cup A_2(i)$.
    
    By symmetry, we may assume that $N_G(w)\cap V(H)$ are contained in $B_{i-1}\cup B_i$.
    Note that $a$ is not adjacent to $w$. Then 
    $w-z-v-a-q_{i+1}-q_{i+2}-q_{i-2}$
    is an induced $P_7$.

     \medskip 
\noindent  {\bf Case 2:} $w\in A_1(i\pm 2)\cup A_2(i+2)\cup A_3(i\pm 2)$.

    By symmetry, we may assume that $w$ has a neighbor $b$ in $B_{i+2}$ and $w$ is anticomplete to $B_{i+1}$.
    Then $v-z-w-b-q_{i+1}-a-v$ is an induced $C_6$.

     \medskip 
\noindent  {\bf Case 3:} $w\in A_2(i+1)\cup A_2(i-2)\cup A_3(i\pm 1)$.
    
    By symmetry, we may assume that 
    $w$ has a neighbor $b$ in $B_{i+2}$ and $w$ is anticomplete to $B_{i-2}\cup B_{i-1}$. Then 
    $a-v-z-w-b-q_{i-2}-q_{i-1}-a$
    is an induced $C_7$.

     \medskip 
\noindent  {\bf Case 4:} $w\in A_3(i)$.
 
    By Lemma~\ref{lem:blowup-threeneighbor}, either $B_{i-1}\setminus N_G(v)\neq \emptyset$ or $B_{i+1}\setminus N_G(v)\neq \emptyset$.
    By symmetry, we assume $B_{i-1}\setminus N_G(v)\neq \emptyset$. Let $b$ be a neighbor of $w$ in $B_{i+1}$ and $c$ be a non-neighbor of $w$ in $B_{i-1}$.  Then 
    $v-z-w-b-q_{i+2}-q_{i-2}-c$
is an induced $P_7$, a contradiction. 
\end{proof}

\begin{lemma}\label{lem:emptyA0-2}
    Let $K$ be a component of $G[A_0]$.
    Let $v\in A_2$ and $w\in A_2\cup A_3\cup A_5$ such that $v,w$ have neighbors in $K$. Then $vw\in E(G)$.
\end{lemma}
\begin{proof}
    Suppose $v\in A_2(i)$ for some $i\in [5]$ and $vw\notin E(G)$.
    Let $z$ be a common neighbor of $v$ and $w$ in $K$. Since $G$ is $C_4$-free, $v$ and $w$ have no 
    common neighbor in $H$. So $w\notin A_5$ by Lemma \ref{lem:blowup-fiveneighbor1}.
    We divide cases depending on the place of $w$.

    \medskip 
    \noindent {\bf Case 1:} $w\in A_2(i)$.
    
    Let $a$ be a neighbor of $v$ in $B_{i+1}$.
    Then 
    $w-z-v-a-q_{i+2}-q_{i-2}-q_{i-1}$
    is an induced $P_7$.

    \medskip 
     \noindent {\bf Case 2:} $w\in A_2(i\pm 1)\cup A_3(i)\cup A_3(i+1)$.
    
    By symmetry, we may assume that $w$ has a neighbor in $B_{i-1}$ and it is anticomplete to $B_{i+2}\cup B_{i-2}$.
     Let $a$ be a neighbor of $v$ in $B_{i+1}$ and $b$ be a neighbor of $w$ in $B_{i-1}$. Then 
     $v-z-w-b-q_{i-2}-q_{i+2}-a-v$ is an induced $C_7$.
    
    \medskip 
     \noindent {\bf Case 3:}  $w\in A_2(i\pm 2)\cup A_3(i-1)\cup A_3(i+2)$.

    By symmetry, we may assume that $w$ has a neighbor in $B_{i-2}$ and $w$ is anticomplete to $B_{i-1}$.
    Let $a$ be a neighbor of $v$ in $B_i$ and $b$ be a neighbor of $w$ in $B_{i-2}$. Then $v-z-w-b-q_{i-1}-a-v$ is an induced $C_6$. 

    \medskip 
     \noindent {\bf Case 4:}  $w\in A_3(i-2)$.
    
    By Lemma~\ref{lem:blowup-threeneighbor}, $B_{i-2}\setminus N_G(w)$ contains a vertex $a$. By Lemma~\ref{lem:blowup-threeneighbor} and symmetry, we may assume that $a$ is anticomplete to $N_G(w)\cap B_{i+2}$. This implies that $a$ has a neighbor $c$ in $B_{i+2}\setminus N_G(w)$.
    Then $v-z-w-b-a-c-q_{i+1}$ is a bad $P_7$.
\end{proof}

\begin{lemma}\label{lem:emptyA0-3}
    Let $K$ be a component of $G[A_0]$.
    Let $v, w\in A_3$ such that $v,w$ have neighbors in $K$. Then $vw\in E(G)$.
\end{lemma}
\begin{proof}
    Suppose $v\in A_3(i)$ for some $i\in [5]$ and $vw\notin E(G)$.
    Let $z$ be a common neighbor of $v$ and $w$ in $K$. Since $G$ is $C_4$-free, $v$ and $w$ have no 
    common neighbors in $H$. We divide cases depending on the place of $w$.
    
    \medskip 
    \noindent {\bf Case 1:} $w\in A_3(i)$. 
    
    This is impossible by Lemma \ref{lem:nonedge in A_3(i)}.

    \medskip 
    \noindent {\bf Case 2:} $w\in A_3(i\pm 1)$.

    By symmetry, we assume that $w\in A_3(i+1)$.
    Let $a\in B_{i-1}$ be a neighbor of $v$ and $b\in B_{i+2}$ be a neighbor of $w$. Then $z-v-a-q_{i-2}-b-w-z$ is an induced $C_6$.

    \medskip 
    \noindent {\bf Case 3:} $w\in A_3(i\pm 2)$.
    
    By symmetry, we assume that $w\in A_3(i+2)$.
    Let $a\in B_{i+1}$ be a neighbor of $v$ and $b\in B_{i+1}$ be a neighbor of $w$.
    Let $c\in B_{i-1}$ be a neighbor of $v$. If $c$ and $w$ have a common neighbor $d\in B_{i-2}$, then $v-a-b-w-d-c-v$ is an induced $C_6$. Otherwise,  $v-a-b-w-d-e-c-v$ is an induced $C_7$ where $d\in B_{i-2}$ is a neighbor of $w$ and $e\in B_{i-2}$ is a neighbor of $c$.
\end{proof}

\begin{lemma}\label{lem:A0cliqueseparator}
    $A_0=\emptyset$.
\end{lemma}

\begin{proof}
    If not, then $G$ has a clique cutset 
    by Lemmas~\ref{lem:emptyA0-1}, ~\ref{lem:emptyA0-2}, ~\ref{lem:emptyA0-3} and \ref{lem:blowup-fiveneighbor1}. 
\end{proof}

\section{Neighborhoods of $A_1(i)$}\label{sec:A_1}

    Our main result in this section is the following proposition.

\begin{proposition}\label{prop:component of A_1(i)}
    Let $K$ be a connected subgraph of $G[A_1(i)]$. 
    Then $N(K)\setminus (A_1(i)\cup A_2(i+2))$ is a clique.
\end{proposition}

    We start with some auxiliary lemmas.

\begin{lemma}\label{lem:A2i1tocompA1}
    Let $K$ be a connected subgraph of $G[A_1(i)]$. Each vertex in $A_2(i-1)\cup A_2(i)$ is either complete or anticomplete to $K$.
\end{lemma}

\begin{proof}
    Let $v \in A_2(i-1) \cup A_2(i)$ be a vertex that has a neighbor in $K$. By symmetry, we may assume that $v\in N(K)\cap A_2(i)$. We show that $v$ is complete to $K$. Suppose not. Then there exists an edge $ab \in E(K)$ such that $vb \in E(G)$ and $va \notin E(G)$. It follows that $a-b-v-N_{B_{i+1}}(v)-q_{i+2}-q_{i-2}-q_{i-1}$ is an induced $P_7$, a contradiction.
\end{proof}

\begin{lemma}\label{lem:A3tocompA1}
    Let $K$ be a connected subgraph of $G[A_1(i)]$. Each vertex in $A_3(i-1)\cup A_3(i)\cup A_3(i+1)$ is either complete or anticomplete to $K$.
\end{lemma}

\begin{proof}
    Let $v \in A_3(i-1)\cup A_3(i)\cup A_3(i+1)$ be a vertex that has a neighbor in $K$. Suppose that there exists an edge $ab \in E(K)$ such that $vb \in E(G)$ and $va \notin E(G)$. We first assume that $v \in A_3(i)$. If $B_{i+1}\setminus N(v) \neq \emptyset$, then $a - b - v - N_{B_{i-1}}(v) - q_{i-2} - q_{i+2} - x$ is an induced $P_7$. So, $v$ is complete to $B_{i+1}$. By symmetry, $v$ is complete to $B_{i-1}$. This contradicts Lemma \ref{lem:blowup-threeneighbor} (2).

    We now assume that $v \in A_3(i-1) \cup A_3(i+1)$. By symmetry, we may assume that $v \in A_3(i-1)$. By Lemma \ref{lem:blowup-threeneighbor} (3), $B_{i-1}\setminus N(v)$ is anticomplete to either $N_{B_{i-2}}(v)$ or $N_{B_i}(v)$. If $B_{i-1}\setminus N(v)$ is anticomplete to $N_{B_{i-2}}(v)$, then $a - b - v - N_{B_{i-1}}(v) - q_{i-2} - q_{i+2} - q_{i+1}$ is an induced $P_7$. So, $B_{i-1}\setminus N(v)$ is anticomplete to $N_{B_i}(v)$. It follows that $a - b - v - N_{B_{i-2}}(v) - q_{i+2} - q_{i+1} - q_i$ is a bad $P_7$, a contradiction.
\end{proof}

    \begin{lemma}\label{lem:basic properties of components of A_1(i)}
    Let $K$ be a connected subgraph of $G[A_1(i)]$. The following are satisfied.
    \begin{enumerate}
        \item[$(1)$] If $A_2(i-1) \neq \emptyset$ and $A_2(i) \neq \emptyset$, then $K$ is anticomplete to $A_2(i-1) \cup A_2(i)$.
        \item[$(2)$] $K$ is anticomplete to either $A_2(i)$ or $A_3(i-1)$. Similarly, $K$ is anticomplete to either $A_2(i-1)$ or $A_3(i+1)$.
        \item[$(3)$] $K$ is anticomplete to either $A_3(i-1)$ or $A_3(i+1)$.
        \item[$(4)$] $K$ is anticomplete to either $A_3(i-1)$ or $A_3(i)$. Similarly, $K$ is anticomplete to either $A_3(i)$ or $A_3(i+1)$.
        \item[$(5)$] $K$ has neighbors in at most one of $A_3(i-1),A_3(i),A_3(i+1)$.
    \end{enumerate}
    \end{lemma}
    
    \begin{proof}
    (1) Let $u \in A_2(i-1)$ and $v \in A_2(i)$ be vertices that have a neighbor in $K$. By Lemma~\ref{A2A2}, $uv\notin E(G)$. If $u$ and $v$ have a common neighbor $w$ in $K$, then $w - v - N_{B_{i+1}}(v) - q_{i+2} - q_{i-2} - N_{B_{i-1}}(u) - u - w$ is an induced $C_7$. So, $u$ and $v$ have disjoint neighborhood in $K$. Then, $(N(u)\cap K) - u - N_{B_{i-1}}(u) - q_{i-2} - q_{i+2} - N_{B_{i+1}}(v) - v$ and $(N(v)\cap K) - v - N_{B_{i+1}}(v) - q_{i+2} - q_{i-2} - N_{B_{i-1}}(u) - u$ are an induced $P_7$, a contradiction.

    \medskip
    \noindent (2) Let $u \in A_3(i-1)$ and $v \in A_2(i)$ be vertices that have a neighbor in $K$. By Lemma~\ref{anti-A2A3}, $uv\notin E(G)$. Let $P$ be a shortest path from $N(u)\cap V(K)$ to $N(v)\cap V(K)$ in $K$. Then $u - P - v - N_{B_{i+1}}(v) - q_{i+2} - N_{B_{i-2}}(u) - u$ is an induced cycle of length at least $6$, a contradiction.

    \medskip
    \noindent (3) Let $u \in A_3(i-1)$ and $v \in A_3(i+1)$ be vertices that have a neighbor in $K$. By Lemma~\ref{lem:anticompleteA3-1}, $uv\notin E(G)$. Let $P$ be a shortest path from $N(u)\cap V(K)$ to $N(v)\cap V(K)$ in $K$. By Lemma~\ref{lem:blowup-threeneighbor}, we obtain $B_{i-1} \setminus N(u) \neq \emptyset$ and $B_{i+1} \setminus N(v) \neq \emptyset$. Thus, it follows that $(B_{i-1} \setminus N(u)) - N_{B_{i-1}}(u) - u - P - v - N_{B_{i-1}}(v) - (B_{i+1} \setminus N(v))$ is an induced path of order at least $7$, a contradiction.

    \medskip 
    \noindent (4) Let $u \in A_3(i-1)$ and $v \in A_3(i)$ be vertices that have a neighbor in $K$. By Lemma~\ref{lem:anticompleteA3}, $uv\notin E(G)$. Let $P$ be a shortest path from $N(u)\cap V(K)$ to $N(v)\cap V(K)$ in $K$. Then $u - P - v - N_{B_{i+1}}(v) - q_{i+2} - N_{B_{i-2}}(u) - u$ is an induced cycle of length at least $6$, a contradiction.

    \medskip
    \noindent (5) This follows from (3) and (4). 
    \end{proof}

We are now ready to prove Proposition~\ref{prop:component of A_1(i)}.
    \begin{proof}[Proof of Proposition~\ref{prop:component of A_1(i)}]
        We may assume by Lemma \ref{lem:basic properties of components of A_1(i)} that $K$ has no neighbor in $A_2(i)$ and has neighbors in at most one of $A_3(i-1),A_3(i),A_3(i+1)$.
    So $N(K)\setminus (A_1(i)\cup A_2(i+2))\subseteq B_i\cup A_2(i-1)\cup A_3(j)\cup A_5$ for some $j\in \{i-1,i,i+1\}$.

    \medskip
    \noindent (1) Each vertex in $N(K)\cap A_5$ is universal in $N(K)\setminus (A_1(i)\cup A_2(i+2))$.
    
   \medskip
   Let $x \in N(K) \cap A_5$ and $ y\in N(K) \cap A_2(i-1)$ be non-adjacent. By Lemma \ref{lem:A2i1tocompA1}, $x$ and $y$ have a common neighbor $z$ in $K$. By Lemma \ref{lem:blowup-fiveneighbor1}, $x$ and $y$ have a common neighbor $w \in B_{i-1}$ and it follows that $z - y - w - x - z$ is an induced $C_4$, a contradiction. So $N(K)\cap A_5$ is complete to $N(K)\cap A_2(i-1)$. 
   The statement follows from Lemma \ref{lem:blowup-fiveneighbor1}.  This proves (1).

    \medskip
    \noindent (2) Each vertex in $N(K)\cap B_i$ is universal in $N(K)\setminus (A_1(i)\cup A_2(i+2))$.
    
   \medskip 
    Let $x \in N_{B_i}(K)$ and $y \in N(K)\setminus (A_1(i)\cup A_2(i+2))$ be non-adjacent.
    By symmetry, we may assume that $y\in A_2(i-1)\cup A_3(i-1)\cup A_3(i)$.
    By Lemma \ref{lem:A2i1tocompA1}, $x$ and $y$ have a common neighbor $z$ in $K$. 
    Since $G$ is $C_4$-free, $N_{B_{j}}(y)$ is anticomplete to $x$ for $j\in \{i-1,i+1\}$. 
    Then, $N_{B_{i-1}}(y) - y - z - x - q_{i+1} - q_{i+2} - q_{i-2}$ is a bad $P_7$. This proves (2).

    \medskip
    \noindent (3) Each vertex in $N(K)\cap A_2(i-1)$ is universal in $N(K)\setminus (A_1(i)\cup A_2(i+2))$.
    
   \medskip 
    Suppose that $x\in N(K) \cap A_2(i-1)$ and $y\in N(K) \setminus (A_1(i)\cup A_2(i+2))$ are non-adjacent. By Lemma \ref{lem:A2i1tocompA1}, $x$ and $y$ have a common neighbor $z$ in $K$. Since $G$ is $C_4$-free, $N_{B_{j}}(x)$ and $N_{B_{j}}(y)$ are disjoint for $j\in \{i-1,i+1\}$. If $y\in A_2(i-1)\cup A_3(i)\cup A_3(i+1)$, then $y - z - x - N_{B_{i-1}}(x) - q_{i-2} - q_{i+2} - q_{i+1}$ is a bad $P_7$. So $y\in A_3(i-1)$.
    If $B_i\setminus N(y) \neq \emptyset$, then $x - z - y - N_{B_{i-2}}(y) - q_{i+2} - q_{i+1} - B_i\setminus N(y)$ is a bad $P_7$. So, $y$ is complete to $B_i$ and this implies that $B_{i-1}\setminus N(y)$ is anticomplete to $N_{B_{i-2}}(y)$ by Lemma \ref{lem:blowup-threeneighbor} (3). Then $x - z - y - N_{B_{i-1}}(y) - q_{i-2} - q_{i+2} - q_{i+1}$ is an induced $P_7$, a contradiction.
    This proves (3).

    \medskip
    \noindent (4) Each vertex in $N(K)\cap A_3(j)$ is universal in $N(K)\setminus (A_1(i)\cup A_2(i+2))$ for each $j\in \{i-1,i,i+1\}$.
    
    \medskip
    Suppose that $x,y\in N(K)\cap A_3(j)$ are non-adjacent. By Lemma \ref{lem:A3tocompA1}, $x$ and $y$ have a common neighbor in $K$ and so $x$ and $y$ have disjoint neighborhoods in $B_j$ for $j\in \{i-1,i+1\}$ by $C_4$-freeness of $G$. 
    By Lemma \ref{lem:comparable vertices in A_3(i)}, $j=i$.
    This contradicts Lemma \ref{lem:nonedge in A_3(i)}.
    This proves (4).

    \medskip    
    The proposition follows from (1)-(4).
    \end{proof}

We conclude this section with two more lemmas.

\begin{lemma}\label{lem:induced P3in A_1(i)}
Let $a-b-c$ be an induced $P_3$ in $A_1(i)$. Then $N_{B_i}(a)\subseteq N_{B_i}(b)$ and $N_{B_i}(c)\subseteq N_{B_i}(b)$.
\end{lemma}

\begin{proof}
    Suppose the lemma does not hold. Then there exists a vertex $p\in B_i$ that is adjacent to $a$ but not to $b$. Since $G$ is $C_4$-free, $pc \notin E(G)$. Then $c-b-a-p-q_{i-1}-q_{i-2}-q_{i+2}$ is an induced $P_7$, a contradiction. 
\end{proof}

\begin{lemma}\label{lem:nbr of component of A_1(i-1) in A_2(i-2)}
    Let $K$ be a component of $A_1(i-1)$. Then $N(K)\cap A_2(i-2)$ is complete to $B_{i-2}$. 
\end{lemma}

\begin{proof}
    By Proposition \ref{prop:component of A_1(i)}, $K$ has a neighbor $v'\in A_2(i+1)$. Let $v\in K$ be a neighbor of $v'$. Let $w \in N(K) \cap A_2(i-2)$. By Lemma \ref{lem:A2i1tocompA1}, we have $vw \in E(G)$. If $w$ has non-neighbor $x$ in $B_{i-2}$, then $x - N_{B_{i-2}}(w) - w - v - v' - q_{i+1} - q_i$ is an induced $P_7$, a contradiction.   
\end{proof}

\section{Neighborhoods of $A_2(i)$}

In this section, we present three more lemmas on $A_2(i)$.

\begin{lemma}\label{lem:neighbors of A_2(i) in A_3(i+1) and A_3(i)}
    Let $K$ be a connected subgraph of $A_2(i)$. Then $K$ is anticomplete to either $A_3(i)$ or $A_3(i+1)$.
\end{lemma}
\begin{proof}
    Let $u \in A_3(i)$ and $v \in A_3(i+1)$ be vertices that have a neighbor in $K$. By Lemma \ref{lem:anticompleteA3}, we have $uv\notin E$. This contradicts Lemma \ref{lem:blowup-twoneighbor}.
\end{proof}

\begin{lemma}\label{lem:neighbors of components of A_2(i) in A_3(i+1)}
    Let $K$ be a connected subgraph of $A_2(i)$. If two vertices $u, v \in N(K)\cap A'_3(i+1)$ are non-adjacent, then $u$ and $v$ have a common neighbor in $K$.
\end{lemma}

\begin{proof}
    Suppose for a contradiction that $u$ and $v$ do not have a common neighbor in $K$. Let $P$ be a shortest path from $u$ and $v$ with internal in $K$.   Note that $P$ has at least $4$ vertices. 
    By Lemma \ref{lem:blowup-twoneighbor}, $u$ and $v$ have a common neighbor in $B_i$ and so have disjoint neighbors in $B_{i+2}$. 
    Then $u - P - v - N_{B_{i+2}}(v) - q_{i-2} - q_{i-1}$ is an induced path of with at least $7$ vertices, a contradiction.
\end{proof}

\begin{lemma}\label{lem:induced P3in A_2(i)}
Let $a-b-c$ be an induced $P_3$ in $A_2(i)$. Then $N_{A'_3(j)}(a)\subseteq N_{A'_3(j)}(b)$ and $N_{A'_3(j)}(c)\subseteq N_{A'_3(j)}(b)$ for $j\in \{i,i+1\}$.
\end{lemma}

\begin{proof}
    By symmetry, we may assume that $j = i$. Suppose the lemma does not hold. Then there exists a vertex $p\in B_i$ that is adjacent to $a$ but not to $b$. Since $G$ is $C_4$-free, $pc \notin E(G)$. Then $c-b-a-p-N_{B_{i-1}}(p)-q_{i-2}-q_{i+2}$ is an induced $P_7$, a contradiction.
\end{proof}

\section{Neighborhoods of $A_3(i)$}\label{sec:A_3}

In this section, we prove more properties on $A_3(i)$.

\subsection{Single $A_3(i)$}

In this subsection, we prove properties for a single $A_3(i)$. We start with a single vertex in $A_3(i)$.

\begin{lemma}\label{lem:single vertex in A_3(i)}
    Let $x\in A_3(i)$ and suppose that $B_i\setminus N(x)$ is anticomplete to $N_{B_{i+1}}(x)$. Then the following properties hold. 

    \begin{itemize}
        \item[$(1)$] $q_i\in N_{B_i}(x)$ and $q_{i+1}\in B_{i+1}\setminus N(x)$. 

        \item[$(2)$] $N_{B_i}(x)$ is complete to $N_{B_j}(x)$ for $j\in \{i-1,i+1\}$.
       
        \item[$(3)$] For any vertex $a\in B_{i}\setminus N(x)$ and $b\in N_{B_{i}}(x)$, then $N_{B_{i-1}\cup B_{i+1}}(a)\subset N_{B_{i-1}\cup B_{i+1}}(b)$.
        
        \item[$(4)$] For any vertex $a\in B_{i-1}\setminus N(x)$ and $b\in N_{B_{i-1}}(x)$, then $N_{B_{i-2}\cup B_i}(a)\subseteq N_{B_{i-2}\cup B_i}(b)$. This implies that there is a vertex in $N_{B_{i-1}}(x)$ complete to $B_{i-2}\cup B_i$.
    \end{itemize}
\end{lemma}

\begin{proof} 
We prove these one by one.

\medskip
\noindent (1) This follows from the definition of $q_i$ and $q_{i+1}$.

\medskip
\noindent (2) By (1), $q_{i+1}x\notin E(G)$. 
If $a\in N_{B_i}(x)$ and $b\in N_{B_{i+1}}(x)$ are non-adjacent, then $x-a-q_{i+1}-b-x$ is an induced $C_4$. 
If $a\in N_{B_i}(x)$ and $b\in N_{B_{i-1}}(x)$ are non-adjacent, then $b-x-a-q_{i+1}$ contradicts Lemma \ref{lem:generalized P4-structure}.

\medskip
\noindent (3) By definition, $x$ has a neighbor $c\in B_{i+1}$. By (2), $c$ is adjacent to $b$ but not to $a$. So we are done by Lemma \ref{lem:blowup-basicproperty} (1). 

\medskip
\noindent (4) By Lemma \ref{lem:generalized P4-structure}, $N_{B_{i-2}}(a)\subseteq N_{B_{i-2}}(b)$. If there is a vertex $c\in B_i$ adjacent to $a$ but not to $b$, then $c\in B_i\setminus N(x)$ by (2), and so there exist vertices $d,e\in B_{i+1}$ such that $x-b-a-c-d-e-x$ is an induced $C_6$ by the assumption on $x$.
\end{proof}

We then consider edges in $A_3(i)$. Let $B_{i\pm 1}=B_{i-1}\cup B_{i+1}$.

\begin{lemma}\label{lem:neighborcontainA3i}
    Let $xy$ be an edge in $A_3(i)$. Then 

    \begin{itemize}
        \item[$(1)$] $N_{B_{i\pm 1}}(x)\subseteq N_{B_{i\pm 1}}(y)$ or $N_{B_{i\pm 1}}(y)\subseteq N_{B_{i\pm 1}}(x)$.
        \item[$(2)$] For $j\in \{i-1,i+1\}$, $N_{B_i\cup B_j}(x)\subseteq N_{B_i\cup B_j}(y)$ or $N_{B_i\cup B_j}(y)\subseteq N_{B_i\cup B_j}(x)$.
        \item[$(3)$] $N_{B_{i-1}\cup B_i\cup B_{i+1}}(x)\subseteq N_{B_{i-1}\cup B_i\cup B_{i+1}}(y)$ or $N_{B_{i-1}\cup B_i\cup B_{i+1}}(y)\subseteq N_{B_{i-1}\cup B_i\cup B_{i+1}}(x)$.
    \end{itemize}
\end{lemma}

\begin{proof} We prove the properties one by one.
    
   \medskip
   \noindent
   (1) This follows directly from Lemmas \ref{lem:nbd containment for an edge} and \ref{lem:generalized P4-structure}.

   \medskip
   \noindent
    (2) By symmetry, we prove for $j=i+1$. Suppose that the statement is not true. Let $z\in B_{i}$ be a neighbor of $x$ but not a neighbor of $y$. Let $w$ be a neighbor of $y$ but not a neighbor of $x$. Since $G$ is $C_4$-free, $w\in B_{i+1}$ and $zw\notin E(G)$. By Lemma \ref{lem:comparable vertices in A_3(i)},  one can take a common non-neighbor $a\in B_i$ of $x$ and $y$. If $aw\notin E(G)$, then $a-z-x-y-w-q_{i+2}-q_{i-2}$ is an induced $P_7$. So $aw\in E(G)$. Let $b\in B_{i-1}$ be a neighbor of $a$. By Lemma \ref{lem:blowup-threeneighbor}, $by\notin 
    E(G)$. By (1), $bx\notin E(G)$. Then
    $b-\{z,a\}-x-y-w-q_{i+2}-q_{i-2}-b$ is an induced cyclic sequence.

   \medskip
    \noindent
    (3) It follows directly from (1) and (2).
\end{proof}

\begin{lemma}\label{lem:edge in $A_3(i)$}
    Let $xy$ be an edge in $A_3(i)$. Then $N_{B_{i}}(x)=N_{B_{i}}(y)$.
\end{lemma}

\begin{proof}
    By Lemma \ref{lem:neighborcontainA3i} (3), we may assume that $N_{B_{i-1}\cup B_i\cup B_{i+1}}(y)\subseteq N_{B_{i-1}\cup B_i\cup B_{i+1}}(x)$.
    Suppose that there exists a vertex $a\in B_i$ with $ax\in E(G)$ and $ay\notin E(G)$.
    Let $d\in B_{i+1}$ be a neighbor of $y$ and $e\in B_{i-1}$ be a neighbor of $y$.
    Since $N_{B_{i-1}\cup B_i\cup B_{i+1}}(y)\subseteq N_{B_{i-1}\cup B_i\cup B_{i+1}}(x)$, $dx,ex\in E(G)$. By Lemma \ref{lem:single vertex in A_3(i)} (2), $ad,ae\in E(G)$. It follows that $a-d-y-e-a$ is an induced $C_4$.
\end{proof}

We extend properties for edges to induced $P_3$s in $A_3(i)$.

\begin{lemma}\label{lem:P3 in A_3(i)}
    If $P=a-b-c$ is an induced $P_3$ in $A_3(i)$, then $N_H(a),N_H(c)\subseteq N_H(b)$.
\end{lemma}

\begin{proof}
    Since $ac\notin E(G)$, $a$ and $c$ have disjoint neighborhoods in one of $B_{i-1}$ or $B_{i+1}$. So $N_H(a)$ and $N_H(c)$ cannot have a containment relation. This implies that $N_H(a),N_H(c)\subseteq N_H(b)$.
\end{proof}

We now prove the following key property.

\begin{lemma}[$P_4$-freeness of $A_3(i)$]\label{lem:P4-freenessofA3}
    $A_3(i)$ is $P_4$-free for each $i\in [5]$.
\end{lemma}

\begin{proof}
    Let $a-b-c-d$ be an induced $P_4$ in $A_3(i)$. By Lemma \ref{lem:P3 in A_3(i)} to $a-b-c$ and $b-c-d$, we have $N_H(a)\subseteq N_H(b)=N_H(c)$. This contradicts Lemma \ref{lem:nonedge in A_3(i)}.
\end{proof}

The next two lemmas are about components of $A_3(i)$.

\begin{lemma}\label{lem:single vertex covers A_3(i)}
    For any component $K$ of $A_3(i)$, there exists a vertex $x\in K$ universal in $K$ such that $N_H(K)= N_H(x)$.
\end{lemma}

\begin{proof}
    By Lemma \ref{lem:P4-freenessofA3},  $K$ is $P_4$-free. It follows from \cite{CLB81} that $V(K)$ can be partitioned into subsets $X$ and $Y$ such that $X$ and $Y$ are complete. Since $G$ is $C_4$-free, we may assume that $X$ is a clique. Choose such a partition $(X,Y)$ such that $|X|$ is maximum. 
    By Lemma \ref{lem:neighborcontainA3i}, there exists
    a vertex $x\in X$ such that $N_H(x)$ is maximum over all vertices in $X$. 
    For each vertex $y\in Y$, $y$ has a non-neighbor $z\in Y$ by the choice of $(X,Y)$. Applying Lemma \ref{lem:P3 in A_3(i)} to $y-x-z$, we have $N_H(y),N_H(z)\subseteq N_H(x)$.
    So $x$ is a desired vertex.
\end{proof}

\begin{lemma}\label{lem:component of A_3(i)}
    Let $K$ be a component of $A_3(i)$ such that $B_i\setminus N(K)$ is anticomplete to $N_{B_{i+1}}(K)$. The set $X_{i-1}$ of vertices in $B_{i-1}\setminus N(K)$ having a neighbor in $B_i\setminus N(K)$ is complete to $B_{i-2}\cup B_i$.
\end{lemma}

\begin{proof}
    By Lemma \ref{lem:single vertex covers A_3(i)}, there exists a vertex $u\in K$ such that $N_H(K)=N_H(u)$.
    Let $x\in X_{i-1}$ and $a\in B_i\setminus N(K)$ be a neighbor of $x$.
    If $x$ has a non-neighbor $x'$ in $B_{i-2}$, then $x' - q_{i-2} - x - a - N_{B_{i+1}}(a) - N_{B_{i+1}}(u) - u$ is an induced $P_7$. So $X_{i-1}$ is complete to $B_{i-2}$. If $x$ has a non-neighbor $x''$ in $N_{B_i}(u)$, then $x - a - x'' - b$ is a $P_4$-structure where $b\in N_{B_{i+1}}(u)$. So $X_{i-1}$ is complete to $N_{B_i}(u)$. If $x$ has a non-neighbor $x'''\in B_i\setminus N(K)$, then $x'''-a-x-q_{i-2}-q_{i+2}-N_{B_{i+1}}(u)-u$ is an induced $P_7$. Therefore, $X_{i-1}$ is complete to $B_i$. 
\end{proof}

The next lemma is used for coloring in Section \ref{sec:only A3}.

\begin{lemma}\label{lem:maximalityofA3}
    Suppose that $H$ is maximal. 
    For any component $K$ of $A_3(i)$,  $|B_i\setminus N(K)|\ge \omega(K)$.
\end{lemma}

\begin{proof}
    By Lemma \ref{lem:edge in $A_3(i)$}, all vertices in $K$ have the same neighborhood in $B_i$. Suppose that $|B_i\setminus N(K)|<\omega(K)$. Let $M$ be a maximum clique of $K$. Then $B'_i=N_{B_i}(K) \cup M$ is a clique of size bigger than $B_i$. Moreover, $|M|=\omega(K)\ge 2$. Let $x$ and $y$ be two vertices in $M$.
    We now show that $(B_1,\cdots ,B'_i, \cdots ,B_5)$ is still a nice blowup of $C_5$, which contradicts the maximality of $H$. Since each vertex in $B_i\setminus N(K)$ is not complete to $N_{B_{i+1}}(K)$ or $N_{B_{i-1}}(K)$ by Lemma \ref{lem:blowup-threeneighbor}, $q_i\in N_{B_i}(K)$. Therefore, it suffices to check the $P_4$-structure condition. 

    \medskip
    \noindent 
    \textbf{Case 1}: $B_{i-1}$ (and $B_{i+1}$).
    
    \medskip
    If there exists an induced $P_4=a-b-c-x$ with $a\in B_{i-2}$ and $b,c\in B_{i-1}$, then $a-b-c-x-N_{B_{i+1}}(x)-q_{i+2}-a$ is an induced $C_6$. So the $P_4$-structure condition is satisfied for $B_{i-1}$. By symmetry, the $P_4$-structure condition is satisfied for $B_{i+1}$. 
   
    \medskip
    \noindent 
    \textbf{Case 2}: $B_i'$.
    
    \medskip    
    Suppose that $a-b-c-d$ is an induced $P_4$ with $a\in B_{i-1}$, $b,c\in B'_i$, and $d\in B_{i+1}$. By Lemma \ref{lem:neighborcontainA3i} and niceness of $(B_1,\ldots,B_5)$, exactly one of $x$ and $y$ is on this $P_4$. We may assume that $b=x$. If $c \in N_{B_i}(K)$, then $ac \in E(G)$ by Lemma \ref{lem:single vertex in A_3(i)} (2), which contradicts the assumption. Thus, we may assume that $c=y$, which contradicts Lemma \ref{lem:neighborcontainA3i}.
    Therefore, the $P_4$-structure condition is satisfied for $B'_i$.
\end{proof}

\subsection{Multiple $A_3(i)$}
In this subsection, we prove properties involving multiple $A_3(j)$s. Recall that $A'_3(i)= A_3(i)\cup B_i$ for each $i\in [5]$.

    \begin{lemma}\label{lem:no two directed edge pointing to the same B_i}
        Let $v,a\in A'_3(i-1)$ be non-adjacent, and $b,w\in A'_3(i+1)$ be non-adjacent. Then it cannot be the case that $N_{B_i}(v)$ and $N_{B_i}(a)$ are disjoint, and that $N_{B_i}(w)$ and $N_{B_i}(b)$ are disjoint.
    \end{lemma}

    \begin{proof}\label{lem:B_i is complete to two large cliques}
        By Lemma \ref{lem:generalized P4-structure}, every $x\in A'_3(i-1)$ and $y\in A'_3(i+1)$  have neighborhood containment in $B_i$.
        If $N_{B_i}(v)$ and $N_{B_i}(a)$ are disjoint, then each of $N_{B_i}(w)$ and $N_{B_i}(b)$ contains both $N_{B_i}(v)$ and $N_{B_i}(a)$. This implies that $w$ and $b$ have a common neighbor in $B_i$.
    \end{proof}

\begin{lemma}\label{lem:B_i complete to two large cliques in B_{i-1}}
    Let $x,y\in A'_3(i-2)$ be two non-adjacent vertices such that $N_{B_{i-1}}(x)$ and $N_{B_{i-1}}(y)$ are disjoint. Then $A'_3(i)$ is complete to $N_{B_{i-1}}(x)\cup N_{B_{i-1}}(y)$.  
\end{lemma}

\begin{proof}
    We first show that each vertex $v\in A'_3(i)$ is either complete or anticomplete to $N_{B_{i-1}}(x)\cup N_{B_{i-1}}(y)$. Suppose that $v$ has a neighbor $a\in N_{B_{i-1}}(x)$. If $v$ has a non-neighbor $b\in N_{B_{i-1}}(y)$, then $v-a-b-y$ contradicts Lemma \ref{lem:generalized P4-structure}. So $v$ is complete to $N_{B_{i-1}}(y)$. It follows that $v$ is complete to $N_{B_{i-1}}(x)$. Now suppose that $v$ is anticomplete to $N_{B_{i-1}}(x)\cup N_{B_{i-1}}(y)$. Let $c\in B_{i-1}$ be a neighbor of $v$ and $d\in N_{B_{i-1}}(x)$. Then $v-c-d-x$ contradicts Lemma \ref{lem:generalized P4-structure}. So $A'_3(i)$ is complete to $N_{B_{i-1}}(x) \cup N_{B_{i-1}}(y)$.
\end{proof}

\begin{lemma}\label{lem:compare nbd of different A_3(i) in B_j}
    The neighbors of $A_3(i+2)$ in $B_{i-2}$ are complete to $A_3(i-2)$.  
\end{lemma}

\begin{proof}
    Suppose not. 
    Let $w\in B_{i-2}$ be a neighbor of $v'\in A_3(i+2)$ and have a non-neighbor
    $v\in A_3(i-2)$.
    Assume first that $v$ and $w$ have disjoint neighborhoods in $B_{i+2}$. 
    By Lemma \ref{lem:nonedge in A_3(i)}, 
    $v$ and $w$ have a common neighbor $c$ in $B_{i-1}$. 
    Let $u\in B_{i+2}$ be a neighbor of $v$. 
    By Lemma \ref{lem:single vertex in A_3(i)} (2), $uv'\notin E(G)$. 
    Then $u-v-c-w-v'-B_{i+1}-u$ is an induced induced cyclic sequence.
    So we may assume that $v$ and $w$ have disjoint neighborhoods in $B_{i-1}$. 
    If $v'$ is not complete to $B_{i+1}$, then $v-N_{B_{i-1}}(v)-N_{B_{i-1}}(w)-w-v'-N_{B_{i+1}}(v')-(B_{i+1}\setminus N(v'))$ is an induced $P_7$. 
    So $v'$ is complete to $B_{i+1}$. 
    Then $v-N_{B_{i-1}}(v)-N_{B_{i-1}}(w)-w-v'-q_{i+1}-(B_{i+2}\setminus N(v'))$ is a bad $P_7$.
\end{proof}

\begin{lemma}\label{lem:complete and anti on w}
    Let $v\in A_3(i-2)$ and $u\in A_3(i+2)$ such that $B_{i-2}\setminus N(v)$ is anticomplete to $N_{B_{i+2}}(v)$ and $B_{i+2}\setminus N(u)$ is anticomplete to $N_{B_{i-2}}(u)$. 
    Then 
    \begin{itemize}
        \item[$(1)$] $B_{i-2}\setminus N(v)$ is anticomplete to $B_{i+2}\setminus N(u)$.  
        \item[$(2)$] For any $w\in B_{i-2}\setminus N(v)$ and $w'\in B_{i+2}\setminus N(u)$, 
        $N_{B_{i+2}}(w)$ is complete to $N_{B_{i-2}}(w')$.
    \end{itemize}
\end{lemma}

\begin{proof}
    We prove one by one.

    \medskip
    \noindent (1)
    Let $v'\in N_{B_{i+2}}(v)$ and $u'\in N_{B_{i-2}}(u)$. 
    Then $v'w, u'w'\notin E(G)$.
    By Lemma \ref{lem:compare nbd of different A_3(i) in B_j}, $v'u,u'v\in E(G)$. 
    By Lemma \ref{lem:single vertex in A_3(i)} (2), $v'u'\in E(G)$.  
    If $w\in B_{i-2}\setminus N(v)$ and $w'\in B_{i+2}\setminus N(u)$ with $ww'\in E(G)$, then $w-w'-v'-u'-w$ is an induced $C_4$.
    This proves (1). 

    \medskip
    \noindent (2)
    Suppose that there are $s\in N_{B_{i+2}}(w)$ and $s'\in N_{B_{i-2}}(w')$ with $ss'\notin E(G)$. Then $w-s-w'-s'-w$ is an induced $C_4$.
    This proves (2).
\end{proof}

\begin{lemma}\label{lem:cross four cliques}
    Let $a_1,b_1$ be two non-adjacent vertices in $A'_3(i-1)$ and $a_2,b_2$ be two non-adjacent vertices in $A'_3(i-2)$ such that $N_{B_{i-2}}(a_1)$ and $N_{B_{i-2}}(b_1)$ are disjoint, and $N_{B_{i-1}}(a_2)$ and $N_{B_{i-1}}(b_2)$ are disjoint. 
    Then $N_{B_{i-2}}(a_1)\cup N_{B_{i-2}}(b_1)\cup N_{B_{i-1}}(a_2)\cup N_{B_{i-1}}(b_2)$ is a clique. 
\end{lemma}

\begin{proof}
    Let $K_{a_1}=N_{B_{i-2}}(a_1)$, $K_{a_2}=N_{B_{i-1}}(a_2)$, $K_{b_1}=N_{B_{i-2}}(b_1)$ and $K_{b_2}=N_{B_{i-1}}(b_2)$. 
    Note that $K_{a_j}\cup K_{b_j}$ is a clique for $j=1,2$. 
    By symmetry, we may assume that $b_1\in A_3(i-1)$ and $b_2\in A_3(i-2)$.
    We show that $K_{a_1}\cup K_{b_1}$ is complete to $K_{a_2}\cup K_{b_2}$. 
    By Lemma \ref{lem:compare nbd of different A_3(i) in B_j}, $K_{b_1}\subseteq N(b_2)$. 
    By Lemma \ref{lem:single vertex in A_3(i)} (2), $K_{b_1}$ is complete to $K_{b_2}$. 
    Suppose first that $K_{a_1}$ is not complete to $K_{b_2}$. 
    Then $a_1\in B_{i-1}$. 
    Let $s\in K_{a_1}$ and $s'\in K_{b_2}$ with $ss'\notin E(G)$. 
    Since $s'a_1\in E(G)$, $a_1-s-K_{b_1}-s'-a_1$ is an induced $C_4$. 
    So $K_{a_1}$ is complete to $K_{b_2}$. 
    By symmetry, $K_{a_2}$ is complete to $K_{b_1}$.
    Suppose that $K_{a_1}$ is not complete to $K_{a_2}$. 
    Then we may assume that $a_1\in B_{i-1}$ and $a_2\in B_{i-2}$. 
    By Lemma \ref{lem:complete and anti on w} (2), $K_{a_1}$ is complete to $K_{a_2}$, a contradiction.
\end{proof}

\subsection{Small vertex argument for $A_3(i)$}

In this subsection, we apply the small vertex argument for some vertices in $A'_3(i)$ to conclude that certain sets have large sizes.

\begin{lemma}[Small vertex argument for $A_3(i)$]\label{lem:small vertex argument for simplicial vertices in A_3(i-2)}
    Let $v$ be a $(A_3(i),H)$-minimal simplicial vertex.  
    If $v$ has no neighbors in $A_1\cup A_2$, then both $N(v)\setminus B_{i-1}$ and $N(v)\setminus B_{i+1}$ are cliques and so $|N_{B_{i-1}}(v)|,|N_{B_{i+1}}(v)|>\ceil{\frac{\omega}{4}}$.
\end{lemma}

\begin{proof}
    Suppose not. 
    By symmetry, we may assume that there are two non-adjacent vertices $s,s'$ in $N(v)\setminus B_{i-1}$. 
    Note that $N(v)\setminus B_{i-1}\subseteq B_{i+1}\cup B_{i}\cup A_3(i)\cup A_5$. 
    By Lemmas \ref{lem:blowup-fiveneighbor1}, \ref{lem:single vertex in A_3(i)} (2) and \ref{lem:edge in $A_3(i)$},
    we may assume $s\in A_3(i)$ and $s'\in B_{i+1}$.
    Suppose first that $s$ is simplicial in $A_3(i)$. By the choice of $v$, $s$ has a neighbor $s''\in B_{i-1}\cup B_{i+1}$ with $s''v\notin E(G)$. This contradicts Lemmas \ref{lem:nbd containment for an edge} and \ref{lem:generalized P4-structure}.
    So $s$ is not simplicial in $A_3(i)$. 
    Then $s$ has a neighbor $s''\in A_3(i)$ with $s''v\notin E(G)$. 
    Now $s''-s-v$ is an induced $P_3$ in $A_3(i)$, which contradicts Lemma \ref{lem:P3 in A_3(i)}.
\end{proof}

\begin{lemma}[Small vertex argument for non-neighbors of $A_3(i)$ in $B_i$]\label{lem:smv on w}
    Suppose that $A_3(i)\neq \emptyset$ and $B_{i}$ is anticomplete to $A_1\cup A_2$. 
    If $w\in B_{i}\setminus N(A_3(i))$ has minimal neighborhood in $H$, then
    $N(w)\setminus B_{i-1}$ and $N(w)\setminus B_{i+1}$ are cliques, and so $|N_{B_{i-1}}(w)|,|N_{B_{i+1}}(w)|>\ceil{\frac{\omega}{4}}$.
\end{lemma}

\begin{proof}
    Suppose for a contradiction that $N(w)\setminus B_{i+1}$ is not a clique. 
    By Lemma \ref{lem:compare nbd of different A_3(i) in B_j} and $A_3(i)\neq \emptyset$, $N(w)\setminus B_{i+1}\subseteq B_{i-1}\cup B_{i}\cup A_5$. 
    Let $s\in B_{i}$ and $s'\in B_{i-1}$ be two non-adjacent neighbors of $w$.
    By Lemma \ref{lem:single vertex in A_3(i)} (3), $s\in B_{i}\setminus N(A_3(i))$. By the minimality of $w$, $s$ has a neighbor $s''\in B_{i-1}\cup B_{i+1}$ not adjacent to $w$. 
    It follows that $s''-s-w-s'$ is a $P_4$-structure if $s''\in B_{i+1}$ or $s''-s-w-s'-s''$ is an induced $C_4$ if $s''\in B_{i-1}$. 
\end{proof}

\section{Non-neighbors of $C_5$}\label{sec:non-neighbors of C5}

For an induced $C_5=v_1-v_2-v_3-v_4-v_5-v_1$ and $X\subseteq V(C)$, we denote by $S(X)$  the set of vertices $v\in V(G)\setminus V(C)$ with $N_C(v)=X$.
We set $S_0=S(\emptyset)$ and $S_5=S(V(C))$. For convenience, we omit the set notation $\{\}$ when we write $X$ in $S(X)$. For instance, we write $S(\{v_1,v_2\})$ as $S(v_1,v_2)$ etc. The following lemma will be used to prove certain subgraphs are chordal in Section \ref{sec:A_1 is not empty}.

\begin{lemma}\label{lem:non-neighbors of C5}
    Let $C=v_1-v_2-v_3-v_4-v_5-v_1$ be an induced $C_5$ of $G$ and $K$ be a non-clique component of $S_0$. Then the following holds for neighbors of $K$.

    \begin{itemize}
        \item[$(1)$] Every vertex in $N(K)\cap S(v_{i-2},v_{i+2})$ is complete to $K$. 
        \item[$(2)$] $N(K)\cap S(v_{i-1},v_i,v_{i+1})$ or $N(K)\cap S(v_{i+1},v_{i+2},v_{i+3})$ is empty.  
        \item[$(3)$] If $p,q\in N(K)$ are not adjacent, then $p\in S(v_{i-1},v_i,v_{i+1})$ and $q\in S(v_{i-2},v_{i+2})$ for some $i\in \mathbb{Z}_5$.
        \item[$(4)$] $N(K)$ cannot contain 4 neighbors $p,p',q,q'$ such that $p\in S(v_{i-1},v_i,v_{i+1})$, $q\in S(v_{i-2},v_{i+2})$, $p'\in S(v_{i},v_{i+1},v_{i+2})$ and $q'\in S(v_{i-2},v_{i-1})$.
    \end{itemize}
\end{lemma}

\begin{proof}
    We prove the properties one by one.

    \medskip
    \noindent (1) Let $c\in N(K)\cap S(v_{i-2},v_{i+2})$ be adjacent to $b\in K$. If $c$ has a non-neighbor $a\in K$, then $a-b-c-v_{i-2}-v_{i-1}-v_i-v_{i+1}$ is an induced $P_7$. This proves (1).
    
    \medskip
    \noindent (2) Suppose for a contradiction that $a\in N(K)\cap S(v_{i-1},v_i,v_{i+1})$ and $b\in N(K)\cap S(v_{i+1},v_{i+2},v_{i+3})$. 
    Since $a-v_{i-1}-v_{i-2}-b-a$ is not an induced $C_4$, $ab\notin E(G)$. 
    Let $P$ be a shortest path between $a$ and $b$ with internal in $K$. 
    Since $G$ is $C_4$-free, $P$ is an induced path on at least 4 vertices. 
    Then $a-v_{i-1}-v_{i-2}-b-P-a$ is an induced cycle of length at least 6. 
    This proves (2). 

    \medskip
    \noindent (3) Since $G$ has no clique cutsets, $N(K)$ contains two non-adjacent vertices $p$ and $q$. 
    Note first that if a vertex $a\in N(K)$ has at most 2 neighbors on $C$, then $a$ is complete to $K$ for otherwise there is an induced $P_7$.
    Since $G$ is $C_4$-free and $K$ is not a clique, we may assume that $p\in S(v_{i-1},v_i,v_{i+1})\cup S_5$. Suppose first that $p\in S_5$. Since $S_5$ is a clique  complete to $S(v_{j-1},v_j,v_{j+1})$ for each $j\in [5]$, we have $q$ has at two neighbors on $C$. It follows that $p$ and $q$ have a common neighbor in both $K$ and $C$, a contradiction. So $p\in S(v_{i-1},v_i,v_{i+1})$.
    Since $S_5$ is complete to $p$, we have $q$ has at most 3 neighbors on $C$.
    By (2), $q\notin S(v_{i+1},v_{i+2},v_{i+3})$.
    If $q\in S(v_{i-1},v_i,v_{i+1})$, $p-v_{i-1}-q-v_{i+1}$ is an induced $C_4$. If $q\in S(v_{i},v_{i+1},v_{i+2})$, then $p-v_{i-1}-v_{i-2}-v_{i+2}-q-K-p$ is an induced cyclic sequence.  So $q$ has at most two neighbors on $C$. It follows that $p$ and $q$ have a common neighbor $u\in K$. Since $G$ is $C_4$-free, $p$ and $q$ have no common neighbors on $C$. If $q\in S(v_{i+2})$, then $p-v_{i-1}-v_{i-2}-v_{i+2}-q-N_K(p)-p$ is an induced $C_6$. So $q\in S(v_{i-2},v_{i+2})$. This proves (3).

    \medskip
    \noindent (4) Suppose not. It follows from (3) that $p-p'-q-q'-p$ is an induced $C_4$. This proves (4).
\end{proof}

\section{A key technical lemma}\label{sec:a general lemma}

In this section, we prove a key technical lemma.
Let $K$ be a connected subgraph of $G$ and $a,b\in N(K)$ be non-adjacent. We say that a 6-tuple $(a,b,U,P_a,P_b,Q)$ is {\em nice} if the following hold:
    \begin{itemize}
        \item $U,P_a,P_b,Q$ form a partition of $K$ where $U,P_a,P_b\neq \emptyset$.
        \item $a$ is complete to $P_a$ and anticomplete to $P_b$, and $b$ is complete to $P_b$ and anticomplete to $P_a$.
        \item $P_a,P_b,Q$ are pairwise anticomplete.
        \item $U$ is complete to $P_a\cup P_b\cup \{a,b\}$.
    \end{itemize}
    Note that in a nice 6-tuple, the adjacency between $Q$ and $\{a,b\}\cup U$ can be arbitrary.

\begin{lemma}\label{lem:complete subgraphs}
    Let $G$ be a $(P_7,C_4,C_6,C_7)$-free graph with no clique cutsets.
    Let $K$ be a non-empty connected subgraph of $G$ with the following properties.
    \begin{itemize}
        \item[$(1)$]  There is a subset $X\subseteq V(G)$ such that each vertex in $N(K)\setminus X$ is universal in $N(K)$. 
        \item[$(2)$] $K\cup X$ contains no induced $P_4=a-b-c-d$ with $d\in X$, $c\in K$ and $a,b\in K\cup (N(K)\cap X)$.
        \item[$(3)$] There is a subset $Y\subseteq V(G)$ anticomplete to $K$ such that any two non-adjacent vertices in $N(K)\cap X$ have non-empty, disjoint and complete neighborhoods in $Y$.   
    \end{itemize}
    Then $N(K)\cap X$ contains two non-adjacent vertices that are complete to $K$. In particular, $K$ is a clique.
\end{lemma}

\begin{proof}
    By (1) and $G$ has no clique cutsets, $N(K)\cap X$ contains two non-adjacent vertices. 

    \medskip
    \noindent (i) If $N(a)\cap K\nsubseteq N(b)\cap K$ and $N(b)\cap K\nsubseteq N(a)\cap K$ where $a,b\in N(K)\cap X$ are not adjacent, then there is a nice 6-tuple.

    \medskip
    Let $U$ be the set of common neighbors of $a$ and $b$ in $K$, $P_a$ be the set of private neighbors of $a$ in $K$, $P_b$ be the set of private neighbors of $b$ in $K$, and $Q$ be the set of common non-neighbors of $a$ and $b$ in $K$. By assumption, $P_a$ and $P_b$ are not empty.  We show that $(a,b,U,P_a,P_b,Q)$ is nice.  Let $a'\in Y$ be a neighbor of $a$ and Let $b'\in Y$ be a neighbor of $b$. By (3), $a'b'\in E(G)$. 

    If $u\in U$ is not adjacent to $p_a\in P_a$, then $p_a-a-u-b$ is a forbidden $P_4$ in (2). So $U$ is complete to $P_a$. By symmetry, $U$ is complete to $P_b$. If $p_a\in P_a$ is adjacent to $p_b\in P_b$, then $p_a-a-a'-b'-b-p_b-p_a$ is an induced $C_6$. So $P_a$ and $P_b$ are anticomplete. If $q\in Q$ has a neighbor $p_a\in P_a$, then $q-p_a-a-a'-b'-b-p_b$ is a bad $P_7$, since the only possible chord is $qp_b$. So $Q$ is anticomplete to $P_a$ and $P_b$. By connectivity of $K$, we have $U\neq \emptyset$.  This completes the proof of (i).

    \medskip
    \noindent (ii) For any non-adjacent vertices $a$ and $b$ in $N(K)$,
    $N(a)\cap K\subseteq N(b)\cap K$ or $N(b)\cap K\subseteq N(a)\cap K$.

    \medskip
    Suppose not. By (i), nice 6-tuples exist. We choose a nice 6-tuple $(a,b,U,P_a,P_b,Q)$ so that $|P_a|+|P_b|$ is minimum. We shall deduce a contradiction by applying the clique cutset argument to $P_a$. Let $p_a\in P_a$ and $p_b\in P_b$.
    Let $a'\in Y$ be a neighbor of $a$ and $b'\in Y$ be a neighbor of $b$. By (3), $a'b'\in E(G)$.

    Note that $N(P_a)\subseteq (N(K)\setminus X)\cup (X\setminus \{b\})\cup U$ by the definition of 6-nice tuple. Since all vertices in $N(K)\setminus X$ are common neighbors of $a$ and $b$ and $G$ is $C_4$-free, $U$ is a clique complete to $N(K)\setminus X$.
    If $u\in U$ is not adjacent to $v\in N(P_a)\cap X$, then $v-p_a-u-b$ is a forbidden $P_4$ in (2). So $U$ is complete to $N(P_a)\cap X$ and thus each vertex in $U$ is universal in $N(P_a)$.

    Next we show that {\color{orange} $N(P_a)\cap X$ is complete to $a$.} Suppose by contradiction that $c\in N(P_a)\cap X$ is not adjacent to $a$. We may assume that $p_ac\in E(G)$. Note that $p_b-b-b'-a'-a-p_a$ is an induced $P_6$. Since $p_b-b-b'-a'-a-p_a-c$ is not a bad $P_7$, $cb'\in E(G)$. Since $c-U-b-b'-c$ is not an induced $C_4$, $cb\in E(G)$. Then $b-c-p_a-a$ contradicts (2). This shows that $N(P_a)\cap X$ is complete to $a$.
    
    Since $N(P_a)$ is not a clique, there exist vertices $c,c'\in N(P_a)\cap X$ such that $c$ and $c'$ are not adjacent. It follows that $ca,c'a\in E(G)$. Since $G$ is $C_4$-free, $b$ is not adjacent to $c$ or $c'$, say $c$.
    Applying to $P_b$ the statement $N(P_a)\cap X$ is complete to $a$, we have $c$ is anticomplete to $P_b$. 
    If there is an edge $st$ with $s\in N(c)\cap P_a$ and $t\in P_a\setminus N(c)$, then $t-s-c-N_{Y}(c)-b'-b-p_b$ is an induced $P_7$. So $N(c)\cap P_a$ is anticomplete to $P_a\setminus N(c)$ and thus $(c,b,U,P_a\cap N(c),P_b,Q\cup (P_a\setminus N(c)))$ is a nice 6-tuple. By the choice of $(a,b,U,P_a,P_b,Q)$, $c$ is complete to $P_a$.
    But now $c'$ is a neighbor of $P_a$ and should be adjacent to $c$ by the argument showing $N(P_a)\cap X$ is complete to $a$. This completes the proof of (ii).

    \medskip
    \noindent (iii) Let $\{a,b\}$ in $N(K)\cap X$ such that $a$ and $b$ have maximum number of common neighbors in $K$. Then $K\subseteq N(a)\cap N(b)$.

    \medskip
    By (ii), we may assume that $N(b)\cap K\subseteq N(a)\cap K$. Let $U$ be the set of common neighbors in $K$, $P$ be the private neighbors of $a$ in $K$ and $Q$ be the set of common non-neighbors of $a$ and $b$ in $K$. Suppose that $P\cup Q$ is not empty. Let $T$ be a non-empty component of $P\cup Q$. Note that $N(T)\subseteq (N(K)\setminus X)\cup (X\setminus \{b\})\cup U$. 
    Since all vertices in $N(K)\setminus X$ are common neighbors of $a$ and $b$ by (1), 
    $U$ is a clique complete to
    $N(K)\setminus X$.
    Suppose that $u\in U$ is not adjacent to $c\in N(T)\cap X$.
    By (2), $c$ and $u$ have a common neighbor $d\in T$. 
    If $d\in P$, then 
    then $bc\in E(G)$ by (ii) and so $c-d-u-b-c$ is an induced $C_4$. If $d\in Q$, then $ca,cb\in E(G)$ by (ii). It follows that $c-a-u-b-c$ is an induced $C_4$. So $U$ is complete to $N(T)\cap X$ and so each vertex in $N(T)\cap U$ is universal in $N(T)$.
    Since $N(T)$ is not a clique, $P\cup Q$ has two non-adjacent neighbors $c,c'$ in $N(T)\cap X$. Since $b$ is anticomplete to $T$, we have $b\neq c,c'$. Since $T$ is connected, $G$ is $C_6$-free and $c,c'$ have disjoint and complete neighborhoods in $Y$, we have that $c,c'$ have a common neighbor in $T$. Then $\{c,c'\}$ contradicts the choice of $\{a,b\}$. This completes the proof of the lemma.
\end{proof}

\section{$A_1$ is not empty}\label{sec:A_1 is not empty}

In this section, we color $G$ when $A_1\neq \emptyset$.
Let $C=v_1-v_2-v_3-v_4-v_5-v_1$ be an induced $C_5$ of $G$ such that $|S_5|$ (see Section \ref{sec:non-neighbors of C5}) is minimum among all induced $C_5$s of $G$. Let $H$ be a maximal nice blowup of $C$. From now on, we fix the choice of $C$ and $H$. We first prove some general properties and then consider two cases depending the number of non-empty $A_1(i)$s (Subsections \ref{sec:two non-empty A_1(i)} and \ref{sec:one non-empty A_1(i)}).

\begin{lemma}\label{lem:A5 vs A1A2}
    $A_5$ is complete to any edge $xy$ with $x\in A_1(i)$ and $y\in A_2(i+2)$ for each $i\in [5]$. 
\end{lemma}

\begin{proof}
    Suppose not. 
    Let $xy$ be an edge with $x\in A_1(i)$ and $y\in A_2(i+2)$ such that some $u\in A_5$ is not complete to $\{x,y\}$.
    Let $x'\in N_{B_{i}}(x)$ and $y'\in N_{B_{i+2}}(y)$. 
    Then $C'=x'-q_{i+1}-y'-y-x-x'$ is an induced $C_5$. 
    Let $S_5'$ be the set of vertices that are complete to $V(C')$. 
    By definition of $S'_5$, $S'_5\subseteq A'_3(i+1)\cup A_5$.
    By Lemma \ref{anti-A2A3}, $S'_5\subseteq A_5$. 
    Since $A_5\subseteq S_5$ and $u\notin S'_5$, $|S'_5|<|S_5|$ which  contradicts the choice of $C$. 
\end{proof}

\begin{lemma}\label{lem:reduct to at most two non-empty A_1(i)}
    For each $i\in \mathbb{Z}_5$, either $A_1(i)$ is anticomplete to $A_2(i+2)$ or $A_1(i+1)$ is anticomplete to $A_2(i-2)$.
\end{lemma}

\begin{proof}
    Suppose not. Let $x\in A_1(i)$ be adjacent to $w\in A_2(i+2)$ and $y\in A_1(i+1)$ be adjacent to $z\in A_2(i-2)$. It follows from anticomplete pair properties that $xw$ and $yz$ induce a $2P_2$. By Lemma \ref{lem:opposite A_1(i) and A_2(j)}, $w$ and $z$ have common neighbor $a$ in $B_{i-2}$. Then $B_i-x-w-a-z-y-B_{i+1}-B_i$ is an induced cyclic sequence.
\end{proof}

By Lemma \ref{lem:reduct to at most two non-empty A_1(i)} and Proposition \ref{prop:component of A_1(i)}, there are at most 2 non-empty $A_1(i)$s. If there are two such $A_1(i)$s, their support are of distance 2 on the maximal nice blowup $H$.

\begin{lemma}\label{lem:A_1(i) is chordal}
    $A_1(i)$ is chordal for each $i\in [5]$.
\end{lemma}

\begin{proof}
    Suppose that $C'=v_1-v_2-v_3-v_4-v_5$ is an induced $C_5$ in $A_1(i)$. Let \[N=B_i\cup A_1(i)\cup A_2(i-1)\cup A_2(i)\cup A_2(i+2)\cup A_3(i-1)\cup A_3(i)\cup A_3(i+1)\cup A_5.\] Note that $N(C')\subseteq N$. Let $K$ be the component of $G-(N(C')\cup C)$ containing $B_{i-2}\cup B_{i-1}\cup B_{i+1}\cup B_{i+2}$. 
    The notations $S(v_3,v_4)$ and $S(v_5,v_1,v_2)$ below are the same as defined in Section \ref{sec:non-neighbors of C5}.
    Since $K$ is not a clique, $K$ has two non-adjacent neighbors $p$ and $q$ by Lemma \ref{lem:non-neighbors of C5}. By symmetry, we may assume that $p\in S(v_5,v_1,v_2)$ and $q\in S(v_3,v_4)$. Since $q$ is complete to $K$ by Lemma \ref{lem:non-neighbors of C5} and each vertex in $N(C')\setminus A_5$ has a non-neighbor in $K$, we have $q\in A_5$. 
    If $p\in N(C')\setminus (A_1(i)\cup A_2(i+2)\cup A_5)$, then one can extend the induced $P_4=v_3-v_4-v_5-p$ to an induced $P_7$ by using vertices in  $B_{i-2}\cup B_{i-1}\cup B_{i+1}\cup B_{i+2}$. So $p\in (A_1(i)\cup A_2(i+2)\cup A_5)$. Since $q\in A_5$ is adjacent to any vertex in $A_5$ or any vertex in $A_2(i+2)$ that has a neighbor in $A_1(i)$ and $pq\notin E(G)$, we have $p\in A_1(i)$. 

    Let $S$ be the set of vertices in $K$ that is a neighbor of any such vertex $p\in S(v_5,v_1,v_2)$. Since $A_1(i)$ is anticomplete to $V(H)\setminus B_i$, $K\setminus S$ contains $V(H)\setminus B_i$.
    Let $L$ be a component of $K\setminus S$ containing $V(H)\setminus B_i$.
     We show that $N(L)$ is a clique. By Lemma \ref{lem:non-neighbors of C5} (3), $N(L)\setminus S$ is a clique. Note that any vertex $N(L)\cap S$ is complete to $L$ for otherwise there is an induced $P_7$ in $G$. Since any neighbor of a vertex in $A_1(i)$ which is not in $A_5$ has a non-neighbor in $L$, $N(L)\cap S\subseteq A_5$ and so $N(L)\cap S$ is a clique. Next we show that $N(L)\cap S$ is complete to $N(L)\setminus S$. Suppose by contradiction that $a\in N(L)\cap S$ and $b\in N(L)\setminus S$ are not adjacent. 
    Since $a$ is complete to $L$, we can take a common neighbor $c\in L$ of $a$ and $b$. By definition of $S$, $a$ has a neighbor $r\in S(v_5,v_1,v_2)$. Since $r-a-c-b-r$ is not an induced $C_4$, $br\notin E(G)$. If $b$ is adjacent to both $v_2$ and $v_5$, then $r-v_2-b-v_5-r$ is an induced $C_4$. So we may assume that $bv_2\notin E(G)$ by symmetry. Then $v_4-v_3-v_2-r-a-c-b$ is a bad $P_7$. This proves that $N(L)\cap S$ is complete to $N(L)\setminus S$ and so $N(L)$ is a clique, a contradiction.
\end{proof}

\begin{lemma}\label{lem:A_2(i) is chordal}
    $A_2(i-2)$ is chordal for each $i\in [5]$. 
\end{lemma}

\begin{proof}
    Suppose that $C'=v_1-v_2-v_3-v_4-v_5$ is an induced $C_5$ in $A_2(i-2)$. Let 
    \[N=B_{i-2}\cup B_{i-1}\cup A_1(i-2)\cup A_1(i-1)\cup A_1(i+1)\cup A_2(i-2)\cup A_3(i-2)\cup A_3(i-1)\cup A_5.
    \] 
    Note that $N(C')\subseteq N$. Let $K$ be the component of $G-(N(C')\cup V(C'))$ containing $B_i\cup B_{i+1}\cup B_{i+2}$. Since $K$ is not a clique, $K$ has two non-adjacent neighbors $p$ and $q$ by Lemma \ref{lem:non-neighbors of C5}.
     The notations $S(v_3,v_4)$ and $S(v_5,v_1,v_2)$ below are the same as defined in Section \ref{sec:non-neighbors of C5}.
    By symmetry, we may assume that $p\in S(v_5,v_1,v_2)$ and $q\in S(v_3,v_4)$. Since $q$ is complete to $K$ by Lemma \ref{lem:non-neighbors of C5} and each vertex in $N(C')\setminus A_5$ has a non-neighbor in $K$, we have $q\in A_5$. By the choice of $C=C_5$ and $q\in A_5$, we have 
    $p\in A_1(i-2)\cup A_1(i-1)\cup A_2(i-2)$. 

    Let $S$ be the set of vertices in $K$ that is a neighbor of any such vertex $p\in S(v_5,v_1,v_2)$. Since $A_1(i-2)\cup A_1(i-1)\cup A_2(i-2)$ is anticomplete to $B_i\cup B_{i+1}\cup B_{i+2}$, we have $B_i\cup B_{i+1}\cup B_{i+2}\subseteq K\setminus S$. Let $L$ be the component of $K\setminus S$ containing $B_i\cup B_{i+1}\cup B_{i+2}$. We show that $N(L)$ is a clique. By Lemma \ref{lem:non-neighbors of C5}, $N(L)\setminus S$ is a clique. Note that any vertex $N(L)\cap S$ is complete to $L$ for otherwise there is an induced $P_7$ in $G$. Since any vertex in $H\cup A_1\cup A_2$ has a non-neighbor in $B_i\cup B_{i+1}\cup B_{i+2}$, $(N(L)\cap S)\cap (H\cup A_1\cup A_2)=\emptyset$. Since any vertex in $A_3$ has a non-neighbor in $B_i\cup B_{i+1}\cup B_{i+2}$, $N(L)\cap S$ contains no vertex from $A_3$. So $N(L)\cap S\subseteq A_5$ and so is a clique.

    Next we show that $N(L)\cap S$ is complete to $N(L)\setminus S$. Suppose by contradiction that $a\in N(L)\cap S$ and $b\in N(L)\setminus S$ are not adjacent. 
    Since $a$ is complete to $L$, we can take a common neighbor $c\in L$ of $a$ and $b$. By definition of $S$, $a$ has a neighbor $r\in S(v_5,v_1,v_2)$. Since $r-a-c-b-r$ is not an induced $C_4$, $br\notin E(G)$. If $b$ is adjacent to both $v_2$ and $v_5$, then $r-v_2-b-v_5-r$ is an induced $C_4$. So we may assume that $bv_2\notin E(G)$ by symmetry. Then $v_4-v_3-v_2-r-a-c-b$ is a bad $P_7$. This proves that $N(L)\cap S$ is complete to $N(L)\setminus S$ and so $N(L)$ is a clique, a contradiction.
\end{proof}

\begin{lemma}\label{lem:no common neighbors in A_1(i)}
    For any two non-adjacent vertices $u,v\in A_3(i)\cup B_i$, $u$ and $v$ cannot have a common neighbor in $A_1(i)$.
\end{lemma}

\begin{proof}
    Suppose not. Let $x\in A_1(i)$ be a common neighbor of $u$ and $v$. Since $A_1(i)$ is anticomplete to $B_{i-1}\cup B_{i+1}$ and $G$ is $C_4$-free, $u$ and $v$ have disjoint neighborhoods in each of $B_{i-1}$ and $B_{i+1}$. This contradicts Lemma \ref{lem:nonedge in A_3(i)}.
\end{proof}

\begin{lemma}\label{lem:clique components of A_1(i)}
    Any component of $A_1(i)\setminus N(A_2(i+2))$ is complete to a pair of non-adjacent vertices in $N(A_2(i+2))\cap A_1(i)$ and so is a clique.  
\end{lemma}

\begin{proof}
    Let $S=N(A_2(i+2))\cap A_1(i)$.
    Let $K$ be a component of $A_1(i)\setminus S$.  Let $N=B_i\cup A_2(i-1)\cup A_2(i)\cup  A_3(i-1)\cup A_3(i)\cup A_3(i+1)\cup A_5$. Note that $N(K)\subseteq N\cup S$. By Proposition \ref{prop:component of A_1(i)}, $N(K)\setminus S$ is a clique. Next we show that $N(K)\setminus S$ is complete to $N(K)\cap S$. Suppose by contradiction that $s\in N(K)\cap S$ is not adjacent to $t\in N(K)\setminus S$. Let $r\in A_2(i+2)$ be a neighbor of $s$. If $t\in B_i$, then $t-K-s-r-B_{i-2}-B_{i-1}-t$ is an induced cyclic sequence. So $t\notin B_i$. By Lemma \ref{lem:A2i1tocompA1}, $t\notin A_2(i-1)\cup A_2(i)$. By Lemma \ref{lem:A3tocompA1}, $t\notin A_3(i-1)\cup A_3(i)\cup A_3(i+1)$. 
    By the choice of $C$, $t\notin A_5$. This proves that $N(K)\setminus S$ is complete to $N(K)\cap S$. Let $a,b\in N(K)\cap S$ be two non-adjacent vertices. If $a,b$ have a common neighbor in $K$, then $a,b$ have disjoint neighborhoods in $A_2(i+2)$ which are complete to each other by $G$ is $C_6$-free. Suppose that $a,b$ do not have a common neighbor in $K$. Then each of $a$ and $b$ mixes on $K$. If $a$ mixes on $uv\in K$, then
    Since $u-v-a-(N(a)\cap A_2(i+2))-q_{i+2}-q_{i+1}-B_i$ is not a bad $P_7$, $a$ is complete to $B_i$. By symmetry, $b$ is complete to $B_i$. So $a,b$ have a common neighbor in $B_i$, which implies that $a,b$ have disjoint and complete neighborhoods in $A_2(i+2)$. If there is an induced $P_4=a-b-c-d$ such that $d\in N(K)\cap S$, $c\in K$ and $a,b\in (N(K)\cap S)\cup K$, then $a-b-c-d-(N(d)\cap A_2(i+2))-q_{i+2}-q_{i+1}$ is an induced $P_7$. 
    
    Now all three conditions of Lemma \ref{lem:complete subgraphs} are satisfied for $K$ where $X=S$ and $Y=A_2(i+2)$, and so we are done.
    \end{proof}

\begin{lemma}\label{lem:A_5 is complete to A_1}
    $A_5$ is complete to $A_1$.
\end{lemma}

\begin{proof}
    We show that $A_5$ is complete to $A_1(i)$. Let $S=N(A_2(i+2))\cap A_1(i)$. By Lemma \ref{lem:A5 vs A1A2}, $A_5$ is complete to $S$. By Lemma \ref{lem:clique components of A_1(i)}, each component of $A_1(i)\setminus S$ is complete to a pair of non-adjacent vertices in $S$. Since $G$ is $C_4$-free, $A_5$ is complete to that component. 
\end{proof}

\begin{lemma}[Simplicial vertices in $A_1(i)$]\label{lem:simplicial vertices in A_1(i)}
    Let $K$ be a component of $A_1(i)$ and $s$ be a $(K,B_i)$-minimal simplicial vertex.
    Then $N(s)\setminus A_2(i+2)$ is a clique and so 
    $N(s)\cap A_2(i+2)$ is a clique of size $>\ceil{\frac{\omega}{4}}$ complete to $B_{i+2}\cup B_{i-2}$.
\end{lemma}

\begin{proof} 
    By Proposition \ref{prop:component of A_1(i)} applying to $\{s\}$, $N(s)\setminus (K\cup A_2(i+2))$ is a clique. Next we investigate the neighbor of $s$ in $K$ and the neighbor of $s$ in $N(s)\setminus (A_2(i+2)\cup K)\subseteq B_i\cup A_1(i)\cup A_2(i-1)\cup A_2(i)\cup A_3(i-1)\cup A_3(i)\cup A_3(i+1)\cup A_5$.

    By Lemma \ref{lem:basic properties of components of A_1(i)} (5), $K$ has neighbors in at most one of $A_3(i-1)$, $A_3(i)$ and $A_3(i+1)$. If $K$ is anticomplete to $A_3(i-1)\cup A_3(i+1)$, we may assume by Lemma \ref{lem:basic properties of components of A_1(i)} (1) that $K$ is anticomplete to $A_2(i)$. So $N(s)\setminus A_2(i+2)\subseteq B_i\cup A_2(i-1)\cup A_3(i)\cup A_5\cup K$. If $K$ has a neighbor in $A_3(i-1)$, then $K$ is anticomplete to $A_2(i)$ by Lemmma \ref{lem:basic properties of components of A_1(i)} (2). Then $N(s)\setminus A_2(i+2)\subseteq B_i\cup A_2(i-1)\cup A_3(i-1)\cup A_5\cup K$. In either case, we have $N(s)\setminus A_2(i+2)\subseteq B_i\cup A_2(i-1)\cup A_3(j)\cup A_5\cup K$ for $j\in \{i-1,i\}$.

    Recall that $N(s)\setminus (A_2(i+2)\cup K)$ is a clique. If $N(s)\cap B_i$ is not complete to $N(s)\cap K$, there is an induced $P_4$ with one end in $
    B_i$ and the other three vertices in $A_1(i)$ by Lemma \ref{lem:minimal simplicial vertex}. This contradicts Lemma \ref{lem:induced P3in A_1(i)}.
    So $N(s)\cap B_i$ is complete to $N(s)\cap K$. By Lemma \ref{lem:A2i1tocompA1}, $N(s)\cap A_2(i-1)$ is complete to $N(s)\cap K$. 
    By Lemma \ref{lem:A3tocompA1}, $N(s)\cap A_3(j)$ is complete to $N(s)\cap K$ for $j\in \{i-1,i\}$.
    By Lemma \ref{lem:A_5 is complete to A_1},  $N(s)\cap K$ is complete to $N(s)\cap A_5$. 
    This completes the proof.
\end{proof}

Let $A_2^1(i-2)$ be the set of vertices in $A_2(i-2)$ that have a neighbor in $A_1(i-1)$ and $A_2^1(i+1)$ be the set of vertices in $A_2(i+1)$ that have a neighbor in $A_1(i+1)$.

\begin{lemma}\label{lem:A_2^1(i-2) is universal in A_2(i-2)}
    Each vertex in $A^1_2(i-2)$ is a universal vertex in $A_2(i-2)$. By symmetry, each vertex in $A^1_2(i+1)$ is a universal vertex in $A_2(i+1)$. 
\end{lemma}

\begin{proof}
    Suppose not. Let $a\in A^1_2(i-2)$ have a non-neighbor $d\in A_2(i-2)$. By Lemma \ref{lem:nbr of component of A_1(i-1) in A_2(i-2)}, $a$ is complete to $B_{i-2}$ and thus $a$ and $d$ have a common neighbor $v\in B_{i-2}$.
    By definition of $A_2^1(i-2)$, $a$ has a neighbor in a component $K$ of $A_1(i-1)$. By Lemma \ref{lem:A2i1tocompA1}, $a$ is complete to $K$. By Proposition \ref{prop:component of A_1(i)}, $K$ has a vertex $b$ that is adjacent to some vertex $c\in A_2(i+1)$. Then $d-v-a-b-c-q_{i+1}-q_i$ is an induced $P_7$, a contradiction.
\end{proof}

\subsection{Two non-empty $A_1(i)$}\label{sec:two non-empty A_1(i)}

\medskip
The global assumption in this subsection is that $A_1(i-1)$ and $A_1(i+1)$ are not empty. By Proposition \ref{prop:component of A_1(i)}, there is an edge between $A_1(i-1)$ and $A_2(i+1)$ and there is an edge between $A_1(i+1)$ and $A_2(i-2)$. Let $x\in A_1(i-1)$, $y\in A_1(i+1)$, $z\in A_2(i+1)$ and $w\in A_2(i-2)$ with $xz,yw\in E(G)$. By Lemma \ref{lem:opposite A_1(i) and A_2(j)}, $z$ is complete to $B_{i+1}\cup B_{i+2}$ and $w$ is complete to $B_{i-1}\cup B_{i-2}$. By Lemma \ref{lem:blowup-twoneighbor}, $B_{i-1}$ is complete to $B_{i-2}$ and $B_{i+1}$ is complete to $B_{i+2}$. Recall that $\omega=\omega(G)$.

\begin{lemma}\label{lem:two opposite edges}
    Suppose that $x'\in A_1(i-1)$ has a neighbor $z'\in A_2(i+1)$ and $y'\in A_1(i+1)$ has a neighbor $w'\in A_2(i-2)$. Then exactly one of $x'w'$ and $y'z'$ is an edge.
\end{lemma}

\begin{proof}
    By Lemmas \ref{A1A2} and \ref{lem:anticompleteA3-1}, $x'y', z'w' \notin E(G)$. By Lemma \ref{lem:opposite A_1(i) and A_2(j)}, $z'$ is complete to $B_{i+1}$ and $w'$ is complete to $B_{i-1}$. If $x'w',y'z'\notin E(G)$, then $x'-N_{B_{i-1}}(x')-w'-y'-N_{B_{i+1}}(y')-z'-x'$ is an induced $C_6$. If $x'w', y'z'\in E(G)$, then $x'-w'-y'-z'-x'$ is an induced $C_4$. So exactly one of $x'w'$ and $y'z'$ is an edge.  
\end{proof}

\begin{lemma}\label{lem:reduce to basic graph}
   If $xw\in E(G)$, then $x$ is complete to $B_{i-1}$ and $B_i$ is complete to $B_{i+1}$. By symmetry, if $yz\in E(G)$, then $y$ is complete to $B_{i+1}$ and $B_i$ is complete to $B_{i-1}$.
\end{lemma}

\begin{proof}
    By Lemma \ref{lem:two opposite edges}, exactly one of $xw$ and $yz$ is an edge. By symmetry, we may assume that $xw\in E(G)$ and $yz\notin E(G)$. If $x$ has a non-neighbor $s\in B_{i-1}$, then $x-w-s-q_i-q_{i+1}-z-x$ is an induced $C_6$. So $x$ is complete to $B_{i-1}$. If $c\in B_{i+1}$ is adjacent to $b\in B_{i}$ but not adjacent to $a\in B_i$, then $a-b-c-(N(c)\cap B_{i+2})-q_{i-2}-w-x$ is an induced $P_7$. So $B_i$ is complete to $B_{i+1}$.
\end{proof} 

\begin{lemma}\label{lem:A_3(i-1) is empty}
    $A_3(i-1)\cup A_3(i+1)=\emptyset$.
\end{lemma}

\begin{proof}
    By Lemma \ref{lem:two opposite edges}, we may assume that $xw\in E(G)$.
    By Lemma \ref{lem:reduce to basic graph}, $B_{i+1}$ is complete to $B_i\cup B_{i+2}$. It follows from Lemma \ref{lem:blowup-threeneighbor} that $A_3(i+1)=\emptyset$.

    Suppose that $A_3(i-1)$ contains a vertex $t$.
    By Lemma \ref{lem:blowup-threeneighbor} (1), let $b\in B_{i-1}$ be a non-neighbor of $t$.
    Since $B_{i-1}$ is complete to $B_{i-2}$, we can take a common neighbor $c\in B_{i-2}$ of $t$ and $b$.
    Let $d\in N_{B_i}(t)$ and so $bd\notin E(G)$ by Lemma \ref{lem:blowup-threeneighbor} (3). By Lemma \ref{lem:reduce to basic graph}, $xb\in E(G)$. 
    If $tx \in E(G)$, then $t - c - b - x - t$ is an induced $C_4$. So $tx \notin E(G)$. By Lemma \ref{lem:anticompleteA3-1}, $tz \notin E(G)$. Then $t-c-b-x-z-q_{i+1}-d-t$ is an induced $C_7$, a contradiction.
\end{proof}

\begin{lemma}\label{lem:A_2=A_2(i-2)+A_2(i+1)}
    $A_2=A_2(i-2)\cup A_2(i+1)$. 
\end{lemma}

\begin{proof}
    By Lemma \ref{lem:two opposite edges}, we may assume by symmetry that $xw\in E(G)$.
    We show that $A_2(i-1)\cup A_2(i+2)\cup A_2(i)=\emptyset$. If $A_2(i-1)$ contains a vertex $t$, then $x-w-q_{i-2}-q_{i+2}-q_{i+1}-(N(t)\cap B_i)-t$ is a bad $P_7$. If $A_2(i+2)$ contains a vertex $t$, then $t-(N(t)\cap B_{i-2})-w-x-z-q_{i+1}-q_i$ is an induced $P_7$. This shows that $A_2(i-1)\cup A_2(i+2)=\emptyset$. 

    It remains to show that $A_2(i)=\emptyset$. Suppose by contradiction that
    $A_2(i)$ contains a vertex $a$. If $ay\in E$, $(N(a)\cap B_i)-a-y-w-x-z-q_{i+2}$ is an induced $P_7$. So $y$ is anticomplete to $A_2(i)$. If $a$ has a non-neighbor $b\in B_{i+1}$, then $a-(N(a)\cap B_i)-b-z-x-w-q_{i-2}$ is an induced $P_7$. So $A_2(i)$ is complete to $B_{i+1}$. 
    If $a$ is adjacent to some vertex $y'\in A_1(i+1)\setminus \{y\}$, then $y'-a-(N(a)\cap B_i)-q_{i-1}-x-z-q_{i+2}$ is a bad $P_7$. So $A_2(i)$ is anticomplete to $A_1=A_1(i-1)\cup A_1(i+1)$. By Lemmas \ref{A2A2}, \ref{lem:anticompleteA3-1}, \ref{anti-A2A3}, \ref{lem:blowup-twoneighbor}, \ref{lem:neighbors of A_2(i) in A_3(i+1) and A_3(i)} and \ref{lem:A_3(i-1) is empty},  any component of $A_2(i)$ have a neighbor in $A_3(i)$. Let $K$ be the component of $A_2(i)$ containing $a$.

    By Lemmas \ref{A1A2},  \ref{A2A2} and \ref{lem:anticompleteA3-1}, $N(K)\subseteq H\cup A_3(i)\cup A_5$. Note that any vertex in $N(K)\cap A_5$ is universal in $N(K)$ and $N(K)\cap H$ is a clique by Lemma \ref{lem:blowup-twoneighbor}. Next we show that $N(K)\cap H$ is complete to $N(K)\cap A_3(i)$. Suppose first that $a\in A_2(i)$ is adjacent to $t\in A_3(i)$. If $a$ has a neighbor $b\in B_i\setminus N(t)$, then $b-a-t-N_{B_{i-1}}(t)-q_{i-2}-q_{i+2}-z$ is an induced $P_7$. So $a$ is anticomplete to $B_i\setminus N(t)$, i.e., $N(a)\cap B_i\subseteq N(t)\cap B_i$. By Lemma \ref{lem:blowup-twoneighbor}, $t$ is complete to $B_{i+1}$. This shows that $N(K)\cap A_3(i)$ is complete to $B_{i+1}$. Now suppose that $t\in A_3(i)\cap N(K)$ is not adjacent to $s\in N(K)\cap B_i$. By Lemma \ref{lem:blowup-threeneighbor} and the fact that $B_i$ is complete to $B_{i+1}$, $s$ and $t$ have disjoint neighborhoods in $B_{i-1}$. Let $P$ be a shortest path from $s$ to $t$ such that the internal of $P$ is contained in $K$. 
    Since $N(a)\cap B_i\subseteq N(t)\cap B_i$, $P$ has at least $4$ vertices. Then $s - P - t - N_{B_{i-1}}(t) - q_{i-2} - q_{i+2}$ is an induced path of order at least $7$. This proves that $N(K)\cap H$ is complete to $N(K)\cap A_3(i)$.
    Since $N(K)$ is not a clique, $N(K)\cap A_3(i)$ contains two non-adjacent vertices $t$ and $t'$. The previous argument shows that $t$ and $t'$ are complete to $B_{i+1}$, and so $t$ and $t'$ have disjoint neighborhoods in $B_{i-1}$. Moreover, $t$ and $t'$ have a common neighbor $a\in K$ by Lemma \ref{lem:neighbors of components of A_2(i) in A_3(i+1)}. If $x$ is not adjacent to $t$ and $t'$, then $t'-a-t-(N(t)\cap B_{i-1})-x-z-q_{i+2}$ is an induced $P_7$. So we may assume by symmetry that $xt\in E$. Then $x-t-q_{i+1}-z-x$ is an induced $C_4$.
\end{proof}

\begin{lemma}\label{lem:A_3(i)}
    The following holds for $A_3(i)$.
    \begin{itemize}
        \item[$(1)$] $A_3(i)$ is anticomplete to $A_1(i-1)\cup A_1(i+1)$.
        
        \item[$(2)$] If $xw\in E(G)$, then $A_3(i)$ is complete to $B_{i+1}$. By symmetry, if $yz\in E(G)$, then $A_3(i)$ is complete to $B_{i-1}$.
        
        \item[$(3)$] Suppose that $A_3(i)\neq \emptyset$. If $xw\in E(G)$, then  $|B_{i-1}|>\frac{\omega}{2}$. By symmetry, if $yz\in E(G)$, then  $|B_{i+1}|>\frac{\omega}{2}$.
    \end{itemize}
\end{lemma}

\begin{proof} We prove one by one.

\medskip
\noindent (1) Suppose that $p\in A_1(i-1)\cup A_1(i+1)$ is adjacent to $t\in A_3(i)$. By Lemma \ref{lem:two opposite edges}, we may assume by symmetry that $xw\in E(G)$ and so $B_{i+1}$ is complete to $B_i$ by Lemma \ref{lem:reduce to basic graph}. By Lemma \ref{lem:blowup-threeneighbor} (3), $B_i\setminus N(t)$ is anticomplete to $B_{i-1}\cap N(t)$. By Lemma \ref{lem:single vertex in A_3(i)} (1), $q_i\in B_i\cap N(t)$ and $q_{i-1}\in B_{i-1}\setminus N(t)$. 
Then $p-t-q_i-q_{i-1}-q_{i-2}-q_{i+2}-z$ is a bad $P_7$. 

\medskip
\noindent (2) If $t\in A_3(i)$ is not adjacent to $a\in B_{i+1}$, then $t-(N(t)\cap B_i)-a-q_{i+2}-q_{i-2}-w-x$ is an induced $P_7$.

\medskip
\noindent (3) By Lemma \ref{lem:A_2=A_2(i-2)+A_2(i+1)} and (1), $A'_3(i)$ is anticomplete to $A_1\cup A_2$. 
By Lemmas \ref{lem:small vertex argument for simplicial vertices in A_3(i-2)} and  \ref{lem:smv on w}, there exist vertices $v\in A_3(i)$ and $u\in B_i\setminus N(v)$ such that $|N(v)\cap B_{i-1}|,|N(u)\cap B_{i-1}|>\ceil{\frac{\omega}{4}}$. 
Since $B_i$ is complete to $B_{i+1}$ by Lemma \ref{lem:reduce to basic graph}, $N(v)\cap B_{i-1}$ and $N(u)\cap B_{i-1}$ are disjoint. So $|B_{i-1}|>\frac{\omega}{2}$.
\end{proof}

Recall that $A_2^1(i-2)$ is the set of vertices in $A_2(i-2)$ that have a neighbor in $A_1(i-1)$ and $A_2^1(i+1)$ is the set of vertices in $A_2(i+1)$ that have a neighbor in $A_1(i+1)$.
Let $A_2^0(i-2)=A_2(i-2)\setminus A_2^1(i-2)$,  $X=A^0_2(i-2)\cap N(A_1(i+1))$ and $Y=A^0_2(i-2)\setminus X$. Let $A_2^0(i+1)=A_2(i+1)\setminus A_2^1(i+1)$, $X'=A^0_2(i+1)\cap N(A_1(i-1))$ and $Y'=A^0_2(i+1)\setminus X'$.

\begin{lemma}\label{lem:A_3(i-2)}
    The following properties hold for $A_3(i-2)\cup A_3(i+2)$.

    \begin{itemize}
        \item[$(1)$] $A_3(i-2)$ is complete to $A_2^1(i-2)\cup X$. By symmetry, $A_3(i+2)$ is complete to $A_2^1(i+1)\cup X'$. 
        \item[$(2)$] $A_3(i-2)$ is complete to $B_{i-1}$. By symmetry, $A_3(i+2)$ is complete to $B_{i+1}$.
        \item[$(3)$] $A_3(i-2)$ is anticomplete to $A_1(i-1)$. By symmetry, $A_3(i+2)$ is complete to $A_1(i+1)$.
    \end{itemize}
\end{lemma}

\begin{proof} We prove one by one.

\medskip
\noindent (1) Suppose by contradiction that $t\in A_3(i-2)$ and $a\in A_2^1(i-2)\cup X$ are not adjacent. Suppose first that $a\in A_2^1(i-2)$. By definition of $A_2^1(i-2)$, $a$ has a neighbor in a component $K$ of $A_1(i-1)$. By Lemma \ref{lem:A2i1tocompA1}, $a$ is complete to $K$. By Proposition \ref{prop:component of A_1(i)}, $K$ has a vertex $b$ that is adjacent to some vertex $c\in A_2(i+1)$. Then $t-N_{B_{i-2}}(t)-a-b-c-q_{i+1}-q_i$ is an induced sequence. Now suppose that $a\in X$. By definition of $X$, $a$ has a neighbor $b\in A_1(i+1)$. Then $a-B_{i-1}-t-(N(t)\cap B_{i+2})-(N(b)\cap B_{i+1})-b-a$ is an induced cyclic sequence. This proves (1).

\medskip
\noindent (2) By definition, $w\in A_2^1(i-2)\cup X$. By (1), $w$ is complete to $A_3(i-2)$. By Lemma \ref{lem:blowup-twoneighbor}, (2) holds.
    
\medskip
\noindent (3)
    By Lemma \ref{lem:A3tocompA1}, $t$ cannot mix on any component $K$ of $A_1(i-1)$. By Proposition \ref{prop:component of A_1(i)}, $K$ has a neighbor $p'\in A_2(i+1)$. Let $p\in K$ be a neighbor of $p'$. If $tp\in E(G)$, then $t-p-p'-N_{B_{i+2}}(t)-t$ is an induced $C_4$. So $t$ is not adjacent to $p$ and thus anticomplete to $K$.
\end{proof}

\begin{lemma}\label{lem:no bad P4 when two non-empty A_1(i)}
    $G$ has no induced $P_4=a-b-c-d$ contained in $A_2(i-2)\cup B_{i-1}\cup B_{i-2}\cup A_3(i-2)$ such that $a$ has a neighbor in $B_{i+2}$ that is not a neighbor of $b$. By symmetry, $G$ has no induced $P_4=a-b-c-d$ contained in $A_2(i+1)\cup B_{i+1}\cup B_{i+2}\cup A_3(i+2)$ such that $a$ has a neighbor in $B_{i-2}$ that is not a neighbor of $b$.
\end{lemma}

\begin{proof}
    Suppose not. Let $a'\in B_{i+2}$ be a neighbor of $a$ not adjacent to $b$. Since $G$ is $C_4$-free, $a'c\notin E(G)$. 
    Note that $a\in B_{i-2}\cup A_3(i-2)$. If $d\in B_{i-2}\cup A_3(i-2)$, then $a'd\notin E(G)$ since $B_{i-1}$ is complete to $B_{i-2}\cup A_3(i-2)$ by Lemma \ref{lem:A_3(i-2)} (2).
    If $d\in A_2(i-2)\cup B_{i-1}$, then $a'd\notin E(G)$ by definition.
    Then $d-c-b-a-a'-q_{i+1}-q_i$ is an induced $P_7$.
\end{proof}

\begin{lemma}\label{lem:B_{i-1} vs A_1(i-1)}
    Each vertex in $B_{i-1}$ is either complete or anticomplete to any component $K$ of $A_1(i-1)$. By symmetry, each vertex in $B_{i+1}$ is either complete or anticomplete to any component of $A_1(i+1)$. 
\end{lemma}

\begin{proof}
    If $c\in B_{i-1}$ mixes on an edge $ab\in K$, then $a-b-c-(N(c)\cap B_{i-2})-q_{i+2}-(N(y)\cap B_{i+1})-y$ is an induced $P_7$.  
\end{proof}

\begin{lemma}\label{lem:simplicial vertices in A_1(i-1)}
    Let $K$ be a component of $A_1(i-1)$. Then any simplicial $s$ vertex of $K$ has $N(s)\setminus A_2(i+1)$ is a clique and so 
    $|N(s)\cap A_2(i+1)|>\ceil{\frac{\omega}{4}}$.
    By symmetry, any simplicial $s$ vertex of $A_1(i+1)$ has $N(s)\setminus A_2(i-2)$ is a clique and so $|N(s)\cap A_2(i-2)|>\ceil{\frac{\omega}{4}}$.
\end{lemma}

\begin{proof} 
    By Proposition \ref{prop:component of A_1(i)} applying to $\{s\}$, $N(s)\setminus (K\cup A_2(i+1))$ is a clique. Next, we investigate the neighbor of $s$ in $K$ and the neighbor of $s$ in $N(s)\setminus A_2(i+1)$.
    By Lemma \ref{lem:A_3(i-2)} (3) and Lemma \ref{lem:A_3(i)} (1), 
    $N(s)\setminus A_2(i+1)\subseteq B_{i-1}\cup A^1_2(i-2)\cup A_5\cup K$. 
    By Lemma \ref{lem:A2i1tocompA1}, $N(s)\cap K$ is complete to $N(s)\cap A^1_2(i-2)$.
    By Lemma \ref{lem:B_{i-1} vs A_1(i-1)}, 
    $N(s)\cap K$ is complete to $N(s)\cap B_{i-1}$. By Lemma \ref{lem:A_5 is complete to A_1}, $N(s)\cap K$ is complete to $N(s)\cap A_5$. This completes the proof.
\end{proof}

\begin{lemma}\label{lem:S is chordal}
    $A_2(i-2)\cup B_{i-1}\cup B_{i-2}\cup A_3(i-2)$ is chordal. By symmetry, $A_2(i+1)\cup B_{i+1}\cup B_{i+2}\cup A_3(i+2)$ is chordal.
\end{lemma}

\begin{proof}
    We first show that $A_2(i-2)\cup B_{i-1}\cup B_{i-2}$ is chordal. Since $B_{i-1}\cup B_{i-2}$ is a clique, any induced $C_5$ contains at most two vertices from $B_{i-1}\cup B_{i-2}$ and so any induced $C_5$ in $B_{i-1}\cup B_{i-2}\cup A_2(i-2)$ contains an induced $P_4=a-b-c-d$ where $a,b,c\in A_2(i-2)$ and $d\in B_{i-1}\cup B_{i-2}$ by Lemma \ref{lem:A_2(i) is chordal}. This contradicts Lemma \ref{lem:induced P3in A_2(i)}. So $A_2(i-2)\cup B_{i-1}\cup B_{i-2}$ is chordal. 

    Suppose that $A_3(i-2)\cup B_{i-2}\cup A_2(i-2)$ contains an induced $C_5=v_1-v_2-v_3-v_4-v_5-v_1$. 
    Suppose first that $v_1\in A_2(i-2)$. If $v_2\in B_{i-2}\cup A_3(i-2)$, then $v_4-v_5-v_1-v_2-N_{B_{i+2}}(v_2)-q_{i+1}-q_i$ is a bad $P_7$. So $v_2,v_5\in A_2(i-2)$. By the same argument, $v_3,v_4\in A_2(i-2)$. This contradicts Lemma \ref{lem:A_2(i) is chordal}. So we may assume that $C_5$ is contained in $B_{i-2}\cup A_3(i-2)$. Since $B_{i-2}$ is a clique and $A_3(i-2)$ is $P_4$-free, we may assume that $v_1,v_2\in B_{i-2}$ and $v_3,v_4,v_5\in A_3(i-2)$. But this contradicts Lemma \ref{lem:edge in $A_3(i)$}. So $A_3(i-2)\cup B_{i-2}\cup A_2(i-2)$ is chordal. 

    If $A_2(i-2)\cup B_{i-1}\cup B_{i-2}\cup A_3(i-2)$ contains an induced $C_5=v_1-v_2-v_3-v_4-v_5$ , then we may assume that $v_1\in B_{i-1}$ and $v_2\in A_3(i-2)$. But then $v_4-v_5-v_1-v_2-N_{B_{i+2}}(v_2)-N_{B_{i+1}}(y)-y$ is a bad $P_7$.
\end{proof}

\begin{lemma}\label{lem:B_{i-1} is large when two non-empty A_1(i)}
    If $A_3(i)=\emptyset$, then $|B_{i-1}|,|B_{i+1}|>\ceil{\frac{\omega}{4}}$.
\end{lemma}

\begin{proof}
    Let $b\in B_i$ such that $|N(b)\cap (B_{i-1}\cup B_{i+1})|$ is minimum. 
    By Lemma \ref{lem:A_3(i-1) is empty} and the assumption, $A_3(j)=\emptyset$ for $j\in \{i-1,i,i+1\}$.
    Then $N(b)=(N(b)\cap (B_i\cup B_{i-1}))\cup B_{i+1}$. 
    By the choice of $b$, $N(b)\cap (B_i\cup B_{i-1})$ is a clique.
    So $d(b)\le (\omega-1)+|N(b)\cap B_{i+1}|$. 
    Since $G$ has no small vertices, $|N(b)\cap B_{i+1}|>\ceil{\frac{\omega}{4}}$.
\end{proof}

\begin{lemma}\label{lem:neighbors of A_1(i-1) in A_2(i-2)}
    $A^1_2(i-2)\cup B_{i-1}\cup B_{i-2}\cup A_5$ is a clique. By symmetry, $A^1_2(i+1)\cup B_{i+1}\cup B_{i+2}\cup A_5$ is a clique.
\end{lemma}

\begin{proof}
    By Lemma \ref{lem:nbr of component of A_1(i-1) in A_2(i-2)}, $A^1_2(i-2)$ is complete to $B_{i-2}$. Suppose that $p\in A^1_2(i-2)$ is not adjacent to $q\in B_{i-1}\cup A_5$. Suppose first that $q\in A_5$. By definition of $A_2^1(i-2)$, $p$ is complete to a component $K$ of $A_1(i-1)$. By Proposition \ref{prop:component of A_1(i)}, $K$ has a vertex $a$ that has a neighbor in $A_2(i+1)$. By the choice of $C=C_5$, $qa\in E(G)$. Then $p-(N(p)\cap B_{i-2})-q-a-p$ is an induced $C_4$. So $q\in B_{i-1}$. If $q$ is complete to $K$, then $q-(N(p)\cap K)-p-(N(q)\cap B_{i-2})-q$ is an induced $C_4$. If $q$ is anticomplete to $K$, then $q-q_{i-2}-p-a-(N(a)\cap A_2(i+1))-q_{i+1}-q_i-q$ is an induced $C_7$.
    By Lemma \ref{lem:A_2^1(i-2) is universal in A_2(i-2)} and the fact that $B_{i-2}$ is complete to $B_{i-1}$, $A^1_2(i-2)\cup B_{i-1}\cup B_{i-2}\cup A_5$ is a clique. 
\end{proof}

\begin{lemma}\label{lem:simplicial vertex in Y}
    If each vertex in $A_2(i-2)\setminus Y$ is universal in $A_2(i-2)$, then any simplicial vertex of $A_2(i-2)\cup B_{i-2}\cup B_{i-1}\cup A_3(i-2)$ is not in $Y$. By symmetry, if each vertex in $A_2(i+1)\setminus Y'$ is universal in $A_2(i+1)$, then any simplicial vertex of $A_2(i+1)\cup B_{i+2}\cup B_{i+1}\cup A_3(i+2)$ is not in $Y'$. 
\end{lemma}

\begin{proof}
    Let $S=A_2(i-2)\cup B_{i-2}\cup B_{i-1}\cup A_3(i-2)$.
    Suppose by contradiction that there is a simplicial vertex $u$ of $S$ in $Y$. 
    Let $K$ be the component of $Y$ containing $u$.
    We choose a vertex $s\in N_S[u]\cap K$  simplicial in $S$ such that $|N(s)\cap A_5|$ is minimum.
    
    By definition of $Y$, $K$ is anticomplete to $A_1$. 
    By Lemma \ref{lem:A_3(i-1) is empty}, $N(K)\subseteq A_5\cup B_{i-1}\cup B_{i-2}\cup (A_2(i-2)\setminus Y)\cup A_3(i-2)$. By Lemma \ref{lem:neighbors of A_1(i-1) in A_2(i-2)}, $B_{i-2}\cup B_{i-1}\cup A^1_2(i-2)\cup A_5$ is a clique. By the choice of $C$, $X$ is complete $A_5$. By Lemma \ref{lem:opposite A_1(i) and A_2(j)}, $X$ is complete to $B_{i-2}\cup B_{i-1}$. By the assumption that each vertex in $A_2(i-2)\setminus Y$ is universal in $A_2(i-2)$, $N(K)\setminus A_3(i-2)$ is a clique.
    By Lemma \ref{lem:A_3(i-2)} (1) and (2), $A_3(i-2)$ is complete to $B_{i-1}\cup (A_2(i-2)\setminus Y) \cup A_5$. It follows that each vertex in $N(K)\setminus A'_3(i-2)$ is universal in $N(K)$. 
    By Lemma \ref{lem:A_3(i-2)} (2), any two non-adjacent vertices in $N(K)\cap A'_3(i-2)$ has disjoint neighborhoods in $B_{i+2}$. 
    If there is an induced $P_4=a-b-c-d$ with $a,b\in K\cup (N(K)\cap A'_3(i-2))$, $c\in K$ and $d\in N(K)\cap A'_3(i-2))$, then $a-b-c-d-N_{B_{i+2}}(d)-q_{i+1}-q_i$ is an induced $P_7$. 
    So all three conditions of Lemma \ref{lem:complete subgraphs} are satisfied for $K$ with $X=A'_3(i-2)$ and $Y=B_{i+2}$, and so $K$ is clique.

    We claim that $s$ is simplicial in $G$. Suppose not. 
    Since $s$ is simplicial in $S$, there exist vertices $p\in N(s)\cap Y$ and $q\in A_5$ that are not adjacent by Lemmas \ref{lem:neighbors of A_1(i-1) in A_2(i-2)} and the choice of $C$.
    By the choice of $s$ and $G$ is $C_4$-free, $p$ is not simplicial in $S$ and so has a neighbor $a\in S\setminus N(s)$. Since $G$ is $C_4$-free, $aq\notin E(G)$. Since $q\in A_5$ is complete to $S\setminus Y$, we have $a\in Y$. Now the induced path $s-p-a$ contradicts that $K$ is a clique.
\end{proof}

\begin{lemma}\label{lem:simplicial vertex in A_3(i-2) and B_{i-2}}
    If $u\in A_3(i-2)\cup B_{i-2}$ is simplicial in $A_2(i-2)\cup B_{i-2}\cup B_{i-1}\cup A_3(i-2)$, then $N(u)\setminus (B_{i+2}\cup A_3(i+2))$ is a clique and so $|N(u)\cap (B_{i+2}\cup A_3(i+2))|>\ceil{\frac{\omega}{4}}$. 
    By symmetry, if $u\in A_3(i+2)\cup B_{i+2}$ is simplicial in $A_2(i+1)\cup B_{i+2}\cup B_{i+1}\cup A_3(i+2)$. Then $N(u)\setminus (B_{i-2}\cup A_3(i-2))$ is a clique and so $|N(u)\cap (B_{i-2}\cup A_3(i-2))|>\ceil{\frac{\omega}{4}}$. 
\end{lemma}

\begin{proof}
    Let $S=A_2(i-2)\cup B_{i-2}\cup B_{i-1}\cup A_3(i-2)$. 
    Suppose that $N(u)\setminus (B_{i+2}\cup A_3(i+2))$ is not a clique.
    Since $u$ is simplicial in $S$, there exist vertices $p\in Y$ and $q\in A_5$ with $pq\notin E(G)$. 
    Let $K$ be the component of $Y\setminus N(q)$ containing $p$. Since $N(K)$ is not a clique, $N(K)$ contains two non-adjacent vertices $a,b$. Note that $a,b\in A_5\cup B_{i-1}\cup B_{i-2}\cup A_3(i-2)\cup (A_2(i-2)\setminus Y)\cup (N(q)\cap Y)$ are adjacent to $q$. So $a$ and $b$ do not have a common neighbor in $K$.
    We show that if $a\notin A_5$, then $au\in E(G)$. Suppose not. Since $q$ is adjacent to both $u$ and $a$, $u$ and $a$ have no common neighbors in $K$.
    Then there exists an induced path $P=a-\cdots-u$ with at least 4 vertices such that $P\setminus \{a,u\}\subseteq K$. This contradicts Lemma \ref{lem:no bad P4 when two non-empty A_1(i)}. 
    If $a,b\notin A_5$, $u$ is not adjacent to at least one of $a,b$ since $u$ is simplicial in $S$, a contradiction.  Therefore, $a\in A_5$ and $b\in Y$ with $bu\in E(G)$. Since $a,b$ have no common neighbors in $K$, there exist $c,d\in K$ such that $a-c-d-b$ is an induced $P_4$ and so $a-c-d-b-q-a$ is an induced $C_5$. Since $u$ is simplicial in $S$, $u$ is not adjacent to $c$. Then $u$ is not adjacent to $d$ for otherwise $a-c-d-u-a$ is an induced $C_4$. So $c-d-b-u$ contradicts Lemma \ref{lem:no bad P4 when two non-empty A_1(i)}.
\end{proof}

Next, we consider two cases. The first case is that some vertex in $X$ is not universal in $A_2(i-2)$ or some vertex in $X'$ is not universal in $A_2(i+1)$, and the second case is that every vertex in $X$ is  universal in $A_2(i-2)$ and every vertex in $X'$ is  universal in $A_2(i+1)$.

\subsubsection{$X$ or $X'$ is not universal}

By symmetry, we may assume that $a\in X'$ is not universal in $A_2(i+1)$. By definition of $X'$, $a$ has a neighbor $b\in A_1(i-1)$ and $ay\notin E(G)$. By applying Lemma \ref{lem:two opposite edges} to edges $ab$ and $yw$, we have $bw\in E(G)$ (at this point, we break the symmetry). For consistency, we rename $a$ as $z$ and rename $b$ as $x$. So the global assumption in this subsection is that $z$ has a non-neighbor $z'\in A_2(i+1)$. It follows from Lemma \ref{lem:reduce to basic graph} that $G$ contains the graph in Figure \ref{fig: a third basic graph} as an induced subgraph.

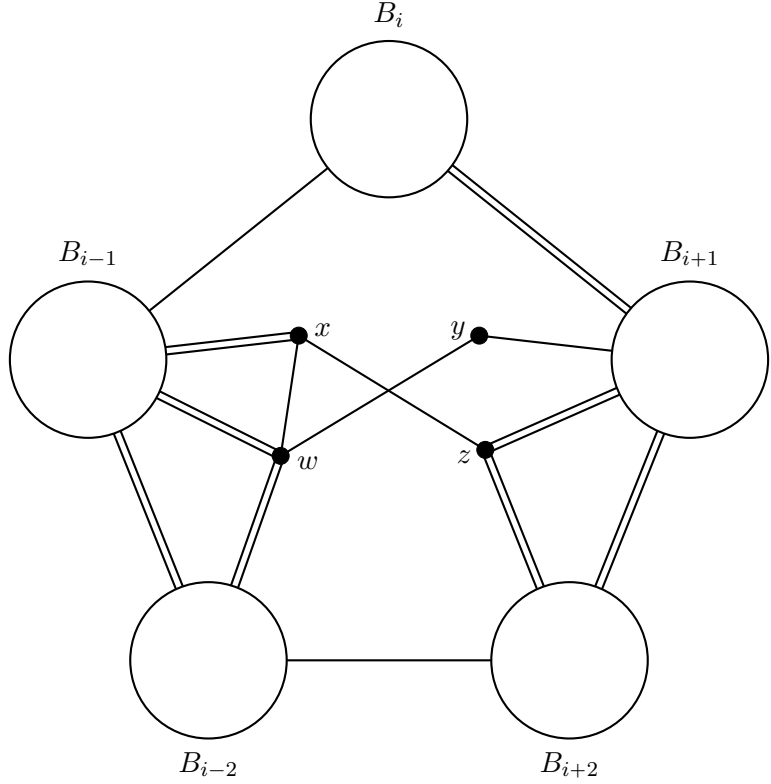
\begin{figure}[h]
    \centering
    \begin{tikzpicture}[scale=0.8]
        \tikzstyle{v}=[circle, draw, solid, fill=black, inner sep=0pt, minimum width=3pt]

        \draw[thick] (-1.5, 0.4)--(1.6, -1.5);
        \draw[thick] (-1.8, -1.6)--(1.5, 0.4);
        \draw[thick] (-1.5, 0.4)--(-1.8, -1.6);

        \draw[double distance=2pt, thick] (-1.8, -1.6)--(-5, 0);
        \draw[double distance=2pt, thick] (-1.8, -1.6)--(-3, -5);
        \draw[double distance=2pt, thick] (1.6, -1.5)--(5, 0);
        \draw[double distance=2pt, thick] (1.6, -1.5)--(3, -5);
        \draw[thick] (1.5, 0.4)--(5, 0);
        \draw[thick] (-3, -5)--(3, -5);
        \draw[thick] (-5, 0)--(0, 4);

        \draw[double distance=2pt, thick] (-5, 0)--(-3, -5);
        \draw[double distance=2pt, thick] (5, 0)--(3, -5);
        \draw[double distance=2pt, thick] (0, 4)--(5, 0);
        \draw[double distance=2pt, thick] (-1.5, 0.4)--(-5, 0);

        \filldraw[color=black, fill=white!100, thick] (0, 4) circle (1.3);
        \node [label=$B_i$] (v) at (0, 5.2){};

        \filldraw[color=black, fill=white!100, thick] (-5, 0) circle (1.3);
        \node [label=$B_{i-1}$] (v) at (-5, 1.2){};

        \filldraw[color=black, fill=white!100, thick] (5, 0) circle (1.3);
        \node [label=$B_{i+1}$] (v) at (5, 1.2){};

        \filldraw[color=black, fill=white!100, thick] (-3, -5) circle (1.3);
        \node [label=$B_{i-2}$] (v) at (-3, -7.3){};

        \filldraw[color=black, fill=white!100, thick] (3, -5) circle (1.3);
        \node [label=$B_{i+2}$] (v) at (3, -7.3){};

        \filldraw[black] (-1.5, 0.4) circle (4pt);
        \node [label = $x$] (v) at (-1.1, 0.05){};

        \filldraw[black] (-1.8, -1.6) circle (4pt);
        \node [label = $w$] (v) at (-1.35, -2.15){};

        \filldraw[black] (1.5, 0.4) circle (4pt);
        \node [label = $y$] (v) at (1.15, 0){};

        \filldraw[black] (1.6, -1.5) circle (4pt);
        \node [label = $z$] (v) at (1.25, -2.05){};
        
    \end{tikzpicture}
    \caption{A basic graph arising from two non-empty $A_1(j)$s. A double line means complete.}
    \label{fig: a third basic graph}
\end{figure}

\begin{lemma}\label{lem:z is not universal}
    $A_3(i)=\emptyset$.
\end{lemma}

\begin{proof}
    If $a\in B_{i-1}$ is adjacent to $b\in B_{i}$ but not adjacent to $c\in B_i$, then $z'-(N(z')\cap B_{i+2})-z-x-a-b-c$ is an induced $P_7$. So $B_{i-1}$ is complete to $B_{i}$. By Lemma \ref{lem:reduce to basic graph}, $B_{i+1}$ is complete to $B_i$. By Lemma \ref{lem:blowup-threeneighbor} (3), $A_3(i)=\emptyset$.
\end{proof}

\begin{lemma}\label{lem:A_3(i+2) is empty}
    $A_3(i+2)=\emptyset$.
\end{lemma}

\begin{proof}
    Suppose by contradiction that $A_3(i+2)$ contains a vertex $t$. Since $B_{i+1}$ is complete to $B_{i+2}$, $B_{i+2}\setminus N(t)$ is anticomplete to $N(t)\cap B_{i-2}$ by Lemma \ref{lem:blowup-threeneighbor}. By Lemmas \ref{A1A3} and \ref{anti-A2A3}, $tx,tw\notin E(G)$.

    \medskip
    \noindent {\bf Case 1. $z'$ has a neighbor $b$ in $B_{i+2}\setminus N(t)$.}
    If $tz\notin E(G)$, then $z'-b-z-x-w-(N(t)\cap B_{i-2})-t$ is a bad $P_7$. So $tz\in E(G)$. If $tz'\in E(G)$, then $z'-b-t-z'$ is an induced $C_4$. So $tz'\notin E(G)$ and thus $z'-b-z-t-(N(t)\cap B_{i-2})-q_{i-1}-q_i$ is an induced $P_7$.

    \medskip
    \noindent {\bf Case 2. $z'$ has no neighbors in $B_{i+2}\setminus N(t)$.} 
    By {\bf Case 1}, $A_2(i+1)\setminus N(z)$ is anticomplete to $A_1(i-1)$. 
    By Lemma \ref{lem:A_2^1(i-2) is universal in A_2(i-2)}, $A_2(i+1)\setminus N(z)$ is anticomplete to $A_1$. 
    Let $K_{z'}$ be the component of $A_2(i+1)\setminus N(z)$ containing $z'$.
    By Lemma \ref{lem:anticompleteA3-1}, $K_{z'}$ is anticomplete to $A_2\setminus N_{A_2(i+1)}(z)$. By Lemmas \ref{lem:anticompleteA3-1}, \ref{anti-A2A3}, \ref{lem:A_3(i-1) is empty} and \ref{lem:z is not universal}, $N(K_{z'})\cap A_3\subseteq A_3(i+2)$ and so $N(K_{z'})\subseteq B_{i+1}\cup N_{A_2(i+1)}(z)\cup A'_3(i+2)\cup A_5$.

    Next we show that $N(K_{z'})\cap A_5$ is complete to $N(K_{z'})\cap N_{A_2(i+1)}(z)$. For contradiction, let $p\in N(K_{z'})\cap N_{A_2(i+1)}(z)$ and $q\in N(K_{z'})\cap A_5$ be non-adjacent. By the choice of $C$, $q$ is complete to $\{x,z\}$. If $p$ and $q$ have a common neighbor $c\in K_{z'}$, then $c-p-z-q-c$ is an induced $C_4$. So $p$ and $q$ have disjoint neighborhoods in $K_{z'}$. This implies that  $p$ mixes on an edge $ab\in K_{z'}$. Then $xp\in E(G)$ for otherwise $a-b-p-z-x-q_{i-1}-q_i$ is an induced $P_7$. By Lemma \ref{lem:opposite A_1(i) and A_2(j)}, $p$ is complete to $B_{i+2}$. Then $p-x-q-q_{i+2}-p$ is an induced $C_4$.
    By Lemma \ref{lem:blowup-fiveneighbor1}, each vertex in $N(K_{z'})\cap A_5$ is universal in $N(K_{z'})$. 

    Next we show that $N(K_{z'})\cap B_{i+1}$ is complete to $N(K_{z'})\cap N_{A_2(i+1)}(z)$. For contradiction, let $p\in N(K_{z'})\cap N_{A_2(i+1)}(z)$ and $q\in N(K_{z'})\cap B_{i+1}$ be non-adjacent. By definition, $pz\in E(G)$ and $qz\in E(G)$. Since $G$ is $C_4$-free, $p$ and $q$ do not have a common neighbor in $K_{z'}$. So $p$ mixes on $K_{z'}$ and thus $xp\in E(G)$. By Lemma \ref{lem:opposite A_1(i) and A_2(j)}, $pq\in E(G)$, a contradiction.  By Lemma \ref{lem:blowup-twoneighbor}, each vertex in $N(K_{z'})\cap B_{i+1}$ is universal in $N(K_{z'})$.

   We show that $N(K_{z'})\cap N_{A_2(i+1)}(z)$ is a clique complete to $A'_3(i+2)$.  Suppose that $p,q\in N_{A_2(i+1)}(z)$  are not adjacent. Since $G$ is $C_4$-free, $p$ and $q$ have disjoint neighborhoods in $K_{z'}$. This implies that both $p$ and $q$ mix on $K_{z'}$ and so $xp,xq\in E(G)$. Then $p-x-q-q_{i+2}-p$ is an induced $C_4$. This shows that $N(K_{z'})\cap N_{A_2(i+1)}(z)$ is a clique.  Suppose that $p\in N(K_{z'})\cap N_{A_2(i+1)}(z)$ and $q\in N(K_{z'})\cap A'_3(i+2)$ be non-adjacent. Recall that $qx\notin E(G)$. If $qz\notin E(G)$, then $q-K_{z'}-\{p,z\}-x-w-B_{i-2}-q$ is an induced cyclic sequence. So $qz\in E(G)$. Since $pz\in E(G)$ and $G$ is $C_4$-free, $p$ and $q$ do not have common neighbors in $K_{z'}$. So $p$ mixes on $K_{z'}$ and so $xp\in E(G)$. Then $q-K_{z'}-p-x-w-B_{i-2}-q$ is an induced cyclic sequence. So each vertex in $N(K_{z'})\cap N_{A_2(i+1)}(z)$ is universal in $N(K_{z'})$.
   
   Since $B_{i+1}$ is complete to $A'_3(i+2)$ by Lemma \ref{lem:A_3(i-2)} (2), each pair of non-adjacent vertices in $N(K_{z'})\cap A'_3(i+2)$ has disjoint neighborhoods in $B_{i-2}$. If there is an induced $P_4=a-b-c-d$ with $a,b\in K_{z'}\cup N(K_{z'})$, $c\in K_{z'}$ and $d\in A'_3(i+2)$, then $a-b-c-d-N_{B_{i-2}}(d)-q_{i-1}-q_i$ is an induced $P_7$. So all three conditions of  Lemma \ref{lem:complete subgraphs} are satisfied for $K_{z'}$ with $X=A'_3(i+2)$ and $Y=B_{i-2}$. It follows that $K_{z'}$ is complete to two non-adjacent vertices $p,q\in A'_3(i+2)$. If $zq\notin E(G)$, then $p-z'-q-(N(q)\cap B_{i-2})-w-x-z$ is a bad $P_7$. So $zq\in E(G)$. By symmetry, $zp\in E(G)$. Then $p-z'-q-z-p$ is an induced $C_4$.
\end{proof}

\begin{lemma}\label{lem:A_1(i+1) is dominated by A_2(i-2)}
    Each vertex in $A_1(i+1)$  has a neighbor in $A_2(i-2)$.
\end{lemma}

\begin{proof}
    Let $S\subseteq A_1(i+1)$ be the set of vertices that have a neighbor in $A_2(i-2)$ and $T=A_1(i+1)\setminus S$. The lemma is equivalent to show 
    $T=\emptyset$. Suppose not. Let $K$ be a connected component of $T$. 
    By Lemmas \ref{lem:z is not universal} and \ref{lem:A_3(i+2) is empty},
    \[
    N(K)=(N(K)\cap B_{i+1})\cup (N(K)\cap A_1(i+1))\cup (N(K)\cap A_2(i+1))\cup (N(K)\cap A_5)
    .\]
    By Proposition \ref{prop:component of A_1(i)}, $N(K)\setminus A_1(i+1)$ is a clique.
    Observe that any vertex $p\in S$ is not adjacent to $z$ by Lemma \ref{lem:A_2^1(i-2) is universal in A_2(i-2)} and the assumption that $z$ is not universal in $A_2(i+1)$. If a vertex $p\in N(K)\cap A_1(i+1)$ is mixed on $ab\in K$, then $a-b-p-(N(p)\cap A_2(i-2))-q_{i-2}-q_{i+2}-z$ is a bad $P_7$. So any vertex in $N(K)\cap A_1(i+1)$ is complete to $K$.

    We show that $N(K)\cap A_1(i+1)$ is complete to $N(K)\setminus A_1(i+1)$. Suppose by contradiction that $p\in N(K)\cap A_1(i+1)$ and $q\in N(K)\setminus A_1(i+1)$ with $pq\notin E(G)$. 
    By definition of $p$, $p$ has a neighbor $w'\in A_2(i-2)$.
    Let $u\in K$ be a common neighbor of $p$ and $q$. 
    By Lemma \ref{lem:A_5 is complete to A_1}, $q\notin A_5$.
    If $q\in B_{i+1}$, then $p-u-q-q_i-q_{i-1}-w'-p$ is an induced $C_6$.
    If $q\in A_2(i+1)$, then $p-u-q-N_{B_{i+2}}(q)-q_{i-2}-w'-p$ is an induced $C_6$.
     
    Therefore, each vertex in $N(K)\setminus A_1(i+1)$ is universal in $N(K)$. Since $N(K)$ is not clique, there are two non-adjacent vertices $p,q\in N(K)\cap A_1(i+1)$. So $p$ and $q$ have a common neighbor $u\in K$ and thus $p$ and $q$ have disjoint neighborhoods in $A_2(i-2)$. Since $pz\notin E(G)$, $q-u-p-(N(p)\cap A_2(i-2))-q_{i-2}-q_{i+2}-z$ is a bad $P_7$.
\end{proof}

\begin{lemma}[$P_4$-freeness of $A_1(i+1)$]\label{lem:A_1(i+1) is $P_4$-free}
    The following properties hold for $A_1(i+1)$.

    \begin{itemize}
        \item[$(1)$]  $x$ is complete to $N(A_1(i+1))\cap A_2(i-2)$.
        \item[$(2)$] For any vertex $w'\in A_2(i-2)$ and any component $K$ of $A_1(i+1)$, $w$ is either complete or anticomplete to $K$.
        \item[$(3)$] $A_1(i+1)$ is $P_4$-free.
    \end{itemize}
\end{lemma}

\begin{proof} We prove these properties one by one.

\medskip \noindent
    (1) Let $w'\in N(A_1(i+1))\cap A_2(i-2)$ and $y'\in A_1(i+1)$ be a neighbor of $w'$. By Lemma \ref{lem:A_2^1(i-2) is universal in A_2(i-2)}, $zy'\notin E(G)$. By Lemma \ref{lem:two opposite edges} applying to $xz$ and $y'w'$, $xw'\in E(G)$.

\medskip \noindent
    (2) Suppose by contradiction that there exists an edge $ab\in E(K)$ such that $w'$ is adjacent to $b$ but not to $a$. By (1), $xw'\in E(G)$ and thus $a-b-w'-x-z-(N(z')\cap B_{i+2})-z'$ is a bad $P_7$.

\medskip \noindent
    (3) It suffices to show that each component $K$ of $A_1(i+1)$ is $P_4$-free.
     By Lemma \ref{lem:A_1(i+1) is dominated by A_2(i-2)} and (2), there exists a vertex $w'\in A_2(i-2)$ that is complete to $K$. 
     If there is an induced $P_5=a-b-c-d-e$ with $a,b,c,d\in K$ and $e\in B_{i+1}$, then $a-b-c-d-e-q_i-q_{i-1}$ is an induced $P_7$. By Lemma \ref{lem:P_4-free lemma}, $K$ is $P_4$-free. 
\end{proof}

\begin{lemma}\label{lem:A^0_2(i-2)}
    $A^0_2(i-2)$ is anticomplete to $A_1(i+1)$. 
\end{lemma}

\begin{proof}
    Suppose that $a\in A^0_2(i-2)$ is adjacent to $b\in A_1(i+1)$. 
    By definition of $A^0_2(i-2)$, $ax\notin E(G)$. Applying Lemma \ref{lem:two opposite edges} to the edges $xz$ and $ab$, we have $bz\in E(G)$. 
    By Lemma \ref{lem:A_2^1(i-2) is universal in A_2(i-2)}, $z$ is  universal in $A_2(i+1)$. This contradicts the assumption that $z$ is not universal in $A_2(i+1)$.
\end{proof}

\begin{lemma}\label{lem:A_2^0(i-2)=empty}
    $A_2^0(i-2)=\emptyset$. 
\end{lemma}

\begin{proof}
    Suppose not. 
    Let $L$ be a component of $A_2^0(i-2)$. By Lemma \ref{lem:A_1(i+1) is $P_4$-free} (1), $L$ is anticomplete to $A_1(i+1)$. By Lemma \ref{lem:A^0_2(i-2)}, $N(L)\subseteq A_5\cup B_{i-1}\cup B_{i-2}\cup A^1_2(i-2)\cup A_3(i-2)$. By Lemma \ref{lem:neighbors of A_1(i-1) in A_2(i-2)}, $N(L)\setminus A_3(i-2)$ is a clique. 
    Since $G$ has no clique cutsets, $L$ has neighbors in $A_3(i-2)$ and so $A_3(i-2)\neq \emptyset$. 
     
    Let $S=B_{i-1}\cup B_{i-2}\cup A_3(i-2)\cup A_2(i-2)$. 
    Since $A_3(i-2)\neq \emptyset$, $S$ is not a clique and so $S$ has two non-adjacent simplicial vertices $u,u'$ by Lemma \ref{lem:S is chordal}. 
    Since $A^1_2(i-2)\cup B_{i-1}\cup B_{i-2}$ is a clique, we may assume that $u\in A^0_2(i-2)\cup A_3(i-2)$.
    By Lemma \ref{lem:A^0_2(i-2)}, $X=\emptyset$ and so $A_2(i-2)\setminus Y=A_2^1(i-2)$. By Lemma \ref{lem:A_2^1(i-2) is universal in A_2(i-2)}, each vertex in $A_2^1(i-2)$ is universal in $A_2(i-2)$. 
    By Lemma \ref{lem:simplicial vertex in Y}, $u\in A_3(i-2)$. 
    By Lemma \ref{lem:A_3(i-2)} (1) and (2) and Lemma \ref{lem:simplicial vertex in Y}, $u'\in A_3(i-2)\cup B_{i-2}$. 
    By Lemma \ref{lem:A_3(i+2) is empty}, $A_3(i+2)=\emptyset$. 
    It follows from Lemma \ref{lem:simplicial vertex in A_3(i-2) and B_{i-2}} that $|N(u)\cap B_{i+2}|>\ceil{\frac{\omega}{4}}$ and $|N(u')\cap B_{i+2}|>\ceil{\frac{\omega}{4}}$. Since $A_3(i-2)\cup B_{i-2}$ is complete to $B_{i-1}$, $u$ and $u'$ have disjoint neighborhoods in $B_{i+2}$. By Lemma \ref{lem:simplicial vertices in A_1(i-1)}, $|A_2(i+1)|>\ceil{\frac{\omega}{4}}$. By Lemmas \ref{lem:z is not universal} and Lemma \ref{lem:B_{i-1} is large when two non-empty A_1(i)}, $|B_{i+1}|>\ceil{\frac{\omega}{4}}$. 
    So $B_{i+1}\cup B_{i+2}\cup A_2(i+1)$ contains a clique of size larger than $\omega$, a contradiction. 
\end{proof}

\begin{lemma} \label{lem:B_{i+2} is large}
    There is a vertex $b\in B_{i-2}$ anticomplete to $A_3(i-2)$ such that $|N(b)\cap B_{i+2}|>\ceil{\frac{\omega}{4}}$.
    This implies that $A_3(i-2)=\emptyset$.
\end{lemma}

\begin{proof}
    By Lemma \ref{lem:comparable vertices in A_3(i)}, there is a vertex $b\in B_{i-2}$ that is anticomplete to $A_3(i-2)$. By Lemmas \ref{lem:neighbors of A_1(i-1) in A_2(i-2)} and \ref{lem:A_2^0(i-2)=empty}, we have $N(b)\setminus B_{i+2}$ is a clique. Since $G$ has no small vertices, $|N(b)\cap B_{i+2}|>\ceil{\frac{\omega}{4}}$.

    By Lemma \ref{lem:simplicial vertices in A_1(i-1)}, $A_2(i+1)$ contains a clique $M$ of size larger than $\ceil{\frac{\omega}{4}}$ complete to $B_{i+1}\cup B_{i+2}$. Suppose that $A_3(i-2)\neq \emptyset$. By Lemma \ref{lem:P4-freenessofA3}, let $t$ be a simplicial vertex of $A_3(i-2)$. By Lemma \ref{lem:edge in $A_3(i)$}, $N(t)\cap A_3(i-2)$ is complete to $N(t)\cap B_{i-2}$. By Lemma \ref{lem:A_3(i-2)} (1), (2) and Lemma \ref{lem:neighbors of A_1(i-1) in A_2(i-2)}, $N(t)\setminus B_{i+2}$ is a clique. Since $G$ has no small vertices, $|N(t)\cap B_{i+2}|>\ceil{\frac{\omega}{4}}$. Since $N(t)\cap B_{i+2}$ and $N(b)\cap B_{i+2}$ are disjoint, $B_{i+1}\cup B_{i+2}\cup M$ is a clique of size larger than $\omega$, a contradiction. 
\end{proof}

\begin{lemma}\label{lem:A_1(i-1) is a clique}
    $A_1(i-1)$ is a clique.
\end{lemma}

\begin{proof}
    Let $K$ be the component of $A_1(i-1)$ containing $x$. By Lemmas \ref{lem:reduce to basic graph} and \ref{lem:B_{i-1} vs A_1(i-1)}, $K$ is complete to $B_{i-1}$. If there is simplicial vertex of $A_1(i-1)$ not adjacent to some simplicial vertex of $K$, then $A_2(i+1)\cup B_{i+1}\cup B_{i+2}$ contains a clique larger than $\omega$ by Lemmas \ref{lem:simplicial vertices in A_1(i-1)} and \ref{lem:B_{i+2} is large}. This implies that $A_1(i-1)$ is a clique.
\end{proof}

\begin{lemma}\label{lem:A and C are disjoint}
    $A_2^1(i+1)$ is anticomplete to $A_1(i-1)$.
\end{lemma}

\begin{proof}
    Suppose by contradiction that $v\in A_2^1(i+1)$ has a neighbor $v'\in A_1(i-1)$. By definition, $v$ has a neighbor $a\in A_1(i+1)$. By Lemma \ref{lem:A_1(i+1) is dominated by A_2(i-2)}, $a$ has a neighbor $a'\in A_2(i-2)$. By applying Lemma \ref{lem:two opposite edges} to $aa'$ and $vv'$, $a'v'\notin E(G)$. By Lemma \ref{lem:A_1(i+1) is $P_4$-free} (3), $xa'\in E(G)$. Since $A_1(i-1)$ is a clique by Lemma \ref{lem:A_1(i-1) is a clique}, $a'$ mixes on $xv'$ which contradicts Lemma \ref{lem:A2i1tocompA1}.
\end{proof}

Recall that $A_2^1(i+1)$ is the set of vertices in $A_2(i+1)$ that have a neighbor in $A_1(i+1)$, $A_2^0(i+1)=A_2(i+1)\setminus A_2^1(i+1)$, $X'=A^0_2(i+1)\cap N(A_1(i-1))$ and $Y'=A^0_2(i+1)\setminus X'$.

\begin{lemma}\label{lem:simplicial vertex in A_2(i+1)+B_{i+1}+B_{i+2}}
    Let $s\in A_2(i+1)\cup B_{i+1}\cup B_{i+2}$ be a simplicial vertex of $A_2(i+1)\cup B_{i+1}\cup B_{i+2}$. If $s\in X'$, then $N(s)\setminus A_1(i-1)$ is a clique and so $|N(s)\cap A_1(i-1)|>\ceil{\frac{\omega}{4}}$.
\end{lemma}

\begin{proof}
    Suppose not. 
    By Lemma \ref{lem:A and C are disjoint}, $s$ is anticomplete to $A_1(i+1)$ and thus $N(s)\setminus A_1(i-1)\subseteq B_{i+1}\cup B_{i+2}\cup A_2(i+1)\cup A_5$.
    By Lemma \ref{lem:neighbors of A_1(i-1) in A_2(i-2)} and \ref{lem:A5 vs A1A2}, there exist vertices $b\in N(s)\cap Y'$ and $q\in N(s)\cap A_5$ that are not adjacent.

    Let $K$ be the component of $Y'\setminus N(q)$ containing $b$.
    Note that $N(K)\subseteq A_2(i+1)\cup B_{i+1}\cup B_{i+2}\cup A_5$. By Lemmas \ref{lem:A_2^1(i-2) is universal in A_2(i-2)} and \ref{lem:neighbors of A_1(i-1) in A_2(i-2)},
    $N(K)\cap (A_2^1(i+1)\cup B_{i+1}\cup B_{i+2}\cup A_5)$ is a clique and complete to $N(K)\cap X'$. 
    Since $N(K)$ is not a clique, $K$ has two non-adjacent neighbors $p_1,p_2$. Suppose first that $p_1\in X'$. Then $p_2\in X'\cup Y'$. 
    If $p_1,p_2$ have a common neighbor $u\in K$, then $p_1-u-p_2-q-p_1$ is an induced $C_4$. So $p_1,p_2$ have no common neighbors in $K$. 
    Let $P$ be a shortest path from $p_1$ to $p_2$ whose internal is contained in $K$. So $|P|\ge 4$. Then $p_2-P-p_1-r-N_{B_{i-1}}(r)-q_i$ is an induced path with at least 7 vertices, where $r$ is a neighbor of $p_1\in A_1(i-1)$. 
    This proves that each vertex in $N(K)\cap X'$ is universal in $N(K)$.
    So $p_1,p_2\notin X'$ and $s$ is adjacent to $p_1$ and $p_2$.
    
    Therefore, we may assume that $p_1\in Y'$. Since $A_2^1(i+1)$ is universal in $A_2(i+1)$, $p_2\in Y'\cup B_{i+1}\cup B_{i+2}\cup A_5$. 
    Since $s$ is simplicial in $A_2(i+1)\cup B_{i+1}\cup B_{i+2}$, we have $p_2\in A_5$. 
    Since $q$ is a common neighbor of $p_1$ and $p_2$, $p_1$ and $p_2$ have no common neighbors in $K$. 
    Since $s$ is simplicial in $A_2(i+1)\cup B_{i+1}\cup B_{i+2}$, $N(s)\cap K\subseteq N(p_1)\cap K$ and so $s$ and $p_2$ have no common neighbors in $K$. Let $P=s-v_1-\cdots v_k-p_2$ be a shortest path between $s$ and $p_2$ with $v_i\in K$ for all $i\in [k]$. If $k=2$, then $s-v_1-v_2-p_2-s$ is an induced $C_4$. If $k\ge 3$, then $v_3-v_2-v_1-s-r-N_{B_{i-1}}(r)-q_i$ is an induced path at least 7 vertices, where $r\in A_1(i-1)$ is a neighbor of $s$.
\end{proof}

\begin{lemma}
    $\chi(G)\le \ceil{\frac{5}{4}\omega}$.
\end{lemma}

\begin{proof}
    By Lemma \ref{lem:simplicial vertices in A_1(i-1)}, $A_2(i+1)$ contains a clique $M$ complete to $B_{i+1}\cup B_{i+2}$ with $|M|>\ceil{\frac{\omega}{4}}$. 
    Since $A_5\cup M\cup B_{i+1}\cup B_{i+2}$ is a clique, $|A_5|<\ceil{\frac{\omega}{4}}$. 

    \medskip
    \noindent (1) $A_1(i-1)$ is a clique and complete to $B_{i-1}\cup A_2(i-2)$.

    \medskip
    By Lemma \ref{lem:A2i1tocompA1} and $A_2(i-2)=A_2^1(i-2)$, $A_2(i-2)$ is complete to $A_1(i-1)$. By $x$ is complete to $B_{i-1}$ and Lemma \ref{lem:B_{i-1} vs A_1(i-1)}, $A_1(i-1)$ is complete to $B_{i-1}$. This proves (1).

    \medskip
     Let $s\in A_2(i+1)\cup B_{i+1}\cup B_{i+2}$ be a simplicial vertex in $A_2(i+1)\cup B_{i+1}\cup B_{i+2}$. If $s\in Y'$, then $N(s)\setminus A_5$ is a clique and so $|A_5|>\ceil{\frac{\omega}{4}}$, a contradiction. 
    So $s\notin Y'$.

    \medskip
    \noindent (2) $A_2(i+1)\cup B_{i+1}\cup B_{i+2}$ is a clique.

    \medskip
    Suppose not. $A_2(i+1)\cup B_{i+1}\cup B_{i+2}$ contains two non-adjacent simplicial vertices $s,s'$ by Lemma \ref{lem:S is chordal}. 
    Since $A_2^1(i+1)\cup B_{i+1}\cup B_{i+2}$ is a clique complete to $X'$, we have $s,s'\in X'$. 
    By Lemma \ref{lem:simplicial vertex in A_2(i+1)+B_{i+1}+B_{i+2}}, $|N(s)\cap A_1(i-1)|,|N(s')\cap A_1(i-1)|>\ceil{\frac{\omega}{4}}$. Since $N(s)\cap A_1(i-1)$ and $N(s')\cap A_1(i-1)$ are disjoint and complete, then $B_{i-1}\cup A_1(i-1)\cup A_2(i-2)$ is a clique of size larger than $\omega$ by (1). This proves (2).

    \medskip
    By (2), $Y'=\emptyset$. By Lemmas \ref{lem:A5 vs A1A2}, \ref{lem:neighbors of A_1(i-1) in A_2(i-2)} and \ref{lem:A_5 is complete to A_1}, each vertex in $A_5$ is universal in $G$. Since $G$ is basic, $A_5=\emptyset$.

    \medskip
    \noindent (3) $A_1(i+1)$ is a clique.

    \medskip
    Suppose not. Let $s,s'\in A_1(i+1)$ be two non-adjacent simplicial vertices of $A_1(i+1)$.
    If $N(s)\cap A_2(i-2)$ and  $N(s')\cap A_2(i-2)$ are disjoint, then $A_2(i-2)\cup B_{i-1}\cup A_1(i-1)$ is a clique of size larger than $\omega$ by Lemmas \ref{lem:simplicial vertices in A_1(i-1)}, \ref{lem:simplicial vertex in A_2(i+1)+B_{i+1}+B_{i+2}} and \ref{lem:B_{i-1} is large when two non-empty A_1(i)}. So $s$ and $s'$ have a common neighbor in $A_2(i-2)$ and so $s$ and $s'$ have disjoint neighborhoods in $B_{i+1}\cup A_2^1(i+1)$.
    By Lemma \ref{lem:A_1(i+1) is $P_4$-free} (2) and $A_2^0(i-2)=\emptyset$, $|N(s)\cap (B_{i+1}\cup A_2^1(i+1))|>\ceil{\frac{\omega}{4}}$ and $|N(s')\cap (B_{i+1}\cup A_2^1(i+1))|>\ceil{\frac{\omega}{4}}$. Then $M\cup A_2^1(i+1)\cup B_{i+1}\cup B_{i+2}$ is a clique of size larger than $\omega$. This proves (3).

    \medskip
    \noindent (4) For every vertex $v\in A_1(i+1)$, $|N(v)\cap B_{i+1}|>\ceil{\frac{\omega}{4}}$. 

    \medskip
    If $v$ has no neighbors in $A_2^1(i+1)$, then $N(v)\cap B_{i+1}=N(v)\cap (B_{i+1}\cup A_2^1(i+1))$ and so the statement holds since $|N(v)\cap (B_{i+1}\cup A_2^1(i+1))|>\ceil{\frac{\omega}{4}}$. 
    So we may assume that $v$ has a neighbor $c\in A_2^1(i+1)$. 
    If $v$ has a non-neighbor $v'\in B_{i+1}$, then $v'-c-v-N_{A_2(i-2)}(v)-q_{i-1}-q_i-v'$ is an induced $C_6$. 
    Then $v$ is complete to $B_{i+1}$ and so $|N(v)\cap B_{i+1}|>\ceil{\frac{\omega}{4}}$. 
    This proves (4).

    \medskip

Next we give a $\ceil{\frac{5}{4}\omega}$-coloring of $G$.
Let $x=|B_{i-1}|$, $y=|B_{i-2}|$, and $z=|A_2(i-2)|$.
Our strategy is to use $p=x+y+z-3\ceil{\frac{\omega}{4}}$ colors to color some vertices of $G$ first and then argue that the remaining vertices can be colored with $f=\ceil{\frac{5}{4}\omega}-p$ colors. 
Note that
    \begin{align*} 
    f & = \ceil{\frac{5}{4}\omega}-(x+y+z-3\ceil{\frac{\omega}{4}}) 
    \\[3pt]
     & = \omega+(4\ceil{\frac{\omega}{4}}-(x+y+z))  
    \\[3pt]
     & \ge \omega,
\end{align*}
since $x+y+z\leq \omega$.

\medskip
\noindent {\bf Step 1: Precolor a subgraph.}

\medskip
Next we define vertices to be colored with $p$ colors. Let $S_{i-1}\subseteq B_{i-1}$ be the set of $x-\ceil{\frac{\omega}{4}}$ vertices, $S_{i-2}\subseteq B_{i-2}$ be the set of $y-\ceil{\frac{\omega}{4}}$ vertices, $S_{z}\subseteq A_2(i-2)$ be the set of $z-\ceil{\frac{\omega}{4}}$ vertices, and $S_{i+1}\subseteq B_{i+1}$ be the set of $p$ vertices with largest neighborhoods in $A_1(i+1)$. Since $(x+y+z)-3\ceil{\frac{\omega}{4}}\leq \ceil{\frac{\omega}{4}}<|B_{i+1}|$, the choice of $S_i$ is possible. 
By (4), $S_{i+1}$ is complete to $A_1(i+1)$. 
We color $S_{i+1}$ with all colors in $[p]$, and color all vertices in $S_{i-1}\cup S_{i-2}\cup S_z$ with all colors in $[p]$. Since $B_{i-1}$ is anticomplete to $B_{i-1}\cup B_{i-2}\cup A_2(i-2)$, this coloring is proper. So it suffices to show that
$G_1=G-(S_{i-1}\cup S_{i-2}\cup S_{i+1}\cup S_z)$ is $f$-colorable.

\medskip
\noindent {\bf Step 2: Reduce by degeneracy.}

\medskip
Let $v$ be a vertex in $B_i\cup B_{i+2}\cup A_1(i+1)$. 
Note that $|N_{G_1}(v)|<\ceil{\frac{\omega}{4}}+(\omega-\ceil{\frac{\omega}{4}})\leq f$. 
It follows that $G_1$ is $f$-colorable if and only if $G_1-(B_i\cup B_{i+2}\cup A_1(i+1))$ is $f$-colorable. 
Since $G_1-(B_i\cup B_{i+2}\cup A_1(i+1))$ is chordal, $\chi(G_1-(B_i\cup B_{i+2}\cup A_1(i+1)))\leq \omega\leq f$. 
\end{proof}

\subsubsection{$X$ and $X'$ are universal}

Recall that $A_2^1(i-2)$ is the set of vertices in $A_2(i-2)$ that have a neighbor in $A_1(i-1)$ and $A_2^0(i-2)=A_2(i-2)\setminus A_2^1(i-2)$, $X=A^0_2(i-2)\cap N(A_1(i+1))$ and $Y=A^0_2(i-2)\setminus X$. Similarly. $A_2^1(i+1)$ is the set of vertices in $A_2(i+1)$ that have a neighbor in $A_1(i+1)$ and $A_2^0(i+1)=A_2(i+1)\setminus A_2^1(i+1)$, and $X'=A^0_2(i+1)\cap N(A_1(i-1))$ and $Y'=A^0_2(i+1)\setminus X'$. The global assumption in this subsection is that each vertex in $X$ is universal in $A_2(i-2)$ and each vertex in $X'$ is universal in $A_2(i+1)$.

\begin{lemma}\label{lem:A_3(i-2) is empty}
    $A_3(i-2)=A_3(i+2)=\emptyset$. 
\end{lemma}

\begin{proof}
    Suppose by symmetry that $A_3(i-2)\neq \emptyset$. By Lemma \ref{lem:A_3(i-1) is empty}, $A_3(i-1)=\emptyset$. Since $A_3(i-2)\neq \emptyset$, $S=A_2(i-2)\cup B_{i-1}\cup B_{i-2}\cup A_3(i-2)$ is not clique. By Lemma \ref{lem:S is chordal}, $S$ contains two non-adjacent simplicial vertices $u$ and $u'$. Since each vertex in $X$ is universal in $A_2(i-2)$, $A^1_2(i-2)\cup X\cup B_{i-2}\cup B_{i-1}$ is a clique. So we may assume that $u\in Y\cup A_3(i-2)$. By Lemma \ref{lem:simplicial vertex in Y}, $u\in A_3(i-2)$. 
    By Lemma \ref{lem:A_3(i-2)} (1) and (2) and Lemma \ref{lem:simplicial vertex in Y}, $u'\in A_3(i-2)\cup B_{i-2}$. 
    Since $u\in A_3(i-2)$ and $uu'\notin E(G)$, $N(u)\cap A_3'(i+2)\subseteq B_{i+2}$ and $N(u')\cap A_3'(i+2)\subseteq B_{i+2}$ by Lemma \ref{lem:compare nbd of different A_3(i) in B_j}. 
    By Lemma \ref{lem:simplicial vertex in A_3(i-2) and B_{i-2}}, $|N_{B_{i+2}}(u)|,|N_{B_{i+2}}(u')|>\ceil{\frac{\omega}{4}}$. 
    
    Since $A_3(i-2)\cup B_{i-2}$ is complete to $B_{i-1}$, $u$ and $u'$ have disjoint neighborhoods in $B_{i+2}$. By Lemma \ref{lem:simplicial vertices in A_1(i-1)}, $|X'|>\ceil{\frac{\omega}{4}}$. 
    If $A_3(i)=\emptyset$, then $|B_{i+1}|>\ceil{\frac{\omega}{4}}$ by Lemma \ref{lem:B_{i-1} is large when two non-empty A_1(i)}. 
    If $A_3(i)\neq \emptyset$, then $|B_{i+1}|>\ceil{\frac{\omega}{4}}$ by Lemmas \ref{lem:smv on w} and \ref{lem:A_2=A_2(i-2)+A_2(i+1)}.
    So $B_{i+1}\cup B_{i+2}\cup A_2(i+1)$ contains a clique of size larger than $\omega$, a contradiction. 
\end{proof}

\begin{lemma}\label{lem:Y is empty}
     $Y=Y'=\emptyset$.
\end{lemma}

\begin{proof}
    If $Y\neq \emptyset$, then $N(Y)\subseteq A^1_2(i-2)\cup X\cup B_{i-1}\cup B_{i-2}\cup A_5$ by Lemmas \ref{lem:A_3(i-1) is empty} and \ref{lem:A_3(i-2) is empty}, and so is a clique. This contradicts that $G$ has no clique cutsets. 
\end{proof}

\begin{lemma}\label{lem:A_3(i) is empty when universal}
    $A_3(i)=\emptyset$.
\end{lemma}

\begin{proof}
    Suppose that $t\in A_3(i)$. 
    If $a\in A_1(i+1)$ is adjacent to $b\in A_2(i+1)$, then $t-q_i-q_{i-1}-q_{i-2}-(N(b)\cap B_{i+2})-b-a$ is an induced $P_7$. 
    Then $A_1(i+1)$ is anticomplete to $A_2(i+1)$ and so $A_2^1(i+1)=\emptyset$. 
    By Lemma \ref{lem:A_3(i-2) is empty}, $A_3(i+2)=\emptyset$. 
    By Lemma \ref{lem:Y is empty}, $Y'=\emptyset$. By Lemma \ref{lem:A_3(i-2)} (2) and Lemma \ref{lem:single vertex in A_3(i)} (1), $q_{i+2}$ is anticomplete to $A_3(i-2)$. It follows that $N(q_{i+2})\setminus B_{i-2}\subseteq B_{i+1}\cup B_{i+2}\cup X'\cup A_5$ is a clique. Since $G$ has no small vertices, $|B_{i-2}|>\ceil{\frac{\omega}{4}}$. 
    By Lemma \ref{lem:simplicial vertices in A_1(i-1)}, $A_2^1(i-2)\cup X$ contains a clique $L$ of size larger than $\ceil{\frac{\omega}{4}}$. 
    By Lemma \ref{lem:A_3(i)} (3), then $|B_{i-1}|>\frac{\omega}{2}$.  By Lemmas \ref{lem:A_3(i-2)} (1) and (2), $B_{i-2}\cup B_{i-1}\cup L$ contains a clique of size larger than $\omega$, a contradiction.
\end{proof}

We say that an edge $ab\in E(G)$ with $a\in A_1(i-1)$ and $b\in A_2(i-2)$ is a {\em vertical edge} if $a$ has a neighbor in $A_2(i+1)$ and $b$ has a neighbor in $A_1(i+1)$. A vertical edge between $A_1(i+1)$ and $A_2(i+1)$ is defined in a similar way.

\begin{lemma}
    Vertical edges between $A_1(i-1)$ and $A_2(i-2)$ and vertical edges between $A_1(i+1)$ and $A_2(i+1)$ cannot exist at the same time.
\end{lemma}

\begin{proof}
    Suppose not. 
    Without loss of generality, $xw\in E(G)$ and $zy\notin E(G)$.
    By Lemmas \ref{lem:A_3(i) is empty when universal}, \ref{lem:A_3(i-1) is empty} and \ref{lem:A_3(i-2) is empty}, $A_3=\emptyset$.
    By Lemma \ref{lem:Y is empty}, $A_2(i-2)\cup B_{i-1}\cup B_{i-2}\cup A_5$ is a clique and $A_2(i+1)\cup B_{i+1}\cup B_{i+2}\cup A_5$ is a clique.
    Since $q_{i-2}$ is not a small vertex, $|B_{i+2}|>\ceil{\frac{\omega}{4}}$. By symmetry, $|B_{i-2}|>\ceil{\frac{\omega}{4}}$. 

    Next we show that $A_1(j)$ is a clique for $j\in \{i-1,i+1\}$. Let $K$ be the component of $A_1(i-1)$ containing $x$. By Lemmas \ref{lem:reduce to basic graph} and \ref{lem:B_{i-1} vs A_1(i-1)}, $K$ is complete to $B_{i-1}$. If there is simplicial vertex of $A_1(i-1)$ not adjacent to some simplicial vertex of $K$, then $A_2(i+1)\cup B_{i+1}\cup B_{i+2}$ contains a clique larger than $\omega$ by Lemma \ref{lem:simplicial vertices in A_1(i-1)}. This implies that $A_1(i-1)$ is a clique. By symmetry $A_1(i+1)$ is a clique. 

    Suppose that $v\in A_1(i+1)$ has a neighbor $a\in A^1_2(i-2)$ and a neighbor $b\in X$. Let $c\in A_1(i-1)$ be a neighbor of $a$ that has a neighbor $d\in A_2(i+1)$. By Lemma \ref{lem:two opposite edges}, $v$ is adjacent to and non-adjacent to $d$, a contradiction. So no vertex in $A_1(i+1)$ can have a neighbor in both $A^1_2(i-2)$ and $X$. 
    Next we show that $X=\emptyset$. Suppose not. Let $a\in X$ and $b\in A_1(i+1)$ be a neighbor of $a$. Then $b\neq y$ and $bw\notin E(G)$. Since $A_1(i+1)$ is a clique, $b-y-w-a-b$ is an induced $C_4$.
    This proves that $X=\emptyset$. So $A_2(i-2)=A^1_2(i-2)$ or equivalently each vertex in $A_2(i-2)$ has a neighbor in $A_1(i-1)$. By symmetry, each vertex in $A_2(i+1)$ has a neighbor in $A_1(i+1)$.

    By Lemmas \ref{lem:A2i1tocompA1}, $A_2(i-2)$ is complete to $A_1(i-1)$ and $A_2(i+1)$ is complete to $A_1(i+1)$. This contradicts that $zy\notin E(G)$.
\end{proof}

\begin{lemma}
    $\chi(G)\leq \ceil{\frac{5}{4}\omega}$.
\end{lemma}

\begin{proof}
    By Lemmas \ref{lem:A_3(i) is empty when universal}, \ref{lem:A_3(i-1) is empty} and \ref{lem:A_3(i-2) is empty}, $A_3=\emptyset$.
    By Lemma \ref{lem:Y is empty}, $Y=Y'=\emptyset$. 
    Then each vertex in $A_5$ is universal in $G$ by Lemma \ref{lem:A_5 is complete to A_1}. So $A_5=\emptyset$ by Lemma \ref{lem:reducible structure}.
    Since $q_{i-2}$ is not a small vertex, $|B_{i+2}|>\ceil{\frac{\omega}{4}}$. By symmetry, $|B_{i-2}|>\ceil{\frac{\omega}{4}}$. It follows from Lemma \ref{lem:simplicial vertices in A_1(i-1)} that $A_1(i-1)$ and $A_1(i+1)$ are cliques.

    Suppose that $X$ contains a vertex $a$ and $X'$ contains a vertex $b$.
    Let $c\in A_1(i+1)$ be a neighbor of $a$ and $d\in A_1(i-1)$ be a neighbor of $b$. Then $ac$ and $bd$ contradict Lemma \ref{lem:two opposite edges}.
    So we may assume that $X=\emptyset$. Equivalently, each vertex in $A_2(i-2)$ has a neighbor in $A_1(i-1)$.  By Lemma \ref{lem:A2i1tocompA1}, $A_2(i-2)$ is complete to $A_1(i-1)$.

    Let $v\in A_2(i-2)$ be a vertex such that $|N(v)\cap A_1(i+1)|$ is maximum. If $A^1_2(i+1)$ contains a vertex $u$, we show that $\{v,u,q_i\}$ is a $(1,1)$-good subgraph. It suffices to show that every  maximum clique $M$ of $G$ contains a vertex in  $\{v,u,q_i\}$. Since $v$ is universal in $B_{i-2}\cup B_{i-1}\cup A_1(i-1)\cup A_2(i-2)$ and $u$ is universal in $B_{i+2}\cup B_{i+1}\cup A_1(i+1)\cup A_2(i+1)$, $M$ contains $v$ or $u$ if $M\subseteq B_{i-2}\cup B_{i-1}\cup A_1(i-1)\cup A_2(i-2)$ or $M\subseteq B_{i+2}\cup B_{i+1}\cup A_1(i+1)\cup A_2(i+1)$. For the same reason, $M$ contains $q_i$ if $M\subseteq B_i\cup B_j$ for $j\in \{i-1,i+1\}$. The remaining cases are the following.

    \begin{itemize}
        \item[$\bullet$] $M\subseteq B_{i-2}\cup B_{i+2}$. Then $|B_{i-2}|\ge \frac{\omega}{2}$ or $|B_{i+2}|\ge \frac{\omega}{2}$. If $|B_{i-2}|\ge \frac{\omega}{2}$, then $B_{i-2}\cup B_{i-1}\cup A_2(i-2)$ is a clique of size larger than $\omega$. If $|B_{i+2}|\ge \frac{\omega}{2}$, then $B_{i+2}\cup B_{i+1}\cup A_2(i+1)$ is a clique of size larger than $\omega$.
        \item[$\bullet$] $M\subseteq A_2(i-2)\cup A_1(i+1)$. Then $v\in M$ by the choice of $v$.
        \item[$\bullet$] $M\subseteq A_2(i+1)\cup A_1(i-1)$. Then $|A_2(i+1)|\ge \frac{\omega}{2}$ or $|A_1(i-1)|\ge \frac{\omega}{2}$. If $|A_2(i+1)|\ge \frac{\omega}{2}$, then  $B_{i+2}\cup B_{i+1}\cup A_2(i+1)$ is a clique of size larger than $\omega$. If $|A_1(i-1)|\ge \frac{\omega}{2}$, then $A_1(i-1)\cup B_{i-1}\cup A_2(i-2)$ is a clique of size larger than $\omega$.
    \end{itemize}
    \vspace{-10pt}
So $A_2^1(i+1)=\emptyset$ and thus $A_1(i+1)$ is anticomplete to $A_2(i+1)$.

Next we give a $\ceil{\frac{5}{4}\omega}$-coloring of $G$.
Let $x=|B_{i-1}|$, $y=|B_{i-2}|$, and $z=|A_2(i-2)|$.
Our strategy is to use $p=x+y+z-3\ceil{\frac{\omega}{4}}$ colors to color some vertices of $G$ first and then argue that the remaining vertices can be colored with $f=\ceil{\frac{5}{4}\omega}-p$ colors. 
Note that
    \begin{align*} 
    f & = \ceil{\frac{5}{4}\omega}-(x+y+z-3\ceil{\frac{\omega}{4}}) 
    \\[3pt]
     & = \omega+(4\ceil{\frac{\omega}{4}}-(x+y+z))  
    \\[3pt]
     & \ge \omega,
\end{align*}
since $x+y+z\leq \omega$.

\medskip
\noindent {\bf Step 1: Precolor a subgraph.}

\medskip
Next we define vertices to be colored with $p$ colors. Let $S_{i-1}\subseteq B_{i-1}$ be the set of $x-\ceil{\frac{\omega}{4}}$ vertices, $S_{i-2}\subseteq B_{i-2}$ be the set of $y-\ceil{\frac{\omega}{4}}$ vertices, $S_{z}\subseteq A_2(i-2)$ be the set of $z-\ceil{\frac{\omega}{4}}$ vertices, and $S_{i+1}\subseteq B_{i+1}$ be the set of $p$ vertices with largest neighborhoods in $A_1(i+1)$. Since $(x+y+z)-3\ceil{\frac{\omega}{4}}\leq \ceil{\frac{\omega}{4}}<|B_{i+1}|$, the choice of $S_{i+1}$ is possible. 
Since $A_1(i+1)$ is anticomplete to $A_2(i+1)$, every vertex in $A_1(i+1)$ has at least $\ceil{\frac{\omega}{4}}+1$ neighbors in $B_{i+1}$ and so $S_{i+1}$ is complete to $A_1(i+1)$. 
We color $S_{i+1}$ with all colors in $[p]$, and color all vertices in $S_{i-1}\cup S_{i-2}\cup S_z$ with all colors in $[p]$. Since $B_{i+1}$ is anticomplete to $B_{i-1}\cup B_{i-2}\cup A_2(i-2)$, this coloring is proper. So it suffices to show that
$G_1=G-(S_{i-1}\cup S_{i-2}\cup S_{i+1}\cup S_z)$ is $f$-colorable.

\medskip
\noindent {\bf Step 2: Reduce by degeneracy.}

\medskip
Let $v$ be a vertex in $B_i\cup B_{i+2}\cup A_1(i+1)$. 
Note that $|N_{G_1}(v)|<\ceil{\frac{\omega}{4}}+(\omega-\ceil{\frac{\omega}{4}})\leq f$. 
It follows that $G_1$ is $f$-colorable if and only if $G_1-(B_i\cup B_{i+2}\cup A_1(i+1))$ is $f$-colorable. 
Since $G_1-(B_i\cup B_{i+2}\cup A_1(i+1))$ is chordal, $\chi(G_1-(B_i\cup B_{i+2}\cup A_1(i+1)))\leq \omega\leq f$.
\end{proof}

\subsection{One non-empty $A_1(i)$}\label{sec:one non-empty A_1(i)}

The global assumption in this subsection is that $A_1=A_1(i)\neq \emptyset$.
Suppose that $x\in A_1(i)$. By Proposition \ref{prop:component of A_1(i)}, the component of $A_1(i)$ containing $x$ has a neighbor $y\in A_2(i+1)$. We may assume that $xy\in E(G)$ by renaming the vertex.

\begin{lemma}\label{lem:B_{i-2} complete to B_{i+2}}
    $B_{i-2}$ and $B_{i+2}$ are complete. 
\end{lemma}

\begin{proof}
    By Lemma \ref{lem:opposite A_1(i) and A_2(j)}, $y$ is complete to $B_{i-2}\cup B_{i+2}$. We are done by Lemma \ref{lem:blowup-twoneighbor}. 
\end{proof}

\begin{lemma}\label{lem:A_2(i-2) is empty}
    $A_2(i-2)=\emptyset$. By symmetry, $A_2(i+1)=\emptyset$. 
\end{lemma}

\begin{proof}
    Suppose by contradiction that $A_2(i-2)\neq \emptyset$. Let $K$ be a component of $A_2(i-2)$. By Lemma \ref{A1A2} (2) and $A_1=A_1(i)$, $K$ has no neighbor in $A_1$. By Lemma \ref{lem:blowup-twoneighbor}, $N(K)\cap (H\cup A_5)$ is a clique. 
    Since $G$ has no clique cutsets, $K$ has a neighbor in $A_3$. 
    By Lemmas \ref{lem:neighbors of A_2(i) in A_3(i+1) and A_3(i)} and \ref{lem:anticompleteA3-1}, $N(K)\cap A_3\subseteq A_3(i-2)$ or  $N(K)\cap A_3\subseteq A_3(i-1)$. 

    Suppose first that $N(K)\cap A_3\subseteq A_3(i-2)$. Then $N(K)\subseteq B_{i-2}\cup B_{i-1}\cup A_3(i-2)\cup A_5$. Each vertex in $N(K)\cap A_5$ is universal in $N(K)$. 
    By Lemma \ref{lem:blowup-twoneighbor}, $N(K)\cap B_{i-1}$ and $N(K)\cap A'_3(i-2)$ are complete. Since $G$ has no clique cutsets, there exist $a,b\in N(K)\cap A'_3(i-2)$.
    By Lemma \ref{lem:neighbors of A_2(i) in A_3(i+1) and A_3(i)}, $a,b$ have a common neighbor in $B_{i-1}$ and so have disjoint neighborhoods in $B_{i+2}$. Then $a-K-b-(N(b)\cap B_{i+2})-B_{i+1}-(N(x)\cap B_i)-x$ is an induced sequence, a contradiction. 
    The argument for the case $N(K)\cap A_3\subseteq A_3(i-1)$ is symmetric.
    \end{proof}

\begin{lemma}\label{lem:A_3(i) is empty when A_1=A_1(i)}
    $A_3(i)=\emptyset$.
\end{lemma}

\begin{proof}
    Suppose not. Let $t\in A_3(i)$ and $b\in B_i$ be a non-neighbor of $t$. By Lemma \ref{lem:nonedge in A_3(i)}, we may assume that $b$ and $t$ have a common neighbor $c\in B_{i-1}$ and have disjoint neighborhoods in $B_{i+1}$. Since $t-c-b-(N(b)\cap B_{i+1})-q_{i+2}-y-x$ is not a bad $P_7$, $bx\in E(G)$. By symmetry, $tx\in E(G)$. But this contradicts Lemma \ref{lem:no common neighbors in A_1(i)}.
\end{proof}

Let $S=A_2(i-1)\cap N(A_1(i))$.

\begin{lemma}\label{lem:nbr of S in A_1(i)}
    For any vertex in $S$, it has a neighbor in $A_1(i)\cap N(A_2(i+2))$.  
\end{lemma}

\begin{proof}
    Let $s\in S$ and $a\in A_1(i)$ be a neighbor of $s$. Suppose by contradiction that $a\notin A_1(i)\cap N(A_2(i+2))$. Let $b\in B_{i-1}$ be a neighbor of $s$. By Lemma \ref{lem:clique components of A_1(i)}, $a$ has a neighbor $a'\in A_1(i)\cap N(A_2(i+2))$. Since $sa'\notin E(G)$, $b-s-a-a'-(N(a')\cap A_2(i+2))-q_{i+2}-q_{i+1}$ is an induced $P_7$. 
\end{proof}

\begin{lemma}\label{lem:two non-adjacent neighbors}
    For any component $K$ of $A_2(i-1)\setminus S$, $K$ has two non-adjacent neighbors in $A'_3(i-1)$ complete to $K$.
\end{lemma}

\begin{proof}
    Note that $N(K)\subseteq S\cup A_5\cup B_i\cup A'_3(i-1)$ by Lemma \ref{lem:A_3(i) is empty when A_1=A_1(i)}. We first show that each vertex in $N(K)\cap S$ is universal in $N(K)$.
    Suppose by contradiction that $s\in N(K)\cap S$ is not adjacent to $t\in N(K)$. 
    By Lemma \ref{lem:nbr of S in A_1(i)}, let $s'\in A_1(i)\cap N(A_2(i+2))$ be a neighbor of $s$. 
    If $s$ mixes on $ab\in K$, then $a-b-s-s'-(N(s')\cap A_2(i+2))-q_{i+2}-q_{i+1}$ is an induced $P_7$.
    So $s$ is complete to $K$ and thus
   $s$ and $t$ have a common neighbor $u\in K$. Since $G$ is $C_4$-free, $ts'\notin E(G)$ and thus $t\notin A_5$. So $t\in S\cup B_i\cup B_{i-1}\cup A_3(i-1)$ and then $t-u-s-s'-(N(s')\cap A_2(i+2))-q_{i+2}-q_{i+1}$ is a bad $P_7$. This shows that each vertex in $N(K)\cap S$ is universal in $N(K)$.
   This implies that each vertex in $N(K)\cap A_5$ is universal in $N(K)$. If $s\in N(K)\cap B_i$ is not adjacent to $t\in N(K)\cap A'_3(i-1)$, then $t-K-s-q_{i+1}-q_{i+2}-(N(t)\cap B_{i-2})-t$ is an induced cyclic sequence. So each vertex in $N(K)\cap B_i$ is universal in $N(K)$. Therefore, each vertex in $N(K)\setminus A'_3(i-1)$ is universal in $N(K)$.

   By Lemma \ref{lem:neighbors of components of A_2(i) in A_3(i+1)}, each pair of non-adjacent vertices in $N(K)\setminus A'_3(i-1)$ has a common neighbor in $K$ and so has disjoint neighborhoods in $B_{i-2}$. If $K\cup A'_3(i-1)$ contains an induced $P_4=a-b-c-d$ with $d\in N(K)\cap A'_3(i-1)$ and $c\in K$ and $a,b\in (N(K)\cap A'_3(i-1))\cup K$, then $a-b-c-d-(N(d)\cap B_{i-2})-q_{i+2}-q_{i+1}$ is an induced $P_7$. So $K\cup A'_3(i-1)$ contains no such induced $P_4$s.

   Therefore, all three conditions of Lemma \ref{lem:complete subgraphs} are satisfied for $K$ where $X=A'_3(i-1)$ and $Y=B_{i-2}$. This completes the proof.
\end{proof}

\begin{lemma}\label{lem:A_2(i-1) or A_2(i) is empty}
    One of $A_2(i-1)$ and $A_2(i)$ is empty.
\end{lemma}

\begin{proof}
    Suppose not. Let $v\in A_2(i-1)$ and $w\in A_2(i)$. 
    By Lemma \ref{lem:basic properties of components of A_1(i)} (1), $A_1(i)$ is anticomplete to $A_2(i)\cup A_2(i-1)$.
    By Lemma \ref{lem:two non-adjacent neighbors}, $v$ has two non-adjacent neighbors $a,b$ in  $B_{i-1}\cup A_3(i-1)$. Since $G$ is $C_4$-free, $a,b$ do not have a common neighbor in $B_{i-2}$. Then $a-v-b-(N(b)\cap B_{i-2})-q_{i+2}-(N(w)\cap B_{i+1})-w$ is an induced $P_7$.
\end{proof}

\begin{lemma}\label{lem:A_3(i-2) or A_3(i+2) is empty}
    Either $A_3(i+2)=\emptyset$ or $A_3(i-2)=\emptyset$.
\end{lemma}

\begin{proof}
    Suppose not. Let $s\in A_3(i-2)$ and $t\in A_3(i+2)$. Let $a\in B_{i-2}$ be a non-neighbor of $s$ and  $b\in B_{i-2}$ be a non-neighbor of $t$.
    By Lemmas \ref{lem:nonedge in A_3(i)} and \ref{lem:B_{i-2} complete to B_{i+2}}, there exist vertices $c,d\in B_{i-1}$ such that $a-c-d-s$ is an induced $P_4$ and  there exist vertices $p,q\in B_{i+1}$ such that $b-p-q-t$ is an induced $P_4$. Since $t-q-p-b-a-c-d$ is not an induced $P_7$, $at\in E(G)$. Then $s-d-c-a-t-q-p$ is an induced $P_7$. 
\end{proof}

Let $T=N(A_1(i))\cap A_2(i+2)$ and $T_0=A_2(i+2)\setminus T$.

\begin{lemma}\label{lem:A_3(i-2) is complete to y}
    The following holds for $A_3(i-2)$. 
    \begin{itemize}
        \item[$(1)$] $A_3(i-2)$ is complete to $T$.
        \item[$(2)$] $A_3(i-2)$ is complete to $B_{i+2}$.
    \end{itemize}
This implies by $C_4$-freeness that $T$ or $A'_3(i-2)$ is a clique.
\end{lemma}

\begin{proof} 
    We prove one by one.
    
    \noindent (1)  Let $t\in A_3(i-2)$ be non-adjacent to $y'\in T$. By definition, $y'$ has a neighbor $x'\in A_1(i)$.
    Since $B_{i-2}$ is complete to $B_{i+2}$ by Lemma \ref{lem:B_{i-2} complete to B_{i+2}}, $t$ has a non-neighbor in $B_{i-1}$ by Lemma \ref{lem:nonedge in A_3(i)} and so $t$ mixes on an edge $ab\in B_{i-1}$. Then $a-b-t-(N(t)\cap B_{i+2})-y'-x'-(N(x')\cap B_i)$ is a bad $P_7$. 

    \noindent (2) This follows from (1) and Lemma \ref{lem:blowup-twoneighbor}.
\end{proof}

Next we explore the internal structure of $A_2(i+2)$.

\begin{lemma}\label{lem:clique components of A_2(i+2)}
    There exists $j\in \{i-2,i+2\}$ such that
    every component of $T_0$ has two non-adjacent neighbors in either $T$ or $A'_3(j)$ but not in both and it is complete to two non-adjacent vertices in $T$ or $A'_3(j)$.
\end{lemma}

\begin{proof}
    Let $K$ be a component of $T_0$. By symmetry, we may assume that $A_3(i+2)=\emptyset$ by Lemma \ref{lem:A_3(i-2) or A_3(i+2) is empty}.
    Note that $N(K)\subseteq A_3(i-2)\cup B_{i-2}\cup T\cup B_{i+2}\cup A_5$. 
    By Lemma \ref{lem:A_3(i-2) is complete to y}, each vertex in $N(K)\cap (A_5\cup B_{i+2})$ is universal in $N(K)$, and either each vertex in $N(K)\cap T$ or each vertex in $N(K)\cap A'_3(i-2)$ is universal in $N(K)$.

    Suppose first that $T$ is a clique and so each vertex in $N(K)\setminus A'_3(i-1)$ is universal in $N(K)$. By Lemma \ref{lem:A_3(i-2) is complete to y} (2), each pair of non-adjacent vertices in $N(K)\cap A'_3(i-2)$ has disjoint neighborhoods in $B_{i-1}$. 
    If there is an induced $P_4=a-b-c-d$ with $d\in N(K)\cap A'_3(i-1)$, $c\in K$ and $a,b\in (N(K)\cap A'_3(i-1))\cup K$, then $a-b-c-d-(N(d)\cap B_{i-1})-q_i-q_{i+1}$ is an induced $P_7$, a contradiction. So all three conditions of Lemma \ref{lem:complete subgraphs} are satisfied for $K$ where $X=A'_3(i-2)$ and $Y=B_{i-1}$. Therefore, the lemma follows.

    Now suppose that $A'_3(i-1)$ is a clique. Then each vertex in $N(K)\setminus T$ is universal in $N(K)$.  For any  pair of non-adjacent vertices in $N(K)\cap T$, they have a common neighbor $q_{i+2}$ and so have disjoint complete neighborhoods in $A_1(i)$ by $(C_4,C_6,C_7)$-freeness of $G$.
    If there is an induced $P_4=a-b-c-d$ with $d\in N(K)\cap T$, $c\in K$ and $a,b\in (N(K)\cap T)\cup K$, then $a-b-c-d-(N(d)\cap A_1(i))-q_i-q_{i+1}$ is an induced $P_7$, a contradiction. So all three conditions of Lemma \ref{lem:complete subgraphs} are satisfied for $K$ where $X=T$ and $Y=A_1(i)$. Therefore, the lemma follows.
\end{proof}

\begin{lemma}\label{lem:A_5 is complete to A_2(i+2)}
    Each vertex in $A_5$ is a universal vertex in $G$.
\end{lemma}

\begin{proof}
    We first show that $A_5$ is complete to $A_2(i+2)$.
    By the choice of $C$, $A_5$ is complete to $T$. By Lemma \ref{lem:clique components of A_2(i+2)}, every component $K$ of $T_0$ has two non-adjacent vertices $\{a,b\}$ complete to $K$ where $a,b\in T$ or $a,b\in A'_3(i-2)$. Since $A_5$ is complete to $T\cup A'_3(i-2)$, $A_5$ is complete to $K$ by $C_4$-free. 

    Next we show that $A_5$ is complete to $A_2(i)\cup A_2(i-1)$. Suppose first that $s\in A_2(i-1)$ has no neighbor in $A_1(i)$. By Lemma \ref{lem:two non-adjacent neighbors}, $s$ has two non-adjacent vertices $\{a,b\}$ with $a,b\in A'_3(i-1)$. Since $G$ is $C_4$-free, $A_5$ is complete to $s$. Now suppose that $s$ has a neighbor in $A_1(i)$. By Lemma \ref{lem:nbr of S in A_1(i)}, $s$ has a neighbor $s'\in A_1(i)$ which has a neighbor in $A_2(i+2)$. Then $A_5$ is complete to $s'$. Since $s-(N(s)\cap B_{i-1})-A_5-s'-s$ is not an induced $C_4$, $A_5$ is complete to $s$. This proves that $A_5$ is complete to $A_2(i-1)\cup A_2(i)$.
    
    The lemma now follows from Lemma \ref{lem:A_5 is complete to A_1}.
\end{proof}

So we may assume that $A_5=\emptyset$.

\begin{lemma}\label{lem:A'_3(i-2) is large}
    There exists a vertex $v\in B_{i-1}\setminus N(A_3(i-1))$ such that $|N(v)\cap B_{i-2}|>\ceil{\frac{\omega}{4}}$. By symmetry, There exists a vertex $u\in B_{i+1}\setminus N(A_3(i+1))$ such that $|N(u)\cap B_{i+2}|>\ceil{\frac{\omega}{4}}$.
\end{lemma}

\begin{proof}
    If $A_3(i-1)\neq \emptyset$, let $v$ be a vertex in $B_{i-1}\setminus N(A_3(i-1))$ with minimal neighborhood in $B_i$.
    If $A_3(i-1)=\emptyset$, then let $v=q_{i-1}$ if $B_{i-1}$ is complete to $B_i$, and let $v\in B_{i-1}$ be a vertex with minimal neighborhood in $B_i$ if $B_{i-1}$ is not complete to $B_i$. We show that $v$ is anticomplete to $A_3(i-2)$ and $N(v)\cap (B_{i-1}\cup B_i)$ is a clique. Suppose first that $A_3(i-1)\neq \emptyset$. By Lemma \ref{lem:compare nbd of different A_3(i) in B_j}, $v$ is anticomplete to $A_3(i-2)$. 
    By Lemma \ref{lem:single vertex in A_3(i)} (3) and the choice of $v$, $N(v)\cap (B_{i-1}\cup B_i)$ is a clique. If $v=q_{i-1}$, then the claim holds by the choice of $v$.
    If $A_3(i-1)=\emptyset$ and $B_{i-1}$ is not complete to $B_i$, then the claim holds by Lemma \ref{lem:B_i complete to two large cliques in B_{i-1}} and the choice of $v$.
    So
    $N(v)\setminus B_{i-2}\subseteq B_{i-1}\cup B_i\cup S\cup (A_2(i-1)\setminus S)$. 
    Suppose that $N(v)\setminus B_{i-2}$ contains two non-adjacent vertices $s$ and $s'$. Suppose first that $s\in S$. By Lemma \ref{lem:nbr of S in A_1(i)}, we can choose $p\in A_1(i)\cap N(A_2(i+2))$ to be a neighbor of $s$. Since $s'-v-s-p-s'$ is not an induced $C_4$, $ps'\notin E(G)$. Then $s'-v-s-p-(N(p)\cap A_2(i+2))-q_{i+2}-q_{i+1}$ is a bad $P_7$. So $s\notin S$. Since $N(v)\cap (B_{i-1}\cup B_i)$ is a clique, we may assume that $s\in A_2(i-1)\setminus S$. By Lemma \ref{lem:two non-adjacent neighbors}, $s$ has two non-adjacent neighbors $a,b$ in $A'_3(i-1)$. By symmetry, we may assume that $a\in A_3(i-1)$. Since $s$ is a common neighbor of $a,b$, we have $a,b$ have disjoint neighborhoods in $B_{i-2}$. Since $b-s-a-N(a)\cap B_{i-2}-q_{i+2}-q_{i+1}-B_i$ is not a bad $P_7$, we have $a$ is complete to $B_i$. It follows that $a,s'$ and $a,v$ have a common neighbor in $B_i$ and so $N(a)\cap B_{i-2}$ is anticomplete to $s'$ and $v$. Since $va\notin E(G)$ by the choice of $v$, $s'-v-s-a$ is an induced $P_4$. Then $s'-v-s-a-(N(a)\cap B_{i-2})-q_{i+2}-q_{i+1}$ is an induced $P_7$ or $C_7$. 
    So $N(v)\setminus B_{i-2}$ is a clique. 
    Since $G$ has no small vertices, $|N(v)\cap B_{i-2}|>\ceil{\frac{\omega}{4}}$. 
\end{proof}

\begin{lemma}\label{lem:T_0 is a clique}
    $T_0$ is a clique.
\end{lemma}

\begin{proof}
    Suppose not. Let $K$ and $K'$ be two components of $T_0$. Let $v\in K$ and $v'\in K'$.  By Lemma \ref{lem:clique components of A_2(i+2)}, $K$ has two non-adjacent neighbors $a,b$ complete to $K$ and  $K'$ has two non-adjacent neighbors $a',b'$ complete to $K'$.

    Suppose first that $a,b\in A'_3(i-2)$. Then $a',b'\in A'_3(i-2)$.
    Since $G$ is $C_4$-free, $v'$ is not adjacent to either $a$ or $b$. If $v'a\in E(G)$, then $v'-a-v-b-(N(b)\cap B_{i-1})-q_i-q_{i+1}$ is an induced $P_7$. So $v'a,v'b\notin E(G)$. By symmetry, $va',vb'\notin E(G)$. So $a,a',b,b'$ are pairwise distinct. If $aa'\in E(G)$, then $a'b\notin E(G)$ or $a'-a-v-b-a'$ is an induced $C_4$, and so $a'-a-v-b-N_{B_{i-1}}(b)-q_i-q_{i+1}$ is an induced $P_7$. So $\{a,a',b,b'\}$ is a stable set. Let $c,c'\in B_{i-1}$ be a neighbor of $a,a'$, respectively. Then $b-v-a-c-c'-a'-v'$ is an induced $P_7$.

    Suppose now that $a,b\in T$. Then $a',b'\in T$.
    Since $G$ is $C_4$-free, $v'$ is not adjacent to either $a$ or $b$. If $v'a\in E(G)$, then $v'-a-v-b-(N(b)\cap A_1(i))-B_i-q_{i+1}$ is an induced $P_7$. So $v'a,v'b\notin E(G)$. By symmetry, $va',vb'\notin E(G)$. 
    So $a,a',b,b'$ are pairwise distinct. If $aa'\in E(G)$, then $a'b\notin E(G)$ or $a'-a-v-b-a'$ is an induced $C_4$, and so $a'-a-v-b-N_{A_1(i)}(b)-B_i-q_{i+1}$ is an induced sequence. So $\{a,a',b,b'\}$ is a stable set. Let $c,c'\in A_1(i)$ be a neighbor of $a,a'$, respectively. Then $b-v-a-c-c'-a'-v'$ is an induced $P_7$.
\end{proof}

\begin{lemma}\label{lem: T is not a clique}
    Either $T$ is a clique or $A_3(i-2)\cup A_3(i+2)=\emptyset$.
\end{lemma}

\begin{proof}
    Suppose that $T$ contains two non-adjacent vertices $a$ and $b$. If $t\in A_3(i-2)$, then $t$ has a non-neighbor $s\in B_{i-2}$ and it follows from Lemma \ref{lem:A_3(i-2) is complete to y} that $a-s-b-t-a$ is an induced $C_4$. 
\end{proof}

\begin{lemma}\label{lem:B_{i-1} is large}
    For any non-empty component $K$ of $A_3(i-2)$ and any simplicial vertex $t\in K$,  $N(t)\setminus B_{i-1}$ is a clique and so $|N(t)\cap B_{i-1}|>\ceil{\frac{\omega}{4}}$. 
\end{lemma}

\begin{proof}
    Note that $N(t)\setminus B_{i-1}\subseteq K\cup B_{i-2}\cup B_{i+2}\cup T\cup T_0$. Suppose that $N(t)\setminus B_{i-1}$ contains two non-adjacent vertices $v$ and $u$. Suppose first that $v,u\notin T_0$. Since $B_{i-2}\cup B_{i+2}\cup T$ is a clique by Lemma \ref{lem: T is not a clique}
    and $N_K(t)$ is a clique, we may assume that $v\in K$ and $u\in B_{i-2}\cup B_{i+2}\cup T$. Since $B_{i+2}\cup T$ is complete to $A_3(i-2)$ by Lemma \ref{lem:A_3(i-2) is complete to y}, $u\in B_{i-2}$. But this contradicts Lemma \ref{lem:edge in $A_3(i)$}. So we may assume that $v\in T_0$. By Lemma \ref{lem:clique components of A_2(i+2)}, $v$ has two non-adjacent neighbors $a,b\in T$ or $A'_3(i-2)$. By Lemma \ref{lem: T is not a clique}, $a,b\in A'_3(i-2)$. We may assume that $a\in A_3(i-2)$ and $t\neq b$. If $tb\notin E(G)$, then $u-t-v-b-(N(b)\cap B_{i-1})-q_i-q_{i+1}$ is an induced $P_7$. So $tb\in E(G)$ and so $t\neq a$. By symmetry, $ta\in E(G)$. Since $t$ is simplicial in $K$, $b\in B_{i-2}$ and so $b$ mixes on $at\in E(K)$. This contradicts Lemma \ref{lem:edge in $A_3(i)$}.
\end{proof}

\begin{lemma}\label{lem:B_i is large}
    $|B_i|>\ceil{\frac{\omega}{4}}$.
\end{lemma}

\begin{proof}
    By Lemma \ref{lem:A_2(i-1) or A_2(i) is empty}, we may assume that $A_2(i-1)=\emptyset$.
    We choose a vertex $v\in B_{i-1}$ anticomplete to $A_3(i-2)\cup A_3(i-1)$ such that $|N(v)\cap B_{i-2}|$ is minimum. Then $N(v)\setminus B_i\subseteq B_{i-1}\cup B_{i-2}$. Suppose that $N(v)\setminus B_i$ contains two non-adjacent vertices $s,t$ with $s\in B_{i-1}$ and $t\in B_{i-2}$. By the choice of $v$, $s$ has a neighbor $r$ in $A_3(i-1)$ or $A_3(i-2)$. By Lemma \ref{lem:single vertex in A_3(i)} (3), $r\in A_3(i-2)$.
    
    By Lemma \ref{lem:A_3(i-2) or A_3(i+2) is empty}, $A_3(i+2)=\emptyset$. If $A_2(i)= \emptyset$, we are done by symmetry. 
    So $A_2(i)$ contains a vertex $w$. If $w$ has a neighbor $w'\in A_1(i)$, then $v-s-r-q_{i+2}-N_{B_{i+1}}(w)-w-w'$ is an induced $P_7$. So $w$ is anticomplete to $A_1(i)$. 
    By Lemma \ref{lem:two non-adjacent neighbors}, $w$ has two non-adjacent neighbors $a,b\in A'_3(i+1)$. Then $v-s-r-N_{B_{i+2}}(a)-a-w-b$ is an induced $P_7$.
    This proves that $N(v)\setminus B_i$ is a clique and the lemma follows.
\end{proof}

\begin{lemma}\label{lem:large cliques in A_1(i)}
   For any two non-adjacent vertices $t_1,t_2\in T$, $N(t_1)\cap A_1(i)$ and $N(t_2)\cap A_1(i)$ are disjoint and  $(N(t_1)\cap A_1(i))\cup (N(t_2)\cap A_1(i))\cup B_i$ is a clique.
\end{lemma}

\begin{proof}
    Since $q_{i+2}$ is a common neighbor of $t_1$ and $t_2$ and $G$ is $C_4$-free, $N(t_1)\cap A_1(i)$ and $N(t_2)\cap A_1(i)$ are disjoint. If $s_j\in N(t_j)\cap A_1(i)$ are not adjacent, then $t_1-s_1-B_i-s_2-t_2-q_{i+2}-t_1$ is an induced cyclic sequence. So $N(t_1)\cap A_1(i)$ and $N(t_2)\cap A_1(i)$ are complete. Since $G$ is $C_4$-free, $N(t_1)\cap A_1(i)$ and $N(t_2)\cap A_1(i)$ are clique. 

    Now suppose $b\in B_i$ has a neighbor $s_1\in N(t_1)\cap A_1(i)$. If $b$ has a non-neighbor $s_2\in N(t_2)\cap A_1(i)$, then $b-s_1-s_2-t_2-q_{i+2}-q_{i+1}-b$ is an induced $C_6$. So $b$ is complete to $N(t_2)\cap A_1(i)$. It follows that $b$ is complete to $N(t_1)\cap A_1(i)$. So each vertex in $B_i$ is either complete or anticomplete to $(N(t_1)\cap A_1(i))\cup (N(t_2)\cap A_1(i))$. If there is a vertex $b\in B_i$ anticomplete to $(N(t_1)\cap A_1(i))\cup (N(t_2)\cap A_1(i))$, then $s_1-s_2-t_2-q_{i-2}-q_{i-1}-b-q_{i+1}$ is an induced $P_7$. This proves that $B_i$ is complete to $(N(t_1)\cap A_1(i))\cup (N(t_2)\cap A_1(i))$. 
\end{proof}

\begin{lemma}\label{lem:structure of T}
    $T$ is a blowup of $P_3$. This implies that $\alpha(T)\le 2$.
\end{lemma}

\begin{proof}
    If $T$ is a clique, then there is nothing to prove. So $T$ is not a clique and so $A_3(i-2)\cup A_3(i+2)=\emptyset$ by Lemma \ref{lem: T is not a clique}. 

    \medskip
    \noindent (1) Each vertex in $B_{i-2}\cup B_{i+2}$ is universal in $T\cup T_0\cup B_{i-2}\cup B_{i+2}$.   
    
    \medskip
    By Lemmas \ref{lem:blowup-twoneighbor} and \ref{lem:opposite A_1(i) and A_2(j)}, $B_{i+2}\cup B_{i-2}$ is a clique complete to $T$.  If $T_0$ is empty, we are done. 
    So $T_0$ is not empty. By Lemma \ref{lem:clique components of A_2(i+2)}, there are two non-adjacent vertices $a,b\in T$ complete to $T_0$. Since $G$ is $C_4$-free, $B_{i-2}\cup B_{i+2}$ is universal in $T\cup T_0\cup B_{i-2}\cup B_{i+2}$. This proves (1).

    \medskip
    \noindent (2) $T$ cannot contain three pairwise non-adjacent vertices simplicial in $T\cup T_0\cup B_{i-2}\cup B_{i+2}$.  
    
    \medskip
     Suppose that $T$ contains three pairwise non-adjacent vertices $v_1,v_2,v_3$ simplicial in $T\cup T_0\cup B_{i-2}\cup B_{i+2}$. For $j=1,2,3$, we have $|N(v_j)\cap A_1(i)|>\ceil{\frac{\omega}{4}}$.
    It follows from Lemmas \ref{lem:large cliques in A_1(i)} and \ref{lem:B_i is large} that $(N(v_1)\cap A_1(i))\cup (N(v_2)\cap A_1(i))\cup (N(v_3)\cap A_1(i))\cup B_i$ is a clique of size larger than $\omega$, a contradiction. This proves (2).
    
    \medskip 
    \noindent (3) Let $T'\subseteq T$ be cutset of $G[T]$ such that each vertex in $T'$ is universal in $A_2(i+2)\cup B_{i-2}\cup B_{i+2}$. Then for each component $T'$-component $K$ of $G[T]$, there is a vertex $v_K$ in $K\setminus T'$ that is simplicial in $K\cup T_0\cup B_{i-2}\cup B_{i+2}$.
    
    \medskip
    If $K\cup T_0\cup B_{i-2}\cup B_{i+2}\cup T'$ is a clique, there is nothing to prove. If $K\cup T_0\cup B_{i-2}\cup B_{i+2}$ is not a clique, then $K\cup T_0\cup B_{i-2}\cup B_{i+2}$ contains two non-adjacent simplicial vertices $s$ and $s'$ by Lemma \ref{lem:A_2(i) is chordal}. Since each vertex in $B_{i-2}\cup B_{i+2}\cup T'$ is universal in $K\cup T_0\cup B_{i-2}\cup B_{i+2}$, we have $s,s'\in (K\setminus T')\cup T_0$. Since $T_0$ is a clique by Lemma \ref{lem:T_0 is a clique}, one of $s,s'$ is in $K\setminus T'$. This proves (3).

    \medskip
    By (2) and (3), $T$ has at most two components. Next we show the following. 
    
    \medskip 
    \noindent $(4)$ Let $L_1,L_2\subseteq T$ be two connected subgraphs such that $L_1$ and $L_2$ are disjoint and anticomplete, and each vertex in $T\setminus (L_1\cup L_2)$ is universal in $A_2(i+2)\cup B_{i-2}\cup B_{i+2}$. Then $L_1$ and $L_2$ are cliques. 
    
    \medskip 
    Let $T'=T\setminus (L_1\cup L_2)$. 
    Suppose first that $T_0=\emptyset$. By the assumption of (4),
    every simplicial vertex of $L_i$ is simplicial in $A_2(i+2)\cup B_{i-2}\cup B_{i+2}$. 
    By (2), $L_1$ and $L_2$ are cliques.
    So we assume that $T_0\neq \emptyset$. By Lemma \ref{lem:clique components of A_2(i+2)}, there exist two non-adjacent vertices $a,b\in T$ complete to $T_0$. Since each vertex in $T'$ is universal in $A_2(i+2)\cup B_{i-2}\cup B_{i+2}$, $a,b\in L_1\cup L_2$.
    
    Suppose first that $a,b$ belong to the same $L_i$, say $L_1$. Since $L_1\cup T'\cup T_0$ is not a clique, it contains two non-adjacent vertices $s,s'$ simplicial in $L_1\cup T'\cup T_0$ by Lemmas \ref{lem:A_2(i) is chordal} and \ref{lem:chordal has simplicial vertices}. By our assumption on $T'$, $s,s'\in L_1$. By $L_1$ is anticomplete to $L_2$ and (1), $s,s'$ are simplicial in $A_2(i+2)\cup B_{i-2}\cup B_{i+2}$. By (3), $L_2$ contains a vertex $t$ simplicial in  $A_2(i+2)\cup B_{i-2}\cup B_{i+2}$. Then $s,s',t$ contradicts (2).
    
    So we assume that $a\in L_1$ and $b\in L_2$.
    Let $u\in L_1$ be a simplicial vertex of $L_1$. 
    Choose a simplicial vertex $s\in N_{L_1}[u]$ with $|N(s)\cap T_0|$ is minimum. 
    Suppose that $s$ is not simplicial in $T'\cup L_1\cup T_0$. 
    Then $s$ has two non-adjacent vertices $p\in L_1$ and $q\in T_0$. 
    If $p$ is not simplicial in $L_1$, then $p$ has a neighbor $r\in L_1$ not adjacent to $s$. 
    It follows that $r-p-s-q-b-(N(b)\cap A_1(i))-B_i$ is an induced sequence. 
    So $p$ is simplicial in $L_1$. 
    By the choice of $s$, $p$ has a neighbor $q'\in T_0$ not adjacent to $s$. 
    It follows that $p-q'-q-s-p$ is an induced $C_4$. 
    So $s$ is simplicial in $T'\cup T_0\cup L_1$. 
    By Lemma \ref{lem:two non-adjacent simplicail verteices}, (2) and (3), $L_1$ and $L_2$ are cliques. This proves (4).

    \medskip
    If $T$ has two components, then each component is a clique by applying (4) with $L_1$ and $L_2$ being the two components of $T$. So $T$ is connected. 
    Let $x,y\in T$ with $xy\notin E(G)$. If $p\in N(x)\cap A_1(i)$ and $q\in N(y)\cap A_1(i)$ are not adjacent, then $p-B_i-q-y-B_{i+2}-x-p$ is an induced cyclic sequence. If there is an induced $P_5=a-b-c-d-e$ with $a,b,c,d\in T$ and $e\in A_1(i)$, then $a-b-c-d-e-N_{B_i}(e)-q_{i+1}$ is an induced $P_7$. By Lemma \ref{lem:opposite A_1(i) and A_2(j)}, $q_{i+2}$ is complete to $T$. So all three conditions of Lemma \ref{lem:P_4-free lemma} are satisfied for $X=T$, $Y=A_1(i)$ and $u=q_{i+2}$.
    By Lemma \ref{lem:P_4-free lemma}, $T$ is $P_4$-free and so $T$ can be partitioned into two subsets $T_1$ and $T_2$ such that  $T_1$ and $T_2$ are complete. Since $G$ is $C_4$-free, we may assume that $T_1$ is a clique. We choose such a partition $(T_1,T_2)$ so that $|T_1|$ is maximum. It follows that $T_2$ is disconnected. Since $A_3(i-2)\cup A_3(i+2)=\emptyset$ and $T_1$ is a universal clique in $T$, $T_1$ is complete to $T_0$. It follows that each component of $T_2$ has a vertex simplicial in $A_2(i+2)\cup B_{i-2}\cup B_{i+2}$.
    By (2), $T_2$ has exactly two components. By applying (4) to the two components of $T_2$, each component of $T_2$ is a clique. So $T$ is a blowup of $P_3$.
\end{proof}

\begin{lemma}\label{lem:T is complete to T_0}
    $T$ is complete to $T_0$.
\end{lemma}

\begin{proof}
    If $T_0=\emptyset$, there is nothing to prove. Let us assume that $T_0$ is not empty. By Lemmas \ref{lem:T_0 is a clique} and \ref{lem:clique components of A_2(i+2)}, $T_0$ is complete to two non-adjacent vertices $a,b$ where $a,b\in A'_3(i-2)$ or $a,b\in T$. If $a,b\in A'_3(i-2)$, then $T\cup T_0$ is complete to $\{a,b\}$ by Lemma \ref{lem:A_3(i-2) is complete to y} and so $T\cup T_0$ is a clique. 

    So $a,b\in T$. Suppose for a contradiction that $c\in T$ is not adjacent to $v\in T_0$. Then $c\neq a,b$. Since $G$ is $C_4$-free, $c$ is not adjacent to $a$ or $b$. If $c$ is adjacent to $a$, then $c-a-v-b-(N(b)\cap A_1(i))-B_i-q_{i+1}$ is an induced sequence. So $c$ is anticomplete to $\{a,b\}$. Then $\{a,b,c\}$ is an independent set in $T$, contradicting Lemma \ref{lem:structure of T}.
\end{proof}

\begin{lemma}\label{lem:T_0 is complete to B_{i-2} and B_{i+2}}
    $T_0$ is complete to $B_{i-2}\cup B_{i+2}$.
\end{lemma}

\begin{proof}
    If $T_0$ is empty, there is nothing to prove. Let $v\in T_0$. By Lemma \ref{lem:clique components of A_2(i+2)}, there are two vertices $a,b$ complete to $T_0$ where $a,b\in A'_3(i-2)$ or $a,b\in T$. If $a,b\in T$, then $B_{i-2}\cup B_{i+2}\cup T_0$ is complete to $\{a,b\}$ and so we are done. So assume that $a,b\in A'_3(i-2)$.

    By Lemmas \ref{lem:B_{i-2} complete to B_{i+2}} and \ref{lem:A_3(i-2) is complete to y}, $T_0$ is complete to $B_{i+2}$. Suppose for a contradiction that $u\in B_{i-2}$ is not adjacent to $v$. Then $u\neq a,b$. Since $G$ is $C_4$-free, $u$ is not adjacent to $a$ or $b$. If $u$ is adjacent to $a$, then $u-a-v-b-(N(b)\cap B_{i-1})-q_i-q_{i+1}$ is an induced $P_7$. So $u$ is anticomplete to $\{a,b\}$. In particular, $a,b\in A_3(i-2)$.

    Choose a vertex $w\in B_{i-2}$ anticomplete to $A_3(i-2)$. By Lemma \ref{lem:compare nbd of different A_3(i) in B_j}, $w$ is anticomplete to $A_3(i-1)$.
    By Lemma \ref{lem:A_3(i-2) or A_3(i+2) is empty} and $A_3(i-2)\neq \emptyset$, we have $A_3(i+2)=\emptyset$. By Lemma \ref{lem:A_2(i-2) is empty}, $w$ has no neighbor in $A_2\setminus A_2(i+2)$. 
    Then $N(w)\setminus B_{i-1}\subseteq B_{i-2}\cup B_{i+2}\cup T\cup T_0$. As $T_0\cup B_{i+2}$ is complete to $\{a,b\}$,  $T_0\cup B_{i+2}$ is a clique. Suppose that $v',u\in N(w)\setminus B_{i-1}$ are not adjacent. By Lemma \ref{lem:T is complete to T_0},
    $u\in B_{i-2}$ and $v'\in T_0$ and so $u-w-v'-a-N_{B_{i-1}}(a)-q_{i}-q_{i+1}$ is an induced $P_7$. This proves that $N(w)\setminus B_{i-1}$ is a clique and so $|N(w)\cap B_{i-1}|>\ceil{\frac{\omega}{4}}$. 

    Since $a,b\in A_3(i-2)$ are not adjacent, there exist two non-adjacent vertices $t_1,t_2\in A_3(i-2)$ such that $|N(t_j)\cap B_{i-1}|>\ceil{\frac{\omega}{4}}$ for $j=1,2$ by Lemma \ref{lem:B_{i-1} is large}. By Lemmas \ref{lem:B_i is large} and \ref{lem:B_i complete to two large cliques in B_{i-1}}, $B_i\cup (N(t_1)\cap B_{i-1})\cup (N(t_2)\cap B_{i-1})\cup (N(w)\cap B_{i-1})$ is a clique of size larger than $\omega$.
\end{proof}

\begin{lemma}\label{lem:A_1(i) is dominated by A_2(i+2)}
    $A_1(i)\setminus N(A_2(i+2))=\emptyset$. In other words, each vertex in $A_1(i)$ has a neighbor in $A_2(i+2)$.
\end{lemma}

\begin{proof}
    Suppose not. Let $X=N(A_2(i+2))\cap A_1(i)$.
    We first show that $A_1(i)\setminus X$ is a clique. Suppose for a contradiction that $K_1,K_2$ are two components of $A_1(i)\setminus X$. By Lemma \ref{lem:clique components of A_1(i)}, there are two non-adjacent vertices $a_1,b_1\in X$ complete to $K_1$ and two non-adjacent vertices $a_2,b_2\in X$ complete to $K_2$. Let $v_1$ be a vertex in $K_1$. 
    Since $v_1\notin K_2$, $v_1$ is not complete to $\{a_2,b_2\}$. If $v_1$ is adjacent to $a_2$, then $v_1-a_2-K_2-b_2-(N(b_2)\cap T)-q_{i+2}-q_{i+1}$ is an induced $P_7$. This proves that $K_1$ is anticomplete to $\{a_2,b_2\}$. By symmetry, $K_2$ is anticomplete to  $\{a_1,b_1\}$. It follows that $a_1,b_1,a_2,b_2$ are four pairwise distinct vertices.  Let $b_1'\in N(b_1)\cap T$ and $b_2'\in N(b_2)\cap T$. 

    Next, we show that $\{a_1,b_1,a_2,b_2\}$ is a stable set. Suppose for a contradiction that $a_1a_2\in E(G)$. Since $a_1$ is anticomplete to $K_2$ and $a_2$ is anticomplete to $K_1$, we have $a_1b_2,a_2b_1\notin E(G)$. Since $a_1-a_2-K_2-b_2-N_T(b_2)-q_{i+2}-q_{i+1}$ is not an induced $P_7$, $a_1$ is complete to $N(b_2)\cap T$. By symmetry, $a_2$ is complete to $N(b_1)\cap T$. 
    Then $a_1b_2',a_2b_1'\in E(G)$. Since $a_1,b_1$ have a common neighbor in $K_1$, the neighborhoods of $a_1$ and $b_1$ in $T$ are disjoint and complete cliques and so $b_1'b_2'\in E(G)$. It follows that $a_1-a_2-b_1'-b_2'-a_1$ is an induced $C_4$. 
    This proves that $\{a_1,b_1,a_2,b_2\}$ is stable.

    If $b_1'=b_2'$, then $a_1,a_2$ are not adjacent to $b_1'$ since $G$ is $C_4$-free. Then $a_1-K_1-b_1-b_1'-b_2-K_2-a_2$ is an induced $P_7$. Therefore, each pair of vertices from $\{a_1,b_1,a_2,b_2\}$ has no common neighbors in $T$. 
    Then $a_1-K_1-b_1-\{b_1',q_{i+2},b_2'\}-b_2-K_2-a_2$ is an induced sequence.
    This proves that $A_1(i)\setminus X$ is connected. By Lemma \ref{lem:clique components of A_1(i)}, $A_1(i)\setminus X$ is a clique.

    Let $u\in A_1(i)\setminus X$.
    By Lemma \ref{lem:clique components of A_1(i)}, there are two non-adjacent vertices $a,b\in X$ complete to $A_1(i)\setminus X$. If $v\in B_i$ that is not adjacent to $a$, then $b-u-a-(N(a)\cap T)-q_{i+2}-q_{i+1}-v$ is a bad $P_7$. So $a,b$ are complete to $B_i\cup (A_1(i)\setminus X)$ and $B_i\cup (A_1(i)\setminus X)$ is a clique. Suppose that $v\in N(a)\cap X$ is not adjacent to $v'\in A_1(i)\setminus X$. It follows that $v'b\in E(G)$ and so $vb\notin E(G)$. 
    Since $v$ and $b$ have a common neighbor in $B_i$, $v-a-v'-b-(N(b)\cap T)-q_{i+2}-q_{i+1}$ is an induced $P_7$.  So $N(a)\cap X$ is complete to $A_1(i)\setminus X$.
    By symmetry, $\{v,b\}$ is complete to $B_i$.
    So $N(a)\cap X$ is complete to $B_i\cup (A_1(i)\setminus X)$.

    We choose a non-adjacent pair $\{a_0,b_0\}$ of vertices in $X$ complete to $A_1(i)\setminus X$ such that $|N(a_0)\cap X|+|N(b_0)\cap X|$ is minimum.
    Next we show that $N(a_0)\setminus T$ is a clique.  If $a_0$ has a neighbor $t\in A_3(i-1)$, then $a_0-t-(N(t)\cap B_{i-2})-(N(a_0)\cap T)-a_0$ is an induced $C_4$. So $a_0$ is anticomplete to $A_3(i-1)\cup A_3(i+1)$.
    By Lemma \ref{lem:A_2(i-1) or A_2(i) is empty}, we may assume that $A_2(i)=\emptyset$ and so $N(a)\setminus T\subseteq B_i\cup A_1(i)\cup A_2(i-1)$.
    Let $s,s'$ be two non-adjacent neighbors of $a_0$ in $B_i\cup A_1(i)\cup A_2(i-1)$. By Proposition \ref{prop:component of A_1(i)}, $N(a_0)\cap (B_i\cup A_2(i-1))$ is a clique and so we may assume that $s\in A_1(i)$. By Lemma \ref{lem:A2i1tocompA1}, $s'\notin A_2(i-1)$ and so $s'\in B_i\cup A_1(i)$. Since $B_i\cup (A_1(i)\setminus X)$ is a clique, we may assume that $s\in X$. Since $s$ is complete to $B_i\cup (A_1(i)\setminus X)$, we have $s'\in X$. Since $G$ is $C_4$-free, $b_0$ is not adjacent to $s$ or $s'$, say $sb_0\notin E(G)$.
    By the choice of $\{a_0,b_0\}$, $s$ has a neighbor $s'' \in X$ with $a_0s''\notin E(G)$. 
    Applying the claim $N(a)\cap X$ is complete to $A_1(i)\setminus X$ to the pair $\{s,b_0\}$, we have $s''$ is complete to $A_1(i)\setminus X$. So $s''-s-a_0-s'$ is induced $P_4$, which contradicts Lemma \ref{lem:P_4-free lemma}.

    So $N(a_0)\setminus T$ is a clique and thus $|N(a_0)\cap T|>\ceil{\frac{\omega}{4}}$. By symmetry, $|N(b_0)\cap T|>\ceil{\frac{\omega}{4}}$. 
    Since $a_0,b_0$ have a common neighbor in $A_1(i)\setminus X$, $N(a_0)\cap T$ and $N(b_0)\cap T$ are disjoint and complete cliques. By Lemma \ref{lem:A'_3(i-2) is large}, $(N(a_0)\cap T)\cup (N(b_0)\cap T)\cup B_{i+2}\cup B_{i-2}$ is a clique of size larger than $\omega$, a contradiction.
\end{proof}

\begin{lemma}\label{lem:A_1(i) is P_4-free}
    $A_1(i)$ is $P_4$-free.
\end{lemma}

\begin{proof}
    Suppose by contradiction that $a-b-c-d$ is an induced $P_4$ in $A_1(i)$. By Lemma \ref{lem:induced P3in A_1(i)}, $N(a)\cap B_i,N(d)\cap B_i\subseteq N(b)\cap B_i=N(c)\cap B_i$. Suppose first that $a$ and $d$ have no common neighbors in $B_i$. Since $G$ is $(C_6,C_7)$-free, $a$ and $d$ have neighborhood containment in $A_2(i+2)$, say $N(a)\cap T\subseteq N(d)\cap T$. 
    By Lemma \ref{lem:A_1(i) is dominated by A_2(i+2)}, there is a vertex $s\in N(a)\cap A_2(i+2)$. Then $sd\in E(G)$. Since $b,d$ has a common neighbor $u\in B_i$, $sb\notin E(G)$ and then $u-b-a-s-q_{i+2}-q_{i+1}-u$ is an induced $C_6$. So $a,d$ have a common neighbor  $s\in B_i$ and this vertex is complete to $a-b-c-d$. 
    But this contradicts Lemma \ref{lem:P_4-free lemma}.
\end{proof}

\begin{lemma}\label{lem:A_1(i) is anticomplete to A_3}
    $A_1(i)$ is anticomplete to $A_3(i-1)\cup A_3(i+1)$. This implies that $A_1(i)$ is anticomplete to $A_3$.
\end{lemma}

\begin{proof}
    Let $t\in A_3(i-1)$ be adjacent to $a\in A_1(i)$. Let $b\in B_{i-1}$ be a non-neighbor of $t$. By Lemma \ref{lem:A_1(i) is dominated by A_2(i+2)}, $a$ has a neighbor $c\in T$. Then $b-q_{i-1}-t-a-c-q_{i+2}-q_{i+1}$ is an induced $P_7$.
\end{proof}

\begin{lemma}\label{lem:small vertex argument for vertices in A_3(i-1)}
    For a component $K$ of $A_3(i-1)$, let $s$ be a $(K,B_i)$-minimal simplicial vertex. Then $N(s)\setminus B_{i-2}$ is a clique and so $|N(s)\cap B_{i-2}|>\ceil{\frac{\omega}{4}}$.
\end{lemma}

\begin{proof}
     By Lemma \ref{lem:A_1(i) is anticomplete to A_3}, $N(s)\setminus B_{i-2}\subseteq B_i\cup B_{i-1}\cup A_3(i-1)\cup A_2(i-1)$. Suppose that $v,v'\in N(s)\setminus B_{i-2}$ are not adjacent. Suppose first that $v,v'\notin A_2(i-1)$.
     By Lemmas \ref{lem:single vertex in A_3(i)} and \ref{lem:edge in $A_3(i)$}, $v\in K$ and $v'\in B_i$. 
    By Lemma \ref{lem:minimal simplicial vertex}, there is an induced $P_4=a-b-s-c$ with $a,b\in K$ and $c\in B_i$. Then $a-b-s-c-q_{i+1}-q_{i+2}-y$ is an induced $P_7$. 

    Now we may assume that $v\in A_2(i-1)$ and $v'\in  B_i\cup B_{i-1}\cup A_3(i-1)\cup A_2(i-1)$.
    Suppose first that $v$ has a neighbor $u\in A_1(i)$. By Lemma \ref{lem:A_1(i) is dominated by A_2(i+2)}, $u$ has a neighbor $w\in T$. By Lemma \ref{lem:A_1(i) is anticomplete to A_3}, $su\notin E(G)$. It follows that $v'-s-v-u-w-q_{i+2}-q_{i+1}$ is an induced $P_7$. So $v$ has no neighbors in $A_1(i)$. By Lemma \ref{lem:two non-adjacent neighbors}, $v$ has non-adjacent vertices $t_1,t_2\in A'_3(i-1)$. Since $t_1-v-t_2-(N(t)\cap B_{i-2})-q_{i+2}-q_{i+1}-B_i$ is not a bad $P_7$, $t_2$ is complete to $B_i$. By symmetry, $t_1$ is complete to $B_i$. We may assume that $s\neq t_1$. Note that both pairs $\{v',t_1\}$ and $\{s,t_1\}$ have a common neighbor in $B_i$ and so have disjoint neighborhoods in $B_{i-2}$.
    If $st_1\notin E(G)$, then $v'-s-v-t_1-(N(t_1)\cap B_{i-2})-q_{i+2}-q_{i+1}$ is an induced $P_7$. So $st_1\in E(G)$. Since $t_2t_1\notin E(G)$, we have $s\neq t_2$. By symmetry, $st_2\in E(G)$. Since $s$ is simplicial in $K$, we may assume that $t_1\in B_{i-1}$ and $t_2\in A_3(i-1)$. Now $t_1\in B_{i-1}$ mixes on an edge in $A_3(i-1)$, which contradicts Lemma \ref{lem:edge in $A_3(i)$}. This proves that $N(s)\setminus B_{i-2}$ is a clique and so $|N(s)\cap B_{i-2}|>\ceil{\frac{\omega}{4}}$. 
\end{proof}

\begin{lemma}\label{lem:A_3(i-1) is anticomplete to A_2(i-1)}
     $A_3(i-1)$ is anticomplete to $A_2(i-1)$. By symmetry,  $A_3(i+1)$ is anticomplete to $A_2(i)$.
\end{lemma}

\begin{proof}
    Suppose for a contradiction that $t\in A_3(i-1)$ is adjacent to $v\in A_2(i-1)$. We first show that $t$ is complete to $B_i$. Suppose first that $v$ has a neighbor $u\in A_1(i)$.
    By Lemma \ref{lem:A_1(i) is anticomplete to A_3}, $ut\notin E(G)$. Since $u-v-t-(N(t)\cap B_{i-2})-q_{i+2}-q_{i+1}-B_i$ is not a bad $P_7$, $t$ is complete to $B_i$. Now suppose that $v$ has no neighbors in $A_1(i)$. By Lemma \ref{lem:two non-adjacent neighbors}, $v$ has two non-adjacent neighbors $t_1,t_2\in A_3'(i-1)$. Since $t_1-v-t_2-(N(t_2)\cap B_{i-2})-q_{i+2}-q_{i+1}-B_i$ is not a bad $P_7$, we have $t_2$ is complete to $B_i$. By symmetry, $t_1$ is complete to $B_i$. Since $G$ is $C_4$-free, $v$ is complete to $B_i$.
    By Lemma \ref{lem:blowup-twoneighbor}, $t$ is complete to $B_i$.

    Let $K$ be the component of $A_3(i-1)$ containing $t$ and $s\in K$ be a $(K,B_i)$-minimal simplicial vertex. By Lemma \ref{lem:small vertex argument for vertices in A_3(i-1)}, $|N(s)\cap B_{i-2}|>\ceil{\frac{\omega}{4}}$.  By Lemma \ref{lem:A'_3(i-2) is large}, there exists $w\in B_{i-1}$ be anticomplete to $A_3(i-1)$ with $|N(w)\cap B_{i-2}|>\ceil{\frac{\omega}{4}}$. If $N(w)\cap B_{i-2}$ and $N(s)\cap B_{i-2}$ are disjoint, then $|B_{i-2}|>\frac{\omega}{2}$.
    By Lemma \ref{lem:simplicial vertices in A_1(i)}, $|T|>\ceil{\frac{\omega}{4}}$.
    By Lemma \ref{lem:A'_3(i-2) is large}, $B_{i-2}\cup B_{i+2}\cup T$ is a clique of size larger than $\omega$. 
    
    So $s$ and $w$ have a common neighbor in $B_{i-2}$ and thus they have disjoint neighborhoods in $B_i$. Recall that $t$ is complete to $B_i$. So $s\neq t$ and thus $sv\notin E(G)$. Let $P=v_0-v_1-\cdots-v_k-v$ be a shortest path such that $v_0=s$ and $v_i\in K$ for each $1\le i\le k$. We show that each $v_i$ with $1\le i\le k$ is complete to $B_i$. We have proved that $v_k$ is complete to $B_i$. Suppose now that $v_i$ is complete to $B_i$ for $2\le i\le k$. If $v_{i-1}$ has a non-neighbor $b\in B_i$, then $v_{i-2}-v_{i-1}-v_i-b-q_{i+1}-q_{i+2}-y$ is an induced $P_7$. So $v_{i-1}$ is complete to $B_i$. Therefore, $v_1$ is complete to $B_i$. By Lemma \ref{lem:neighborcontainA3i} (3), $N(s)\cap B_{i-2}\subseteq N(v_1)\cap B_{i-2}$. By Lemma \ref{lem:blowup-threeneighbor}, $w$ is anticomplete to $N(v_1)\cap B_{i-2}$ and so anticomplete to $N(s)\cap B_{i-2}$. This implies that $N(w)\cap B_{i-2}$ and $N(s)\cap B_{i-2}$ are disjoint, a contradiction.
\end{proof}

\begin{lemma}\label{lem:A_2(i-1)}
    $A_2(i-1)$ is a clique complete to $B_{i-1}\cup B_i$. By symmetry, $A_2(i)$ is a clique complete to $B_{i+1}\cup B_i$.
\end{lemma}

\begin{proof}
    By Lemmas \ref{lem:A_3(i-1) is anticomplete to A_2(i-1)} and Lemma \ref{lem:two non-adjacent neighbors}, each vertex in $A_2(i-1)$ has a neighbor in $A_1(i)$. Suppose that $v\in A_2(i-1)$ mixes on an edge $ab\in B_{i-1}$. Let $u\in A_1(i)$ be a neighbor of $v$ and $w\in A_2(i+2)$ be a neighbor of $u$. Then $a-b-v-u-w-q_{i+2}-q_{i+1}$ is an induced $P_7$. So $v$ is complete to $B_{i-1}$. Suppose that $v$ is not adjacent to $v'\in A_2(i-1)\cup B_i$. Then $v'-q_{i-1}-v-u-w-q_{i+2}-q_{i+1}$ is a bad $P_7$.
    This proves the lemma.
\end{proof}

\begin{lemma}\label{lem:A_1(i) is a clique if A_2(i) exists}
    If $A_2(i)\cup A_2(i-1)\neq \emptyset$, then $A_1(i)$ is a clique.
\end{lemma}

\begin{proof}
    By symmetry, assume that $A_2(i)$ contains a vertex $b$. By Lemmas \ref{lem:A_3(i-1) is anticomplete to A_2(i-1)} and \ref{lem:two non-adjacent neighbors}, $b$ has a neighbor $a\in A_1(i)$. By Lemma \ref{lem:A2i1tocompA1}, $b$ is complete to the component $K$ of $A_1(i)$ containing $a$. Suppose that $K$ is not a clique. 
    By Lemmas \ref{lem:A_1(i) is chordal}, \ref{lem:two non-adjacent simplicail verteices} and
    \ref{lem:simplicial vertices in A_1(i)}, there exist two non-adjacent simplicial vertices $s_1,s_2$ of $K$ with $|N(s_j)\cap T|>\ceil{\frac{\omega}{4}}$ for $j=1,2$. 
    Since $s_1$ and $s_2$ have a common $b\in A_2(i)$, $N(s_1)\cap T$ and $N(s_2)\cap T$ are disjoint and complete. 
    By Lemma \ref{lem:A'_3(i-2) is large}, $T\cup B_{i-2}\cup B_{i+2}$ is a clique of size larger than $\omega$. So $K$ is a clique. If $v\in K$ is not adjacent to $u\in B_i$, then $u-b-v-N_T(v)-q_{i-2}-q_{i-1}-u$ is an induced $C_6$ by Lemma \ref{lem:A_2(i-1)}. So $K$ is complete to $B_i$. If $A_1(i)\neq K$, then there exist two non-adjacent simplicial vertices $s_2\in A_1(i)\setminus K$ and $s_1\in K$ with $|N(s_j)\cap T|>\ceil{\frac{\omega}{4}}$ for $j=1,2$ by Lemma \ref{lem:simplicial vertices in A_1(i)}. Since $s_1$ is complete to $K$, $s_1$ and $s_2$ hava a common neighbor in $B_i$ and so $N(s_1)\cap T$ and $N(s_2)\cap T$ are disjoint and complete.  Then $T\cup B_{i-2}\cup B_{i+2}$ is a clique of size larger than $\omega$.
\end{proof}

\subsubsection{$T$ is not a clique}

In this subsection, $T$ is not a clique. Let $t_1,t_2\in T$ be two non-adjacent simplicial vertices of $T$. 

\begin{lemma}\label{lem:two large cliques}
    For $j=1,2$, $N(t_j)\setminus A_1(i)$ is a clique and so $K_{t_j}=N(t_j)\cap A_1(i)$  has size $>\ceil{\frac{\omega}{4}}$. Moreover, $K_{t_1}$ and $K_{t_2}$ are disjoint and complete.
\end{lemma}

\begin{proof}
    By Lemma \ref{lem: T is not a clique}, $N(t_j)\setminus A_1(i)\subseteq B_{i-2}\cup B_{i+2}\cup A_2(i+2)$. By Lemmas \ref{lem:T_0 is a clique}, \ref{lem:T is complete to T_0}, \ref{lem:T_0 is complete to B_{i-2} and B_{i+2}} and \ref{lem:opposite A_1(i) and A_2(j)}, $N(t_j)\setminus A_1(i)$ is a clique.
\end{proof}

\begin{lemma}
    $A_1(i)$ is a clique.
\end{lemma}

\begin{proof}
    Suppose not. By Lemma \ref{lem:A_1(i) is a clique if A_2(i) exists}, $A_2(i-1)=A_2(i)=\emptyset$. 
    Let $a'_1,a'_2$ be two non-adjacent simplicial vertices of $A_1(i)$. 
    Choose $a_j\in N_{A_1(i)}[a'_j]$ simplicial in $A_1(i)$ with minimal neighborhood in $B_i$ and with minimal neighborhood in $T$ for $j=1,2$.
    By Lemma \ref{lem:simplicial vertices in A_1(i)},  $N(a_j)\setminus T$ is a clique and so $|N(a_j)\cap T|>\ceil{\frac{\omega}{4}}$.
   
    If $a_1$ and $a_2$ have a common neighbor in $B_i$, then $N(a_1)\cap T$ and $N(a_2)\cap T$ are disjoint and complete sine $G$ is $(C_4,C_6)$-free, and so $(N(a_1)\cap T)\cup (N(a_2)\cap T)\cup B_{i-2}\cup B_{i+2}$ is a clique of size larger than $\omega$. 
    So $a_1$ and $a_2$ have disjoint neighborhoods in $B_i$. 
    Next we show that $N(a_j)\setminus B_i$ is a clique. Note that $N(a_j)\setminus B_i\subseteq A_1(i)\cup T$. Suppose that $s,s'$ are two non-adjacent vertices in $N(a_j)\setminus B_i$. Then $s\in A_1(i)$ and $s'\in T$. 
    Suppose first that $s$ is simplicial in $A_1(i)$. 
    By the minimality of $a_1$, $s$ has a neighbor $w\in T\cup B_i$ that is not adjacent to $a_1$. 
    If $w\in B_i$, then $w-s-a_1-s'-q_{i+2}-q_{i+1}-w$ is an induced $C_6$. 
    So $w\in T$ and $N(s)\cap B_i\subseteq N(a_1)\cap B_i$. 
    Then $s-a_1-s'-q_{i+2}-q_{i+1}-N_{B_i}(a_2)-q_{i-1}$ is an induced $P_7$. 
    So $s$ is no simplicial in $A_1(i)$. 
    Let $w$ be a neighbor of $s$ in $A_1(i)$ that is not adjacent to $a_1$. 
    Then $w-s-a_1-s'-q_{i+2}-q_{i+1}-N_{B_i}(a_2)$ is a bad $P_7$.
    So $N(a_1)\setminus B_i$ is a clique and so $|N(a_1)\cap B_i|>\ceil{\frac{\omega}{4}}$. 
    By symmetry, $|N(a_2)\cap B_i|>\ceil{\frac{\omega}{4}}$. 
    By Lemmas \ref{lem:large cliques in A_1(i)} and \ref{lem:two large cliques}, $(N(a_1)\cap B_i)\cup (N(a_2)\cap B_i)\cup K_{t_1}\cup K_{t_2}$ is a clique of size larger than $\omega$.
\end{proof}

\begin{lemma}\label{lem:A_3(i+1) is empty}
    $A_3(i+1)=A_3(i-1)=\emptyset$
\end{lemma}

\begin{proof}
    By symmetry, we prove that $A_3(i+1)=\emptyset$. Suppose not. By Lemma \ref{lem:small vertex argument for simplicial vertices in A_3(i-2)}, there exists a vertex $v\in A_3(i+1)$ such that $|N(v)\cap B_i|,|N(v)\cap B_{i+2}|>\ceil{\frac{\omega}{4}}$. 
    Let $w\in B_{i+1}\setminus N(A_3(i+1))$ with minimal neighborhood in $B_i\cup B_{i+2}$. 
    Suppose first that $w$ and $v$ have disjoint neighborhoods in $B_{i+2}$. 
    By Lemma \ref{lem:A'_3(i-2) is large}, $|N(w)\cap B_{i+2}|>\ceil{\frac{\omega}{4}}$ and so $(N(v)\cap B_{i+2})\cup (N(w)\cap B_{i+2})\cup B_{i-2}\cup T$ is a clique of size larger than $\omega$. 
    So $w$ and $v$ have disjoint neighborhoods in $B_{i}$ and so $B_i$ is not complete to $B_{i+1}$. 
    By Lemma \ref{lem:A_2(i-1)}, $A_2(i)=\emptyset$. 
    By Lemma \ref{lem:smv on w}, $|N(w)\cap B_{i}|>\ceil{\frac{\omega}{4}}$. 
    Then $(N(v)\cap B_{i})\cup (N(w)\cap B_{i})\cup K_{t_1}\cup K_{t_2}$ is a clique of size larger than $\omega$ by Lemma \ref{lem:B_i complete to two large cliques in B_{i-1}}.  
\end{proof}

Next, we give a lemma for reducing vertices of small degrees when we color $G$. This lemma will also be used in Section \ref{sec:A_2 is not emptyset}.

\begin{lemma}\label{lem:degeneracy}
    Let $X_1,X_2,X_3,X^*$ be four disjoint cliques of $G$ with $|X_1|>\ceil{\frac{\omega}{4}}$, $|X_2|\leq \frac{3}{4}\omega$ and $|X_3|=\ceil{\frac{\omega}{4}}$ such that 
    \begin{itemize}
        \item $X_1\cup X^*$ is anticomplete to $X_3$,
        \item every two vertices in $X_1\cup X^*$ have comparable neighborhoods in $X_2$, 
        \item for every $v\in X^*$ and every vertex $u\in X_1$, $N_{X_2}(v)\subseteq N_{X_2}(u)$,
        \item $P\subseteq X_1$ is the set of $p$ vertices with largest neighborhoods in $X_2$ with $p\in \left[ \ceil{\frac{\omega}{4}} \right]$, and 
        \item $F$ is an induced subgraph of $G$ with $F\cap P=\emptyset$.
    \end{itemize}
    If a vertex $v\in F\cap X_2$ satisfies $N_F(v)\subseteq X_1\cup X_2\cup X_3\cup X^*$ and $N(v)\cap (X_1\cup X_2\cup X^*)$ is a clique of $G$, then $d_F(v)<\ceil{\frac{5}{4}\omega}-p$. 
\end{lemma}

\begin{proof}
    If $N_{F_1}(v)\cap X_1\neq \emptyset$, then $(N_{F_1}(v)\cap (X_1\cup X_2\cup X^*))\cup P$ is also a clique of $G$ by the definition of $P$. 
    This implies that $|N_{F_1}(v)\cap (X_1\cup X_2\cup X^*)|\leq \omega-p$ or $|N_{F_1}(v)\cap (X_1\cup X_2\cup X^*)|=|N_{F_1}(v)\cap X_2|< \frac{3}{4}\omega$. 
    It follows that 
    \begin{align*} 
    |N_{F_1}(v)| & =|N_{F_1}(v)\cap X_3|+|N_{F_1}(v)\cap (X_1\cup X_2\cup X^*)|
    \\[3pt]
     & < \ceil{\frac{\omega}{4}}+\max \left \{\omega-p, \frac{3}{4}\omega \right \}
    \\[3pt]
     & \le \ceil{\frac{5}{4}\omega}-p.
\end{align*}
This completes the proof.
\end{proof}

\begin{lemma}
    $\chi(G)\le \ceil{\frac{5}{4}\omega}$.
\end{lemma}

\begin{proof}
By Lemmas \ref{lem:A_2(i-1) or A_2(i) is empty} and \ref{lem:A_2(i-2) is empty}, we may assume that $A_2=A_2(i)\cup A_2(i+2)$. By Lemmas \ref{lem:A_3(i) is empty when A_1=A_1(i)} \ref{lem:A_3(i+1) is empty} and \ref{lem: T is not a clique}, $A_3=\emptyset$.

By Lemma \ref{lem:structure of T}, $T$ can be partitioned into three cliques $T_1,T_2,T_3$ such that $T_2$ is complete to $T_1\cup T_3$ and $T_1$ and $T_3$ are anticomplete where $T_1,T_3\neq \emptyset$ and $T_2$ may be empty. Let $L=N(T_1)\cap A_1(i)$, $R=N(T_3)\cap A_1(i)$ and $M=A_1(i)\setminus (L\cup R)$. Note that $L$ and $R$ are disjoint. We show that $T_2$ is complete to $L\cup R$. Suppose that $v\in T_2$ has a non-neighbor $u\in L\cup R$, say $u\in L$. If $v$ is anticomplete to $L\cup R$, then $v$ has a neighbor $v'\in M$ and so $v-v'-u-N_{T_1}(u)-v$ is an induced $C_4$. So $v$ has a neighbor $w\in L\cup R$. If $w\in R$, then $w-u-N_{T_1}(u)-v-w$ is an induced $C_4$. So $v$ is anticomplete to $R$. By symmetry, $v$ is anticomplete to $L$, a contradiction.

Next we give a $\ceil{\frac{5\omega}{4} }$-coloring of $G$.
Let $x=|B_{i-2}|$ and $y=|B_{i+2}|$.
Our strategy is to use $p=x+y-2\ceil{\frac{\omega}{4}}$ colors to color some vertices of $G$ first and then argue that the remaining vertices can be colored with $f=\ceil{\frac{5}{4}\omega}-p$ colors. 
Note that
    \begin{align*} 
    f & = \ceil{\frac{5}{4}\omega}-(x+y-2\ceil{\frac{\omega}{4}}) 
    \\[3pt]
     & = \omega+(3\ceil{\frac{\omega}{4}}-(x+y))  
    \\[3pt]
     & \ge \omega,
\end{align*}
since $x+y\leq \frac{3}{4}\omega$.

\medskip
\noindent {\bf Step 1: Precolor a subgraph.}

\medskip
Next we define vertices to be colored with $p$ colors. Let $S_{i-2}\subseteq B_{i-2}$ be the set of $x-\ceil{\frac{\omega}{4}}$ vertices, $S_{i+2}\subseteq B_{i+2}$ be the set of $y-\ceil{\frac{\omega}{4}}$ vertices, and $S_i\subseteq B_i$ be the set of $p$ vertices with largest neighborhoods in $H$. Since $(x+y)-2\ceil{\frac{\omega}{4}}\leq \ceil{\frac{\omega}{4}}<|B_i|$, the choice of $S_i$ is possible.
We color $S_i$ with all colors in $[p]$, and color all vertices in $S_{i-2}\cup S_{i+2}$ with all colors in $[p]$. Since $B_i$ is anticomplete to $B_{i-2}\cup B_{i+2}$, this coloring is proper. So it suffices to show that
$G_1=G-(S_i\cup S_{i-2}\cup S_{i+2})$ is $f$-colorable.

\medskip
\noindent {\bf Step 2: Reduce by degeneracy.}

\medskip
Next, we use Lemma \ref{lem:degeneracy} to further reduce the problem to color a subgraph of $G_1$. 
Let $v\in B_{i-1}$ with minimal neighborhood in $B_{i-2}\cup B_i$. 
By Lemma \ref{lem:degeneracy} with $(X_1,X_2,X_3,X^*)=(B_i,B_{i-1},B_{i-2}\setminus S_{i-2},\emptyset)$, $P=S_i$, and $F=G_1$, we have $d_{G_1}(v)<f$. 
It follows that $G_1$ is $f$-colorable if and only if $G_1-v$ is $f$-colorable. Repeatedly applying Lemma \ref{lem:degeneracy} we conclude that $G_1$ is $f$-colorable if only if $G_1-B_{i-1}$ is $f$-colorable.

Next, we reduce vertices in $B_{i+1}$. By Lemmas \ref{lem:A_2(i-1)} and \ref{lem:blowup-twoneighbor}, $N_{B_{i+1}}(u')= N_{B_{i+1}}(u)$ for every $u\in B_{i}$ and $u'\in A_2(i)$.
Let $v\in B_{i+1}$ with minimal neighborhood in $B_{i+2}\cup B_i$. 
By Lemma \ref{lem:degeneracy} with $(X_1,X_2,X_3,X^*)=(B_i,B_{i+1},B_{i+2}\setminus S_{i+2},A_2(i))$, $P=S_i$, and $F=G_1-B_{i-1}$, we have $d_{G_1-B_{i-1}}(v)<f$. It follows that $G_1-B_{i-1}$ is $f$-colorable if and only if $(G_1-B_{i-1})-v$ is $f$-colorable. Repeatedly applying Lemma \ref{lem:degeneracy} we conclude that $G_1-B_{i-1}$ is $f$-colorable if only if $(G_1-B_{i-1})-B_{i+1}$ is $f$-colorable.

\medskip
\noindent {\bf Step 3: Reduce simplicial vertices.}

\medskip
Now it suffices to show that $G_2=G_1-B_{i-1}-B_{i+1}$ is $f$-colorable. 
Let $v\in B_{i}\cup A_2(i)$ be a vertex with minimal neighborhood in $A_1(i)$. 
Note that $v$ is simplicial in $G_2$. 
So $G_2$ is $f$-colorable if and only if $G_2-v$ is $f$-colorable. 
By applying this step repeatedly, we have $G_2$ is $f$-colorable if and only if $G_2-(B_{i}\cup A_2(i))$ is $f$-colorable. 

Let $v\in M$ be a vertex with minimal neighborhood in $T_2$. 
Note that $v$ is simplicial in $G_2-(B_{i}\cup A_2(i))$. 
So $G_2-(B_{i}\cup A_2(i))$ is $f$-colorable if and only if $(G_2-(B_{i}\cup A_2(i)))-v$ is $f$-colorable. 
By applying this step repeatedly, we have $G_2-(B_{i}\cup A_2(i))$ is $f$-colorable if and only if $G_2-(B_{i}\cup A_2(i))-M$ is $f$-colorable. 
Let $G_3=G_2-(B_{i}\cup A_2(i))-M$. 

\medskip
\noindent {\bf Step 4: Color $G_3$.}

\medskip
So it suffices to show that $G_3$ is $f$-colorable. 
Note that every vertex in $T_2$ is universal in $G_3$. 
So it suffices to show that $G_4=G_3-T_2$ is $(f-|T_2|)$-colorable.  
Let $H'$ be a subgraph of $G_4$ which is hyperhole with maximum chromatic number. 
It follows that $\omega(H')\leq \omega-|T_2|\leq f-|T_2|$.  
Since $L\cup R$ is complete to $B_i$, $|L\cup R|<\omega-\ceil{\frac{\omega}{4}}$. 
Let $M_j=V(G_4)\cap (T_j\cup B_{i-2}\cup B_{i+2}\cup T_0)$ for $j=1,3$. 
Since $T_1\cup B_{i-2}\cup B_{i+2}\cup T_0$ is complete to $T_2$, $|M_1|=|V(G_4)\cap (T_1\cup B_{i-2}\cup B_{i+2}\cup T_0)|\leq \omega-|T_2|-p$. 
By symmetry, $|M_3|=|V(G_4)\cap (T_3\cup B_{i-2}\cup B_{i+2}\cup T_0)|\leq \omega-|T_2|-p$. 
Then 
    \begin{align*}
	    |V(H')|\leq |V(G_4)| &= |L\cup R|+|M_1|+|M_3|-|B_{i-2}\cup B_{i+2}\cup T_0|
            \\[3pt]& < \left(\omega-\ceil{\frac{\omega}{4}}\right)+2(\omega-|T_2|-p)-2\ceil{\frac{\omega}{4}}
            \\[3pt]& =3\omega+\ceil{\frac{\omega}{4}}-2(x+y)-2|T_2| 
            \\[3pt]& <2\omega+6\ceil{\frac{\omega}{4}}-2(x+y)-2|T_2| 
            \\[3pt]&= 2(f-|T_2|), 
    \end{align*}
By \cite{MIV20} (Theorem 1.2), $\chi(G_4)= \chi(H')=\max\left \{\omega(H'),\ceil{\frac{|V(H')|}{2}}\right \}\leq f-|T_2|$.
\end{proof}

\subsubsection{$T$ is a clique}

\begin{lemma}\label{lem:small vertex argument for vertices in B_{i-2}}
     If $A_3(i-2)\neq \emptyset$, any vertex $w\in B_{i-2}$ anticomplete to $A_3(i-2)$ has $N(w)\setminus B_{i-1}$ is a clique and so $|N(w)\cap B_{i-1}|>\ceil{\frac{\omega}{4}}$.
\end{lemma}

\begin{proof}
    By Lemma \ref{lem:A_3(i-2) or A_3(i+2) is empty}, we have $A_3(i+2)=\emptyset$. 
    By Lemma \ref{lem:compare nbd of different A_3(i) in B_j}, $w$ is anticomplete to $A_3(i-1)$.
    By Lemma \ref{lem:A_2(i-2) is empty}, $w$ has no neighbor in $A_2\setminus A_2(i+2)$. 
    By Lemmas \ref{lem:T is complete to T_0}, \ref{lem:T_0 is complete to B_{i-2} and B_{i+2}}, \ref{lem:T_0 is a clique} and the assumption that $T$ is a clique, $N(w)\setminus B_{i-1}=B_{i-2}\cup B_{i+2}\cup T\cup T_0$ is a clique.  
\end{proof}

\begin{lemma}\label{lem:A_3(i-2) is a clique}
    $A_3(i-2)$ is a clique. By symmetry, $A_3(i+2)$ is a clique.
\end{lemma}

\begin{proof}
    Suppose not. Let $s,s'\in A_3(i-2)$ be two non-adjacent vertices simplicial in $A_3(i-2)$. By Lemma \ref{lem:B_{i-1} is large}, $|N(s)\cap B_{i-1}|,|N(s')\cap B_{i-1}|>\ceil{\frac{\omega}{4}}$. Let $w\in B_{i-2}$ be anticomplete to $A_3(i-2)$. By Lemma \ref{lem:small vertex argument for vertices in B_{i-2}}, $|N(w)\cap B_{i-1}|>\ceil{\frac{\omega}{4}}$. It follows from 
    Lemmas \ref{lem:B_i complete to two large cliques in B_{i-1}} and \ref{lem:B_i is large} that $B_i\cup (N(s)\cap B_{i-1})\cup (N(s')\cap B_{i-1})\cup (N(w)\cap B_{i-1})$ is a clique of size larger than $\omega$. 
\end{proof}

\begin{lemma}\label{lem:A_3{i-1} and A_3{i+1} is empty when A_3{i-2} exists}
    If $A_3(i-2)\cup A_3(i+2)\neq \emptyset$, then $A_3(i-1)\cup A_3(i+1)=\emptyset$. 
\end{lemma}

\begin{proof}
    By symmetry, we assume that $A_3(i-2)\neq \emptyset$. Let $t\in A_3(i-2)$ be a vertex simplicial in $A_3(i-2)$ and $b_{i-2}\in B_{i-2}$ be a vertex anticomplete to $A_3(i-2)$. By Lemma \ref{lem:compare nbd of different A_3(i) in B_j}, $b_{i-2}$ is anticomplete to $A_3(i-1)$. By Lemma \ref{lem:B_{i-1} is large}, $|N(t)\cap B_{i-1}|>\ceil{\frac{\omega}{4}}$. By Lemma \ref{lem:small vertex argument for vertices in B_{i-2}}, $|N(b_{i-2})\cap B_{i-1}|>\ceil{\frac{\omega}{4}}$.

    Let $s\in A_3(i-1)$ and $b_{i-1}\in B_{i-1}$ be anticomplete to $A_3(i-1)$. If $A_2(i-1)$ contains a vertex $a$, then $s-(N(s)\cap B_{i-1})-a-(N(a)\cap A_1(i))-T-q_{i+2}-q_{i+1}$ is an induced sequence. So $A_2(i-1)=\emptyset$. 
    By Lemmas \ref{lem:small vertex argument for simplicial vertices in A_3(i-2)} and \ref{lem:smv on w}, there are $v\in A_3(i-1)$ and $w\in B_{i-1}\setminus N(A_3(i-1))$ such that $N(v)\cap B_{i},N(v)\cap B_{i-2},N(w)\cap B_{i},N(w)\cap B_{i-2}$ are cliques of size larger than $\ceil{\frac{\omega}{4}}$. 
    If $v$ and $w$ have disjoint neighborhoods in $B_i$, then $(N(v)\cap B_{i})\cup (N(w)\cap B_{i})\cup (N(t)\cap B_{i-1})\cup (N(b_{i-2})\cap B_{i-1})$ is a clique of size larger than $\omega$ by Lemma \ref{lem:B_i complete to two large cliques in B_{i-1}}. 
    So $v$ and $w$ have disjoint neighborhoods in $B_{i-2}$. 
    Then $(N(v)\cap B_{i-2})\cup (N(w)\cap B_{i-2})\cup (N(t)\cap B_{i-1})\cup (N(b_{i-2})\cap B_{i-1})$ is a clique of size larger than $\omega$ by Lemma \ref{lem:cross four cliques}. 
    \end{proof}

\begin{lemma}\label{lem:A_1(i) is a clique when A_3(i-2) is not empty}
    If $A_3(i\pm2)\neq \emptyset$, then $A_1(i)$ is a clique.
\end{lemma}

\begin{proof}
    By symmetry, assume that $A_3(i-2)\neq \emptyset$.
    Suppose that $a\in A_3(i-2)$ is simplicial in $A_3(i-2)$.  Let $b\in B_{i-2}$ anticomplete to $A_3(i-2)$. 
    If $c\in B_i$ mixes on an edge $st\in A_1(i)$, then $a-q_{i+2}-b-q_{i-1}-c-s-t$ is an induced $P_7$. So each vertex in $B_i$ is complete or anticomplete to any component of $A_1(i)$.
    If a component $K$ of $A_1(i)$ is not a clique, then $|T|>\frac{\omega}{2}$ by Lemma \ref{lem:simplicial vertices in A_1(i)} and so $B_{i-2}\cup B_{i+2}\cup T$ is a clique of size larger than $\omega$. So any component of $A_1(i)$ is a clique. 

    Suppose that $A_1(i)$ has two components $L_1$ and $L_2$. Note that $K_a=N(a)\cap B_{i-1}$ and $K_b=N(b)\cap B_{i-1}$ are disjoint and have size $>\ceil{\frac{\omega}{4}}$ by Lemmas \ref{lem:small vertex argument for vertices in B_{i-2}} and \ref{lem:B_{i-1} is large}. If $u\in L_j$ has a neighbor $w\in A_2(i)$, then $u-w-(N(w)\cap B_{i+1})-q_{i+2}-b-K_b-K_a$ is an induced $P_7$. So $L_j$ is anticomplete to $A_2(i)$ for $j=1,2$. It follows that $N(L_j)\subseteq B_i\cup A_2(i-1)\cup T$.
    Let $s_j\in L_j$ such that $|N(s_j)\cap T|$ is minimum. Since $L_j$ and $T$ are cliques, $N(s_j)\cap (L_j\cup T)$ is a clique. By Proposition \ref{prop:component of A_1(i)}, $K_{s_j}=N(s_j)\cap (B_i\cup A_2(i-1))=N(s_j)\setminus (L_j\cup T)$ is a clique and so $|K_{s_j}|>\ceil{\frac{\omega}{4}}$. If $K_{s_1}$ and $K_{s_2}$ are not disjoint, then $|T|>\frac{\omega}{2}$ and so $T\cup B_{i-2}\cup B_{i+2}$ is a clique of size larger than $\omega$. So $K_{s_1}$ and $K_{s_2}$ are disjoint. Since $G$ is $(C_6,C_7)$-free, $K_{s_1}$ and $K_{s_2}$ are complete.
    By Lemma \ref{lem:B_i complete to two large cliques in B_{i-1}} and \ref{lem:A_2(i-1)}, $K_{s_1}\cup K_{s_2}\cup K_a\cup K_b$ is a clique of size larger than $\omega$. This proves that $A_1(i)$ is a clique.
\end{proof}

\begin{lemma}\label{lem:A_3(i-1) or A_3(i+1) is empty}
    $A_3(i-1)$ or $A_3(i+1)$ is empty.
\end{lemma}

\begin{proof}
    Suppose not.
    By Lemmas \ref{lem:small vertex argument for vertices in A_3(i-1)} and \ref{lem:A'_3(i-2) is large}, there exist $s\in A_3(i-1)$ and $a\in B_{i-1}\setminus N(s)$ such that $|N(a)\cap B_{i-2}|>\ceil{\frac{\omega}{4}}$ and $|N(s)\cap B_{i-2}|>\ceil{\frac{\omega}{4}}$, and there exist $t\in A_3(i+1)$ and $b\in B_{i+1}\setminus N(t)$ such that $|N(b)\cap B_{i+2}|>\ceil{\frac{\omega}{4}}$ and $|N(t)\cap B_{i+2}|>\ceil{\frac{\omega}{4}}$.
    By Lemma \ref{lem:no two directed edge pointing to the same B_i}, we may assume that $N(a)\cap B_{i-2}$ and $N(s)\cap B_{i-2}$ are disjoint, and so $\omega(B_{i-2}\cup B_{i+2}\cup T)>\omega$, a contradiction. 
\end{proof}

\begin{lemma}\label{lem:A_1(i)}
    $A_1(i)$ is a blowup of $P_3$.
\end{lemma}

\begin{proof}
    By Lemma \ref{lem:A_1(i) is a clique if A_2(i) exists}, we may assume that $A_2(i-1)=A_2(i)=\emptyset$ and $A_1(i)$ is not a clique.
    
    \medskip
    \noindent (1) Any two non-adjacent $(A_1(i),B_i)$-minimal simplicial vertices $a,b$ have a common neighbor in $T$. 
    
    \medskip    
    Suppose not.  By Lemma \ref{lem:simplicial vertices in A_1(i)}, $|N(a)\cap T|,|N(b)\cap T|>\ceil{\frac{\omega}{4}}$. It follows that $\omega(T\cup B_{i-2}\cup B_{i+2})>\omega$, a contradiction. This proves (1).

    \medskip
    \noindent (2) If there are two non-adjacent simplicial vertices $a,b$ of $A_1(i)$ which have disjoint neighborhoods in $B_i$, then $|N(a)\cap B_i|,|N(b)\cap B_i|>\ceil{\frac{\omega}{4}}$. Moreover, $(N(a)\cap B_i)\cup(N(b)\cap B_i)$ is complete to $B_{i-1}\cup B_{i+1}$.

    \medskip
    It suffices to show that $N(a)\setminus B_i$ and $N(b)\setminus B_i$ are cliques. Note that $N(a)\setminus B_i\subseteq A_1(i)\cup T$.
    Since $a$ is simplicial in $A_1(i)$ and $T$ is a clique, it suffices to show that $N(a)\cap A_1(i)$ is complete to $N(a)\cap T$. Suppose that $s\in N(a)\cap A_1(i)$ is not adjacent to $s'\in N(a)\cap T$. If $s$ has a neighbor $p\in B_i$ not adjacent to $a$, then $p-s-a-s'-q_{i+2}-q_{i+1}-p$ is an induced $C_6$. So $N(s)\cap B_i\subseteq N(a)\cap B_i$. Then $s-a-s'-q_{i+2}-q_{i+1}-N_{B_i}(b)-q_{i-1}$ is an induced $P_7$. This proves that $N(a)\setminus B_i$ is a clique. By symmetry, $N(b)\setminus B_i$ is a clique. 
    
    We first show that each vertex $v\in B_{i-1}$ is either complete or anticomplete to $(N(a)\cap B_i)\cup(N(b)\cap B_i)$. Suppose that $v$ has a neighbor $a'\in N(a)\cap B_i$. If $v$ has a non-neighbor $b'\in N(b)\cap B_i$, then $v-a'-b'-b-N_T(b)-q_{i-2}-v$ is an induced $C_6$. So $v$ is complete to $N(b)\cap B_i$. It follows that $v$ is complete to $N(a)\cap B_i$. Now suppose that $v$ is anticomplete to $(N(a)\cap B_i)\cup(N(b)\cap B_i)$. Let $c\in B_{i}$ be a neighbor of $v$ and $d\in N(a)\cap B_i$. Then $v-c-d-a-N_T(a)-q_{i-2}-v$ is an induced $C_6$.
    This proves (2).

    \medskip
    \noindent (3)  $|B_{i-1}|>\ceil{\frac{\omega}{4}}$ or $|B_{i+1}|>\ceil{\frac{\omega}{4}}$. 
    
    \medskip     
    By Lemma \ref{lem:small vertex argument for vertices in B_{i-2}}, we may assume that $A_3(i-2)=A_3(i+2)=\emptyset$. 
    By Lemma \ref{lem:A_3(i-1) or A_3(i+1) is empty}, we may assume that $A_3(i-1)=\emptyset$. For any vertex $v\in B_{i-2}$, $N(v)\setminus B_{i-1}=A_2(i+2)\cup B_{i+2}\cup B_{i-2}$ is a clique since $T$ is a clique. 
    It follows that $|B_{i-1}|>\ceil{\frac{\omega}{4}}$.
    This proves (3).
    
    \medskip
    By (1), (2), (3) and Lemma \ref{lem:two non-adjacent simplicail verteices}, $A_1(i)$ cannot have three pairwise non-adjacent simplicial vertices. This implies that $A_1(i)$ has at most two components and if $A_1(i)$ has two components, then each component is a clique. So assume that $A_1(i)$ is connected. Since $A_1(i)$ is $P_4$-free, $A_1(i)$ can be partitioned into two non-empty subsets $X_1$ and $X_2$ such that $X_1$ and $X_2$ are complete. Since $G$ is $C_4$-free, $X_1$ is a clique. Choose such a partition $(X_1,X_2)$ such that $|X_1|$ is maximum.  It follows that $X_2$ is disconnected.
    Since each vertex simplicial in $X_2$ is simplicial in $A_1(i)$, $X_2$ has exactly two clique components. So $A_1(i)$ is a blowup of $P_3$.
\end{proof}

\begin{lemma}\label{lem:A_3(i-1) is a clique when only one A_1}
    $A_3(i-1)$ is a clique and for every vertex $a\in A_3(i-1)$, $N_{B_i}(a)$ is anticomplete to $B_{i-1}\setminus N(a)$. 
\end{lemma}

\begin{proof}
    Suppose first that $A_3(i-1)$ is not a clique. 
    By Lemma \ref{lem:A_3{i-1} and A_3{i+1} is empty when A_3{i-2} exists}, $A_3(i-2)=A_3(i+2)=\emptyset$. 
    If $A_2(i-1)$ contains a vertex $p$, let $s\in A_3(i-1)$ and then $s-(N(s)\cap B_{i-1})-p-(N(p)\cap A_1(i))-T-q_{i+2}-q_{i+1}$ is an induced sequence. So $A_2(i-1)=\emptyset$. 
    By Lemmas \ref{lem:two non-adjacent simplicail verteices}, \ref{lem:small vertex argument for simplicial vertices in A_3(i-2)} and \ref{lem:smv on w}, there exist three pairwise non-adjacent vertices $v_1,v_2\in A_3(i-1)$ and $w\in B_{i-1}\setminus N(A_3(i-1))$ such that $N(v_1)\cap B_{i},N(v_1)\cap B_{i-2},N(v_2)\cap B_{i},N(v_2)\cap B_{i-2},N(w)\cap B_{i},N(w)\cap B_{i-2}$ are cliques of size larger than $\ceil{\frac{\omega}{4}}$. 
    If any two of $v_1,v_2,w$ have disjoint neighborhoods in $B_{i-2}$, then $B_{i-2}\cup B_{i+2}\cup T$ is a clique of size larger than $\omega$. 
    So $v_1,v_2,w$ have pairwise comparable neighborhoods in $B_{i-2}$ and so have disjoint neighborhoods in $B_{i}$. 
    Let $v$ be a vertex in $B_{i+2}$. 
    Since $N(v)\subseteq B_{i+1}\cup (B_{i-2}\cup B_{i+2}\cup A_2(i+2))$ and $B_{i-2}\cup B_{i+2}\cup A_2(i+2)$ is a clique, $|N(v)\cap B_{i+1}|>\ceil{\frac{\omega}{4}}$. 
    By Lemmas \ref{lem:B_i complete to two large cliques in B_{i-1}}, $(N(v_1)\cap B_{i})\cup (N(v_2)\cap B_{i})\cup (N(w)\cap B_{i})\cup B_{i+1}$ is a clique of size larger than $\omega$. 
    So $A_3(i-1)$ is a clique. 

    Let $u$ be a vertex in $A_3(i-1)$ with minimal neighborhood in $H$. 
    Since $A_3(i-1)$ is a clique, $u$ is simplicial in $A_3(i-1)$. 
    By Lemma \ref{lem:small vertex argument for simplicial vertices in A_3(i-2)}, $|N(u)\cap B_{i-2}|>\ceil{\frac{\omega}{4}}$. 
    By Lemma \ref{lem:smv on w}, there exists $w\in B_{i-1}\setminus N(A_3(i-1))$ such that $|N(w)\cap B_{i-2}|>\ceil{\frac{\omega}{4}}$. 
    If $u$ and $w$ have disjoint neighborhoods in $B_{i-2}$, then $B_{i-2}\cup B_{i+2}\cup A_2(i+2)$ is a clique of size larger than $\omega$. 
    So $w$ and $v$ have a common neighbor in $B_{i-2}$. 
    By the minimality of $u$, every vertex $a$ in $A_3(i-1)$ and $w$ have a common neighbor in $B_{i-2}$, and so $N_{B_i}(a)$ is anticomplete to $B_{i-1}\setminus N(a)$. 
\end{proof}

\begin{lemma}\label{lem:A_1(i) is a clique when A_3(i-1) is not empty}
    If $A_3(i\pm 1)\neq \emptyset$, then $A_1(i)$ is a clique.
\end{lemma}

\begin{proof}
    Suppose for a contradiction that $A_1(i)$ is not a clique.
    By symmetry, $A_3(i-1)\ne \emptyset$.
    Let $v\in A_3(i-1)$ and $w\in B_{i-1}\setminus N(A_3(i-1))$.
    By Lemma \ref{lem:A_3(i-1) is a clique when only one A_1}, $v$ and $w$ have disjoint neighborhoods in $B_i$.
    By Lemma \ref{lem:simplicial vertices in A_1(i)}, $A_1(i)$ has two non-adjacent vertices $s_1,s_2$ such that $|N(s_j)\cap T|>\ceil{\frac{\omega}{4}}$ for $j=1,2$. If $s_1$ and $s_2$ have disjoint neighborhoods in $T$, then $\omega(T\cup B_{i-2}\cup B_{i+2})>\omega$. So $s_1$ and $s_2$ have disjoint neighborhoods in $B_i$. 
    If there is an induced $P_4=a-b-c-d$ such that $a\in A'_3(i-1)$, $b,c\in B_i$ and $d\in A_1(i)$, then $a-b-c-d-N_T(d)-N_{B_{i-2}}(a)-a$ is an induced $C_6$. 
    It follows that $N(v),N(w)\subseteq N(s_j)$ for $j=1,2$. This contradicts that $s_1$ and $s_2$ have disjoint neighborhoods in $B_i$.
\end{proof}

\begin{lemma}
    If $A_3\neq \emptyset$, $\chi(G)\le \ceil{\frac{5}{4}\omega}$.
\end{lemma}

\begin{proof}
    By Lemma \ref{lem:A_2(i-2) is empty}, we may assume that $A_2=A_2(i-1)\cup A_2(i)\cup A_2(i+2)$, where $A_2(i-1)$ and $A_2(i)$ may be empty. 
     By Lemmas \ref{lem:A_1(i) is a clique when A_3(i-2) is not empty} and \ref{lem:A_1(i) is a clique when A_3(i-1) is not empty}, $A_1(i)$ is a clique.
    If $A_3(i-2)\neq \emptyset$, then by Lemmas \ref{lem:A_3{i-1} and A_3{i+1} is empty when A_3{i-2} exists}, \ref{lem:A_3(i) is empty when A_1=A_1(i)} and \ref{lem:A_3(i-2) or A_3(i+2) is empty}, $A_3=A_3(i-2)$ and by Lemma \ref{lem:A_3(i-2) is a clique}, $A_3(i-2)$ is a clique. If $A_3(i-2)=A_3(i+2)=\emptyset$, then by Lemma \ref{lem:A_3(i-1) or A_3(i+1) is empty}, we may assume that $A_3(i+1)=\emptyset$. 
    Moreover, by Lemma \ref{lem:A_3(i-1) is a clique when only one A_1}, $A_3(i-1)$ is a clique.

Next we give a $\ceil{\frac{5}{4}\omega}$-coloring of $G$.
Let \[x=\begin{cases}
            \,|A_3(i-2)\cup N_{B_{i-2}}(A_3(i-2))|, & \text{if }\, A_3(i-2)\neq\emptyset, \\
            \,|B_{i-2}|, & \text{if }\, A_3(i-2)=\emptyset,
        \end{cases}
        \]
 $y=|B_{i+2}|$, and $z=|T|$. 
Our strategy is to use $p=x+y+z-3\ceil{\frac{\omega}{4}}$ colors to color some vertices of $G$ first and then argue that the remaining vertices can be colored with $f=\ceil{\frac{5}{4}\omega}-p$ colors. 
Note that
    \begin{align*} 
    f & = \ceil{\frac{5}{4}\omega}-(x+y+z-3\ceil{\frac{\omega}{4}}) 
    \\[3pt]
     & = \omega+(4\ceil{\frac{\omega}{4}}-(x+y+z))  
    \\[3pt]
     & \ge \omega,
\end{align*}
since $x+y+z\leq \omega$.

\medskip
\noindent {\bf Step 1: Precolor a subgraph.}

\medskip
Next we define vertices to be colored with $p$ colors. Let 
\[S_{i-2}\subseteq
        \begin{cases}
            \,A_3(i-2)\cup N_{B_{i-2}}(A_3(i-2)), & \text{if }\, A_3(i-2)\neq\emptyset, \\
            \,B_{i-2}, & \text{if }\, A_3(i-2)=\emptyset,
        \end{cases}
        \]
be the set of $x-\ceil{\frac{\omega}{4}}$ vertices, $S_{i+2}\subseteq B_{i+2}$ be the set of $y-\ceil{\frac{\omega}{4}}$ vertices, $S_{T}\subseteq T$ be the set of $z-\ceil{\frac{\omega}{4}}$ vertices, and 
$S_i\subseteq B_i$ be the set of $p$ vertices with largest neighborhoods in $H\cup A_1(i)$ if $A_3(i-2)\neq \emptyset$ or with largest neighborhoods in $A_3(i-1)\cup B_{i+1}\cup A_1(i)$ if $A_3(i-2)= \emptyset$. Since $(x+y+z)-3\ceil{\frac{\omega}{4}}\leq \ceil{\frac{\omega}{4}}<|B_i|$, the choice of $S_i$ is possible.
We color $S_i$ with all colors in $[p]$, and color all vertices in $S_{i-2}\cup S_{i+2}\cup S_T$ with all colors in $[p]$. Since $B_i$ is anticomplete to $A'_3(i-2)\cup B_{i+2}\cup T$, this coloring is proper. So it suffices to show that
$G_1=G-(S_i\cup S_{i-2}\cup S_{i+2}\cup S_T)$ is $f$-colorable.

\medskip
\noindent {\bf Step 2: Reduce by degeneracy.}

\medskip
Next, we use Lemma \ref{lem:degeneracy} to further reduce the problem to color a subgraph of $G_1$. 
Suppose first that $A_3(i-2)\neq \emptyset$. 
Let $R_{i-1}$ be the set of vertices in $B_{i-1}$ that is anticomplete to $B_{i-2}\setminus N(A_3(i-2))$. Let $v\in R_{i-1}$ with minimal neighborhood in $B_{i-2}\cup B_i$ among $R_{i-1}$. Since  $v$ is anticomplete to $B_{i-2}\setminus N(A_3(i-2))$, $v$ is also minimal among $B_{i-1}$.
By Lemma \ref{lem:degeneracy} with $(X_1,X_2,X_3,X^*)=\left (B_i,B_{i-1},\left (A_3(i-2)\cup N_{B_{i-2}}\left (A_3(i-2)\right)\right )\setminus S_{i-2},A_2(i-1)\right )$, $P=S_i$, and $F=G_1$, we have $d_{G_1}(v)<f$. It follows that $G_1$ is $f$-colorable if and only if $G_1-v$ is $f$-colorable. Repeatedly applying Lemma \ref{lem:degeneracy} we conclude that $G_1$ is $f$-colorable if only if $G_1-R_{i-1}$ is $f$-colorable. 
Then we suppose that $A_3(i-2)= \emptyset$. 
Let $R_{i-1}=A_3(i-1)$. 
Let $v\in R_{i-1}$ with minimal neighborhood in $B_i$ among $R_{i-1}$.  
By Lemma \ref{lem:degeneracy} with $(X_1,X_2,X_3,X^*)=(B_i,A_3(i-1)\cup N_{B_{i-1}}(A_3(i-1)),B_{i-2}\setminus S_{i-2},A_2(i-1))$, $P=S_i$, and $F=G_1$, we have $d_{G_1}(v)<f$. It follows that $G_1$ is $f$-colorable if and only if $G_1-v$ is $f$-colorable. Repeatedly applying Lemma \ref{lem:degeneracy} we conclude that $G_1$ is $f$-colorable if only if $G_1-R_{i-1}$ is $f$-colorable.

Let $v\in B_{i+1}$ be a vertex with minimal neighborhood in $B_{i+2}\cup B_i$ among $B_{i+1}$. 
By Lemma \ref{lem:degeneracy}  with $(X_1,X_2,X_3,X^*)=(B_i,B_{i+1},B_{i+2}\setminus S_{i+2},A_2(i))$, $P=S_i$, and $F=G_1-R_{i-1}$, we have $d_{G_1-R_{i-1}}(v)<f$. It follows that $G_1-R_{i-1}$ is $f$-colorable if and only if $(G_1-R_{i-1})-v$ is $f$-colorable. Repeatedly applying Lemma \ref{lem:degeneracy} we conclude that $G_1-R_{i-1}$ is $f$-colorable if only if $(G_1-R_{i-1})-B_{i+1}$ is $f$-colorable.

Let $v\in A_1(i)$ be a vertex with minimal neighborhood in $B_i\cup T$ among $A_1(i)$. 
By Lemma \ref{lem:A_2(i-1) or A_2(i) is empty}, we may assume that $A_2(i)=\emptyset$. 
Let $u\in B_i$ and $u'\in A_2(i-1)$. 
If $u'$ has a neighbor $u''\in A_1(i)$ with $uu''\notin E(G)$, then $u-(B_{i-1}\cup \{u'\})-u''-T-q_{i+2}-q_{i+1}-u$ is an induced cyclic sequence. 
So for every $u\in B_{i}$ and $u'\in A_2(i-1)$, $N_{A_1(i)}(u')\subseteq N_{A_1(i)}(u)$.
By Lemma \ref{lem:degeneracy} with $(X_1,X_2,X_3,X^*)=(B_i,A_1(i),T\setminus S_{T},A_2(i-1))$, $P=S_i$, and $F=G_1-R_{i-1}-B_{i+1}$, we have $d_{G_1-R_{i-1}-B_{i+1}}(v)<f$. It follows that $G_1-R_{i-1}-B_{i+1}$ is $f$-colorable if and only if $(G_1-R_{i-1}-B_{i+1})-v$ is $f$-colorable. Repeatedly applying Lemma \ref{lem:degeneracy} we conclude that $G_1-R_{i-1}-B_{i+1}$ is $f$-colorable if only if $G_1-R_{i-1}-B_{i+1}-A_1(i)$ is $f$-colorable.

\medskip
Note that $G_2=G_1-R_{i-1}-B_{i+1}-A_1(i)$ is chordal and so $\chi(G_2)\leq \omega\leq f$. 
\end{proof}

\begin{lemma}
    If $A_3= \emptyset$, $\chi(G)\le \ceil{\frac{5}{4}\omega}$.
\end{lemma}

\begin{proof}
    By Lemma \ref{lem:A_2(i-2) is empty}, we may assume that $A_2=A_2(i-1)\cup A_2(i)\cup A_2(i+2)$, where $A_2(i-1)$ and $A_2(i)$ may be empty.
    By Lemma \ref{lem:A_1(i)}, $A_1(i)$ can be partitioned into three cliques $L\cup M\cup R$ such that $M$ is complete to $L\cup R$ and $R$ and $L$ are anticomplete where $L\neq \emptyset$. We may assume that if $R= \emptyset$, then $M=\emptyset$.
    For any $v\in L$ with minimal neighborhood in $B_i$, $N(v)\cap (A_1(i)\cup B_i)$ is also a clique by Lemma \ref{lem:induced P3in A_1(i)}.

Next we give a $\ceil{\frac{5}{4}\omega}$-coloring of $G$.
Let $x=|B_{i-2}|$,
 $y=|B_{i+2}|$, and $z=|T|$. 
Our strategy is to use $p=x+y+z-3\ceil{\frac{\omega}{4}}$ colors to color some vertices of $G$ first and then argue that the remaining vertices can be colored with $f=\ceil{\frac{5}{4}\omega}-p$ colors. 
Note that
    \begin{align*} 
    f & = \ceil{\frac{5}{4}\omega}-(x+y+z-3\ceil{\frac{\omega}{4}}) 
    \\[3pt]
     & = \omega+(4\ceil{\frac{\omega}{4}}-(x+y+z))  
    \\[3pt]
     & \ge \omega,
\end{align*}
since $x+y+z\leq \omega$.

\medskip
\noindent {\bf Step 1: Precolor a subgraph.}

\medskip
Next we define vertices to be colored with $p$ colors. Let 
$S_{i-2}\subseteq B_{i-2}$ be the set of $x-\ceil{\frac{\omega}{4}}$ vertices, $S_{i+2}\subseteq B_{i+2}$ be the set of $y-\ceil{\frac{\omega}{4}}$ vertices, $S_{T}\subseteq T$ be the set of $z-\ceil{\frac{\omega}{4}}$ vertices, and 
$S_i\subseteq B_i$ be the set of $p$ vertices with largest neighborhoods in $H\cup L\cup M$. Since $(x+y+z)-3\ceil{\frac{\omega}{4}}\leq \ceil{\frac{\omega}{4}}<|B_i|$, the choice of $S_i$ is possible.
We color $S_i$ with all colors in $[p]$, and color all vertices in $S_{i-2}\cup S_{i+2}\cup S_T$ with all colors in $[p]$. Since $B_i$ is anticomplete to $A'_3(i-2)\cup B_{i+2}\cup T$, this coloring is proper. So it suffices to show that
$G_1=G-(S_i\cup S_{i-2}\cup S_{i+2}\cup S_T)$ is $f$-colorable.

\medskip
\noindent {\bf Step 2: Reduce by degeneracy.}

\medskip
Next, we use Lemma \ref{lem:degeneracy} to further reduce the problem to color a subgraph of $G_1$. 
Let $R_{i-1}=B_{i-1}$. 
Let $v\in R_{i-1}$ with minimal neighborhood in $B_i$ among $R_{i-1}$.  
By Lemma \ref{lem:degeneracy} with $(X_1,X_2,X_3,X^*)=(B_i,B_{i-1},B_{i-2}\setminus S_{i-2},A_2(i-1))$, $P=S_i$, and $F=G_1$, we have $d_{G_1}(v)<f$. It follows that $G_1$ is $f$-colorable if and only if $G_1-v$ is $f$-colorable. Repeatedly applying Lemma \ref{lem:degeneracy} we conclude that $G_1$ is $f$-colorable if only if $G_1-R_{i-1}$ is $f$-colorable.

Let $v\in B_{i+1}$ be a vertex with minimal neighborhood in $B_{i+2}\cup B_i$ among $B_{i+1}$. 
By Lemma \ref{lem:degeneracy}  with $(X_1,X_2,X_3,X^*)=(B_i,B_{i+1},B_{i+2}\setminus S_{i+2},A_2(i))$, $P=S_i$, and $F=G_1-R_{i-1}$, we have $d_{G_1-R_{i-1}}(v)<f$. It follows that $G_1-R_{i-1}$ is $f$-colorable if and only if $(G_1-R_{i-1})-v$ is $f$-colorable. Repeatedly applying Lemma \ref{lem:degeneracy} we conclude that $G_1-R_{i-1}$ is $f$-colorable if only if $(G_1-R_{i-1})-B_{i+1}$ is $f$-colorable.

Next, we reduce $R_z=L$.
Let $v\in L$ be a vertex with minimal neighborhood in $B_i\cup T$ among $A_1(i)$. So $N(v)\cap (A_1(i)\cup B_i)$ is a clique.
By Lemma \ref{lem:A_2(i-1) or A_2(i) is empty}, we may assume that $A_2(i)=\emptyset$. 
Let $u\in B_i$ and $u'\in A_2(i-1)$. 
If $u'$ has a neighbor $u''\in A_1(i)$ with $uu''\notin E(G)$, then $u-(B_{i-1}\cup \{u'\})-u''-T-q_{i+2}-q_{i+1}-u$ is an induced cyclic sequence. 
So for every $u\in B_{i}$ and $u'\in A_2(i-1)$, $N_{A_1(i)}(u')\subseteq N_{A_1(i)}(u)$.
By Lemma \ref{lem:degeneracy} with $(X_1,X_2,X_3,X^*)=(B_i,L\cup M,T\setminus S_{T},A_2(i-1))$, $P=S_i$, and $F=G_1-R_{i-1}-B_{i+1}$, we have $d_{G_1-R_{i-1}-B_{i+1}}(v)<f$. It follows that $G_1-R_{i-1}-B_{i+1}$ is $f$-colorable if and only if $(G_1-R_{i-1}-B_{i+1})-v$ is $f$-colorable. Repeatedly applying Lemma \ref{lem:degeneracy} we conclude that $G_1-R_{i-1}-B_{i+1}$ is $f$-colorable if only if $G_1-R_{i-1}-B_{i+1}-R_z$ is $f$-colorable.

\medskip
Note that $G_2=G_1-R_{i-1}-B_{i+1}-R_z$ is chordal and so $\chi(G_2)\leq \omega\leq f$. 
\end{proof}

\section{$A_1$ is empty and $A_2$ is not empty}\label{sec:A_2 is not emptyset}

Now we can assume that $A_1=\emptyset$.  Let $K$ be a non-empty component of $A_2(i+2)$. 
By Lemma \ref{lem:complete subgraphs}, $K$ has two non-adjacent neighbors $a,b$ complete to $K$ where $a,b\in A'_3(i-2)$ or $a,b\in A'_3(i+2)$.  We call $\{a,b\}$ a {\em special pair} for $K$. Note that special pairs for $K$ are not unique. By Lemma \ref{lem:blowup-fiveneighbor1}, $A_5$ is complete to $K$. Since $K$ is arbitrary, each vertex in $A_5$ is universal in $G$ and so $A_5=\emptyset$. 

\begin{lemma}[Basic properties of $A_2(i)$]\label{lem:basic properties of A_2(i)}
    Let $K$ be a non-empty component of $A_2(i+2)$ and $\{a,b\}\subseteq A'_3(i-2)$ be a special pair for $K$. Then $a,b$ are complete to $B_{i+2}$ and have disjoint neighborhoods in $B_{i-1}$ and so $K\cup B_{i+2}$ is a clique. Moreover, $B_i$ and $B_{i+1}$ are complete.
\end{lemma}

\begin{proof}
    Let $v\in K$. Since $G$ is $C_4$-free, $a$ and $b$ have disjoint neighborhoods in $B_{i-1}$. Since $a-v-b-N_{B_{i-1}}(b)-q_i-q_{i+1}-B_{i+2}$ is not a bad $P_7$, $b$ is complete to $B_{i+2}$. By symmetry, $a$ is complete to $B_{i+2}$. Since $K\cup B_{i+2}$ is complete to $\{a,b\}$, $K\cup B_{i+2}$ is a clique. If $c\in B_i$ mixes on an edge $st\in B_{i+1}$, then $a-v-b-N_{B_{i-1}}(b)-c-s-t$ is an induced $P_7$. This proves that $B_i$ is complete to $B_{i+1}$. 
\end{proof}

\begin{lemma}\label{lem:A_2(i+2) has neighbors in one A_3(j)}
    $A_2(i+2)$ is anticomplete to either $A_3(i-2)$ or $A_3(i+2)$.
\end{lemma}

\begin{proof}
    Suppose not. By Lemma \ref{lem:neighbors of A_2(i) in A_3(i+1) and A_3(i)}, let $K$ and $K'$ be two components of $A_2(i+2)$ such that $K$ has a neighbor in $A_3(i-2)$ and $K'$ has a neighbor in $A_3(i+2)$. 

    By Lemma \ref{lem:complete subgraphs}, $K$ has two non-adjacent neighbors $a,b$ complete to $K$ where $a,b\in A'_3(i-2)$ and $K'$ has two non-adjacent neighbors $a’,b'$ complete to $K$ where $a,b\in A'_3(i+2)$. We may assume that $b\in A_3(i-2)$ and $b'\in A_3(i+2)$. 
    By Lemma 
    \ref{lem:neighbors of A_2(i) in A_3(i+1) and A_3(i)}, $K$ is anticomplete to $b'$ and $K'$ is anticomplete to $b$. It follows that $K-b-N_{B_{i-1}}(b)-q_i-N_{B_{i+1}}(b')-b'-K'$ is an induced $P_7$.
\end{proof}

\begin{lemma}\label{lem:A_2(i) is a clique}
    $A_2(i+2)$ is a clique.
\end{lemma}

\begin{proof}
    Suppose not. Let $K$ and $K'$ be two components of $A_2(i+2)$. By Lemmas \ref{lem:complete subgraphs} and \ref{lem:A_2(i+2) has neighbors in one A_3(j)}, $K$ has two non-adjacent neighbors in $a,b\in A'_3(i-2)$ and $K'$ has two non-adjacent neighbors $a',b'\in A'_3(i-2)$. Let $v\in K$. If $b$ has a neighbor $v'\in K'$, then $v'-b-v-a-N_{B_{i-1}}(a)-q_i-q_{i+1}$ is an induced $P_7$. So $b$ is anticomplete to $K'$. This implies that $a,b,a',b'$ are pairwise distinct. 
    If $aa'\in E(G)$, then $a'-a-v-b-N_{B_{i-1}}(b)-q_i-q_{i+1}$ is an induced $P_7$. So $\{a,b,a',b'\}$ is stable. 
    Let $u\in N_{B_{i-1}}(a)$ and $u'\in N_{B_{i-1}}(a')$. 
    Then $u$ is anticomplete to $\{a',b,b'\}$ and  $u'$ is anticomplete to $\{a,b,b'\}$. Then $b-v-a-u-u'-a'-K'$ is an induced $P_7$. 
\end{proof}

\begin{lemma}\label{lem:smv on special vertices}
    Let $a,b$ be two non-adjacent simplicial vertices of the set that consists of all vertices in $A'_3(i-2)$ complete to $A_2(i-2)$. 
    Then both of $N(a)\setminus B_{i+2}$ and $N(b)\setminus B_{i+2}$ are cliques, and so $N_{B_{i+2}}(a)$ and $N_{B_{i+2}}(b)$ are two disjoint complete cliques of size larger than $\ceil{\frac{\omega}{4}}$.
\end{lemma}

\begin{proof}
    Suppose for a contradiction that $s,s'$ are two non-adjacent vertices in $N(a)\setminus B_{i+2}$. 
    By Lemma \ref{lem:compare nbd of different A_3(i) in B_j},
    $N(a)\setminus B_{i+2}\subseteq B_{i-1}\cup A_2(i-2)\cup B_{i-2}\cup A_3(i-2)$. 
    Suppose first that $s\in B_{i-1}\cup A_2(i-2)$. 
    Since $B_{i-1}\cup A_2(i-2)$ is a clique by Lemma \ref{lem:basic properties of A_2(i)}, we may assume $s'\in B_{i-2}\cup A_3(i-2)$. 
    Since $s'$ is not complete to $B_{i-1}\cup A_2(i-2)$, we have $s'b\notin E(G)$.
    If $s'$ is complete to $A_2(i-2)$, then $\{s',b\}$ is also a special pair for $A_2(i-2)$ and then $s'$ is complete to $A_2(i-2)\cup B_{i-1}$, a contradiction. 
    So $s'$ has a non-neighbor $s''\in A_2(i-2)$. 
    Then $s'-a-s''-b-N_{B_{i+2}}(b)-q_{i+1}-q_i$ is an induced $P_7$. 
    So we may assume that $s,s'\in B_{i-2}\cup A_3(i-2)$. 
    By the assumption on $a$, we have at least one of $s,s'$ is not complete to $A_2(i-2)$. 
    By symmetry, we may assume that $s$ has a non-neighbor $s''\in A_2(i-2)$. 
    Then $sb\notin E(G)$, and so $s-a-s''-b-N_{B_{i+2}}(b)-q_{i+1}-q_i$ is an induced $P_7$. 
\end{proof}

\begin{lemma}\label{lem:at most two non-empty A_2(i)}
    There exists an index $i$ such that $A_2=A_2(i-2)\cup A_2(i+1)$. Moreover, if $A_2(i-2)$ and $A_2(i+1)$ are not empty, then $A_2(i-2)$ has two non-adjacent neighbors $a,b\in A'_3(i-2)$ complete to $A_2(i-2)$ and $A_2(i+1)$ has two non-adjacent neighbors $a,b\in A'_3(i+2)$ complete to $A_2(i+1)$.
\end{lemma}

\begin{proof}
    Suppose first that $A_2(i+2)\neq \emptyset$. By Lemme \ref{lem:complete subgraphs}, we may assume that a special pair $\{a,b\}$ for $A_2(i+2)$ are in $A'_3(i-2)$. If $v\in A_2(i+1)$, then $a-A_2(i+2)-b-N_{B_{i-1}}(b)-q_i-N_{B_{i+1}}(v)-v$ is an induced $P_7$. So $A_2(i+1)=\emptyset$. 

    Suppose that $A_2(i)\neq \emptyset$. Let $\{a',b'\}$ be a special pair for $A_2(i)$. 
    Suppose first that $a',b'\in A'_3(i)$. 
    If $b,b'$ have a common neighbor $d\in B_{i-1}$, then $a-A_2(i+2)-b-d-b'-A_2(i)-a'$ is an induced $P_7$. 
    So $b,a',b'$ have pairwise disjoint neighborhoods in $B_{i-1}$. 
    Then $a-A_2(i+2)-b-N_{B_{i-1}}(b) - N_{B_{i-1}}(b')-b'-A_2(i)$ is an induced $P_7$. 
    So $a',b'\in A_3'(i+1)$. By Lemma \ref{lem:basic properties of A_2(i)}, $B_{i-2}$ is complete to $B_{i-1}$. It follows that $a$ is complete to $B_{i-1}\cup B_{i+2}$ and so $a,b$ have a common neighbor in both $B_{i-1}$ and $B_{i+2}$, a contradiction. So $A_2(i)=\emptyset$.  

    Now suppose that $A_2(i-2)\neq \emptyset$. By symmetry and $A_2(i+2)\neq \emptyset$, a special pair $\{a',b'\}$ for $A_2(i-2)$ is contained in $A'_3(i-2)$. By Lemma \ref{lem:basic properties of A_2(i)}, $b'$ is complete to $B_{i-1}$ and $b$ is complete to $B_{i+2}$. If $b=b'$, then $b$ is complete to $B_{i-1}\cup B_{i+2}$ and so $a,b$ have a common neighbor in both $B_{i-1}$ and $B_{i+2}$. So $b\neq b'$. If $bb'\in E(G)$, then one of $b$ and $b'$ is complete to  $B_{i-1}\cup B_{i+2}$ by Lemma \ref{lem:neighborcontainA3i}. So $bb'\notin E(G)$. Since $b$ and $b'$ have a common neighbor in both $B_{i-1}$ and $B_{i+2}$, it contradicts Lemma \ref{lem:neighborcontainA3i}. So $A_2(i-2)=\emptyset$. 

    This implies that if $A_2(i-1)\neq \emptyset$, then a special pair for $A_2(i-1)$ is contained in $A'_3(i-1)$.  It follows that $i+1$ is such an index.
\end{proof}

\begin{lemma}
    There are at most one non-empty $A_2(j)$.
\end{lemma}

\begin{proof}
    Suppose not. By Lemma \ref{lem:at most two non-empty A_2(i)}, we may assume that $A_2=A_2(i-2)\cup A_2(i+1)$ where $A_2(i-2)$ and $A_2(i+1)$ are not empty.

    Let $a_1,b_1$ be two non-adjacent simplicial vertices of the set that consists of all vertices in $A'_3(i-2)$ that complete to $A_2(i-2)$ and $a_2,b_2$ be two non-adjacent simplicial vertices of the set that consists of all vertices in $A'_3(i+2)$ that complete to $A_2(i+1)$. 
    By symmetry, we may assume that $b_1\in A_3(i-2)$ and $b_2\in A_3(i+2)$. 
    By Lemmas \ref{lem:cross four cliques} and \ref{lem:smv on special vertices}, $N_{B_{i+2}}(a_1)\cup N_{B_{i+2}}(b_1)\cup N_{B_{i-2}}(a_2)\cup N_{B_{i-2}}(b_2)$ is a clique of size larger than $\omega$, a contradiction.
\end{proof}

\subsection{One non-empty $A_2(i)$}

Suppose that $A_2=A_2(i+2)\neq \emptyset$. 
Let $\{a,b\}\subseteq A'_3(i-2)$ be a special pair for $A_2(i+2)$. By symmetry, we may assume that $b\in A_3(i-2)$.

\begin{lemma}\label{lem:A_3(i) is empty}
    $A_3(i)=A_3(i+2)=\emptyset$. 
\end{lemma}

\begin{proof}
      Suppose that $A_3(i)\neq \emptyset$. 
      By Lemma \ref{lem:basic properties of A_2(i)}, $a,b$ have disjoint neighborhoods in $B_{i-1}$ and $B_i$ is complete to $B_{i+1}$. But this contradicts Lemma \ref{lem:no two directed edge pointing to the same B_i}.

     Suppose that $A_3(i+2)\neq \emptyset$. Then $B_{i+2}\setminus N(A_3(i+2))\neq \emptyset$ by Lemma \ref{lem:comparable vertices in A_3(i)}. By Lemma \ref{lem:compare nbd of different A_3(i) in B_j}, $B_{i+2}\setminus N(A_3(i+2))$ is anticomplete to $b$, which contradicts Lemma \ref{lem:basic properties of A_2(i)}.
\end{proof}

\begin{lemma}\label{lem:B_i is large when A_1 is empty}
    $|B_i|>\ceil{\frac{\omega}{4}}$. 
\end{lemma}
\begin{proof}
    Let $w\in B_{i+1}\setminus N(A_3(i+1))$ with minimal neighborhood in $B_{i+2}$. By Lemma \ref{lem:A_3(i) is empty}, $N(w)\setminus B_{i}\subseteq B_{i+1}\cup B_{i+2}$. 
    Suppose that $N(w)\setminus B_{i}$ has two non-adjacent vertices $s,s'$
    with $s\in B_{i+1}$ and $s'\in B_{i+2}$. 
    If $s\in B_{i+1}\setminus N(A_3(i+1))$, then by the minimality of $w$, we have $s$ has a neighbor $s''\in B_{i+2}$ with $s''w\notin E(G)$. 
    Then $s-w-s'-s''-s$ is an induced $C_4$. 
    So we may assume that $s\in N_{B_{i+1}}(A_3(i+1))$. 
    Then $N_{A_3(i+1)}(s)-s-w-s'-a-N_{B_{i-1}}(a)-N_{B_{i-1}}(b)$ is an induced $P_7$.
\end{proof}

\begin{lemma}\label{clm:neighbor complete}
    For every special pair $\{a_0,b_0\}$ of $A_2(i+2)$, each vertex in $(N(a_0)\cup N(b_0))\cap A'_3(i-2)$ is complete to $A_2(i+2)$. 
\end{lemma}

\begin{proof}
    Suppose not. 
    By symmetry, we may assume $a_0$ has a neighbor $a'\in A'_3(i-2)$ that has a non-neighbor $a''\in A_2(i+2)$. 
    By definition of $a_0,b_0$ and $C_4$-freeness of $G$, $a'b_0\notin E(G)$. 
    Then $a'-a_0-a''-b_0-N_{B_{i-1}}(b_0)-q_i-q_{i+1}$ is an induced $P_7$.
\end{proof}

Let $w$ be a vertex in $B_{i-2}\setminus N(A_3(i-2))$ with minimal neighborhood in $H$ among all vertices in $B_{i-2}\setminus N(A_3(i-2))$.

\begin{lemma}\label{clm:w big in left}
    $N(w)\setminus B_{i-1}$ is a clique and so $|N_{B_{i-1}}(w)|>\ceil{\frac{\omega}{4}}$. 
\end{lemma}

\begin{proof}
    Suppose not. 
    Let $s,s'$ be two non-adjacent vertices in $N(w)\setminus B_{i-1}$. 
    By Lemmas \ref{lem:A_3(i) is empty} and \ref{lem:compare nbd of different A_3(i) in B_j}, $N(w)\setminus B_{i-1}\subseteq B_{i-2}\cup B_{i+2}\cup A_2(i+2)$. 
    Since $B_{i+2}\cup A_2(i+2)$ is a clique, we may assume that $s\in B_{i-2}$ and $s'\in B_{i+2}\cup A_2(i+2)$. 
    If $s'\in A_2(i+2)$, then $s-w-s'-b-N_{B_{i-1}}(b)-q_i-q_{i+1}$ is an induced $P_7$. 
    So we may assume that $s'\in B_{i+2}$. 
    Suppose first that $s\in B_{i-2}\setminus N(A_3(i-2))$.
    By the minimality of $w$, $s$ has a neighbor $s''$ in $B_{i+2}\cup B_{i-1}$ that is not adjacent to $w$. 
    If $s''\in B_{i+2}$, then $s''-s-w-s'-s''$ is an induced $C_4$. 
    If $s''\in B_{i-1}$, then $s''-s-w-s'$ is a $P_4$-structure.
    So we may assume that $s$ has a neighbor $s''\in A_3(i-2)$. 
    Since $s''-s-w-s'-q_{i+1}-q_i-B_{i-1}$ is not a bad $P_7$, we have $w$ is complete to $B_{i-1}$. 
    So $w$ and $s''$ have a common neighbor in $B_{i-1}$, and hence $w,s''$ have disjoint neighborhoods in $B_{i+2}$. 
    By Lemma \ref{lem:single vertex in A_3(i)} (2), $w-s'-N_{B_{i+2}}(s'')-s-w$ is an induced $C_4$.
\end{proof}

By Lemma \ref{lem:chordal has simplicial vertices}, we can assume that $a,b$ are simplicial vertices of the set that consists of all vertices in $A'_3(i-2)$ complete to $A_2(i+2)$.

\begin{lemma}\label{clm:a in B}
    $a\in B_{i-2}$. 
\end{lemma}
\begin{proof}
    Suppose not. 
    By Lemmas \ref{clm:w big in left} and \ref{lem:smv on special vertices},  $|N_{B_{i-1}}(w)|,|N_{B_{i-1}}(a)|,|N_{B_{i-1}}(b)|>\ceil{\frac{\omega}{4}}$.
    Since $a,b,w$ are pairwise non-adjacent and they have a common neighbor in $B_{i+2}$, $N_{B_{i-1}}(w)$, $N_{B_{i-1}}(a)$, $N_{B_{i-1}}(b)$ are pairwise disjoint. By Lemma \ref{lem:B_i complete to two large cliques in B_{i-1}}, $N_{B_{i-1}}(w)$, $N_{B_{i-1}}(a)$, $N_{B_{i-1}}(b)$ are complete to $B_{i}$. By Lemma \ref{lem:B_i is large when A_1 is empty},
    $N_{B_{i-1}}(w)\cup N_{B_{i-1}}(a)\cup N_{B_{i-1}}(b)\cup B_i$ is a clique of size larger than $\omega$, a contradiction. 
\end{proof}

\begin{lemma}\label{clm:A3i-2 big in left}
    Let $v$ be a $(A_3(i-2), B_{i+2}\cup A_2(i+2))$-minimal simplicial vertex.
    Then $N(v)\setminus B_{i-1}$ is a clique and so $|N_{B_{i-1}}(v)|>\ceil{\frac{\omega}{4}}$.
\end{lemma}

\begin{proof}
    Suppose not. 
    Let $s,s'$ be two non-adjacent neighbors of $v$ in $N(v)\setminus B_{i-1}$. 
    Note that $N(v)\setminus B_{i-1}\subseteq B_{i-2}\cup B_{i+2}\cup A_2(i+2)\cup A_3(i-2)$.  
    Since $v$ is simplicial in $A_3(i-2)$ and by Lemma \ref{lem:edge in $A_3(i)$}, we have at most one of $s,s'\in A'_3(i-2)$. 
    Since $B_{i+2}\cup A_2(i+2)$ is a clique, we may assume that $s\in A'_3(i-2)$ and $s'\in B_{i+2}\cup A_2(i+2)$. 
    Assume first that $s\in B_{i-2}$. 
    By Lemma \ref{lem:single vertex in A_3(i)} (2), $s'\in A_2(i+2)$.
    By Lemma \ref{clm:a in B}, $as\in E(G)$ and so $sb\notin E(G)$ by $G$ is $C_4$-free. 
    Then $s-a-s'-b-N_{B_{i-1}}(b)-q_i-q_{i+1}$ is an induced $P_7$. 
    So we may assume that $s\in A_3(i-2)$. 
    If $s$ is simplicial in $A_3(i-2)$, then $s$ has a neighbor $s''\in A_2(i+2)\cup B_{i+2}$ that is not adjacent to $v$ by the minimality of $v$, and so $s-s''-s'-v-s$ is an induced $C_4$. 
    So $s$ is not simplicial in $A_3(i-2)$ and thus $s$ has a neighbor $s''\in A_3(i-2)$ that is not adjacent to $v$. 
    It follows that $v-s-s''$ is an induced $P_3$ in $A_3(i-2)$. 
    By Lemma \ref{lem:P3 in A_3(i)}, $s'\in A_2(i+2)$.
    Since $s$ is not complete to $A_2(i+2)$, $v\neq a,b$ by Lemma \ref{clm:neighbor complete}. If $v$ is complete to $A_2(i+2)$, then $v$ is adjacent to $a$ and $b$ by Lemma \ref{clm:neighbor complete}. Then $a\in B_{i-2}$ mixes on $vb\in A_3(i-2)$, which contradicts Lemma \ref{lem:edge in $A_3(i)$}. So $v$ is not complete to $A_2(i+2)$ and so $av\notin E(G)$ or $bv\notin E(G)$ by Lemma \ref{lem:basic properties of A_2(i)}.
    By symmetry, we assume that $av\notin E(G)$. Then $s-v-s'-a-N_{B_{i-1}}(a)-q_i-q_{i+1}$ is an induced $P_7$. 
\end{proof}

\begin{lemma}\label{lem:A_3i-2 A_2i+2 Bi+2 is clique}
    $A_3(i-2)\cup A_2(i+2)\cup B_{i+2}$ is a clique. 
\end{lemma}
\begin{proof}
    We first show that $A_3(i-2)$ is a clique. 
    Suppose not. 
    By Lemma \ref{lem:two non-adjacent simplicail verteices}, there are two non-adjacent $(A_3(i-2), A_2(i+2)\cup B_{i+2})$-minimal simplicial vertices $v_1,v_2$.
    If $b$ is adjacent to both $v_1$ and $v_2$, then $v_1$ and $v_2$ complete to $A_2(i+2)$ by Lemma \ref{clm:neighbor complete}, which contradicts the choice of $b$.
    So we may assume that $v_1$ is not adjacent to $b$.
    For convenience, let $v=v_1$.
    Suppose first that $av\notin E(G)$. Since $a,b,v$ have a common neighbor in $B_{i+2}$, $a,b,v$ have pairwise disjoint neighborhoods in $B_{i-1}$. By Lemma \ref{lem:B_i complete to two large cliques in B_{i-1}}, $N_{B_{i-1}}(v)\cup N_{B_{i-1}}(a)\cup N_{B_{i-1}}(b)\cup B_i$ is a clique.
    By Lemmas \ref{clm:A3i-2 big in left}, \ref{lem:smv on special vertices} and \ref{lem:B_i is large when A_1 is empty}, $|N_{B_{i-1}}(v)\cup N_{B_{i-1}}(a)\cup N_{B_{i-1}}(b)\cup B_i|>\omega$, a contradiction. 
    So $av\in E(G)$. By Lemma \ref{clm:neighbor complete}, $v$ is complete to $A_2(i+2)$ and so $\{v,b\}$ is a special pair for $A_2(i+2)$. By Lemma \ref{lem:edge in $A_3(i)$}, $v$ and $b$ are simplicial in $A'_3(i-2)$. The fact that $v,b\in A_3(i-2)$ contradicts Lemma \ref{clm:a in B}. This proves that $A_3(i-2)$ is a clique.

    By Lemma \ref{clm:neighbor complete}, $A_3(i-2)$ is complete to $A_2(i+2)$. Since $a\in B_{i-2}$, $a$ is anticomplete to $A_3(i-2)$ by Lemma \ref{lem:edge in $A_3(i)$}. For any vertex $v\in A_3(i-2)$, $\{a,v\}$ is a special pair for $A_2(i+2)$ and so $v$ is complete to $B_{i+2}$ by Lemma \ref{lem:basic properties of A_2(i)}. This proves the lemma. 
\end{proof}

\begin{lemma}\label{clm:v in A3i-2 complete}
    $B_{i-2}$ is complete to $B_{i+2}\cup A_2(i+2)$. 
\end{lemma}

\begin{proof}
    Since $a\in B_{i-2}$, $B_{i-2}$ is complete to $A_2(i+2)$ by Lemma \ref{clm:neighbor complete}. By Lemma \ref{lem:blowup-twoneighbor}, $B_{i-2}$ is complete to $B_{i+2}$.
\end{proof}

\begin{lemma}
    $A_3(i-1)=\emptyset$. 
\end{lemma}

\begin{proof}
    Suppose not. 
    Let $u$ be a $(A_3(i-1), H)$-minimal simplicial vertex.
    By Lemma \ref{lem:small vertex argument for simplicial vertices in A_3(i-2)}, $|N_{B_i}(u)|, |N_{B_{i-2}}(u)|>\ceil{\frac{\omega}{4}}$.

    Let $w'\in B_{i-1}\setminus N(A_3(i-1))$ be a vertex with minimal neighborhood in $H$. 
    By Lemma \ref{lem:smv on w}, $|N_{B_i}(w')|>\ceil{\frac{\omega}{4}}$ and $|N_{B_{i-2}}(w')|>\ceil{\frac{\omega}{4}}$.
    If $N_{B_i}(w')$ and $N_{B_i}(u)$ are disjoint, then $|B_i|>2\cdot \ceil{\frac{\omega}{4}}$ and so $N_{B_{i-1}}(w)\cup N_{B_{i-1}}(b)\cup B_i$ is a clique of size larger than $\omega$ by Lemmas \ref{clm:w big in left} and \ref{lem:smv on special vertices} (recall that $w$ is the vertex defined before Lemma \ref{clm:w big in left}).

    So we may assume that $N_{B_i}(w')$ and $N_{B_i}(u)$ have a common vertex and so $N_{B_{i-2}}(w')$ and $N_{B_{i-2}}(u)$ are disjoint. 
    By Lemma \ref{lem:cross four cliques}, $N_{B_{i-1}}(b)\cup N_{B_{i-1}}(w)\cup N_{B_{i-2}}(u)\cup N_{B_{i-2}}(w')$ is a clique of size larger than $\omega$, a contradiction.
\end{proof}

\begin{lemma}\label{lem:A_3(i+1) is a clique}
    $A_3(i+1)$ is a clique. 
\end{lemma}

\begin{proof}
    Suppose not. 
    By Lemma \ref{lem:two non-adjacent simplicail verteices}, there are two non-adjacent $(A_3(i+1), H)$-minimal simplicial vertices $v_1,v_2$.
    By Lemma \ref{lem:small vertex argument for simplicial vertices in A_3(i-2)}, $|N_{B_i}(v_1)|, |N_{B_{i+2}}(v_1)|, |N_{B_i}(v_2)|, |N_{B_{i+2}}(v_2)|>\ceil{\frac{\omega}{4}}$.
    If $N_{B_i}(v_1)$ and $N_{B_i}(v_2)$ are disjoint, then $N_{B_i}(v_1)\cup N_{B_i}(v_2)\cup N_{B_{i-1}}(a)\cup N_{B_{i-1}}(b)$ is a clique of size larger than $\omega$ by Lemmas \ref{lem:smv on special vertices} and \ref{lem:B_i complete to two large cliques in B_{i-1}}. 
    So we may assume that $N_{B_i}(v_1)$ and $N_{B_i}(v_2)$ have a common vertex and so $N_{B_{i+2}}(v_1)$ and $N_{B_{i+2}}(v_2)$ are disjoint. 
    Let $w'$ be a vertex in $B_{i+1}\setminus N(A_3(i+1))$ that has minimal neighborhood in $H$. 
    By Lemma \ref{lem:smv on w}, $|N_{B_{i+2}}(w')|>\ceil{\frac{\omega}{4}}$. Since $B_{i}$ is complete to $B_{i+1}$, each of $w',v_1$ and $w',v_2$ have a common neighbor in $B_i$. So $N_{B_{i+2}}(w')$, $N_{B_{i+2}}(v_1)$ and $N_{B_{i+2}}(v_2)$ are pairwise disjoint cliques that are complete to $A_3(i-2)\cup N_{B_{i-2}}(A_3(i-2))$ by Lemma \ref{lem:B_i complete to two large cliques in B_{i-1}}.
    
    Let $v$ be a vertex in $B_{i-1}$ with minimal neighborhood in $B_i\cup B_{i-2}$. 
    Since $A_3(i-2)\cup B_{i+2}$ is a clique by Lemma \ref{lem:A_3i-2 A_2i+2 Bi+2 is clique}, $N_{B_{i-1}}(A_3(i-2))$ is anticomplete to $B_{i-2}\setminus N(A_3(i-2))$. 
    It follows that $N(v)\setminus (A_3(i-2)\cup N_{B_{i-2}}(A_3(i-2)))\subseteq B_{i}\cup B_{i-1}$. 
    By the minimality of $v$, $N(v)\setminus (A_3(i-2)\cup N_{B_{i-2}}(A_3(i-2)))$ is a clique. 
    So $|A_3(i-2)\cup N_{B_{i-2}}(A_3(i-2))|>\ceil{\frac{\omega}{4}}$. 
    Hence, $(A_3(i-2)\cup N_{B_{i-2}}(A_3(i-2)))\cup N_{B_{i+2}}(w')\cup N_{B_{i+2}}(v_1)\cup N_{B_{i+2}}(v_2)$ is a clique of size larger than $\omega$, a contradiction. 
\end{proof}

Next, we collect some information about the size of certain sets.
\begin{lemma}
    The following statements hold
    \begin{itemize}
        \item[$(1)$] $|A_3(i+1)\cup N_{B_{i+1}}(A_3(i+1))|,|A_3(i-2)\cup N_{B_{i-2}}(A_3(i-2))|,|N_{B_{i+2}}(A_3(i+1))|,|B_{i+2}\cap N(B_{i+1}\setminus N(A_3(i+1)))|>\ceil{\frac{\omega}{4}}$. 
        \item[$(2)$] $|B_{i-1}|,|B_{i+1}|,|A_3(i+1)\cup (B_{i+1}\cap N(A_3(i+1)))|<\frac{3}{4}\omega$. 
        \item[$(3)$] $|A_3(i-2)\cup (B_{i-2}\cap N(A_3(i-2)))|,|B_i|,|B_{i+1}\cap N(A_3(i+1))|< \frac{\omega}{2}$. 
    \end{itemize}
\end{lemma}

\begin{proof}
    By Lemmas \ref{lem:smv on w} and \ref{lem:small vertex argument for simplicial vertices in A_3(i-2)}, $|N_{B_{i+2}}(A_3(i+1))|$, $|B_{i+2}\cap N(B_{i+1}\setminus N(A_3(i+1)))|>\ceil{\frac{\omega}{4}}$. 
    
    Let $u\in B_{i-1}$ be a vertex with minimal neighborhood in $B_{i}\cup B_{i-2}$. 
    Since $N_{B_{i-1}}(A_3(i-2))$ is anticomplete to $B_{i-2}\setminus N(A_3(i-2))$, $u$ is anticomplete to $B_{i-2}\setminus N(A_3(i-2))$. 
    By the minimality of $u$, $N(u)\setminus (A_3(i-2)\cup N_{B_{i-2}}(A_3(i-2)))\subseteq B_{i}\cup B_{i-1}$ is a clique and so $|N_{B_{i-2}}(A_3(i-2))\cup A_3(i-2)|\geq |N(u)\cap (N_{B_{i-2}}(A_3(i-2))\cup A_3(i-2))|>\ceil{\frac{\omega}{4}}$. 
    Moreover, $N(u)\cap (N_{B_{i-2}}(A_3(i-2))\cup A_3(i-2))$ is complete to $B_{i-1}$ and so $|B_{i-1}|<\frac{3}{4}\omega$. 
    
    Let $v\in B_{i+2}$ be a vertex with minimal neighborhood in $B_{i+1}\cup B_{i-2}$. 
    Since $N_{B_{i+2}}(A_3(i+1))$ is anticomplete to $B_{i+1}\setminus N(A_3(i+1))$, $v$ is anticomplete to $B_{i+1}\setminus N(A_3(i+1))$. 
    By the minimality of $v$, $N(v)\setminus (A_3(i+1)\cup N_{B_{i+1}}(A_3(i+1)))\subseteq B_{i}\cup B_{i-1}$ is a clique and so $|A_3(i+1)\cup N_{B_{i+1}}(A_3(i+1))|>\ceil{\frac{\omega}{4}}$. 
    
    Let $w\in A_3(i+1)\cup N_{B_{i+1}}(A_3(i+1))$ be a vertex with minimal neighborhood in $B_{i}\cup B_{i+2}$. 
    Since $N_{B_{i+1}}(A_3(i+1))$ is complete to $N_{B_{i+2}}(B_{i+1}\setminus N(A_3(i+1)))$, $w\in A_3(i+1)$. 
    Since $A_3(i+1)$ is a clique by Lemma \ref{lem:A_3(i+1) is a clique}, $w$ is simplicial in $A_3(i+1)$ and so $|N_{B_{i+2}}(w)|>\ceil{\frac{\omega}{4}}$ by Lemma \ref{lem:small vertex argument for simplicial vertices in A_3(i-2)}. 
    By the minimality of $w$, $N_{B_{i+2}}(w)$ is complete to $A_3(i+1)\cup N_{B_{i+1}}(A_3(i+1))$ and so $|A_3(i+1)\cup N_{B_{i+1}}(A_3(i+1))|<\frac{3}{4}\omega$. 

    Let $x\in B_{i+1}$ be a vertex with minimal neighborhood in $B_{i}\cup B_{i+2}$. 
    Since $N_{B_{i+1}}(A_3(i+1))$ is complete to $N_{B_{i+2}}(A_3(i+1))$, $x\in B_{i+1}\setminus N(A_3(i+1))$. 
    By Lemma \ref{lem:smv on w}, $|N_{B_{i+2}}(x)|>\ceil{\frac{\omega}{4}}$ and so by the minimality of $x$, $N_{B_{i+2}}(x)$ is complete to $B_{i+1}$. 
    It follows that $|B_{i+1}|<\frac{3}{4}\omega$. 

    Since $|B_{i+2}|>2\cdot\ceil{\frac{\omega}{4}}$, and $A_3(i-2)\cup N_{B_{i-2}}(A_3(i-2))$ and $N_{B_{i+1}}(A_3(i+1))$ are complete to $B_{i+2}$, $|A_3(i-2)\cup N_{B_{i-2}}(A_3(i-2))|$ and $|N_{B_{i+1}}(A_3(i+1))|<\frac{\omega}{2}$. 
    Since $B_{i}$ is complete to two disjoint cliques of size larger than $\ceil{\frac{\omega}{4}}$ in $B_{i-1}$ by Lemma \ref{lem:B_i complete to two large cliques in B_{i-1}}, \ref{lem:smv on w} and \ref{lem:small vertex argument for simplicial vertices in A_3(i-2)}, $|B_i|<\frac{\omega}{2}$. 
\end{proof}

We are now ready to give a desired coloring of $G$ by considering whether $A_3(i+1)$ is empty or not.

\begin{lemma}
    If $A_3(i+1)\neq \emptyset$, $\chi(G)\le \ceil{\frac{5}{4}\omega}$.
\end{lemma}

\begin{proof}
Let 
    \begin{align*} 
    x & = |A_3(i-2)\cup (B_{i-2}\cap N(A_3(i-2)))|, 
    \\[3pt]
    y & = |B_{i+2}\cap N(B_{i+1}\setminus N(A_3(i+1)))|,  
    \\[3pt]
    z & = |B_{i+2}\cap N(A_3(i+1))|. 
\end{align*}
By Lemma \ref{clm:v in A3i-2 complete}, $x+y+z\leq \omega$.
Our strategy is to use $p=x+y+z-3\ceil{\frac{\omega}{4}}$ colors to color some vertices of $G$ first and then argue that the remaining vertices can be colored with $f=\ceil{\frac{5}{4}\omega}-p$ colors. 
Note that
    \begin{align*} 
    f & = \ceil{\frac{5}{4}\omega}-(x+y+z-3\ceil{\frac{\omega}{4}}) 
    \\[3pt]
     & = 4\ceil{\frac{\omega}{4}}+(\omega-(x+y+z))  
    \\[3pt]
     & \ge \omega,
\end{align*}
since $x+y+z\leq \omega$.

\medskip
\noindent {\bf Step 1: Precolor a subgraph.}
Next we define vertices to be colored with $p$ colors. Let $S_i\subseteq B_i$ be the set of $p$ vertices with largest neighborhoods in $H$, $S_{i-2}\subseteq A_3(i-2)\cup (B_{i-2}\cap N(A_3(i-2)))$ be the set of $x-\ceil{\frac{\omega}{4}}$ vertices with largest neighborhoods in $H$, $S^y_{i+2}\subseteq B_{i+2}\cap N(B_{i+1}\setminus N(A_3(i+1)))$ be the set of $y-\ceil{\frac{\omega}{4}}$ vertices with largest neighborhoods in $H$, and
$S^z_{i+2}\subseteq B_{i+2}\cap N(A_3(i+1))$ be the set of $z-\ceil{\frac{\omega}{4}}$ vertices with largest neighborhoods in $H$.
Since $(x+y+z)-3\ceil{\frac{\omega}{4}}\leq \ceil{\frac{\omega}{4}}<|B_i|$, the choice of $S_i$ is possible. We color $S_i$ with all colors in $[p]$, and color all vertices in $S_{i-2}\cup S^y_{i+2}\cup S^z_{i+2}$ with all colors in $[p]$. Since $B_i$ is anticomplete to $A'_3(i-2)\cup B_{i+2}$, this coloring is proper. So it suffices to show that
$G_1=G-(S_i\cup S_{i-2}\cup S^y_{i+2}\cup S^z_{i+2})$ is $f$-colorable. 

\medskip
\noindent {\bf Step 2: Reduce by degeneracy.}
Next, we use Lemma \ref{lem:degeneracy} to further reduce the problem to color a subgraph of $G_1$. Let $R_{i-1}$ be the set of vertices in $B_{i-1}$ that is anticomplete to $B_{i-2}\setminus N(A_3(i-2))$. Let $v\in R_{i-1}$ with minimal neighborhood in $B_{i-2}\cup B_i$ among $R_{i-1}$. Since  $v$ is anticomplete to $B_{i-2}\setminus N(A_3(i-2))$, $v$ is minimal among $B_{i-1}$.
By Lemma \ref{lem:degeneracy}  with $(X_1,X_2,X_3,X^*)=\left (B_i,B_{i-1},\left (A_3(i-2)\cup N_{B_{i-2}}\left (A_3(i-2)\right)\right )\setminus S_{i-2},A_2(i-1)\right )$, $P=S_i$, and $F=G_1$, we have $d_{G_1}(v)<f$. It follows that $G_1$ is $f$-colorable if and only if $G_1-v$ is $f$-colorable. Repeatedly applying Lemma \ref{lem:degeneracy} we conclude that $G_1$ is $f$-colorable if only if $G_1-R_{i-1}$ is $f$-colorable.

Let $R^y_{i+1}=B_{i+1}\setminus N(A_3(i+1))$. Let $v\in R^y_{i+1}$ be a vertex with minimal neighborhood in $B_{i+2}\cup B_i$ among $R^y_{i+1}$. Since $v$ is anticomplete to $N_{B_{i+2}}(A_3(i+1))$, $v$ is minimal among $B_{i+1}$.
By Lemma \ref{lem:degeneracy} with $(X_1,X_2,X_3,X^*)=\left (B_i,B_{i+1},N_{B_{i+2}}(R^y_{i+1})\setminus S^y_{i+2},A_2(i)\right )$, $P=S_i$, and $F=G_1-R_{i-1}$, we have $d_{G_1-R_{i-1}}(v)<f$. It follows that $G_1-R_{i-1}$ is $f$-colorable if and only if $(G_1-R_{i-1})-v$ is $f$-colorable. Repeatedly applying Lemma \ref{lem:degeneracy} we conclude that $G_1-R_{i-1}$ is $f$-colorable if only if $(G_1-R_{i-1})-R^y_{i+1}$ is $f$-colorable. 

Let $R^z_{i+1}=A_3(i+1)$. Let $v\in R^z_{i+1}$ be a vertex with minimal neighborhood in $B_{i+2}\cup B_i$ among $R^z_{i+1}$. Since $v$ is anticomplete to $B_{i+2}\setminus N_{B_{i+2}}(R^z_{i+1})$, $v$ is minimal among $N_{B_{i+1}}(R^z_{i+1})\cup R^z_{i+1}$.
By Lemma \ref{lem:degeneracy} with $(X_1,X_2,X_3,X^*)=\left (B_i,R^z_{i+1}\cup N_{B_{i+1}}(R^z_{i+1}),N_{B_{i+2}}(R^z_{i+1})\setminus S^z_{i+2},A_2(i)\right)$, $P=S_i$, and $F=G_1-R_{i-1}-R^y_{i+1}$, we have $d_{G_1-R_{i-1}-R^y_{i+1}}(v)<f$. It follows that $G_1-R_{i-1}-R^y_{i+1}$ is $f$-colorable if and only if $(G_1-R_{i-1}-R^y_{i+1})-v$ is $f$-colorable. Repeatedly applying Lemma \ref{lem:degeneracy} we conclude that $G_1-R_{i-1}-R^y_{i+1}$ is $f$-colorable if only if $G_1-R_{i-1}-R^y_{i+1}-R^z_{i+1}$ is $f$-colorable. 

\medskip
\noindent {\bf Step 3: Reduce simplicial vertices.}
Now it suffices to show that $G_2=G_1-R_{i-1}-R^y_{i+1}-R^z_{i+1}$ is $f$-colorable. Note that every vertex in $A_3(i-2)$ is simplicial in $G_2$.
So $G_2$ is $f$-colorable if and only if $G_2-A_3(i-2)$ is $f$-colorable. 
Since every vertex in $A_2(i+2)$ is simplicial in $G_2-A_3(i-2)$, $G_2-A_3(i-2)$ is $f$-colorable if and only if $G_3=G_2-A_3(i-2)-A_2(i+2)$ is $f$-colorble.

\medskip
\noindent {\bf Step 4: Color $G_3$.}
So it suffices to show that $G_3$ is $f$-colorable. Let $H'$ be a subgraph of $G_3$ which is hyperhole with maximum chromatic number. It follows that $\omega(H')\leq \omega\leq f$. 
It is shown in \cite{Penev25} (Theorem 2.1) that $\chi(G_3)=\chi(H')$. 
If $H'$ is perfect, then $\chi(H')=\omega(H')\leq f$. 
So $H'\cap B_j$ is not empty for $j\in [5]$. 
Let $M_1=(H'\cap B_{i-1})\cup (H'\cap B_{i})$. 
By the definition of $S_i$, $M_1\cup S_i$ is a clique of $G$. 
By Lemma \ref{lem:B_i complete to two large cliques in B_{i-1}}, $N_{B_{i-1}}(A_3(i-2))$ is complete to $B_i$ and so $(M_1\cup S_i)\cup N_{B_{i-1}}(A_3(i-2))$ is also a clique of $G$.
Then $|M_1|\leq \omega-|S_i|-|N_{B_{i-1}}(A_3(i-2))|\leq \omega-p-\ceil{\frac{\omega}{4}}$. 
Let $M_2=(H'\cap B_{i-2})\cup (H'\cap B_{i+2})$. 
By Lemma \ref{lem:B_i complete to two large cliques in B_{i-1}}, $S_{i+2}^y$ and $S_{i+2}^z$ are complete to $B_{i-2}$ and so $M_2\cup S_{i+2}^y\cup S_{i+2}^z$ is a clique of $G$.
Then $|M_2|\leq \omega-|S_{i+2}^y|-|S_{i+2}^z|\leq \omega-(y+z)+2\ceil{\frac{\omega}{4}}$. 
Let $M_3=H'\cap B_{i+1}$. Note that $M_3\subseteq N_{B_{i+1}}(A_3(i+1))$. 
By Lemma \ref{lem:single vertex in A_3(i)} (3), $M_3$ is complete to $N_{B_{i+2}}(A_3(i+1))$ and $N_{B_{i+2}}(B_{i+1}\setminus N(A_3(i+1)))$
and so $|M_3|<\frac{\omega}{2}$. 
Then 
    \begin{align}
	    |V(H')|= & |M_1|+|M_2|+|M_3|\nonumber
            \\< & (\omega-\ceil{\frac{\omega}{4}}-(x+y+z)+3\ceil{\frac{\omega}{4}})+(\omega-(y+z)+2\ceil{\frac{\omega}{4}})+\frac{\omega}{2}\nonumber 
            \\=&2\omega+4\ceil{\frac{\omega}{4}}+(x+\frac{\omega}{2} )-2(x+y+z)\nonumber 
            \\\leq&2(\omega+4\ceil{\frac{\omega}{4}}-(x+y+z)) \nonumber
            \\= &2f, \nonumber
    \end{align}
where the last inequality is due to $x\le \frac{\omega}{2}$.
It follows that $\chi(G_3)= \chi(H')=\max \left \{\omega(H'),\ceil{\frac{|V(H')|}{2}} \right \}\leq f$. 
\end{proof}

\begin{lemma}
    If $A_3(i+1)=\emptyset$, then $\chi(G)\leq \ceil{\frac{5}{4}\omega}$.
\end{lemma}

\begin{proof}
    Let 
    \begin{align*}
        x & = |B_i|,
        \\
        y & = |N_{B_{i-1}}(A_3(i-2))|. 
    \end{align*}
    By Lemma \ref{lem:B_i complete to two large cliques in B_{i-1}}, $x+y< \omega-\ceil{\frac{\omega}{4}}\leq \frac{3}{4}\omega$. 
    Our strategy is to use $p=x+y-2\ceil{\frac{\omega}{4}}$ colors to color some vertices of $G$ first and then argue that the remaining vertices can be colored with $f=\ceil{\frac{5}{4}\omega}-p$ colors. 
    Note that 
    \begin{align*}
        f & =\ceil{\frac{5}{4}\omega}-(x+y-2\ceil{\frac{\omega}{4}})
        \\
        & = \omega+(3\ceil{\frac{\omega}{4}}-(x+y))
        \\
        &\geq \omega, 
    \end{align*}
    since $x+y\leq \frac{3}{4}\omega$.

    \medskip
    \noindent
    \textbf{Step 1: Precolor a subgraph} Next we define vertices to be colored with colors in $[p]$. 
    Let $S_i\subseteq B_i$ with $|S_i|=x-\ceil{\frac{\omega}{4}}$, $S_{i-1}\subseteq N_{B_{i-1}}(A_3(i-2))$ with $|S_{i-1}|=y-\ceil{\frac{\omega}{4}}$, and $S_{i+2}\subseteq B_{i+2}$ be the set of $p$ vertices with largest neighborhoods in $H$. 
    Since $p=x+y-2\ceil{\frac{\omega}{4}}<\ceil{\frac{\omega}{4}}\leq |B_{i+2}|$, the choice of $S_{i+2}$ is possible. 
    We color all vertices in $S_{i}\cup S_{i-1}$ with all colors in $[p]$ and color $S_{i+2}$ with all colors in $[p]$. 
    Since $B_{i+2}$ is anticomplete to $B_{i}\cup B_{i-1}$, this coloring is proper. 
    So it suffices to show that $G_1=G-(S_{i}\cup S_{i-1}\cup S_{i+2})$ is $f$-colorable. 

    \medskip
    \noindent
    \textbf{Step 2: Reduce by degeneracy} Next we use Lemma \ref{lem:degeneracy} to reduce the problem to color a subgraph of $G_1$. 
    
    Let $v\in A_3(i-2)$ be a vertex with minimal neighborhood in $B_{i-1}\cup B_{i+2}$. 
    By Lemma \ref{lem:degeneracy} with $(X_1,X_2,X_3,X^*)=\left (B_{i+2}, A_3(i-2)\cup N_{B_{i-2}}(A_3(i-2)),N_{B_{i-1}}(A_3(i-2))\setminus S_{i-1},A_2(i+2)\right )$, $P=S_{i+2}$, and $F=G_1$, we have $d_{G_1}(v)<f$. 
    It follows that $G_1$ is $f$-colorable if and only if $G_1-v$ is $f$-colorable. 
    Repeatedly applying Lemma \ref{lem:degeneracy}, we conclude that $G_1$ is $f$-colorable if and only if $G_1-A_3(i-2)$ is $f$-colorable.

    Let $v\in B_{i+1}$ be a vertex with minimal neighborhood in $B_{i}\cup B_{i+2}$. 
    By Lemma \ref{lem:degeneracy} with $(X_1,X_2,X_3,X^*)=\left (B_{i+2},B_{i+1},B_i\setminus S_{i},\emptyset \right )$, $P=S_{i+2}$, and $F=G_1-A_3(i-2)$, we have $d_{G_1-A_3(i-2)}(v)<f$. 
    It follows that $G_1-A_3(i-2)$ is $f$-colorable if and only if $(G_1-A_3(i-2))-v$ is $f$-colorable. 
    Repeatedly applying Lemma 11.19, we conclude that $G_1-A_3(i-2)$ is $f$-colorable if and only if $G_2=G_1-A_3(i-2)-B_{i+1}$ is $f$-colorable. 

    Since $G_2$ is chordal, $\chi(G_2)\leq \omega(G_2)\leq \omega\leq f$. 
\end{proof}

\section{$A_1$ and $A_2$ are empty}\label{sec:only A3}

Now we can assume that $A_1=A_2=\emptyset$. Since $A_5$ is universal, $A_5=\emptyset$. So $G=H\cup A_3$.

\begin{lemma}\label{lem:structure of only A3}
    One of the following holds.
    \begin{itemize}
        \item[$(1)$] $A_3=A_3(i-1)\cup A_3(i+1)$ for some $i\in [5]$ where $A_3(i-1)$ and $A_3(i+1)$ are non-empty cliques, and for every $v\in A_3(i-1)$ and $u\in A_3(i+1)$, $N_{B_i}(v)$ is anticomplete to $B_{i-1}\setminus N(v)$ and $N_{B_{i+2}}(u)$ is anticomplete to $B_{i+1}\setminus N(u)$.
        \item[$(2)$]  $A_3=A_3(i-2)\cup A_3(i+2)$ for some $i\in [5]$ where $A_3(i-2)$ and $A_3(i+2)$ are non-empty cliques, and for every $v\in A_3(i-2)$ and $u\in A_3(i+2)$, $N_{B_{i-1}}(v)$ is anticomplete to $B_{i-2}\setminus N(v)$ and $N_{B_{i+1}}(u)$ is anticomplete to $B_{i+2}\setminus N(u)$.
        \item[$(3)$] $A_3=A_3(i-2)$ for some $i\in [5]$ and $A_3(i-2)$ is a blowup of $P_3$. 
    \end{itemize}
\end{lemma}
\begin{proof}
    We start with the following property.
    
    \vspace{.2cm}
    \noindent (i) Let $v$ be a simplicial vertex in $A_3(i-2)$ such that $N_{B_{i+2}}(v)$ is anticomplete to $B_{i-2}\setminus N(v)$. 
    Then $A_3(i+2)=\emptyset$. 
    \vspace{.2cm} 

    Suppose to the contrary that $A_3(i+2)\neq \emptyset$. 
    Let $w\in B_{i-2}\setminus N(A_3(i-2))$ be a vertex with minimal neighborhood in $B_{i-1}\cup B_{i+2}$.
    By Lemma \ref{lem:smv on w} and Lemma \ref{lem:small vertex argument for simplicial vertices in A_3(i-2)}, we have $N_{B_{i+2}}(v)$ and $N_{B_{i+2}}(w)$ are disjoint cliques larger than $\ceil{\frac{\omega}{4}}$. 

    Let $u$ be a simplicial vertex in $A_3(i+2)$ and $w'\in B_{i+2}\setminus N(A_3(i+2))$ be a vertex with minimal neighborhood in $B_{i+1}\cup B_{i-2}$.
    By Lemma \ref{lem:blowup-threeneighbor}, $B_{i+2}\setminus N(u)$ is anticomplete to either $N_{B_{i+1}}(u)$ or $N_{B_{i-2}}(u)$. 
    Suppose first that $B_{i+2}\setminus N(u)$ is anticomplete to $N_{B_{i+1}}(u)$. 
    Then $N_{B_{i+2}}(v)\cup N_{B_{i+2}}(w)\cup N_{B_{i+1}}(u)\cup N_{B_{i+1}}(w')$ is a clique larger than $\omega$ by Lemma \ref{lem:B_i complete to two large cliques in B_{i-1}}, Lemma \ref{lem:smv on w} and Lemma \ref{lem:small vertex argument for simplicial vertices in A_3(i-2)}, a contradiction. 
    So we may assume that $B_{i+2}\setminus N(u)$ is anticomplete to $N_{B_{i-2}}(u)$. 
    Then $N_{B_{i+2}}(v)\cup N_{B_{i+2}}(w)\cup N_{B_{i-2}}(u)\cup N_{B_{i-2}}(w')$ is a clique larger than $\omega$ by Lemma \ref{lem:complete and anti on w}, Lemma \ref{lem:smv on w} and Lemma \ref{lem:small vertex argument for simplicial vertices in A_3(i-2)}, a contradiction.
    This completes the proof of (i). 
    
    \vspace{.2cm}

    Since $A_3(i)\neq \emptyset$, we may assume by symmetry that $A_3(i-2)\neq \emptyset$. 
    Let $v_{i-2}\in A_3(i-2)$ be a vertex simplicial in $A_3(i-2)$. By symmetry and Lemma \ref{lem:blowup-threeneighbor}, we may assume that $N_{B_{i+2}}(v_{i-2})$ is anticomplete to $B_{i-2}\setminus N(v_{i-2})$.  By (1), we have $A_3(i+2)=\emptyset$.

    \vspace{.2cm}
    \noindent (ii) If $A_3(i+1)\neq \emptyset$, then (1) holds. 
    \vspace{.2cm}

    By Lemma \ref{lem:no two directed edge pointing to the same B_i}, we have $N_{B_i}(u)$ is anticomplete to $B_{i+1}\setminus N(u)$ for every $u\in A_3(i+1)$. 
    By (i) and Lemma \ref{lem:no two directed edge pointing to the same B_i}, $A_3(i)=A_3(i-1)=\emptyset$ and so $A_3=A_3(i-2)\cup A_3(i+1)$. 

    Suppose first that $A_3(i+1)$ is not a clique. 
    Then $A_3(i+1)$ has two non-adjacent simplicial vertices $u_1',u_2'$.
    For $j \in [2]$, let $u_j\in N_{A_3(i+1)}[u'_j]$ be a vertex simplicial in $A_3(i+1)$ with minimal neighborhood in $B_i\cup B_{i+2}$.
    By Lemma \ref{lem:two non-adjacent simplicail verteices}, $u_1u_2\notin E(G)$.
    By Lemma \ref{lem:no two directed edge pointing to the same B_i}, we have $u_1,u_2$ have disjoint neighborhoods in $B_i$. 
    Let $w$ be a vertex in $B_{i-2}\setminus N(A_3(i-2))$ with minimal neighborhood in $B_{i-1}\cup B_{i+2}$ and $w'$ be a vertex in $B_{i+1}\setminus N(A_3(i+1))$ with minimal neighborhood in $B_{i}\cup B_{i+2}$. 
    By Lemmas \ref{lem:B_i complete to two large cliques in B_{i-1}}, \ref{lem:smv on w} and \ref{lem:small vertex argument for simplicial vertices in A_3(i-2)}, we have $N_{B_{i-1}}(w)\cup N_{B_i}(w')\cup N_{B_i}(u_1)\cup N_{B_i}(u_2)$ is a clique larger than $\omega$. 
    So $A_3(i+1)$ is a clique. Let $K = A_3(i+1)$.

    Suppose that $A_3(i-2)$ is not a clique. 
    Then $A_3(i-2)$ has two non-adjacent simplicial vertices $v_1',v_2'$.
    For $j \in [2]$, let $v_j\in N_{A_3(i-2)}[v'_j]$ be a vertex simplicial in $A_3(i-2)$ with minimal neighborhood in $B_{i-1}\cup B_{i+2}$.
    By Lemma \ref{lem:two non-adjacent simplicail verteices}, $v_1v_2\notin E(G)$.
    Let $w\in B_{i-2}\setminus N(A_3(i-2))$ with minimal neighborhood in $B_{i-1}\cup B_{i+2}$. 
    If any two of $w,v_1,v_2$ have disjoint neighborhoods in $B_{i-1}$, then $B_{i-1}\cup B_i$ has a clique larger than $\omega$ by Lemmas \ref{lem:B_i complete to two large cliques in B_{i-1}}, \ref{lem:smv on w} and \ref{lem:small vertex argument for simplicial vertices in A_3(i-2)}, a contradiction. 
    So we may assume that $w,v_1,v_2$ have disjoint neighborhoods in $B_{i+2}$. 
    Let $u\in B_{i}$ be a vertex with minimal neighborhood in $B_{i-1}\cup B_{i+1}$. 
    Note that $N(u)\setminus (K\cup N_{B_{i+1}}(K))\subseteq B_{i-1}\cup B_{i}$. 
    Then by the minimality of $u$, $N_{B_{i-1}}(u) \cup N_{B_i}(u)$ is a clique and so $K\cup N_{B_{i+1}}(K)$ is a clique of size larger than $\ceil{\frac{\omega}{4}}$ by Lemma \ref{lem:edge in $A_3(i)$}. 
    Then $N_{B_{i+2}}(w)\cup N_{B_{i+2}}(v_1)\cup N_{B_{i+2}}(v_2)\cup K\cup N_{B_{i+1}}(K)$ is a clique larger than $\omega$ by Lemmas  \ref{lem:B_i complete to two large cliques in B_{i-1}}, \ref{lem:smv on w} and \ref{lem:small vertex argument for simplicial vertices in A_3(i-2)}, a contradiction. 
 
    Since $A_3(i-2)$ is a clique, every $v\in A_3(i-2)$ and $v_{i-2}$ have a common neighbor in $B_{i+2}$.  
    If $N_{B_{i+2}}(v)$ is not anticomplete to $B_{i-1}\setminus N(v)$, then $v$ and $w$ have a common neighbor in $B_{i+2}$ and so $N_H(v_{i-2})\subsetneq N_H(v)$. 
    Since $v_{i-2}$ and $w$ have a common neighbor in $B_{i-1}$, $v$ and $w$ have a common neighbor in $B_{i-1}$, contradicts $G$ is $C_4$-free. 
    So $N_{B_{i+2}}(v)$ is anticomplete to $B_{i+2}\setminus N(v)$. 
    This completes the proof of (ii). 

    \vspace{.2cm}
    So we may assume that $A_3(i+1)=\emptyset$. 

    \vspace{.2cm}
    \noindent (iii) If $A_3(i)\neq \emptyset$, then (1) holds. 
    \vspace{.2cm}

    Let $v$ be a simplicial vertex of $A_3(i)$ with minimal neighborhood in $B_{i-1}\cup B_{i+1}$. 
    If $B_{i}\setminus N(v)$ is anticomplete to $N_{B_{i-1}}(v)$, then we are done by (2). 
    So we may assume that $B_i\setminus N(v)$ is anticomplete to $N_{B_{i+1}}(v)$. 
    By Lemmas \ref{lem:B_i complete to two large cliques in B_{i-1}}, \ref{lem:smv on w} and  \ref{lem:small vertex argument for simplicial vertices in A_3(i-2)}, we have $B_{i+1}\cup B_{i+2}$ has a clique larger than $\omega$, a contradiction. 
    This completes the proof of (iii). 

    \vspace{.2cm}
    So we may assume that $A_3(i)=\emptyset$. 

    \vspace{.2cm}
    \noindent (iv) If $A_3(i-1)\neq \emptyset$, then (2) holds. 
    \vspace{.2cm}

    Let $w$ be a vertex in $B_{i-2}\setminus N(A_3(i-2))$ with minimal neighborhood in $B_{i-1}\cup B_{i+2}$ and $w'$ be a vertex in $B_{i-1}\setminus N(A_3(i-1))$ with minimal neighborhood in $B_{i}\cup B_{i-2}$. 
    
    Suppose first that $A_3(i-1)$ is not a clique. 
    Then $A_3(i-1)$ has two non-adjacent simplicial vertices $v_1',v_2'$. 
    For $j \in [2]$, let $v_j\in N_{A_3(i-1)}[v'_j]$ be a vertex simplicial in $A_3(i-1)$ with minimal neighborhood in $B_i\cup B_{i-2}$.
    By Lemma \ref{lem:two non-adjacent simplicail verteices}, $v_1v_2\notin E(G)$.
    By (1), we have $v_1,w'$ and $v_2,w'$ have disjoint neighborhoods in $B_i$. 
    If $v_1,v_2$ have disjoint neighborhoods in $B_{i-2}$, then $N_{B_{i-2}}(v_1)\cup N_{B_{i-2}}(v_2)\cup N_{B_{i+2}}(v_{i-2})\cup N_{B_{i+2}}(w)$ is a clique larger than $\omega$ by Lemmas \ref{lem:B_i complete to two large cliques in B_{i-1}}, \ref{lem:smv on w} and \ref{lem:small vertex argument for simplicial vertices in A_3(i-2)}, a contradiction. 
    So we may assume that $v_1,v_2$ have disjoint neighborhoods in $B_{i}$. 
    Then $N_{B_i}(w')\cup N_{B_i}(v_1)\cup N_{B_i}(v_2)\cup N_{B_{i-1}}(v_{i-2})$ is a clique larger than $\omega$ by Lemmas \ref{lem:single vertex in A_3(i)} (3), \ref{lem:compare nbd of different A_3(i) in B_j}, \ref{lem:smv on w} and \ref{lem:small vertex argument for simplicial vertices in A_3(i-2)}, a contradiction. 
    So $A_3(i-1)$ is a clique. 
    
    Let $v_{i-1}$ be a simplicial vertex of $A_3(i-1)$ with minimal neighborhood in $B_{i}\cup B_{i-2}$. 
    By (1), we have $N_{B_i}(v_{i-1})$ is anticomplete to $B_{i-1}\setminus N(v_{i-1})$. 
    Since $A_3(i-1)$ is a clique, we have $u$ and $v_{i-1}$ have a common neighbor in $B_i$ for every $u\in A_3(i-1)$. Let $u \in A_3(i-1)$.
    If $N_{B_i}(u)$ is not anticomplete to $B_{i-1}\setminus N(u)$, then $u$ and $w'$ have a common neighbor in $B_i$ and so $N_H(v_{i-1})\subseteq N_H(u)$. 
    Since $v_{i-1}$ and $w'$ have a common neighbor in $B_{i-2}$, $u$ and $w'$ have a common neighbor in $B_{i-2}$, which contradicts $G$ is $C_4$-free.  
    By symmetry,  $A_3(i-2)$ is a clique and $N_{B_{i+2}}(v)$ is anticomplete to $B_{i-2}\setminus N(v)$ for every $v\in A_3(i-2)$.
    This completes the proof of (iv). 

    \vspace{.2cm}
    So we may assume that $A_3=A_3(i-2)$. 

    \vspace{.2cm}
    \noindent (v) If $A_3=A_3(i-2)$, then $(3)$ holds.  
    \vspace{.2cm}

    Let $u\in B_i$ with minimal neighborhood in $B_{i-1}\cup B_{i+1}$. 
    Since $N(u)\setminus B_{i+1}$ is a clique, we have $|N_{B_{i+1}}(u)|>\ceil{\frac{\omega}{4}}$ and so $|B_{i+1}|>\ceil{\frac{\omega}{4}}$. 
    By symmetry, we have $|B_{i}|>\ceil{\frac{\omega}{4}}$. 

    We first show that there are no three pairwise non-adjacent simplicial vertices of $A_3(i-2)$. Suppose for a contradiction that $v_1',v_2',v_3'$ are three pairwise non-adjacent simplicial vertices of $A_3(i-2)$.
    For $j \in [3]$, let $v_j\in N_{A_3(i-2)}[v'_j]$ be a vertex simplicial in $A_3(i-2)$ with minimal neighborhood in $B_{i-1}\cup B_{i+2}$.
    By Lemma \ref{lem:two non-adjacent simplicail verteices}, $v_1,v_2,v_3$ are pairwise non-adjacent.
    By Lemmas \ref{lem:B_i complete to two large cliques in B_{i-1}}, \ref{lem:smv on w} and  \ref{lem:small vertex argument for simplicial vertices in A_3(i-2)}, we have that no three of $w,v_1,v_2,v_3$ have pairwise disjoint neighborhoods in $B_{i+2}$ or in $B_{i-1}$. In other word, no three of $w,v_1,v_2,v_3$ have pairwise neighborhood containment in $B_{i+2}$ or in $B_{i-1}$. So, we may assume that $v_2$ and $v_3$ have disjoint neighborhoods in $B_{i+2}$. Then $v_1$ has neighborhood containment with either both $v_2$ and $v_3$ in $B_{i+2}$, or exactly one of $v_2$ and $v_3$ in $B_{i+2}$.

    Suppose first that $v_1$ has neighborhood containment with both $v_2$ and $v_3$ in $B_{i+2}$. As $w$ has neighborhood containment with at least one of $v_2$ and $v_3$ in $B_{i+2}$, $w$ and $v_1$ have neighborhood containment in $B_{i+2}$. Then $w$, $v_1$ and one of $v_2$ and $v_3$ have pairwise neighborhood containment in $B_{i+2}$, a contradiction. So, we may assume that $v_1$ has neighborhood containment with exactly one of $v_2$ and $v_3$ in $B_{i+2}$.

    By symmetry, we may assume that $v_1$ has neighborhood containment with $v_2$ in $B_{i+2}$. Then $v_1$ and $v_2$ have disjoint neighborhood in $B_{i-1}$ and $v_3$ have neighborhood containment with both $v_1$ and $v_2$ in $B_{i-1}$. By same argument,  as $w$ has neighborhood containment with at least one of $v_1$ and $v_2$ in $B_{i-1}$, $w$, $v_3$ and one of $v_1$ and $v_2$ have pairwise neighborhood containment in $B_{i-1}$, a contradiction. Therefore, there are no three pairwise non-adjacent simplicial vertices of $A_3(i-2)$.

    This immediately implies that $A_3(i-2)$ has at most two components and if $A_3(i-2)$ has two components, then $A_3(i-2)$ is a disjoint union of two cliques. So we assume that $A_3(i-2)$ is connected. By Lemma \ref{lem:P4-freenessofA3}, $A_3(i-2)$ can be partitioned into two complete subsets $X$ and $Y$. Since $G$ is $C_4$-free, we may assume that $X$ is a clique. Choose such a partition $(X,Y)$ such that $|X|$ is maximum. It follows that $Y$ is disconnected. Since any $v\in Y$ simplicial in $Y$ is also simplicial in $A_3(i-2)$, it follows that $Y$ has two component and each component is a clique. So $A_3(i)$ is a blowup of $P_3$.  This completes the proof of (v). 

    \vspace{.2cm}
    This completes the proof of Lemma \ref{lem:structure of only A3}. 
\end{proof}

Next we color $H\cup A_3$ in terms of the outcomes of Lemma \ref{lem:structure of only A3}. For a subset $S\subseteq B_i$ and $v\in S$,  $v$ is {\em minimal} if $N_H(v)$ is minimal over all vertices in $S$, and $v$ is {\em maximal} if $N_H(v)$ is maximal over all vertices in $S$.

\subsection{Distance 2}

In this subsection, we handle the first outcome of Lemma \ref{lem:structure of only A3}.

\begin{lemma}\label{lem:two disjoint A3}
    Let $G=H\cup K\cup J$ where
    $K$ is a non-empty component of $A_3(i-1)$ and $J$ is a non-empty component of $A_3(i+1)$.
    Then $\chi(G)\le \ceil{\frac{5}{4}\omega}$. 
\end{lemma}

\begin{proof}
    By Lemma \ref{lem:structure of only A3}, both $K$ and $J$ are cliques. Moreover, $B_{i-1} \setminus N(K)$ is anticomplete to $N_{B_i}(K)$ and $B_{i+1} \setminus N(J)$ is anticomplete to $N_{B_{i+2}}(J)$. Let $z \in B_{i-1} \setminus N(K)$ be a minimal vertex in $B_{i-1} \setminus N(K)$ and let $w \in B_{i+1} \setminus N(J)$ be a minimal vertex in $B_{i+1} \setminus N(J)$. Then $N_{B_i}(K)$ and $N_{B_i}(z)$ are disjoint and $N_{B_{i+2}}(J)$ and $N_{B_{i+2}}(w)$ are disjoint.
    
    By Lemma \ref{lem:smv on w} and Lemma \ref{lem:small vertex argument for simplicial vertices in A_3(i-2)}, $|B_{i-2}|$, $|N_{B_i}(K)|$, $|N_{B_i}(z)|$, $ |N_{B_{i+2}}(J)|$, $|N_{B_{i+2}}(w)| > \ceil{\frac{\omega}{4}}$. By Lemma \ref{lem:single vertex in A_3(i)} (4), there is a vertex in $N_{B_{i-2}}(K)$ complete to $B_{i-1}\cup B_{i+2}$. We set this vertex to be $q_{i-2}$. Let $v_K \in K$ that has maximum neighborhood in $H$. Let $v_K'' \in K$ that has minimal neighborhood in $H$ and let $v_J'' \in J$ that has minimal neighborhood in $H$.

    We set $F = (F_1, F_2, F_3, F_4, F_5, O)$ as follows:
    \begin{itemize}
        \item $F_1 = \{q_{i-2}\}$,
        \item $F_2 = \begin{cases}
            \{q_{i-1}, q_{i-1}', a\} \quad & \text{if }\,\,|N_{B_{i-1}}(K)| \ge 2\\
            \{q_{i-1}, a, a'\} \quad &\text{if }\,\,|N_{B_{i-1}}(K)| = 1,
        \end{cases}$\\
        where $a, a' \in B_{i-1}\setminus N(K)$ are the first and second maximal in $B_{i-1}\setminus N(K)$, respectively,
        \item $F_3 = \{q_i, q_i', b\}$, where $b \in N_{B_i}(v_K'')$,
        \item $F_4 = \{q_{i+1}, c\}$, where $c \in B_{i+1}\setminus N(J)$ is a maximal vertex in $B_{i+1}\setminus N(J)$,
        \item $F_5 = \{q_{i+2}, d, d'\}$, where $d, d' \in N_{B_{i+2}}(v_J'')$,
        \item $O = \begin{cases}
            \{v_K\} \quad & \text{if }\,\,|N_{B_{i-1}}(K)| \ge 2\\
            \{v_K, v_K'\} \quad &\text{if }\,\,|N_{B_{i-1}}(K)| = 1,
        \end{cases}$\\
        where $v_K' \in K\setminus \{v_K\}$ that has maximal neighborhood in $H$.
    \end{itemize}

We can assign color $1$ to $a$, $b$ and $d$, color $2$ to $q_{i-1}$ and $q_{i+1}$, color $3$ to $v_K$, $q_i$ and $q_{i+2}$, color $4$ to $q_{i-2}$ and $q_i'$, and color $5$ to $v_K'$, $q_{i-1}'$, $a'$, $c$ and $d'$. So $F$ is 5-colorable (see Figure \ref{fig:good subgraph1}). 

    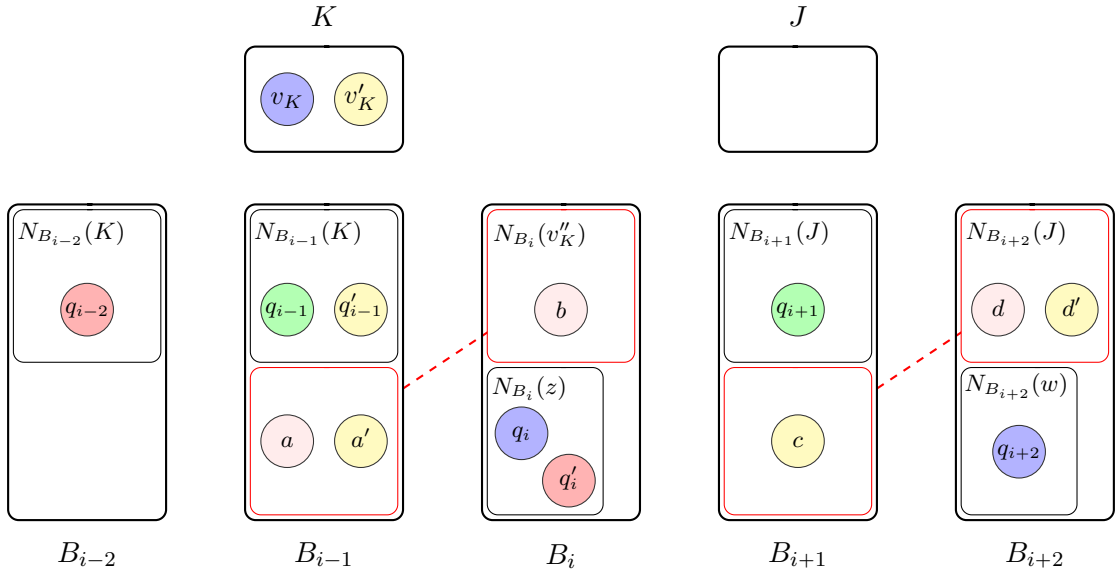
\begin{figure}[h]
        \centering
        \begin{tikzpicture}[scale=0.7]
        \tikzstyle{v}=[circle, draw, solid, fill=black, inner sep=0pt, minimum width=3pt]

        \draw[red, dashed, thick] (-4.5, -3.5)--(0, -0.5);
        \draw[red, dashed, thick] (4.5, -3.5)--(9, -0.5);
        
        \draw[rounded corners, thick] (-9, 1)--(-7.5, 1)--(-7.5, -5)--(-10.5, -5)--(-10.5, 1)--(-8.9, 1);
        \draw[rounded corners] (-9, 0.9)--(-7.6, 0.9)--(-7.6, -2)--(-10.4, -2)--(-10.4, 0.9)--(-8.9, 0.9);

        \node [label=$B_{i-2}$] (v) at (-9, -6.3){};
        \node [label={\footnotesize $N_{B_{i-2}}(K)$}] (v) at (-9.3, -0.2){};

        \draw[draw=black!80, fill =red!30] (-9, -1) circle (0.5);
        \node [label = {\footnotesize $q_{i-2}$}] (v) at (-9, -1.55){};

        \draw[rounded corners, thick] (-4.5, 1)--(-3, 1)--(-3, -5)--(-6, -5)--(-6, 1)--(-4.4, 1);
        \draw[rounded corners] (-4.5, 0.9)--(-3.1, 0.9)--(-3.1, -2)--(-5.9, -2)--(-5.9, 0.9)--(-4.4, 0.9);
        \draw[red, fill=white!, rounded corners] (-4.5, -2.1)--(-3.1, -2.1)--(-3.1, -4.9)--(-5.9, -4.9)--(-5.9, -2.1)--(-4.4, -2.1);

        \node [label=$B_{i-1}$] (v) at (-4.5, -6.3){};
        \node [label={\footnotesize $N_{B_{i-1}}(K)$}] (v) at (-4.8, -0.2){};

        \draw[draw=black!80, fill =green!30] (-5.2, -1) circle (0.5);
        \node [label = {\footnotesize $q_{i-1}$}] (v) at (-5.2, -1.55){};

        \draw[draw=black!80, fill =yellow!30] (-3.8, -1) circle (0.5);
        \node [label = {\footnotesize $q_{i-1}'$}] (v) at (-3.8, -1.55){};

        \draw[draw=black!80, fill =pink!30] (-5.2, -3.5) circle (0.5);
        \node [label = {\footnotesize $a$}] (v) at (-5.2, -4){};

        \draw[draw=black!80, fill =yellow!30] (-3.8, -3.5) circle (0.5);
        \node [label = {\footnotesize $a'$}] (v) at (-3.8, -4){};
        
        \draw[rounded corners, thick] (0, 1)--(1.5, 1)--(1.5, -5)--(-1.5, -5)--(-1.5, 1)--(0.1, 1);
        \draw[red, fill=white!, rounded corners] (0, 0.9)--(1.4, 0.9)--(1.4, -2)--(-1.4, -2)--(-1.4, 0.9)--(0.1, 0.9);
        \draw[rounded corners] (0, -2.1)--(0.8, -2.1)--(0.8, -4.9)--(-1.4, -4.9)--(-1.4, -2.1)--(0.1, -2.1);
        
        \node [label=$B_{i}$] (v) at (0, -6.3){};
        \node [label={\footnotesize $N_{B_i}(v_K'')$}] (v) at (-0.4, -0.2){};
        \node [label={\footnotesize $N_{B_i}(z)$}] (v) at (-0.6, -3.1){};

        \draw[draw=black!80, fill =pink!30] (0, -1) circle (0.5);
        \node [label = {\footnotesize $b$}] (v) at (0, -1.55){};

        \draw[draw=black!80, fill =blue!30] (-0.75, -3.35) circle (0.5);
        \node [label = {\footnotesize $q_i$}] (v) at (-0.75, -3.9){};

        \draw[draw=black!80, fill =red!30] (0.15, -4.25) circle (0.5);
        \node [label = {\footnotesize $q_i'$}] (v) at (0.15, -4.8){};
        
        \draw[rounded corners, thick] (4.5, 1)--(3, 1)--(3, -5)--(6, -5)--(6, 1)--(4.4, 1);
        \draw[rounded corners] (4.5, 0.9)--(3.1, 0.9)--(3.1, -2)--(5.9, -2)--(5.9, 0.9)--(4.4, 0.9);
        \draw[red, fill=white!, rounded corners] (4.5, -2.1)--(3.1, -2.1)--(3.1, -4.9)--(5.9, -4.9)--(5.9, -2.1)--(4.4, -2.1);

        \node [label=$B_{i+1}$] (v) at (4.5, -6.3){};
        \node [label={\footnotesize $N_{B_{i+1}}(J)$}] (v) at (4.15, -0.2){};

        \draw[draw=black!80, fill =green!30] (4.5, -1) circle (0.5);
        \node [label = {\footnotesize $q_{i+1}$}] (v) at (4.5, -1.55){};

        \draw[draw=black!80, fill =yellow!30] (4.5, -3.5) circle (0.5);
        \node [label = {\footnotesize $c$}] (v) at (4.5, -4){};

        \draw[rounded corners, thick] (9, 1)--(7.5, 1)--(7.5, -5)--(10.5, -5)--(10.5, 1)--(8.9, 1);
        \draw[red, fill=white!, rounded corners] (9, 0.9)--(7.6, 0.9)--(7.6, -2)--(10.4, -2)--(10.4, 0.9)--(8.9, 0.9);
        \draw[rounded corners] (9, -2.1)--(9.8, -2.1)--(9.8, -4.9)--(7.6, -4.9)--(7.6, -2.1)--(9.1, -2.1);

        \node [label=$B_{i+2}$] (v) at (9, -6.3){};
        \node [label={\footnotesize $N_{B_{i+2}}(J)$}] (v) at (8.65, -0.2){};
        \node [label={\footnotesize $N_{B_{i+2}}(w)$}] (v) at (8.65, -3.1){};

        \draw[draw=black!80, fill =pink!30] (8.3, -1) circle (0.5);
        \node [label = {\footnotesize $d$}] (v) at (8.3, -1.5){};

        \draw[draw=black!80, fill =yellow!30] (9.7, -1) circle (0.5);
        \node [label = {\footnotesize $d'$}] (v) at (9.7, -1.5){};

        \draw[draw=black!80, fill =blue!30] (8.7, -3.7) circle (0.5);
        \node [label = {\footnotesize $q_{i+2}$}] (v) at (8.7, -4.25){};

        \draw[rounded corners, thick] (-4.5, 4)--(-3, 4)--(-3, 2)--(-6, 2)--(-6, 4)--(-4.4, 4);
        
        \node [label=$K$] (v) at (-4.5, 4){};

        \draw[draw=black!80, fill =blue!30] (-5.2, 3) circle (0.5);
        \node [label = {\small $v_K$}] (v) at (-5.2, 2.45){};
        
        \draw[draw=black!80, fill =yellow!30] (-3.8, 3) circle (0.5);
        \node [label = {\small $v_K'$}] (v) at (-3.8, 2.4){};

        \draw[rounded corners, thick] (4.5, 4)--(3, 4)--(3, 2)--(6, 2)--(6, 4)--(4.4, 4);

        \node [label=$J$] (v) at (4.5, 4){};

    \end{tikzpicture}
        \caption{Illustration of $F$ and a 5-coloring of $F$. For convenience, all vertices of color 1 are shown in pink, all vertices of color 2 in green, all vertices of color 3 in blue, all vertices of color 4 in red and all vertices of color 5 in yellow. Note that $q_{i-1}',v_K'$ and $a'$ may not exist.}
        \label{fig:good subgraph1}
    \end{figure}

    Next, we show that every maximal clique of size $\omega-j$ with $j\in \{0,1,2,3\}$ of $G$ contains at least $4-j$ vertices from $F$. Let $M$ be a maximal clique of $G$. We consider the following cases depending on all possible locations of $M$.

    \medskip
    \noindent
    \textbf{(i) $M\subseteq K\cup N_H(K)$}

    We first assume that $|N_{B_{i-1}}(K)| \ge 2$. Then $q_{i-1}$, $q_{i-1}'$ and $v_K$ are universal in $K\cup N_H(K)$ by Lemmas \ref{lem:single vertex in A_3(i)} (2) and \ref{lem:edge in $A_3(i)$}. If $q_{i-2}$ has non-neighbor $x \in K$, then $x - v_K - q_{i-2} - q_{i+2} - w - q_i - J$ is an induced $P_7$. So, $q_{i-2}$ is complete to $K$. Note that $b$ is complete to $K$ because $b \in N_{B_i}(v_K'')$. Therefore, $M$ contains either $\{q_{i-2}, q_{i-1}, q_{i-1}', v_K\}$ or $\{q_{i-1}, q_{i-1}', v_K, b\}$.

    We now assume that $|N_{B_{i-1}}(K)| = 1$. Recall that $q_{i-1}$ and $v_K$ are universal in $K\cup N_H(K)$, and $q_{i-2}$ and $b$ is complete to $K$. If $v_K' \in M$, then we are done. So we may assume that $v_K' \notin M$. It follows that $|M \cap (K \cup N_{B_{i-1}}(K))| = 2$. As $|M \cap \{q_{i-2}, q_{i-1}, v_K, b\}| = 3$, we have $|M| = \omega$ and so either $|M \cap N_{B_{i-2}}(K)| = \omega - 2$ or $|M \cap N_{B_i}(K)| = \omega - 2$. If $|M \cap N_{B_{i-2}}(K)| = \omega - 2$, then $N_{B_{i-2}}(K)\cup N_{B_{i+2}}(J) \cup N_{B_{i+2}}(w)$ is a clique of size larger than $\omega$ by Lemma \ref{lem:B_i complete to two large cliques in B_{i-1}}. If $|M \cap N_{B_i}(K)| = \omega - 2$, then $N_{B_i}(K) \cup N_{B_i}(z) \cup B_{i+1}$ is a clique of size larger than $\omega$. Therefore, $v_K' \in M$.
    
    \medskip
    \noindent
    \textbf{(ii) $M\subseteq J\cup N_H(J)$}

    We first assume that $M \subseteq J \cup N_{B_i}(J) \cup N_{B_{i+1}}(J)$. By Lemma \ref{lem:B_i complete to two large cliques in B_{i-1}}, $q_i$, $q_i'$ and $b$ are universal in $J \cup N_{B_i}(J) \cup N_{B_{i+1}}(J)$. So, $q_i, q_i', q_{i+1}, b\in M$.

    We now assume that $M \subseteq J \cup N_{B_{i+1}}(J) \cup N_{B_{i+2}}(J)$. Note that $d$ is complete to $J$ because $d \in N_{B_{i+2}}(v_J'')$. So, $q_{i+1}, d \in M$. Hence, we may assume that $|M| \ge \omega - 1$ and it follows that either $|M \cap (J \cup N_{B_{i+1}}(J))| \ge \frac{\omega - 1}{2}$ or $|M \cap N_{B_{i+2}}(J)| \ge \frac{\omega - 1}{2}$. Then either $N_{B_i}(K) \cup N_{B_i}(z) \cup J \cup N_{B_{i+1}}(J)$ or $N_{B_{i+2}}(J) \cup N_{B_{i+2}}(w) \cup B_{i-2}$ is a clique of size larger than $\omega$ by Lemma \ref{lem:B_i complete to two large cliques in B_{i-1}}, a contraction. Therefore, $|M| \le \omega - 2$ and the claim holds.
    
    \medskip
    \noindent
    \textbf{(iii) $M\subseteq B_{i-2} \cup B_{i-1}$}

    Note that $q_{i-2}$ and $q_{i-1}$ are universal in $B_{i-2} \cup B_{i-1}$. Suppose first that $a \in M$. If $|N_{B_{i-1}}(K)| \ge 2$, then $q_{i-1}' \in M$ because $N_{B_{i-2}}(a) \subseteq N_{B_{i-2}}(q_{i-1}')$. So we may assume that $|N_{B_{i-1}}(K)| = 1$. This implies that $|M| = \omega$ and $a' \notin M$. Then $|M \cap B_{i-2}| = \omega - 2$ and it follows that $N_{B_{i+2}}(J) \cup N_{B_{i+2}}(w) \cup B_{i-2}$ is a clique of size larger than $\omega$, a contradiction. So $a' \in M$ and we are done. Hence, we may assume that $a \notin M$. As $q_{i-2}, q_{i-1} \in M$, we have $|M| \ge \omega - 1$. Then either $|M \cap B_{i-2}| \ge \frac{\omega - 1}{2}$ or $|M \cap N_{B_{i-1}}(K)| \ge \frac{\omega - 1}{2}$. If $|M \cap B_{i-2}| \ge \frac{\omega - 1}{2}$, then $N_{B_{i+2}}(J) \cup N_{B_{i+2}}(w) \cup B_{i-2}$ is a clique of size larger than $\omega$. If $|M \cap N_{B_{i-1}}(K)| \ge \frac{\omega - 1}{2}$, then $N_{B_{i-1}}(K) \cup N_{B_i}(K) \cup N_{B_i}(z)$ is a clique of size larger than $\omega$. Therefore, $|M| \le \omega - 2$ and the claim holds.

    \medskip
    \noindent
    \textbf{(iv) $M\subseteq B_{i-1} \cup B_i$}

    Note that $q_{i-1}$, $q_i$ and $q_i'$ are universal in $B_{i-1} \cup B_i$. So, $q_{i-1}, q_i, q_i' \in M$. Since $b$ is universal in $N_{B_{i-1}}(K) \cup B_i$, we have either $a \in M$ or $b \in M$. In both case, $M$ contains at least $4$ vertices from $F$. 
    
    \medskip
    \noindent
    \textbf{(v) $M\subseteq B_i \cup B_{i+1}$}

    By Lemma \ref{lem:B_i complete to two large cliques in B_{i-1}}, $q_i$, $q_i'$ and $b$ are universal in $B_i\cup B_{i+1}$. So $q_i, q_i', q_{i+1}, b\in M$. 

    \medskip
    \noindent
    \textbf{(vi) $M\subseteq B_{i+1} \cup B_{i+2}$}

    Note that $q_{i+1}$ and $q_{i+2}$ are universal in $B_{i+1} \cup B_{i+2}$. So, $q_{i+1}, q_{i+2} \in M$. Since $d$ and $d'$ are universal in $N_{B_{i+1}}(J) \cup B_{i+2}$, we have either $c \in M$ or $d, d' \in M$. If $d, d' \in M$, then we are done. So we may assume that $c \in M$. As $q_{i+1}, q_{i+2}, c \in M$, we have $|M| = \omega$. It follows that either $|M \cap B_{i+1}| \ge \frac{\omega}{2}$ or $|M \cap N_{B_{i+2}}(w)| \ge \frac{\omega}{2}$. Then either $N_{B_i}(K) \cup N_{B_i}(z) \cup B_{i+1}$ or $N_{B_{i+2}}(J) \cup N_{B_{i+2}}(w) \cup B_{i-2}$ is a clique of size larger than $\omega$ by Lemma \ref{lem:B_i complete to two large cliques in B_{i-1}}, a contradiction. Therefore, $|M| \le \omega - 1$ and the claim holds.
    
    \medskip
    \noindent
    \textbf{(vii) $M\subseteq B_{i-2} \cup B_{i+2}$}

    By Lemma \ref{lem:B_i complete to two large cliques in B_{i-1}}, $d$ and $d'$ are universal in $B_{i+2}\cup B_{i-2}$. So $q_{i+2},d,d',q_{i-2}\in M$. 

    \medskip
    Therefore, $F$ is a $(4, 5)$-good subgraph. 
\end{proof}

\subsection{Distance 1}

In this subsection, we handle the second outcome of Lemma \ref{lem:structure of only A3}.

\begin{lemma}\label{lem:two non-dis A3}
    Let $G = H \cup K \cup J$ where $K$ is a non-empty component of $A_3(i-2)$ and $J$ is a non-empty component of $A_3(i+2)$. Then $\chi(G) \le \ceil{\frac{5}{4}\omega}$.
\end{lemma}

\begin{proof}
    By Lemma \ref{lem:structure of only A3}, both $K$ and $J$ are cliques. Moreover, $B_{i-2}\setminus N(K)$ is anticomplete to $N_{B_{i-1}}(K)$ and $B_{i+2}\setminus N(J)$ is anticomplete to $N_{B_{i+1}}(J)$. Let $z_K \in B_{i-2} \setminus N(K)$ be a minimal vertex in $B_{i-2} \setminus N(K)$ and let $z_J \in B_{i+2} \setminus N(J)$ be a minimal vertex in $B_{i+2} \setminus N(J)$. Then $N_{B_{i-1}}(K)$ and $N_{B_{i-1}}(z_K)$ are disjoint and $N_{B_{i+1}}(J)$ and $N_{B_{i+1}}(z_J)$ are disjoint. .

    By Lemmas \ref{lem:smv on w} and \ref{lem:small vertex argument for simplicial vertices in A_3(i-2)},$|N_{B_{i-1}}(K)|$, $|N_{B_{i-1}}(z_K)|$, $ |N_{B_{i+1}}(J)|$, $|N_{B_{i+1}}(z_J)| > \ceil{\frac{\omega}{4}}$. By Lemmas \ref{lem:compare nbd of different A_3(i) in B_j} and \ref{lem:small vertex argument for simplicial vertices in A_3(i-2)}, we have $|N_{B_{i-2}}(K)| \ge |N_{B_{i-2}}(J)| > \ceil{\frac{\omega}{4}}$ and $|N_{B_{i+2}}(J)| \ge |N_{B_{i+2}}(K)| > \ceil{\frac{\omega}{4}}$. If there is an edge $uv$ with $u \in B_{i-2}\setminus N(K)$ and $v \in B_{i+2}\setminus N(J)$, then $K - N_{B_{i-1}}(K) - N_{B_{i-1}}(z_K) - u - v - N_{B_{i+1}}(z_J) - N_{B_{i+1}}(J)$ is an induced $P_7$. So $B_{i-2}\setminus N(K)$ is anticomplete to $B_{i+2}\setminus N(J)$.

    Let $u_K \in K$ be the vertex that has minimum neighborhood in $H$. Let $p \in N_{B_{i+2}}(u_K)$. We show that $p$ is complete to $B_{i-2} \cup B_{i+1}$. Suppose not. By Lemma \ref{lem:single vertex in A_3(i)} (4), there is a vertex $x$ in $N_{B_{i+2}}(K)$ that is complete to $B_{i-2}\cup B_{i+1}$. By assumption $x \notin N_{B_{i+2}}(u_K)$. If $p$ has non-neighbor $y \in B_{i+1}$, then $u_K - p - x - y$ is a $P_4$-structure. So, $p$ is complete to $B_{i+1}$. Hence, we may assume that $p$ has a non-neighbor $y' \in B_{i-2}$. Since $p$ is complete to $N_{B_{i-2}}(K)$, $y' \in B_{i-2}\setminus N(K)$. Then $u_K - p - x - y' - N_{B_{i-1}}(y') - N_{B_{i-1}}(u_K) - u_K$ is an induced $C_6$. So, $p$ is complete to $B_{i-2}$.
    Therefore, $p$ is complete to $K \cup B_{i-2} \cup B_{i+1}$. We set this vertex to be $q_{i+2}$. By symmetry, there is a vertex in $N_{B_{i-2}}(J)$ complete to $J \cup B_{i-1}\cup B_{i+2}$. We set this vertex to be $q_{i-2}$. 
    
    For $X \in \{K, J\}$, let $v_X, v_X' \in X$ be the vertices with the largest and second-largest neighborhoods in $H$, respectively. Let $w_K, w_K' \in B_{i-2}\setminus N(K)$ be the first and second maximal in $B_{i-2}\setminus N(K)$, respectively. Let $w_J, w_J' \in B_{i+2}\setminus N(J)$ be the first and second maximal in $B_{i+2}\setminus N(J)$, respectively. Note that $v_K', v_J', w_K'$ and $w_J'$ may not exist. Let $a, b \in N_{B_{i-1}}(v_K'')$ where $v_K'' \in K$ that has a minimum neighborhood in $H$ and let $c, d \in N_{B_{i+1}}(v_J'')$ where $v_J'' \in J$ that has a minimum neighborhood in $H$.
    
    We set $F = (F_1, F_2, F_3, F_4, F_5, O)$ as follows:
    \begin{itemize}
        \item $F_1 = \{q_{i-2}, w_K, w_K'\}$,
        \item $F_2 = \{q_{i-1}, a, b\}$,
        \item $F_3 = \{q_i\}$,
        \item $F_4 = \{q_{i+1}, c, d\}$,
        \item $F_5 = \{q_{i+2}, w_J, w_J'\}$,
        \item $O = \{v_K, v_K', v_J, v_J'\}$.
    \end{itemize}

     We can assign color $1$ to $q_{i-1}$, $q_{i+1}$, $v_K$ and $v_J$, color $2$ to $q_i$, $v_K'$, $v_J'$, $w_K$ and $w_J$, color $3$ to $q_{i-2}$ and $c$, color $4$ to $q_{i+2}$ and $a$, and color $5$ to $w_K'$, $w_J'$, $b$ and $d$. So $F$ is 5-colorable (see Figure \ref{fig:good subgraph2}).

    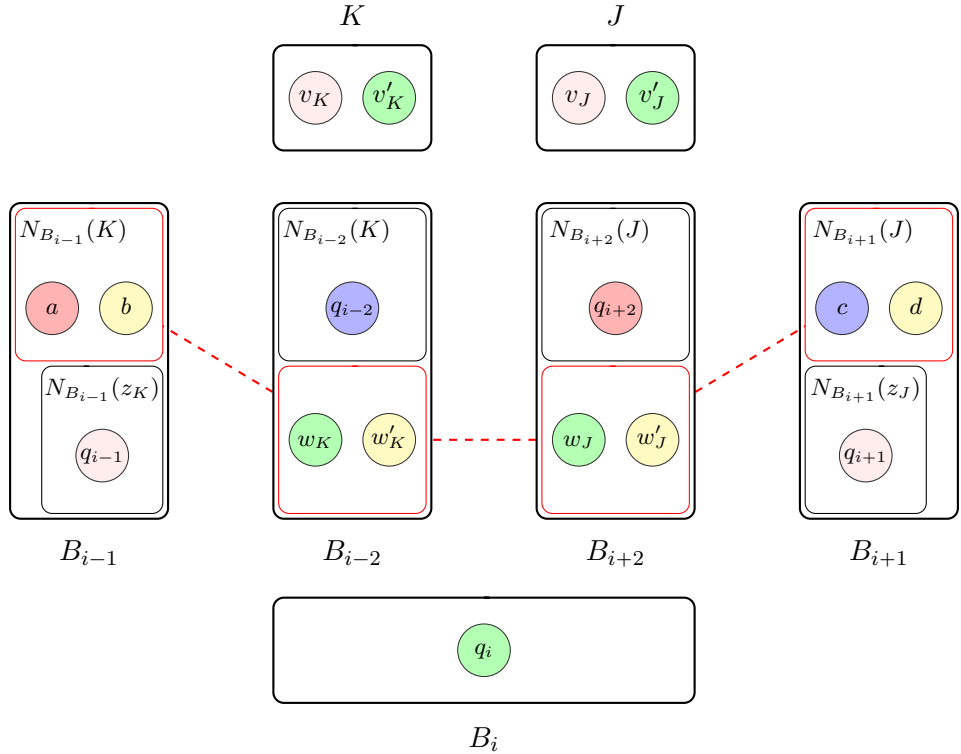
\begin{figure}[h]
        \centering
        \begin{tikzpicture}[scale=0.7]
        \tikzstyle{v}=[circle, draw, solid, fill=black, inner sep=0pt, minimum width=3pt]

        \draw[red, dashed, thick] (-7.5, -0.5)--(-2.5, -3.5);
        \draw[red, dashed, thick] (7.5, -0.5)--(2.5, -3.5);
        \draw[red, dashed, thick] (2.5, -3.5)--(-2.5, -3.5);
        
        \draw[rounded corners, thick] (-2.5, 1)--(-1, 1)--(-1, -5)--(-4, -5)--(-4, 1)--(-2.4, 1);
        \draw[rounded corners] (-2.5, 0.9)--(-1.1, 0.9)--(-1.1, -2)--(-3.9, -2)--(-3.9, 0.9)--(-2.4, 0.9);
        \draw[red, fill=white!, rounded corners] (-2.5, -2.1)--(-1.1, -2.1)--(-1.1, -4.9)--(-3.9, -4.9)--(-3.9, -2.1)--(-2.4, -2.1);

        \node [label=$B_{i-2}$] (v) at (-2.5, -6.3){};
        \node [label={\footnotesize $N_{B_{i-2}}(K)$}] (v) at (-2.8, -0.2){};
        \draw[draw=black!80, fill =blue!30] (-2.5, -1) circle (0.5);
        \node [label = {\footnotesize $q_{i-2}$}] (v) at (-2.5, -1.55){};
        \draw[draw=black!80, fill =green!30] (-3.2, -3.5) circle (0.5);
        \node [label = {\footnotesize $w_K$}] (v) at (-3.2, -4.05){};
        \draw[draw=black!80, fill =yellow!30] (-1.8, -3.5) circle (0.5);
        \node [label = {\footnotesize $w_K'$}] (v) at (-1.8, -4.05){};

        \draw[rounded corners, thick] (-7.5, 1)--(-6, 1)--(-6, -5)--(-9, -5)--(-9, 1)--(-7.4, 1);
        \draw[red, fill=white!, rounded corners] (-7.5, 0.9)--(-6.1, 0.9)--(-6.1, -2)--(-8.9, -2)--(-8.9, 0.9)--(-7.4, 0.9);
        \draw[rounded corners] (-7.5, -2.1)--(-8.4, -2.1)--(-8.4, -4.9)--(-6.1, -4.9)--(-6.1, -2.1)--(-7.6, -2.1);

        \node [label=$B_{i-1}$] (v) at (-7.5, -6.3){};
        \node [label={\footnotesize $N_{B_{i-1}}(K)$}] (v) at (-7.8, -0.2){};
        \node [label={\footnotesize $N_{B_{i-1}}(z_K)$}] (v) at (-7.25, -3.2){};
        \draw[draw=black!80, fill =red!30] (-8.2, -1) circle (0.5);
        \node [label = {\footnotesize $a$}] (v) at (-8.2, -1.5){};
        \draw[draw=black!80, fill =yellow!30] (-6.8, -1) circle (0.5);
        \node [label = {\footnotesize $b$}] (v) at (-6.8, -1.5){};
        \draw[draw=black!80, fill =pink!30] (-7.25, -3.8) circle (0.5);
        \node [label = {\footnotesize $q_{i-1}$}] (v) at (-7.25, -4.35){};

        \draw[rounded corners, thick] (0, -6.5)--(4, -6.5)--(4, -8.5)--(-4, -8.5)--(-4, -6.5)--(0.1, -6.5);

        \node [label=$B_{i}$] (v) at (0, -9.8){};
        \draw[draw=black!80, fill =green!30] (0, -7.5) circle (0.5);
        \node [label = {\footnotesize $q_i$}] (v) at (0, -8.05){};
        
        \draw[rounded corners, thick] (7.5, 1)--(6, 1)--(6, -5)--(9, -5)--(9, 1)--(7.4, 1);
        \draw[red, fill=white!, rounded corners] (7.5, 0.9)--(6.1, 0.9)--(6.1, -2)--(8.9, -2)--(8.9, 0.9)--(7.4, 0.9);
        \draw[rounded corners] (7.5, -2.1)--(8.4, -2.1)--(8.4, -4.9)--(6.1, -4.9)--(6.1, -2.1)--(7.6, -2.1);

        \node [label=$B_{i+1}$] (v) at (7.5, -6.3){};
        \node [label={\footnotesize $N_{B_{i+1}}(J)$}] (v) at (7.2, -0.2){};
        \node [label={\footnotesize $N_{B_{i+1}}(z_J)$}] (v) at (7.25, -3.2){};
        \draw[draw=black!80, fill =blue!30] (6.8, -1) circle (0.5);
        \node [label = {\footnotesize $c$}] (v) at (6.8, -1.5){};
        \draw[draw=black!80, fill =yellow!30] (8.2, -1) circle (0.5);
        \node [label = {\footnotesize $d$}] (v) at (8.2, -1.5){};
        \draw[draw=black!80, fill =pink!30] (7.25, -3.8) circle (0.5);
        \node [label = {\footnotesize $q_{i+1}$}] (v) at (7.25, -4.35){};

        \draw[rounded corners, thick] (2.5, 1)--(1, 1)--(1, -5)--(4, -5)--(4, 1)--(2.4, 1);
        \draw[rounded corners] (2.5, 0.9)--(1.1, 0.9)--(1.1, -2)--(3.9, -2)--(3.9, 0.9)--(2.4, 0.9);
        \draw[red, fill=white!, rounded corners] (2.5, -2.1)--(1.1, -2.1)--(1.1, -4.9)--(3.9, -4.9)--(3.9, -2.1)--(2.4, -2.1);

        \node [label=$B_{i+2}$] (v) at (2.5, -6.3){};
        \node [label={\footnotesize $N_{B_{i+2}}(J)$}] (v) at (2.2, -0.2){};
        \draw[draw=black!80, fill =red!30] (2.5, -1) circle (0.5);
        \node [label = {\footnotesize $q_{i+2}$}] (v) at (2.5, -1.55){};
        \draw[draw=black!80, fill =yellow!30] (3.2, -3.5) circle (0.5);
        \node [label = {\footnotesize $w_J'$}] (v) at (3.2, -4.05){};
        \draw[draw=black!80, fill =green!30] (1.8, -3.5) circle (0.5);
        \node [label = {\footnotesize $w_J$}] (v) at (1.8, -4.05){};

        \draw[rounded corners, thick] (-2.5, 4)--(-1, 4)--(-1, 2)--(-4, 2)--(-4, 4)--(-2.4, 4);
        
        \node [label=$K$] (v) at (-2.5, 4){};

        \draw[draw=black!80, fill =pink!30] (-3.2, 3) circle (0.5);
        \node [label = {\small $v_K$}] (v) at (-3.2, 2.45){};
        
        \draw[draw=black!80, fill =green!30] (-1.8, 3) circle (0.5);
        \node [label = {\small $v_K'$}] (v) at (-1.8, 2.4){};

        \draw[rounded corners, thick] (2.5, 4)--(1, 4)--(1, 2)--(4, 2)--(4, 4)--(2.4, 4);
        
        \node [label=$J$] (v) at (2.5, 4){};
        
        \draw[draw=black!80, fill =pink!30] (1.8, 3) circle (0.5);
        \node [label = {\small $v_J$}] (v) at (1.8, 2.45){};
        
        \draw[draw=black!80, fill =green!30] (3.2, 3) circle (0.5);
        \node [label = {\small $v_J'$}] (v) at (3.2, 2.4){};
        
    \end{tikzpicture}
        \caption{Illustration of $F$ and a 5-coloring of $F$. For convenience, all vertices of color 1 are shown in pink, all vertices of color 2 in green, all vertices of color 3 in blue, all vertices of color 4 in red and all vertices of color 5 in yellow. Note that $v_K'$, $v_J'$, $w_K'$ and $w_J'$ may not exist.}
        \label{fig:good subgraph2}
    \end{figure}

     Next, we show that every maximal clique of size $\omega-j$ with $j\in \{0,1,2,3\}$ of $G$ contains at least $4-j$ vertices from $F$. Let $M$ be a maximal clique of $G$. We consider the following cases depending on all possible locations of $M$. Since $K$ and $J$ play symmetric roles, it suffices to consider $K\cup N_H(K)$, $B_{i-2}\cup B_{i+2}$, $B_{i-2}\cup B_{i-1}$ and $B_{i-1}\cup B_i$.

    \medskip
    \noindent
    \textbf{(i) $M\subseteq K\cup N_H(K)$ (or $J\cup N_H(J)$)}

    Note that $q_{i-2}$ and $v_K$ are universal in $K\cup N_H(K)$. We first assume that $M \subseteq K \cup N_{B_{i-2}}(K) \cup N_{B_{i-1}}(K)$. Since $a, b \in N_{B_{i-1}}(v_K'')$, $a$ and $b$ are complete to $K$. By Lemma \ref{lem:single vertex in A_3(i)}, $a$ and $b$ are complete to $N_{B_{i-2}}(K)$. Hence, $q_{i-2}, v_K, a, b \in M$. 

    We now assume that $M \subseteq K \cup N_{B_{i-2}}(K) \cup N_{B_{i+2}}(K)$. Note that $q_{i+2}$ is universal in $K \cup N_{B_{i-2}}(K) \cup N_{B_{i+2}}(K)$. So, $q_{i-2}, q_{i+2}, v_K \in M$. Suppose that $v_K' \notin M$. This implies that $|M| = \omega$ and we have either $|M \cap N_{B_{i-2}}(K)| \ge \frac{\omega - 1}{2}$ or $|M \cap N_{B_{i+2}}(K)| \ge \frac{\omega - 1}{2}$. It follows that either $N_{B_{i-2}}(K) \cup N_{B_{i-1}}(K) \cup N_{B_{i-1}}(z_K)$ or $N_{B_{i+2}}(K) \cup N_{B_{i+1}}(J) \cup N_{B_{i+1}}(z_J)$ is a clique of size larger than $\omega$, a contradiction.

    \medskip
    \noindent
    \textbf{(ii) $M\subseteq B_{i-2} \cup B_{i-1}$ (or $B_{i+1} \cup B_{i+2}$)}

    Note that $q_{i-2}$ and $q_{i-1}$ are universal in $B_{i-2} \cup B_{i-1}$. So, $q_{i-2}, q_{i-1} \in M$. Since $a$ and $b$ are universal in $N_{B_{i-2}}(K) \cup B_{i-1}$, we have either $a, b \in M$ or $w_K \in M$. If $a, b \in M$, then we are done. So, we may assume that $w_K \in M$. Suppose that $w_K' \notin M$. This implies that $|M| = \omega$. Then $(M \setminus \{w_K\}) \cup \{a, b\}$ is a clique of size larger than $\omega$, a contradiction.

    \medskip
    \noindent
    \textbf{(iii) $M\subseteq B_{i-1} \cup B_i$ (or $B_i \cup B_{i+1}$)}

    Note that $q_{i-1}$ and $q_i$ are universal in $B_{i-1} \cup B_i$. By Lemma \ref{lem:B_i complete to two large cliques in B_{i-1}}, $a$ and $b$ are complete to $B_i$ and so $a$ and $b$ are universal in $B_{i-1} \cup B_i$. Therefore, $q_{i-1}, q_i, a, b \in M$.

    \medskip
    \noindent
    \textbf{(vii) $M\subseteq B_{i-2} \cup B_{i+2}$}

    Note that $q_{i-2}$ and $q_{i+2}$ are universal in $B_{i-2} \cup B_{i+2}$. So, $q_{i-2}, q_{i+2} \in M$. Since $B_{i-2}\setminus N(K)$ is anticomplete to $B_{i+2} \setminus N(J)$, either $M \subseteq N_{B_{i-2}}(K) \cup B_{i+2}$ or $M \subseteq N_{B_{i+2}}(J) \cup B_{i-2}$. By symmetric, we may assume that $M \subseteq N_{B_{i-2}}(K) \cup B_{i+2}$. If $w_J, w_J' \in M$, then we are done. Let $j = |M \cap \{w_J, w_J'\}$. Then, we have $|M| = \omega - 1 + j$ and it follows that either $|M \cap N_{B_{i-2}}(K)| \ge \frac{\omega - 1}{2}$ or $|M \cap N_{B_{i+2}}(J)| \ge \frac{\omega - 1}{2}$. Then either $N_{B_{i-2}}(K) \cup N_{B_{i-1}}(K) \cup N_{B_{i-1}}(z_K)$ or $N_{B_{i+2}}(J) \cup N_{B_{i+1}}(J) \cup N_{B_{i+1}}(z_J)$ is a clique of size larger than $\omega$, a contradiction.
    
    \medskip
    Therefore, $F$ is a $(4, 5)$-good subgraph. 
\end{proof}

\subsection{Distance 0}

In this subsection, we handle the third outcome of Lemma \ref{lem:structure of only A3}. We assume that $A_3=A_3(i)\neq \emptyset$ has at most two components.
We first present several simple lemmas.

\begin{lemma}\label{lem:smv on B_{i+2}+B_{i-2}}
    There exists a clique of size larger than $\ceil{\frac{\omega}{4}}$ in $B_{i-2}$ complete to $B_{i+2}$. 
    By symmetry, there exists a clique of size larger than $\ceil{\frac{\omega}{4}}$ in $B_{i+2}$ complete to $B_{i-2}$. 
\end{lemma}

\begin{proof}
    Let $v$ be a vertex in $B_{i+2}$ with minimal neighborhood in $H$. 
    By the minimality of $v$, $N(v)\setminus B_{i-2}$ is a clique and so $N_{B_{i-2}}(v)$ is a clique of size larger than $\ceil{\frac{\omega}{4}}$ complete to $B_{i+2}$. 
\end{proof}

\begin{lemma}\label{lem:B_{i-2} is complete to B_{i+2}}
    If $B_{i-1}$ is complete to $B_{i-2}$ and $B_{i+1}$ is complete to $B_{i+2}$, then $B_{i-2}$ and $B_{i+2}$ are complete.
\end{lemma}

\begin{proof}
    Let $z \in B_{i-2}$ be a minimal vertex in $B_{i-2}$ and $z' \in B_{i+2}\setminus N(z)$ be a maximal vertex in $B_{i+2} \setminus N(z)$. Let $M\subseteq B_{i-2}\cup B_{i+2}$ be a maximum clique of $G$.
    We show that $M$ contains one of $z$ and $z'$. 
    Suppose for a contradiction that $z, z' \notin M$. If $M \cap B_{i+2} \subseteq N_{B_{i+2}}(z)$, then $M \cup \{z\}$ is also a clique, a contradiction. So $(M \cap B_{i+2}) \setminus N(z) \neq \emptyset$ and let $a \in (M \cap B_{i+2}) \setminus N(z)$. If $M \cap B_{i-2} \subseteq N_{B_{i-2}}(z')$, then $M \cup \{z'\}$ is also a clique, a contradiction. So $(M \cap B_{i-2}) \setminus N(z') \neq \emptyset$ and let $b \in (M \cap B_{i-2}) \setminus N(z')$. It follows that $ab \in E(G)$ and $z'b \notin E(G)$, which contradicts the choice of $z'$. Then $\{z, z', q_i\}$ is a good stable set of $G$.
\end{proof}

\begin{lemma}\label{lem: Y is empty}
    If $|N_{B_i}(A_3(i))| = 1$, then every vertex in $(B_{i-1} \cup B_{i+1})\setminus N(A_3(i))$ has a neighbor in $B_i \setminus N(A_3(i))$. 
\end{lemma}

\begin{proof}
    Let $K$ be a non-empty component of $A_3(i)$ and $J$ be the possible second component of $A_3(i)$.
    By Lemma \ref{lem:blowup-threeneighbor}, we may assume that $B_i \setminus N(K)$ is anticomplete to $N_{B_{i+1}}(K)$. 
    Suppose that this lemma is not true. 
    We first assume that there exists $x \in B_{i+1}\setminus N(A_3(i))$ anticomplete to $B_i \setminus \{q_{i}\}$. 
    Then we may assume that $x$ is minimal in $B_{i+1}\setminus N(A_3(i))$. 
    By Lemma \ref{lem:B_i complete to two large cliques in B_{i-1}}, $N_{B_{i+1}}(K)$ is complete to $B_{i+2}$. 
    If $B_i \setminus N(J)$ is anticomplete to $N_{B_{i+1}}(J)$, then $N_{B_{i+1}}(J)$ is complete to $B_{i+2}$. 
    If $B_i \setminus N(J)$ is anticomplete to $N_{B_{i-1}}(J)$, then by Lemma \ref{lem:single vertex in A_3(i)} (4), for any vertex $u\in N_{B_{i+1}}(J)$ and any vertex $v\in B_{i+1}\setminus N(J)$, $N(v)\cap H\subseteq N(u)\cap H$. 
    So $x$ is also minimal in $B_{i+1}$. Hence, $B_{i+1} \cup N_{B_{i+2}}(x)$ is a clique and so $d(x) \le \omega$. 
    We now assume that there exists $x \in B_{i-1}\setminus N(A_3(i))$ anticomplete to $B_i \setminus N(A_3(i))$. Let $y \in  B_{i-1} \setminus N(A_3(i))$ be a minimal vertex in $B_{i-1} \setminus N(A_3(i))$. It suffices to show that $y$ is minimal in $B_{i-1}$. By assumption, $y$ is anticomplete to $B_i \setminus N(A_3(i))$. Suppose that there exist $a \in B_{i-2}$ and $b \in N_{B_{i-1}}(A_3(i))$ such that $ab \notin E(G)$ and $ay \in E(G)$. Let $c\in A_3(i)\cap N(b)$. Then $a - y - b - c$ contradicts Lemma \ref{lem:generalized P4-structure}. This completes the proof.
\end{proof}

\begin{lemma}\label{lem:reduce to four complete pair}
    Let $J_1$ be a non-empty component of $A_3(i)$ and $J_2$ be the possible second component of $A_3(i)$.
    Suppose that $B_i \setminus N(J_s)$ is anticomplete to $B_{i+1} \cap N(J_s)$ and $B_{i-1} \setminus N(J_s) \neq \emptyset$ for $s=1,2$. 
    If every vertex in $(B_{i-1} \setminus N(A_3(i))) \cup (B_{i+1} \setminus N(A_3(i)))$ has a neighbor in $B_i \setminus N(A_3(i))$, then $B_j$ is complete to $B_{j+1}$ for $j\neq i$.
\end{lemma}

\begin{proof}
    Note that $q_{i+1} \in B_{i+1} \setminus N(A_3(i))$. By assumption and Lemma 
    \ref{lem:B_i complete to two large cliques in B_{i-1}}, $B_{i+1}$ is complete to $B_{i+2}$. 

    Next we show that $B_{i-1}$ is complete to $B_{i-2} \cup B_i$. 
    Suppose first that $B_{i-1} \setminus N(A_3(i)) \neq \emptyset$. 
    By assumption and Lemma \ref{lem:component of A_3(i)}, $B_{i-1}\setminus N(K)$ is complete to $B_{i-2} \cup B_i$. 
    By Lemma \ref{lem:single vertex in A_3(i)} (4), $N_{B_{i-1}}(K)$ is complete to $B_{i-2}\cup B_i$. Hence, $B_{i-1}$ is complete to $B_{i-2}\cup B_i$. 
    Now suppose that $B_{i-1} \setminus N(A_3(i)) = \emptyset$. 
    It follows that $J_2\neq \emptyset$, $J_1$ and $J_2$ have disjoint neighborhoods in $B_{i-1}$ and $B_{i-1}=N_{B_{i-1}}(J_1)\cup N_{B_{i-1}}(J_2)$. 
    By assumption and Lemma \ref{lem:component of A_3(i)}, $B_{i-1}$ is complete to $B_{i-2} \cup B_i$. By Lemma \ref{lem:B_{i-2} is complete to B_{i+2}}, $B_{i-2}$ and $B_{i+2}$ are complete. 
\end{proof}

We first deal with the following special case.

\begin{lemma}[Coloring $H\cup A_3(i)$ with four complete pairs]\label{lem:color four complete pairs}
    Let $K$ and $J$ be two components of $A_3(i)$ \textup{(}It is also possible that $J = \emptyset$\textup{)} and $G = H \cup K \cup J$. If $B_j$ and $B_{j+1}$ complete for $j\neq i$, then $\chi(G) \le \ceil{\frac{5}{4}\omega}$.
\end{lemma}

\begin{proof}
     We assume that $\omega(J) \le \omega(K)$.
     By assumption and Lemma \ref{lem:blowup-threeneighbor}, we may assume that $B_i \setminus N(K)$ is anticomplete to $N_{B_{i+1}}(K)$, and $B_i \setminus N(J)$ is anticomplete to $N_{B_{i+1}}(J)$. By Lemmas \ref{lem:small vertex argument for simplicial vertices in A_3(i-2)} and \ref{lem:smv on B_{i+2}+B_{i-2}}, we have $|B_j| > \ceil{\frac{\omega}{4}}$ for $j \in \{i-2, i-1, i+2\}$. By Lemma \ref{lem:single vertex covers A_3(i)}, there exist vertices $w_K\in K$ universal in $K$ and $w_j\in J$ universal in $J$ such that $N_H(K)=N_H(w_K)$ and $N_H(J)=N_H(w_J)$.  
     It follows from Lemma \ref{lem:small vertex argument for simplicial vertices in A_3(i-2)} that $|N_{B_j}(K)| > \ceil{\frac{\omega}{4}}$ and $|N_{B_j}(J)| > \ceil{\frac{\omega}{4}}$ for $j \in \{i-1, i+1\}$.

     Let $z \in B_i \setminus (N(K) \cup N(J))$ be a minimal vertex in $B_i \setminus (N(K) \cup N(J))$. Then, $N_{B_{i+1}}(z) \subseteq B_{i+1}\setminus (N(K) \cup N(J))$. Since every vertex in $N_{B_i}(K) \cup N_{B_i}(J)$ has a neighbor in $N_{B_{i+1}}(K) \cup N_{B_{i+1}}(J)$, we obtain $z$ is also minimal in $B_i$. By applying a small vertex argument, $|B_{i+1} \setminus (N(K) \cup N(J))| > \ceil{\frac{\omega}{4}}$. If $N_{B_{i+1}}(K)$ and $N_{B_{i+1}}(J)$ are disjoint, then $B_{i+1} \cup B_{i+2}$ is a clique of size larger than $\omega$. So $N_{B_{i+1}}(K)$ and $N_{B_{i+1}}(J)$ are comparable.

     We now determine a maximum clique in $G$. By Lemma \ref{lem:single vertex in A_3(i)} (1), $q_i \in N_{B_i}(K) \cap N_{B_i}(J)$ and $q_{i+1} \in B_{i+1} \setminus (N(K) \cup N(J))$. Let $p \in B_i\setminus (N(K) \cup N(J))$ be a maximal vertex in $B_i\setminus (N(K) \cup N(J))$ and let $r \in N_{B_{i+1}}(K) \cap N_{B_{i+1}}(J)$. Then,
     \begin{itemize}
        \item If $B_{i-2} \cup B_{i-1}$ is not a maximum clique, then $\{q_i, q_{i+2}\}$ is a good stable set.
        \item If $B_{i-1} \cup B_i$ is not a maximum clique, then $\{q_{i-2}, q_{i+1}, w_K, w_J\}$ is a good stable set.
        \item If $B_{i+1} \cup B_{i+2}$ is not a maximum clique, then $\{q_{i-2}, q_i\}$ is a good stable set.
    \end{itemize}
    So $B_{i-2} \cup B_{i-1}$, $B_{i-1} \cup B_i$ and $B_{i+1} \cup B_{i+2}$ are maximum clique in $G$.

    We now develop the coloring strategy. By Lemma \ref{lem:P4-freenessofA3}, $K$ and $J$ are $P_4$-free and so we have $\chi(K) = \omega(K)$ and $\chi(J) = \omega(J)$. Let $\omega(K) = \alpha \omega$ and $|B_{i-1}| = \beta \omega$. By Lemma \ref{lem:maximalityofA3}, $|B_i \setminus N(K)| \ge \alpha\omega$, and so we can assume that $|B_i \setminus N(K)| = (\alpha + \gamma)\omega$ for some $\gamma \ge 0$. Observe that $\gamma \omega \in \mathbb{N}$ because $\gamma\omega = |B_i\setminus N(K)| - \omega(K)$. Since $B_{i-2} \cup B_{i-1}$ and $B_{i-1} \cup B_i$ are maximum clique, this implies that $|B_{i-2}| = (1 - \beta)\omega$ and $|N_{B_i}(K)| = (1 - \alpha - \beta - \gamma)\omega$. As $B_{i-2} \cup B_{i+2}$ may not be a maximum clique, let $|B_{i+2}| = (\beta - \delta)\omega$ with $\delta \ge 0$. Since  $B_{i+1} \cup B_{i+2}$ is a maximum clique, $|B_{i+1}| = (1 - \beta + \delta)\omega$ . If $\omega(K) \ge \frac{3}{4}\omega$, then $|B_i \setminus N(K)| \ge \frac{3}{4}\omega$ and this implies that $B_{i-1} \cup B_i$ is a clique of size larger than $\omega$. So we may assume that $\omega(K) < \frac{3}{4}\omega$.
    
    \vskip 0.3cm
    \noindent
    {\bf (i) $\frac{1}{4}\omega < \omega(K) < \frac{3}{4}\omega$}
    \vskip 0.2cm
    
    We now partition $K$ into three parts $(X, Y, Z)$ of color sizes $x\omega$, $y\omega$, and $\ceil{\frac{\omega}{4}}$ where $x\omega, y\omega \le \ceil{\frac{\omega}{4}}$. Since $x\omega + y\omega = \alpha \omega - \ceil{\frac{\omega}{4}} < \frac{1}{2}\omega$, this implies that we can choose $x\omega$ and $y\omega$ that satisfies $x\omega, y\omega \le \ceil{\frac{\omega}{4}}$. We then color some vertices in $G$ using $x\omega + y\omega + 2\cdot\ceil{\frac{\omega}{4}}$ colors as follows.

    \begin{itemize}
        \item Assign the colors $a_1, \ldots, a_{x\omega}$ to the vertices of $X$, and $x\omega$ vertices in $B_{i-2}$.
        \item Assign the colors $b_1, \ldots, b_{y\omega}$ to the vertices of $Y$, and $y\omega$ vertices in $B_{i+2}$.
        \item Assign the colors $c_1, \ldots, c_{\ceil{\frac{\omega}{4}}}$ to the vertices of $Z$, $\ceil{\frac{\omega}{4}}$ vertices in $B_{i+1} \setminus (N(K) \cup N(J))$, and $\ceil{\frac{\omega}{4}}$ vertices in $B_{i-2}$.
        \item Assign colors $d_1, \ldots, d_{\ceil{\frac{\omega}{4}}}$ to $\ceil{\frac{\omega}{4}}$ vertices in $N_{B_{i+1}}(K)$, and $\ceil{\frac{\omega}{4}}$ vertices in $B_{i-2}$.
        \item The vertices of $J$ are colored starting from $c_1$ and proceeding sequentially through\[c_1, \ldots, c_{\ceil{\frac{\omega}{4}}}, a_1, \ldots, a_{x\omega}, b_1, \ldots, b_{y\omega}\] using exactly $\omega(J)$ colors.
        \item Assign the colors $d_1, \ldots, d_{\ceil{\frac{\omega}{4}}}, a_1, \ldots, a_{x\omega}, b_1, \ldots, b_{y\omega}$ to $\alpha \omega$ vertices in $B_i \setminus N(K)$ in such a way that non-neighbors of $J$ in  $B_i \setminus N(K)$ are colored first, in order.
    \end{itemize}

    \begin{figure}
    \centering
    \begin{tikzpicture}[scale=0.7]
        \tikzstyle{v}=[circle, draw, solid, fill=black, inner sep=0pt, minimum width=3pt]
        \tikzstyle{w}=[circle, draw, solid, fill=black, inner sep=0pt, minimum width=3pt, font=\tiny]

        \draw[rounded corners, thick] (-9, 0)--(-7.5, 0)--(-7.5, -6)--(-10.5, -6)--(-10.5, 0)--(-8.9, 0);
        \draw[thin, blue] (-10.5, -1.2)--(-7.5, -1.2);
        \draw[thin, blue] (-10.5, -2.7)--(-7.5, -2.7);
        \draw[thin, blue] (-10.5, -4.2)--(-7.5, -4.2);
        
        \node [label={\small $|B_{i-2}|=(1-\beta)\omega$}] (v) at (-9, 0){};
        \node [label=$B_{i-2}$] (v) at (-9, 0.7){};
        \node [label={\small $a_1, \ldots, a_{x\omega}$}] (w) at (-9, -1.2){};
        \node [label={\small $c_1, \ldots, c_{\ceil{\frac{\omega}{4}}}$}] (w) at (-9, -2.6){};
        \node [label={\small $d_1, \ldots, d_{\ceil{\frac{\omega}{4}}}$}] (w) at (-9, -4.1){};

        \fill[red] (-9.2, -5.3) rectangle (-8.8, -4.9);
        
        \draw[rounded corners, thick] (-4.5, 0)--(-3, 0)--(-3, -6)--(-6, -6)--(-6, 0)--(-4.4, 0);
        
        \node [label={\small $|B_{i-1}| = \beta\omega$}] (v) at (-4.5, 0){};
        \node [label=$B_{i-1}$] (v) at (-4.5, 0.7){};

        \fill[red] (-4.7, -3.2) rectangle (-4.3, -2.8);
        
        \draw[rounded corners, thick] (0, 0)--(1.5, 0)--(1.5, -6)--(-1.5, -6)--(-1.5, 0)--(0.1, 0);
        \draw[very thick] (-1.5, -2)--(1.5, -2);
        \draw[thin, blue] (-1.5, -2.8)--(1.5, -2.8);
        \draw[thin, blue] (-1.5, -3.8)--(1.5, -3.8);
        \draw[thin, blue] (-1.5, -4.8)--(1.5, -4.8);
        
        \node [label={\small $|B_{i}|=(1- \beta)\omega$}] (v) at (0, 0){};
        \node [label=$B_{i}$] (v) at (0, 0.7){};
        \node [label={\small $b_1, \ldots, b_{y\omega}$}] (w) at (0, -3.85){};
        \node [label={\small $a_1, \ldots, a_{x\omega}$}] (w) at (0, -4.8){};
        \node [label={\small $d_1, \ldots, d_{\ceil{\frac{\omega}{4}}}$}] (w) at (0, -6.1){};

        \fill[red] (-0.2, -1.2) rectangle (0.2, -0.8);
        \fill[red] (-0.2, -2.6) rectangle (0.2, -2.2);

        \draw[rounded corners, thick] (4.5, 0)--(3, 0)--(3, -6)--(6, -6)--(6, 0)--(4.4, 0);
        \draw[thin, blue] (3, -1.5)--(6, -1.5);
        \draw[very thick] (3, -3)--(6, -3);
        \draw[thin, blue] (3, -4.5)--(6, -4.5);
        
        \node [label={\small $|B_{i+1}|=(1-\beta + \delta)\omega$}] (v) at (4.5, 0){};
        \node [label=$B_{i+1}$] (v) at (4.5, 0.7){};
        \node [label={\small $d_1, \ldots, d_{\ceil{\frac{\omega}{4}}}$}] (w) at (4.5, -1.5){};
        \node [label={\small $c_1, \ldots, c_{\ceil{\frac{\omega}{4}}}$}] (w) at (4.5, -4.5){};

        \fill[red] (4.3, -2.45) rectangle (4.7, -2.05);
        \fill[red] (4.3, -5.45) rectangle (4.7, -5.05);

        \draw[rounded corners, thick] (9, 0)--(7.5, 0)--(7.5, -6)--(10.5, -6)--(10.5, 0)--(8.9, 0);
        \draw[thin, blue] (10.5, -1.2)--(7.5, -1.2);
        
        \node [label={\small $|B_{i+2}| = (\beta-\delta)\omega$}] (v) at (9.2, 0){};
        \node [label=$B_{i+2}$] (v) at (9, 0.7){};
        \node [label={\small $b_1, \ldots, b_{y\omega}$}] (w) at (9, -1.2){};

        \fill[red] (8.8, -3.8) rectangle (9.2, -3.4);

        \draw[rounded corners, thick] (-2.5, 5)--(-0.5, 5)--(-0.5, 2)--(-4.5, 2)--(-4.5, 5)--(-2.4, 5);
        \draw[thin, blue] (-4.5, 4)--(-0.5, 4);
        \draw[thin, blue] (-4.5, 3)--(-0.5, 3);
        
        \node [label=$K$] (v) at (-2.5, 5){};
        \node [label={\small $a_1, \ldots, a_{x\omega}$}] (w) at (-2.5, 4){};
        \node [label={\small $b_1, \ldots, b_{y\omega}$}] (w) at (-2.5, 2.9){};
        \node [label={\small $c_1, \ldots, c_{\ceil{\frac{\omega}{4}}}$}] (w) at (-2.3, 1.8){};

        \draw[rounded corners, thick] (2.5, 5)--(0.5, 5)--(0.5, 2)--(4.5, 2)--(4.5, 5)--(2.4, 5);

        \node [label=$J$] (v) at (2.5, 5){};
        \node [label={\small $c_1, \ldots, c_{\ceil{\frac{\omega}{4}}},$}] (w) at (2.5, 3.6){};
        \node [label={\small $a_1, \ldots, a_{x\omega},$}] (w) at (2.45, 2.9){};
        \node [label={\small $b_1, \ldots$}] (w) at (1.8, 2){};

        \draw[very thick] (-7.5, -2.6)--(-6, -2.6);
        \draw[very thick] (-7.5, -3.4)--(-6, -3.4);

        \draw[very thick] (-3, -2.6)--(-1.5, -2.6);
        \draw[very thick] (-3, -3.4)--(-1.5, -3.4);

        \draw[very thick] (7.5, -2.6)--(6, -2.6);
        \draw[very thick] (7.5, -3.4)--(6, -3.4);

        \draw[rounded corners, very thick] (-9.3, -6)--(-9.3, -6.8)--(9.3, -6.8)--(9.3, -6);
        \draw[rounded corners, very thick] (-8.7, -6)--(-8.7, -6.4)--(8.7, -6.4)--(8.7, -6);
    \end{tikzpicture}
    \caption{The coloring strategy for Case (i). For each $j \in \{i, i+1\}$, the black line in $B_j$ distinguishes $N_{B_j}(K)$ from $B_j \setminus N(K)$. Two black line between $B_j$ and $B_{j+1}$ mean that the two sets are complete to each other. The remaining graph consists of the red-boxed parts.}
\end{figure}
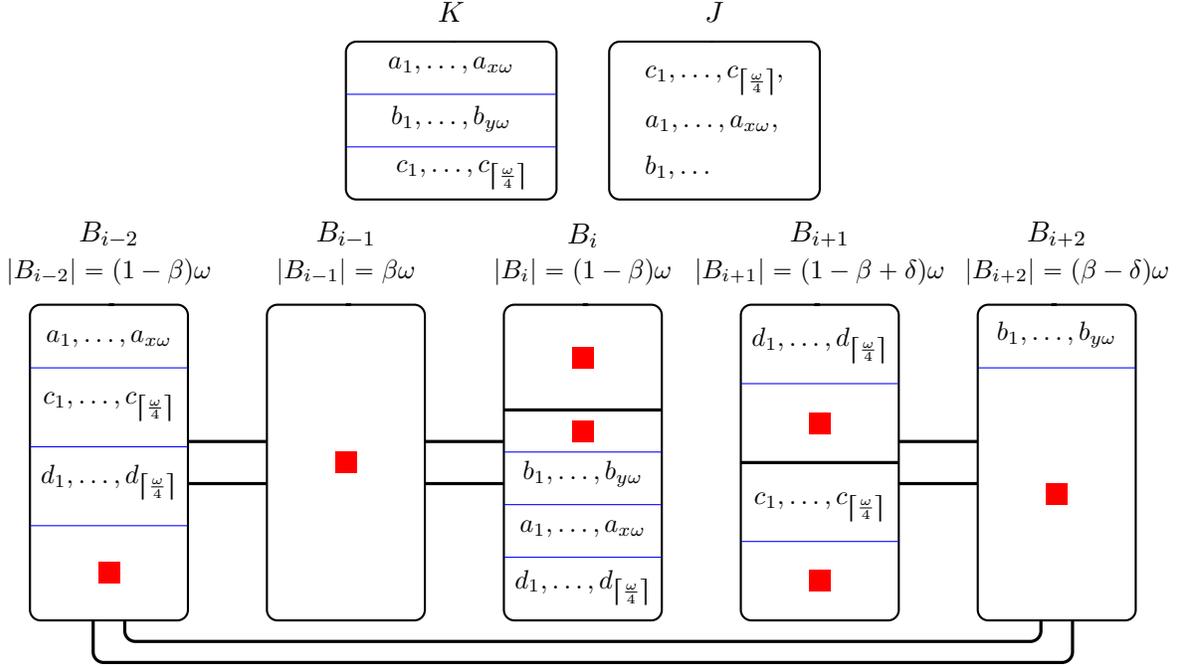

    The coloring strategy is illustrated in the above figure. By Lemma \ref{lem:comparable vertices in A_3(i)}, $N_{B_i}(K)\subseteq N_{B_i}(J)$ or $N_{B_i}(J)\subseteq N_{B_i}(K)$. By Lemma \ref{lem:maximalityofA3} and $\omega(K)\ge \omega(J)$, the number of non-neighbors of $J$ in $B_i\setminus N(K)$ is at least $\omega(J)$ and so the coloring used on $J \cup B_i$ is proper. The remaining graph consists of the red-boxed parts in the figure, and it suffices to show that the remaining graph $G'$ is $\left(\ceil{\frac{5}{4}\omega} - x\omega - y\omega - 2\cdot\ceil{\frac{\omega}{4}}\right)$-colorable. Since $|B_{i+1} \setminus (N(K) \cup N(J))|, |N_{B_{i+1}}(K) \cap N_{B_{i+1}}(J)| > \ceil{\frac{\omega}{4}}$, the each red-boxed parts in $B_{i+1}$ have non-negative size. Since $|B_{i+2}| > \ceil{\frac{\omega}{4}}$, the red-boxed part in $B_{i+2}$ also has non-negative size. Note that the red-boxed part in $B_{i-2}$ could be empty. Then we only need to consider the worst case that $B_i$ is complete to $B_{i+1}$ in $G'$ and so $G'$ is a hyperhole $C_5[s_1, s_2, s_3, s_4, s_5]$, where
    \begin{itemize}
        \item $s_1 = (1 - \beta)\omega - x\omega - 2\cdot\ceil{\frac{\omega}{4}}$,
        \item $s_2 = \beta\omega$,
        \item $s_3 = (1 - \beta)\omega  - x\omega - y\omega - \ceil{\frac{\omega}{4}}$,
        \item $s_4 = (1 - \beta + \delta)\omega - 2\cdot\ceil{\frac{\omega}{4}}$,
        \item $s_5 = (\beta - \delta)\omega - y\omega$.
    \end{itemize}

     We first show that $\omega(G') \le \ceil{\frac{5}{4}\omega} - x\omega - y\omega - 2\cdot\ceil{\frac{\omega}{4}}$. It is equivalent to show that for every $j \in [5]$, 
    \[s_j + s_{j+1} \le \ceil{\frac{5}{4}\omega} - x\omega - y\omega - 2\cdot\ceil{\frac{\omega}{4}}\]
    where the indices are taken modulo $5$ (so that $s_6 = s_1$). For $j = 1,2,4,5$, the above inequality is equivalent to the following respective inequality
    \[y\omega \le \ceil{\frac{\omega}{4}}\, , \, \omega \le \ceil{\frac{5}{4}\omega} - \ceil{\frac{\omega}{4}}\, ,\, x\omega \le \ceil{\frac{\omega}{4}}\, , \, \omega \le \ceil{\frac{5}{4}\omega} + \delta\omega,\]
    and in each case the inequality holds clearly. When $j = 3$, the inequality yields $2\omega \le 2\beta\omega + \ceil{\frac{\omega}{4}} + \ceil{\frac{5}{4}\omega} - \delta \omega$. Since $(\beta - \delta)\omega = |B_{i+2}| > \ceil{\frac{\omega}{4}}$ and $\beta \omega = |B_{i-1}| > \ceil{\frac{\omega}{4}}$, this inequality also holds for $j = 3$. Hence, the inequality holds for all $j\in[5]$.  Now we show that $\ceil{\frac{|V(G')|}{\alpha(G')}} \le \ceil{\frac{5}{4}\omega} - x\omega - y\omega - 2\cdot\ceil{\frac{\omega}{4}}$. Observing that $\alpha(G')=2$, it suffices to show
    \[|V(G')| \le 2\cdot\ceil{\frac{5}{4}\omega} - 2x\omega - 2y\omega - 4\cdot\ceil{\frac{\omega}{4}}.\]
    Since $|V(G')| = 3\omega - \beta\omega - 2x\omega - 2y\omega - 5\cdot\ceil{\frac{\omega}{4}}$, we obtain $3\omega  \le \beta\omega + \ceil{\frac{\omega}{4}} + 2\cdot\ceil{\frac{5}{4}\omega}$ and this inequality is satisfied because $\beta \omega \ge \ceil{\frac{\omega}{4}} + 1$. By Lemma \ref{lem:hyperhole}, $\chi(G') = \max\{\omega(G'), \ceil{\frac{|V(G')|}{\alpha(G')}}\} \le \ceil{\frac{5}{4}\omega} - x\omega - y\omega - 2\cdot\ceil{\frac{\omega}{4}}$.

    \vskip 0.3cm
    \noindent
    {\bf (ii) $\omega(K) \le \frac{1}{4}\omega$}
    \vskip 0.2cm

    Recall that $z \in B_i \setminus (N(K) \cup N(J))$ be a minimal vertex in $B_i \setminus (N(K) \cup N(J))$ and $|N_{B_{i+1}}(z)| > \ceil{\frac{\omega}{4}}$. As $z$ is minimal in $B_i$, $N_{B_{i+1}}(z)$ is complete to $B_i$. Since $N_{B_i}(K) \cup N_{B_{i+1}}(K) \cup N_{B_{i+1}}(z)$ is a clique, $|N_{B_i}(K)| < \frac{\omega}{2}$.
    Let $x \omega = \min\{|B_i \setminus N(K)|, \ceil{\frac{1}{4}\omega}\}$. We then color some vertices of $G$ using $\alpha \omega + x\omega + \ceil{\frac{1}{4}\omega}$ colors as follows. Recall that $|B_i\setminus N(K)| = (\alpha + \gamma)\omega$.

        \begin{itemize}
        \item Assign the colors $a_1, \ldots, a_{\alpha \omega}$ to the vertices of $K$, $\alpha \omega$ vertices in $B_{i+1} \setminus (N(K) \cup N(J))$, and $\alpha \omega$ vertices in $B_{i-2}$.
        \item Assign the colors $a_1, \ldots, a_{\omega(J)}$ to the vertices of $J$.
        \item Assign the colors $b_1, \ldots, b_{x\omega}$ to the vertices of $B_i \setminus N(K)$, $x \omega$ vertices in $N_{B_{i+1}}(K)$, and $x \omega$ vertices in $B_{i-2}$.
        \item Assign colors $c_1, \ldots, c_{\alpha \omega}$ to $\alpha \omega$ vertices in $B_i$ excluding the vertices that have already been colored with $b_1, \ldots, b_{x\omega}$, and $\alpha \omega$ vertices in $B_{i-2}$. 
        \item Assign colors $d_1, \ldots, d_{\ceil{\frac{\omega}{4}} - \alpha \omega}$ to $\ceil{\frac{\omega}{4}}-\alpha\omega$ vertices in $B_{i+1} \setminus (N(K) \cup N(J))$, and $\ceil{\frac{\omega}{4}}-\alpha \omega$ vertices in $B_{i-1}$.

    \end{itemize}

    \begin{figure}
    \centering
    \begin{tikzpicture}[scale=0.7]
        \tikzstyle{v}=[circle, draw, solid, fill=black, inner sep=0pt, minimum width=3pt]
        \tikzstyle{w}=[circle, draw, solid, fill=black, inner sep=0pt, minimum width=3pt, font=\tiny]

        \draw[rounded corners, thick] (-9, 0)--(-7.5, 0)--(-7.5, -6)--(-10.5, -6)--(-10.5, 0)--(-8.9, 0);
        \draw[thin, blue] (-10.5, -1.2)--(-7.5, -1.2);
        \draw[thin, blue] (-10.5, -2.7)--(-7.5, -2.7);
        \draw[thin, blue] (-10.5, -4.2)--(-7.5, -4.2);
        
        \node [label={\small $|B_{i-2}|=(1-\beta)\omega$}] (v) at (-9, 0){};
        \node [label=$B_{i-2}$] (v) at (-9, 0.7){};
        \node [label={\small $a_1, \ldots, a_{\alpha\omega}$}] (w) at (-9, -1.2){};
        \node [label={\small $b_1, \ldots, b_{x\omega}$}] (w) at (-9, -2.6){};
        \node [label={\small $c_1, \ldots, c_{\alpha\omega}$}] (w) at (-9, -4.1){};

        \fill[red] (-9.2, -5.3) rectangle (-8.8, -4.9);
        
        \draw[rounded corners, thick] (-4.5, 0)--(-2.7, 0)--(-2.7, -6)--(-6.3, -6)--(-6.3, 0)--(-4.4, 0);
        \draw[thin, blue] (-6.3, -1.5)--(-2.7, -1.5);
        
        \node [label={\small $|B_{i-1}| = \beta\omega$}] (v) at (-4.5, 0){};
        \node [label=$B_{i-1}$] (v) at (-4.5, 0.7){};
        \node [label={\small $d_1, \ldots, d_{\ceil{\frac{\omega}{4}}-\alpha\omega}$}] (w) at (-4.5, -1.5){};

        \fill[red] (-4.7, -3.95) rectangle (-4.3, -3.55);
        
        \draw[rounded corners, thick] (0, 0)--(1.5, 0)--(1.5, -6)--(-1.5, -6)--(-1.5, 0)--(0.1, 0);
        \draw[very thick] (-1.5, -3)--(1.5, -3);
        \draw[thin, blue] (-1.5, -2)--(1.5, -2);
        \draw[thin, blue] (-1.5, -5)--(1.5, -5);

        \node [label={\small $|B_{i}|=(1- \beta)\omega$}] (v) at (0, 0){};
        \node [label=$B_{i}$] (v) at (0, 0.7){};
        \node [label={\small $c_1, \ldots, c_{\alpha\omega}$}] (w) at (0, -1.5){};
        \node [label={\small $b_1, \ldots, b_{x\omega}$}] (w) at (0, -4.5){};

        \fill[red] (-0.2, -2.7) rectangle (0.2, -2.3);
        \fill[red] (-0.2, -5.7) rectangle (0.2, -5.3);

        \draw[rounded corners, thick] (4.5, 0)--(2.7, 0)--(2.7, -6)--(6.3, -6)--(6.3, 0)--(4.4, 0);
        \draw[thin, blue] (2.7, -1.5)--(6.3, -1.5);
        \draw[very thick] (2.7, -3)--(6.3, -3);
        \draw[thin, blue] (2.7, -4)--(6.3, -4);
        \draw[thin, blue] (2.7, -5.2)--(6.3, -5.2);
        
        \node [label={\small $|B_{i+1}|=(1-\beta + \delta)\omega$}] (v) at (4.5, 0){};
        \node [label=$B_{i+1}$] (v) at (4.5, 0.7){};
        \node [label={\small $b_1, \ldots, b_{x\omega}$}] (w) at (4.5, -1.3){};
        \node [label={\small $a_1, \ldots, a_{\alpha\omega}$}] (w) at (4.5, -4){};
        \node [label={\small $d_1, \ldots, d_{\ceil{\frac{\omega}{4}}-\alpha\omega}$}] (w) at (4.5, -5.3){};

        \fill[red] (4.3, -2.45) rectangle (4.7, -2.05);
        \fill[red] (4.3, -5.8) rectangle (4.7, -5.4);

        \draw[rounded corners, thick] (9, 0)--(7.5, 0)--(7.5, -6)--(10.5, -6)--(10.5, 0)--(8.9, 0);

        \node [label={\small $|B_{i+2}| = (\beta-\delta)\omega$}] (v) at (9.2, 0){};
        \node [label=$B_{i+2}$] (v) at (9, 0.7){};

        \fill[red] (8.8, -3.2) rectangle (9.2, -2.8);

        \draw[rounded corners, thick] (-2.5, 3.5)--(-0.5, 3.5)--(-0.5, 2)--(-4.5, 2)--(-4.5, 3.5)--(-2.4, 3.5);
        
        \node [label=$K$] (v) at (-2.5, 3.5){};
        \node [label={\small $a_1, \ldots, a_{\alpha\omega}$}] (w) at (-2.5, 2.25){};

        \draw[rounded corners, thick] (2.5, 3.5)--(0.5, 3.5)--(0.5, 2)--(4.5, 2)--(4.5, 3.5)--(2.4, 3.5);

        \node [label=$J$] (v) at (2.5, 3.5){};
        \node [label={\small $a_1, \ldots, a_{\omega(J)}$}] (w) at (2.5, 2.2){};

        \draw[very thick] (-7.5, -2.6)--(-6.3, -2.6);
        \draw[very thick] (-7.5, -3.4)--(-6.3, -3.4);

        \draw[very thick] (-2.7, -2.6)--(-1.5, -2.6);
        \draw[very thick] (-2.7, -3.4)--(-1.5, -3.4);

        \draw[very thick] (7.5, -2.6)--(6.3, -2.6);
        \draw[very thick] (7.5, -3.4)--(6.3, -3.4);

        \draw[rounded corners, very thick] (-9.3, -6)--(-9.3, -6.8)--(9.3, -6.8)--(9.3, -6);
        \draw[rounded corners, very thick] (-8.7, -6)--(-8.7, -6.4)--(8.7, -6.4)--(8.7, -6);
    \end{tikzpicture}
    \caption{The coloring strategy for Case (ii). For each $j \in \{i, i+1\}$, the black line in $B_j$ distinguishes $N_{B_j}(K)$ from $B_j \setminus N(K)$. Two black line between $B_j$ and $B_{j+1}$ mean that the two sets are complete to each other. The remaining graph consists of the red-boxed parts.}
\end{figure}
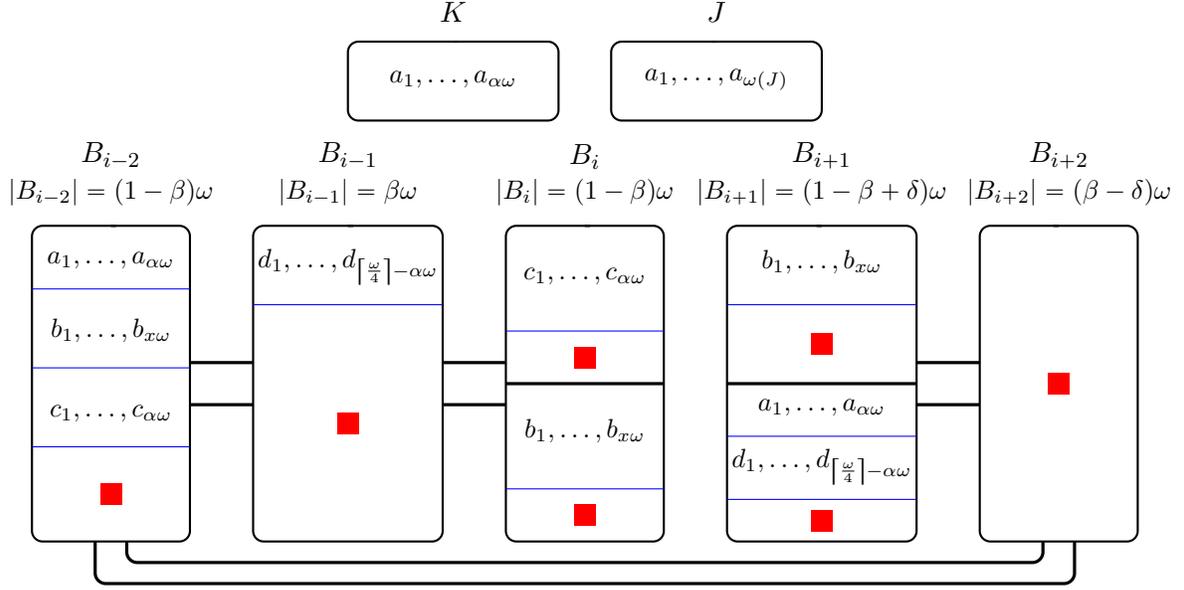

    The coloring strategy is illustrated in the above figure. The remaining graph $G'$ consists of the red-boxed parts in the figure, and it suffices to show that $G'$ is $\left(\ceil{\frac{5}{4}\omega} - \alpha \omega - z\omega - \ceil{\frac{1}{4}\omega}\right)$-colorable. Since $|B_{i+1} \setminus (N(K) \cup N(J))|> \ceil{\frac{\omega}{4}}$, $|N_{B_{i+1}}(K)| > \ceil{\frac{\omega}{4}}$ and $x\omega \le \ceil{\frac{\omega}{4}}$, the each red-boxed parts in $B_{i+1}$ have non-negative size. Since $|B_{i-1}| > \ceil{\frac{\omega}{4}}$, the red-boxed parts in $B_{i-1}$ has non-negative size. Note that the red-boxed part in $B_{i-2}$ and $B_i$ could be empty. Then we only need to consider the worst case that $B_i$ is complete to $B_{i+1}$ in $G'$ and so $G'$ is a hyperhole $C_5[s_1, s_2, s_3, s_4, s_5]$, where
    \begin{itemize}
        \item $s_1 = (1 - \beta)\omega - 2\alpha \omega - x\omega$,
        \item $s_2 = \beta\omega  - \ceil{\frac{1}{4}\omega} + \alpha \omega$,
        \item $s_3 = (1 - \beta)\omega - \alpha \omega - x\omega$,
        \item $s_4 = (1 - \beta + \delta)\omega - x\omega - \ceil{\frac{1}{4}\omega}$, 
        \item $s_5 = (\beta - \delta)\omega$.
    \end{itemize}

    We first show that $\omega(G') \le \ceil{\frac{5}{4}\omega} - \alpha \omega - x\omega - \ceil{\frac{1}{4}\omega}$. It suffices to show that for every $j \in [5]$, 
    \[s_j + s_{j+1} \le \ceil{\frac{5}{4}\omega} - \alpha \omega - x\omega - \ceil{\frac{1}{4}\omega}\]
    where the indices are taken modulo $5$ (so that $s_6 = s_1$). From the $j = 1,2,4,5$, the above inequality is equivalent to the following respective inequality
    \[\omega \le \ceil{\frac{5}{4}\omega}\, , \, \alpha \omega \le \ceil{\frac{1}{4}\omega}\, ,\, \alpha \omega \le \ceil{\frac{1}{4}\omega}\, , \, 0 \le \alpha \omega + \delta\omega,\]
    and in each case the inequality holds. When $j = 3$, the inequality yields $2\omega \le 2\beta\omega + x\omega + \ceil{\frac{5}{4}\omega} - \delta \omega$. Since $\beta\omega - \delta\omega = |B_{i+2}| > \ceil{\frac{1}{4}\omega}$ and $\beta\omega + x\omega = \omega - |N_{B_i}(K)| > \frac{\omega}{2}$, this inequality also holds for $j = 3$. Hence, the inequality holds for all $j\in[5]$.  Now we show that $\ceil{\frac{|V(G')|}{\alpha(G')}} \le \ceil{\frac{5}{4}\omega} - \alpha \omega - x\omega - \ceil{\frac{1}{4}\omega}$. Observing that $\alpha(G')=2$, it suffices to show
    \[|V(G')|  \le 2\cdot\ceil{\frac{5}{4}\omega} - 2\alpha \omega - 2x\omega - 2\cdot\ceil{\frac{1}{4}\omega}.\]
    Since $|V(G')| = 3\omega - \beta\omega - 2\alpha \omega - 3x\omega - 2\cdot\ceil{\frac{1}{4}\omega}$, the inequality is equivalent to $3\omega \le \beta\omega + x\omega + 2\cdot\ceil{\frac{5}{4}\omega}$, which holds as $\beta\omega + x\omega>\frac{1}{2}\omega$.
    By Lemma \ref{lem:hyperhole}, $\chi(G') = \max\{\omega(G'), \ceil{\frac{|V(G')|}{\alpha(G')}}\} \le \ceil{\frac{5}{4}\omega} - \alpha \omega - x\omega - \ceil{\frac{1}{4}\omega}$.
\end{proof}

We then consider cases depending on the number of components of $A_3(i)$ and the number of neighbors of each component in $B_i$.

\begin{lemma}[One component with one neighbor in $B_i$]
    If $A_3(i)$ is connected and $|N(A_3(i)) \cap B_i| =1$, then $\chi(G) \le \ceil{\frac{5}{4}\omega}$.
\end{lemma}

\begin{proof}
    For convenience, let $K=A_3(i)$.
    By Lemma \ref{lem:single vertex covers A_3(i)}, let $w \in K$ be a universal vertex in $K$ such that $N_H(w)=N_H(K)$. By Lemma \ref{lem:blowup-threeneighbor}, we may assume that $B_i\setminus N(K)$ is anticomplete to $N_{B_{i+1}}(K)$. So $N_{B_{i+1}}(K) = \{q_i\}$. 
    By Lemma \ref{lem:B_i complete to two large cliques in B_{i-1}}, $N_{B_{i+1}}(K)$ is complete to $B_{i+2}$. Let $X_{i+1} = B_{i+1}\setminus N(K)$. By Lemma \ref{lem: Y is empty}, each vertex in $X_{i+1}$ has a neighbor in $B_i \setminus \{q_i\}$. By Lemma \ref{lem:B_i complete to two large cliques in B_{i-1}}, $X_{i+1}$ is complete to $B_{i+2}$. It follows that $B_{i+1}$ is complete to $B_{i+2}$.

    Suppose that $B_{i-1}\setminus N(K) \neq \emptyset$ and let $X_{i-1} = B_{i-1}\setminus N(K)$. By Lemma \ref{lem: Y is empty}, each vertex in $X_{i-1}$ has a neighbor in $B_i \setminus \{q_i\}$. By Lemmas \ref{lem:reduce to four complete pair} and \ref{lem:color four complete pairs}, we have $\chi(G) \le \ceil{\frac{5}{4}\omega}$. So we may assume that $X_{i-1} = \emptyset$.

    Let $z \in B_i \setminus \{q_i\}$ be a minimal vertex in $B_i \setminus \{q_i\}$.
    By Lemma \ref{lem:smv on w}, $|N_{B_{i+1}}(z)| > \ceil{\frac{\omega}{4}}$. 
    By Lemma \ref{lem:smv on B_{i+2}+B_{i-2}}, $|B_j|>\ceil{\frac{\omega}{4}}$ for $j\in \{i-2,i+2\}$.
    This implies that $\omega\ge 5$.
    By Lemma \ref{lem:structure of only A3} (3), $K$ is a blowup of $P_3$ and so $\alpha(K)\le 2$.
    
    \medskip
    \noindent
    \textbf{Case 1: $K$ contains two non-adjacent vertices.}

    \medskip
    Since $K$ is a $P_4$-free, let $u$ and $v$ be two simplicial vertices in $K$. Let $u' \in N_G[u] \cap K$ be a vertex minimizing $|N_H(u')|$ and $v' \in N_G[v] \cap K$ be a vertex minimizing $|N_H(v')|$. By Lemma \ref{lem:two non-adjacent simplicail verteices}, $u'v'\notin E(G)$.
    By Lemma \ref{lem:small vertex argument for simplicial vertices in A_3(i-2)}, we have $|N_{B_j}(u')|, |N_{B_j}(v')| > \ceil{\frac{\omega}{4}}$ for each $j \in \{i-1, i+1\}$. If $N_{B_{i+1}}(u')$ and $N_{B_{i+1}}(v')$ are disjoint, then $B_{i+1} \cup B_{i+2}$ is a clique of size larger than $\omega$. So $N_{B_{i+1}}(u')$ and $N_{B_{i+1}}(v')$ are comparable. By Lemma \ref{lem:nonedge in A_3(i)}, $N_{B_{i-1}}(u')$ and $N_{B_{i-1}}(v')$ are disjoint. Since $N_{B_{i+1}}(K)$ and $N_{B_{i+1}}(z)$ are disjoint, we have $N_{B_{i-1}}(u') \cup N_{B_{i-1}}(v') \subseteq N_{B_{i-1}}(z)$ and it follows from the choice of $z$ that $N_{B_{i-1}}(u') \cup N_{B_{i-1}}(v')$ is complete to $B_i$. By Lemma \ref{lem:B_i complete to two large cliques in B_{i-1}}, $N_{B_{i-1}}(u') \cup N_{B_{i-1}}(v')$ is complete to $B_{i-2}$.

    Note that $B_{i-1} \cup \{q_i\}$ and $B_{i+1} \cup \{q_i\}$ are not maximum clique, since $B_{i-1} \cup \{q_i,w\}$ and $B_{i+1}\cup B_{i+2}$ are cliques. By the definition of $z$, there is a vertex in $N_{B_{i+1}}(z)$ complete to $B_i\cup B_{i+2}$. We set this vertex to be $q_{i+1}$. Let $a \in N_{B_{i-1}}(u')$ and let $b \in N_{B_{i-1}}(v')$. Let $a' \in K$ be a vertex non-adjacent to $a$ maximizing $|N_H(a')|$ and let $b' \in K$ be a vertex non-adjacent to $b$ maximizing $|N_H(b')|$. 
    If $a'u'\in E(G)$, then $a'a\in E(G)$ by the choice of $u$, a contradiction. So $a'u'\notin E(G)$ and thus $a'v'\in E(G)$ by $\alpha(K)\le 2$. By symmetry, $b'u'\in E(G)$. Let $c \in N_{B_{i+1}}(u') \cap N_{B_{i+1}}(v')$. Then $c$ is complete to $K$. Let $q_j' \in B_j$ be the second largest vertex in $B_j$ for $j \in \{i-2, i+2\}$. By Lemma \ref{lem:smv on B_{i+2}+B_{i-2}}, we have $q'_{i+2}$ is complete to $B_{i-2}$ and $q'_{i-2}$ is complete to $B_{i+2}$. 
    Let $q'_i,q''_i$ be the second and third largest vertices in $B_i$, respectively.

    We set $F = (F_1, F_2, F_3, F_4, F_5, O)$ as follows:
    \begin{itemize}
        \item $F_1 = \{q_{i-2}, q_{i-2}'\}$,
        \item $F_2 = \{a, b\}$,
        \item $F_3 = \{q_i, q_i', q_i''\}$,
        \item $F_4 = \{c, q_{i+1}\}$,
        \item $F_5 = \{q_{i+2}, q_{i+2}'\}$,
        \item $O = \{w, a', b'\}$.
    \end{itemize}

    We can assign color $1$ to $q_{i-2}$ and $q_i$, color $2$ to $w$, $q_i'$ and $q_{i+2}$, color $3$ to $a$, $a'$ and $q_{i+2}'$, color $4$ to $b$, $b'$ and $q_{i+1}$, and color $5$ to $q_{i-2}'$, $q_i''$ and $c$. So $F$ is 5-colorable. Next, we show that every maximal clique of size $\omega-j$ with $j\in \{0,1,2,3\}$ of $G$ contains at least $4-j$ vertices from $F$. Let $M$ be a maximal clique of $G$. We consider the following cases depending on all possible locations of $M$. If $M$ is contained in $B_{i+1} \cup B_{i+2}$ or $B_{i-2} \cup B_{i+2}$, $M$ contains four vertices of $F$.

    \medskip
    \noindent
    \textbf{(i) $M\subseteq K \cup N_{B_i}(K) \cup N_{B_j}(K)$ for every $j \in \{i-1, i+1\}$}

    Note that $w$ and $q_i$ are universal vertex in $K\cup N_H(K)$.
    
    We first assume that $M$ is contained in $K \cup N_{B_{i-1}}(K) \cup \{q_i\}$. By definition of $u'$ and $v'$, $M$ contains at least one of $a$ and $b$.  If $a, b \in M$, then we are done. So we may assume that $M$ contains exactly one of $a$ and $b$. By symmetry, we may assume that $a \in M$ and this implies that $M \cap N_{B_{i-1}}(v') = \emptyset$. By definition of $b'$, we have $b' \in M$ and so $M$ contains $\{a, b', q_i, w\}$.
    
    So we may assume that $M \subseteq K \cup N_{B_{i+1}}(K) \cup \{q_i\}$. Note that $q_i, w, c$ are universal in $K \cup N_{B_{i+1}}(K)$. Suppose that $a', b' \notin M$. By definition of $a'$ and $b'$, $M \cap K$ is complete to $\{a, b\}$ and it follows that $M \cap K$ is complete to $N_{B_{i-1}}(u') \cup N_{B_{i-1}}(v')$. As $(M \cap K) \cup N_{B_{i-1}}(u') \cup N_{B_{i-1}}(v') \cup \{q_i\}$ is a clique and $|N_{B_{i-1}}(u') \cup N_{B_{i-1}}(v')| > 2 \cdot \ceil{\frac{\omega}{4}}$, we obtain $|M \cap K| < \omega - 2 \cdot \ceil{\frac{\omega}{4}} - 1$. So, $|M \cap N_{B_{i+1}}(K)| > 2 \cdot \ceil{\frac{\omega}{4}}$ and $B_{i+1} \cup B_{i+2}$ is a clique of size larger than $\omega$. Hence, $M \cap \{a', b'\} \neq \emptyset$ and we are done.

    \medskip
    \noindent
    \textbf{(ii) $M\subseteq B_{i-2} \cup B_{i-1}$}.

    Since $a,b,q_{i-2}$ are universal in $B_{i-2}\cup B_{i-1}$, $a,b,q_{i-2}\in M$. Suppose that $|M|=\omega$ and $q'_{i-2}\notin M$. Then $M\cap B_{i-2}=\{q_{i-2}\}$ and so $|M\cap B_{i-1}|=\omega-1$. Then $(M\cap B_{i-1})\cup \{q_i,w\}$ is a clique of size $\omega+1$, a contradiction.
    
    \medskip
    \noindent
    \textbf{(iii) $M\subseteq B_{i-1} \cup B_i$}

    Note that $a, b, q_i$ are universal in $B_{i-1} \cup B_i$. If $M = B_{i-1} \cup \{q_i\}$, then $M$ is not a maximum clique and $a, b, q_i \in M$. So we may assume that $M \cap (B_i \setminus \{q_i\}) \neq \emptyset$ and this implies that $q_i' \in M$. Hence, $M$ contains $\{a, b, q_i, q_i'\}$.

    \medskip
    \noindent
    \textbf{(iv) $M\subseteq B_i \cup B_{i+1}$}

    Note that $q_i$ and $q_{i+1}$ are universal in $B_i \cup B_{i+1}$. If $M = B_{i+1} \cup \{q_i\}$, then $M$ is not a maximum clique and $c, q_i, q_{i+1} \in M$. So we may assume that $M \cap (B_i \setminus \{q_i\}) \neq \emptyset$. Then $q_i' \in M$ and it follows that $M \cap N_{B_{i+1}}(K) = \emptyset$. Thus, we can assume that $|M| = \omega$. If $q_i'' \notin M$, then $|M \cap B_{i+1}| = \omega - 2$ and it follows that $B_{i+1} \cup B_{i+2}$ is a clique of size larger than $\omega$. So, $q_i'' \in M$ and the claim holds.

    \medskip
    Therefore, $F$ is a $(4,5)$-good subgraph. From now on, we may assume that $K$ is a clique. 

    \medskip
    \noindent
    \textbf{Case 2:} $K$ is a clique.
    
    \medskip
    Suppose that $|K| = 1$. For any $x \in N_{B_{i+1}}(K)$, we have $d_G(x) \le (\omega-1) + 2 \le \ceil{\frac{5}{4}\omega}-1$ since $\omega\ge 5$. So we may assume that $|K| \ge 2$. By Lemma \ref{lem:maximalityofA3}, we have $|B_i| \ge 3$. Let $w' \in K$ be a vertex minimizing $|N_H(w')|$. By Lemma \ref{lem:single vertex in A_3(i)} (4), there exists a vertex in $N_{B_{i-1}}(w')$ that is complete to $B_{i-2} \cup B_i$. 
    We set this vertex to be $q_{i-1}$. Moreover, $q_{i-1}$ is complete to $K$.
    
    Let $a \in N_{B_{i+1}}(w')$. Then $a$ is complete to $K$. Let $v \in K$ be the second largest vertex in $K$ (It is possible that $v = w'$). Note that $q_{i+1} \in N_{B_{i+1}}(z)$. Let $q_j' \in B_j$ be the second largest vertex in $B_j$ for $j \in \{i-2, i+2\}$. By Lemma \ref{lem:smv on B_{i+2}+B_{i-2}}, we have $q'_{i+2}$ is complete to $B_{i-2}$. Let $q'_i,q''_i$ be the second and third largest vertices in $B_i$, respectively. 
    Since $|B_{i-2}| \ge 3$, the third largest vertex $q_{i-2}''$ in $B_{i-2}$ exists. By Lemma \ref{lem:smv on B_{i+2}+B_{i-2}}, we have $q_{i-2}'$ and $q_{i-2}''$ are complete to $B_{i+2}$.

    We set $F = (F_1, F_2, F_3, F_4, F_5, O)$ as follows:
    \begin{itemize}
        \item $F_1 = \{q_{i-2}, q_{i-2}', q_{i-2}''\}$,
        \item $F_2 = \{q_{i-1}\}$,
        \item $F_3 = \{q_i, q_i', q_i''\}$,
        \item $F_4 = \{a, q_{i+1}\}$,
        \item $F_5 = \{q_{i+2}, q_{i+2}'\}$,
        \item $O = \{w, v\}$.
    \end{itemize}

    We can assign color $1$ to $q_{i-2}$ and $q_i$, color $2$ to $w$, $q_i'$ and $q_{i+2}$, color $3$ to $q_{i-2}'$, $v$ and $q_{i+1}$, color $4$ to $q_{i-1}$ and $q_{i+2}'$, and color $5$ to $q_{i-2}''$, $q_i''$ and $a$. So $F$ is 5-colorable. Next, we show that every maximal clique of size $\omega-j$ with $j\in \{0,1,2,3\}$ of $G$ contains at least $4-j$ vertices from $F$. Let $M$ be a maximal clique of $G$. We consider the following cases depending on all possible locations of $M$. If $M$ is one of $B_{i+1} \cup B_{i+2}$ and $B_{i-2} \cup B_{i+2}$, $M$ contains four vertices of $F$.

    \medskip
    \noindent
    \textbf{(i) $M\subseteq K \cup N_{B_i}(K) \cup N_{B_j}(K)$ for every $j \in \{i-1, i+1\}$}
    
    We first assume that $M$ is contained in $K \cup (N_{B_{i-1}}(K)) \cup \{q_i\}$. Note that $q_{i-1}$, $q_i$ and $w$ are universal in $K \cup (N_{B_{i-1}}(K)) \cup \{q_i\}$. So $q_{i-1}, q_i, w \in M$ and we can assume that $|M|=\omega$. Suppose that $v \notin M$. Then there exists a vertex $x \in B_{i-1} \cap M$ that is non-adjacent to $v$. If $x$ has a non-neighbor $x' \in B_i \setminus \{q_i\}$, then $v - w - x - q_{i-2} - q_{i+2} - q_{i+1} - x'$ is an induced $P_7$. By Lemma \ref{lem:single vertex in A_3(i)} (4),
    $B_{i-1} \cap M$ is complete to $B_i$. Since $|B_{i-1} \cap M| = \omega - 2$ and $|B_i|\ge 3$, $(B_{i-1} \cap M) \cup B_i$ is a clique of size larger than $\omega$. So $v \in M$ and $M$ contains $\{q_{i-1}, q_i, w, v\}$. 
    
    We then assume that $M \subseteq K \cup (N_{B_{i+1}}(K)) \cup \{q_i\}$. Note that $q_i, w, a$ are universal in $K \cup N_{B_{i+1}}(K)$. Suppose that $v \notin M$ and so we can assume that $|M| = \omega$. It follows that $|B_{i+1} \cap M| = \omega - 2$. Then $B_{i+1} \cup B_{i+2}$ is a clique of size larger than $\omega$. Therefore, $v \in M$ and we are done.

    \medskip
    \noindent
    \textbf{(ii) $M\subseteq B_{i-2} \cup B_{i-1}$}

    Note that $q_{i-2}$ and $q_{i-1}$ are universal in $B_{i-2} \cup B_{i-1}$. Define $q$ to be $q_{i-2}''$ if $|M|=\omega$ and to be $q_{i-2}'$ if $|M|=\omega-1$. We show that $q\in M$, which completes the proof. Suppose for a contradiction that $q\notin M$. Then there exists a vertex $x \in B_{i-1} \cap M$ such that $xq \notin E(G)$. If $x$ has non-neighbor $x' \in B_i \setminus \{q_i\}$, then $q - q_{i-2} - x - w - N_{B_{i+1}}(K) - q_{i+1} - x'$ is an induced $P_7$. So $x$ is complete to $B_i$. For any $y \in B_{i-1} \cap M$ with $yq \in E(G)$, we have $N_{B_{i-2}\cup B_i}(x) \subseteq N_{B_{i-2}\cup B_i}(y)$ and so $B_{i-1} \cap M$ is complete to $B_i$. This implies that $(B_{i-1} \cup M) \cup B_i$ is a clique of size larger than $\omega$. This completes the claim.

    \medskip
    \noindent
    \textbf{(iii) $M\subseteq B_{i-1} \cup B_i$}

    Note that $q_{i-1}$ and $q_i$ are universal in $B_{i-1} \cup B_i$. 
    Since $B_{i-1}\cup \{q_i,w\}$ is a clique,
    $M \cap (B_i \setminus \{q_i\}) \neq \emptyset$ and this implies that $q_i' \in M$. Suppose that $q_i'' \notin M$. Then there exists a vertex $x \in M \cap B_{i-1}$ such that $xq_i'' \notin E(G)$. If $x$ has a non-neighbor $x' \in B_{i-2}$, then $x' - q_{i-2} - x - w - N_{B_{i+1}}(K) - q_{i+1} - q_i''$ is an induced $P_7$. So, $x$ is complete to $B_{i-2}$. For any $y \in B_{i-1} \cap M$ with $yq_i'' \in E(G)$, we have $N_{B_{i-2}\cup B_i}(x) \subseteq N_{B_{i-2}\cup B_i}(y)$ and so $B_{i-1} \cap M$ is complete to $B_{i-2}$. This implies that $(B_{i-1} \cup M) \cup B_{i-2}$ is a clique of size larger than $\omega$. Therefore, $q_i'' \in M$ and so $M$ contains $\{q_{i-1}, q_i, q_i',q_i''\}$.

    \medskip
    \noindent
    \textbf{(iv) $M\subseteq B_i \cup B_{i+1}$}

    Note that $q_i$ and $q_{i+1}$ are universal in $B_i \cup B_{i+1}$. If $M = B_{i+1} \cup \{q_i\}$, then $M$ is not a maximum clique and $a, q_i, q_{i+1} \in M$. So we may assume that $M \cap (B_i \setminus \{q_i\}) \neq \emptyset$. Then $q_i' \in M$ and it follows that $M \cap N_{B_{i+1}}(K) = \emptyset$. We can assume that $|M| = \omega$. If $q_i'' \notin M$, then $|M \cap B_{i+1}| = \omega - 2$ and it follows that $B_{i+1} \cup B_{i+2}$ is a clique of size larger than $\omega$. So $q_i'' \in M$ and the claim holds.

    \medskip
    Therefore, $F$ is a $(4,5)$-good subgraph. 
\end{proof}

Next we deal with the case that $A_3(i)$ is connected with at least two neighbors in $B_i$. The lemma below also works for the case that $A_3(i)$ has two components with each component having at least two neighbors in $B_i$.

\begin{lemma}[One or two components with two neighbors in $B_i$]
        If each component of $A_3(i)$ has at least two neighbors in $B_i$, then $G$ has a $(4,5)$-good subgraph.
\end{lemma}

\begin{proof}
    Let $K$ be a non-empty component of $A_3(i)$ and $J$ be a possible second component of $A_3(i)$. 
    By Lemma \ref{lem:structure of only A3}, $K$ can be partitioned into three disjoint cliques $L_1,L_2,L_3$ such that $L_2$ is complete to $L_1\cup L_3$ and $L_1$ is anticomplete to $L_3$.
    We may assume that $L_1$ is non-empty if $L_3$ is non-empty. 
    By Lemma \ref{lem:structure of only A3},  $L_3=\emptyset$ if $J\neq \emptyset$. 
    Let $w\in K$ be universal in $K$ with maximum neighborhood in $H$ (such a vertex exists by Lemma \ref{lem:single vertex covers A_3(i)}). 
    By Lemma \ref{lem:small vertex argument for simplicial vertices in A_3(i-2)}, $|N_{B_{i-1}}(K)|>\ceil{\frac{\omega}{4}}$ and so $|N_{B_{i-1}}(K)|\ge 2$. It follows that $\{w\}\cup N_{B_{i-1}}(K)\cup N_{B_i}(K)$ is a clique and so $\omega\ge 5$.  By Lemma \ref{lem:smv on B_{i+2}+B_{i-2}}, we have $|B_{i-2}|,|B_{i+2}|>\ceil{\frac{\omega}{4}}$ and so $|B_{i-2}|,|B_{i+2}|\ge 3$. 

    By Lemma \ref{lem:blowup-basicproperty}, we can order vertices in $B_j$ as $b_j^1, b_j^2, \ldots, b_j^{|B_i|} = q_j$ so that 
    $N_{B_{j-1}\cup B_{j+1}}(b_j^1)\subseteq N_{B_{j-1}\cup B_{j+1}}(b_j^2)\subseteq \cdots \subseteq N_{B_{j-1}\cup B_{j+1}}(b_j^{|B_j|}) = B_{j-1} \cup B_{j+1}$ for every $j \in [5]$. Let $q_j' = b_j^{|B_j|-1}$ and $q_j'' = b_j^{|B_j|-2}$ for each $j$. By Lemma \ref{lem:P4-freenessofA3}, $K$ and $J$ are $P_4$-free. By Lemma \ref{lem:blowup-threeneighbor}, we may assume that $B_i \setminus N(w)$ is anticomplete to $N_{B_{i+1}}(w)$ and $q_i \in N_{B_i}(w)$. By Lemma \ref{lem:blowup-threeneighbor}, $N_{B_{i-1}\cup B_{i+1}}(a)\subset N_{B_{i-1}\cup B_{i+1}}(b)$ for any $a\in B_i\setminus N(K)$ and any $b\in N_{B_i}(K)$. This implies that $q'_i\in N_{B_i}(K)$. By symmetry, $q_i,q'_i\in N_{B_i}(J)$ if $J\neq \emptyset$.
    Suppose that $q'_i$ has a non-neighbor in $B_{i+1}\setminus N(K)$. Let $a$ be a minimal vertex over all non-neighbors of $q'_i$ in $B_{i+1}\setminus N(K)$. By Lemma \ref{lem:single vertex in A_3(i)} (2), $a$ is minimal over $B_{i+1}$ and $a$ is anticomplete to $J$.
    It follows that 
    $N_{B_i}(a)=\{q_i\}$ and $N(a)\setminus B_i$ is a clique. 
    So $d(a)\le \omega$, a contradiction. Therefore, $q'_i$ is complete to $B_{i+1}\setminus N(K)$ and so $q'_i$ complete to $B_{i+1}$ by Lemma \ref{lem:single vertex in A_3(i)} (2).

    Let $z \in J$ be universal in $J$ with maximum neighborhood in $H$ (such a vertex exists by  Lemma \ref{lem:single vertex covers A_3(i)}) if $J\neq \emptyset$. By Lemma \ref{lem:small vertex argument for simplicial vertices in A_3(i-2)}, $|N_{B_{i+1}}(K)|, |N_{B_{i+1}}(J)|>\ceil{\frac{\omega}{4}}$. By Lemma \ref{lem:comparable vertices in A_3(i)}, $N_{B_i}(K)\subseteq N_{B_i}(J)$ or $N_{B_i}(J)\subseteq N_{B_i}(K)$. 
    Let $w_0\in L_1\cup L_2$ be a vertex with minimal neighborhood on $H$ and $z_0\in J$ be a vertex with minimal neighborhood in $H$. Let $a\in N_{B_{i-1}}(w_0)$ with maximum neighborhood in $H$ and let $a'\in N_{B_{i-1}}(z_0)$ with maximum neighborhood in $H$. 
    By Lemma \ref{lem:single vertex in A_3(i)} (4), $a$ and $a'$ are complete to $B_{i-2}\cup B_i$. If $J\neq \emptyset$, then $K$ and $J$ are cliques by Lemma \ref{lem:structure of only A3} and so $w_0$ and $z_0$ are minimum and $a$ is complete to $K$ and $a'$ is complete to $J$.

    \medskip
    Next we deduce some structural information about $J$ and its neighbors on $H$ when $J\neq \emptyset$. We partition $B_{i-1}\setminus N(K)$ into two subsets $X$ and $Y$ where every vertex in $X$ has a neighbor in $B_i \setminus N(K)$ and $Y$ is anticomplete to $B_i \setminus N(K)$. By Lemma \ref{lem:component of A_3(i)}, $X$ is complete to $B_{i-2}\cup B_i$. Let $x \in B_i \setminus (N(K) \cup N(J))$ be a minimal vertex in $B_i \setminus (N(K) \cup N(J))$.
    
    \vspace{0.3cm}
    \noindent
    \textbf{Case 1: $J\neq \emptyset$ and $B_i \setminus N(J)$ is anticomplete to $N_{B_{i+1}}(J)$}

    \medskip
    Then $N_{B_{i+1}}(x)$ is disjoint from both $N_{B_{i+1}}(J)$ and $N_{B_{i+1}}(K)$. By Lemma \ref{lem:smv on w}, $|N_{B_{i+1}}(x)|>\ceil{\frac{\omega}{4}}$. 
    Assume that $N_{B_{i+1}}(J)$ and $N_{B_{i+1}}(K)$ are disjoint. By Lemma \ref{lem:B_i complete to two large cliques in B_{i-1}}, $N_{B_{i+1}}(K) \cup N_{B_{i+1}}(J) \cup N_{B_{i+1}}(x)$ is complete to $B_{i+2}$ and so $N_{B_{i+1}}(K) \cup N_{B_{i+1}}(J) \cup N_{B_{i+1}}(x) \cup B_{i+2}$ is a clique of size larger than $\omega$. Hence, we may assume that $N_{B_{i+1}}(J)$ and $N_{B_{i+1}}(K)$ are comparable. 
    By symmetry, we may assume that $N_{B_{i+1}}(J) \subseteq N_{B_{i+1}}(K)$. 
    By Lemma \ref{lem:single vertex in A_3(i)} (2), $N_{B_i}(K)= N_{B_i}(J)$.

    \vspace{0.3cm}
    \noindent
    \textbf{Case 2: $J\neq \emptyset$  and $B_i \setminus N(J)$ is anticomplete to $N_{B_{i-1}}(J)$.}

    \medskip
    By symmetry, we may assume that $N_{B_{i+1}}(K)$ and $N_{B_{i+1}}(J)$ are disjoint. So, by Lemma \ref{lem:nonedge in A_3(i)}, $N_{B_{i-1}}(K)$ and $N_{B_{i-1}}(J)$ are comparable. By Lemma \ref{lem:single vertex in A_3(i)} (4), there exists a vertex in $N_{B_{i-1}}(K)$ complete to $B_{i-2}\cup B_i$. By assumption of $J$, $N_{B_{i-1}}(J) \subsetneq N_{B_{i-1}}(K)$. By Lemma \ref{lem:single vertex in A_3(i)} (2), $N_{B_i}(K) \subseteq N_{B_i}(J)$.

    Since $N_{B_{i-1}}(x)$ and $N_{B_{i-1}}(J)$ are disjoint, $N_{B_{i+1}}(x)$ and $N_{B_{i+1}}(J)$ are comparable and so $q_{i+1} \in N_{B_{i+1}}(J)$. Let $b \in B_i \setminus N(K)$ be a maximal vertex in $B_i \setminus N(K)$. If $N_{B_i}(K) = N_{B_i}(J)$, then $b \in B_i \setminus N(J)$. If $N_{B_i}(K) \neq N_{B_i}(J)$, then $b \in N_{B_i}(J) \setminus N_{B_i}(K)$ because $N_{B_i}(J) \setminus N_{B_i}(K)$ is complete to $N_{B_{i-1}}(J)$.

    We show that $X=\emptyset$. If not, let $x'$ be a vertex in $B_i\setminus N(J)$ with maximum neighborhood in $H$. Since $N_{B_{i+1}}(x')$ and $N_{B_{i+1}}(K)$ are disjoint, $N_{B_{i-1}}(x')$ and $N_{B_{i-1}}(K)$ are comparable. Since $X\subseteq N_{B_{i-1}}(x')$, $N_{B_{i-1}}(K)\subseteq N_{B_{i-1}}(x')$ and so $N_{B_{i-1}}(J)\subseteq N_{B_{i-1}}(x')$. This contradicts the assumption of this case. 
    
    Next, we will use the structure information in Cases 1 and 2 to define good subgraphs. 
    Since $|B_{i-2}|,|B_{i+2}|\ge 3$, $q'_j$ and $q''_j$ exist for $j\in \{i-2,i+2\}$. We set $F = (F_1, F_2, F_3, F_4, F_5, O)$ as follows:
    \begin{itemize}
        \item $F_5 = \{q_{i+2}, q_{i+2}'\}$,
        \item $F_4 = \{q_{i+1}, c\}$ where $c$ is a vertex in $N_{B_{i+1}}(z_0)$ if {\bf Case 1} and  $c\in N_{B_{i+1}}(w_0)$ is a vertex with maximum neighborhood in $H$ if {\bf Case 2} or $J=\emptyset$,
        \item $F_3 = \{q_i, q_i', b\}$ where $b \in B_i \setminus N(K)$ is a maximal vertex in $B_i \setminus N(K)$,
        \item $F_2$ is defined as follows. Let $d\in Y$ be a vertex with maximum neighborhood in $B_{i-2}\cup B_i$ if $Y\neq \emptyset$. 
        \begin{itemize}
            \item If $J=\emptyset$ or {\bf Case 2}, then 
                \[
                F_2=
                    \begin{cases}
                    \{a\}, & Y=\emptyset \\
                    \{a,d\}, & Y\neq \emptyset
                    \end{cases}
                \]
            \item If \textbf{Case 1}, then $F_2 = \{a, a'\}$. 
        \end{itemize}
        
        \item $F_1$ is defined as follows. Let $e \in B_{i-2}\setminus N(d)$ with maximal neighborhood in $B_{i-2}\setminus N(d)$ if $Y\neq \emptyset$. 
        \begin{itemize}
            \item If $J=\emptyset$ or {\bf Case 2}, then 
                \[
                F_1=
                    \begin{cases}
                    \{q_{i-2}, q_{i-2}', q_{i-2}''\}, & Y=\emptyset \\
                    \{q_{i-2}, q_{i-2}', e\}, & Y\neq \emptyset
                    \end{cases}
                \]
            \item If \textbf{Case 1}, then $F_1 = \{q_{i-2}, q_{i-2}'\}$. 
        \end{itemize}
        \item 
        \begin{itemize}
            \item $O = \{w,v_k\}$ if $J = \emptyset$, where $v_k$ is a vertex in $L_3$ with maximal neighborhood in $H$. Note that $v_k$ may not exist. 
            \item $O = \{w, z\}$ if \textbf{Case 1} or \textbf{Case 2} with $N_{B_i}(K) \neq N_{B_i}(J)$
            \item $O = \{w, z, z'\}$ if \textbf{Case 2} with $N_{B_i}(K) = N_{B_i}(J)$ where $z' \in J$ such that $z'$ is second maximal in $H$.
        \end{itemize}
    \end{itemize}

    We now show that $F$ is 5-colorable. We first consider \textbf{Case 1}. In this case, we can assign color $1$ to $q_{i-2},q_i$, color $2$ to $a',q_{i+1},w$, color $3$ to $q_{i-2}'$, $b$, $c$, color $4$ to $a$, $q_{i+2}$, $z$, and color $5$ to $q_i'$ and $q_{i+2}'$. 
    In \textbf{Case 2}, we assign color $1$ to $q_{i-2},q_i$, color $2$ to $a'$, $q_{i+1}$, $w$, $q_{i-2}''$, $d,e$ and color $3$ to $z'$, $q_{i-2}'$, $b,c$, color $4$ to $a$, $q_{i+2},z$, and color $5$ to $q_i'$ and $q_{i+2}'$. 
    If $J=\emptyset$, we assign color $1$ to $q_{i-2},q_i$, color $2$ to $q_{i-2}',q_{i+1},w$, color $3$ to $q_{i-2}'',b,c,d,e$, color $4$ to $a,q_{i+2}$, and color $5$ to $q_i',q_{i+2}'$. If $v_k$ does not exist, the coloring is proper. So assume that $v_k$ exists.
    By the definition of $L_1$ and $L_3$, $w_0\in L_1$ and so either $a$ or $c$ is not adjacent to $v_k$. 
    We assign to $v_k$ the color of a non-neighbor of $v_k$ in $\{a,c\}$.
    Therefore, $F$ is 5-colorable. 

    Next, we show that every maximal clique of size $\omega-j$ with $j\in \{0,1,2,3\}$ of $G$ contains at least $4-j$ vertices from $F$. Let $M$ be a maximal clique of $G$. We consider the following cases depending on all possible locations of $M$.

    \vskip 0.3cm
    \noindent
    {\bf (i) $M\subseteq K \cup N_{B_i}(K) \cup N_{B_j}(K)$ for every $j \in \{i-1, i+1\}$}
    \vskip 0.2cm

    Since $\{w\} \cup N_{B_i}(w)$ dominates $K \cup N_H(K)$, every maximal clique in $K \cup N_{B_i}(K) \cup N_{B_j}(K)$ contains $\{w, q_i, q_i'\}$. We first consider \textbf{Case 2} or $J = \emptyset$. 
    If $J\neq \emptyset$, every maximum clique in $K \cup N_{B_i}(K) \cup N_{B_j}(K)$ contains $\{a, c\} \cap B_j$ by choice of $a$ and $c$. 
    So we may assume that $J=\emptyset$. 
    
    We now consider \textbf{Case 1}. Suppose that $c$ has a neighbor $w_1\in K$ and a non-neighbor $w_2\in K$. Then $w_2-w_1-c-q_{i+1}-(B_{i}\setminus N(K))-a'-q_{i-2}$ is an induced $P_7$. So $c$ is complete to $K$. Since $a$ is complete to $L_1 \cup L_2$ by choice of $a$, every maximum clique in $K \cup N_{B_{i-1}}(K) \cup N_{B_i}(K)$ contains exactly one of $\{a, v_k\}$. This completes the proof of this case.

    \vskip 0.3cm
    \noindent
    {\bf (ii) $M\subseteq J \cup N_{B_i}(J) \cup N_{B_j}(J)$ for every $j \in \{i-1, i+1\}$}
    \vskip 0.2cm

    In this case, $J\neq \emptyset$. 
    We first consider the case that $N_{B_i}(K) \neq N_{B_i}(J)$. Then, we have $N_{B_i}(K) \subsetneq N_{B_i}(J)$ and since $N_{B_i}(J) \setminus N_{B_i}(K)$ is complete to $N_{B_{i-1}}(J)$, we have $b \in N_{B_i}(J) \setminus N_{B_i}(K)$. So, since $\{z\} \cup N_{B_i}(z)$ dominates $J \cup N_H(J)$, every maximal clique in $J \cup N_{B_i}(J) \cup N_{B_j}(J)$ contains $\{z, q_i, q_i', b\}$.

    We now consider the case that $N_{B_i}(K) = N_{B_i}(J)$. Note that $q_i, q_i', z \in M$. In \textbf{Case 1}, every maximum clique in $J \cup N_{B_i}(J) \cup N_{B_j}(J)$ contains $\{a', c\} \cap B_j$ by choice of $a'$ and $c$. In \textbf{Case 2}, suppose that $z'\notin M$. 
    Then we have $|M|=\omega$ and $|M\setminus J|=\omega-1$. 
    Hence, either $(M\setminus J)\cup N_{B_{i+1}}(K)$ or $(M\setminus J)\cup \{a,w\}$ is a clique of size larger than $\omega$ and so $z' \in M$. This completes the proof of this case.

    \vskip 0.3cm
    \noindent
    {\bf (iii) $B_{i-2}\cup B_{i-1}$}
    \vskip 0.2cm

    Since $q_{i-2}$ and $a$ are universal in $B_{i-2} \cup B_{i-1}$, $q_{i-2},a\in M$. 
    So we may assume that $|M|\ge \omega-1$. We first consider \textbf{Case 1}. Since $a'$ is universal in $B_{i-2} \cup B_{i-1}$, $a' \in M$. Next, we show that $q_{i-2}'\in M$, which will complete the proof for this case. Suppose for a contradiction that $q_{i-2}'\notin M$. 
    Suppose first that $B_{i-1}\setminus (N(K)\cup N(J))\neq \emptyset$. 
    Let $v\in B_{i-1}\setminus (N(K)\cup N(J))$ be a vertex with minimal neighborhood in $B_{i-2}\cup B_i$. 
    By Lemma \ref{lem:single vertex in A_3(i)} (4), $v$ is a vertex in $B_{i-1}$ with minimal neighborhood in $B_{i-2}\cup B_i$. 
    It follows that $N(v)\setminus B_{i-2}$ is a clique. 
    Thus, $q_{i-2}'$ is complete to $B_{i-1}$ and so $q_{i-2}'\in M$, a contradiction.
    So $B_{i-1}\setminus (N(K)\cup N(J))=\emptyset$. 
    By Lemma \ref{lem:single vertex in A_3(i)} (2), $q_i,q_i'$ are complete to $B_{i-1}$. 
    Then $|M\cap B_{i-2}|=1$ and $|M\cap B_{i-1}|=\omega-1$. 
    So $(M\setminus \{q_{i-2}\})\cup \{q_i,q_i'\}$ is a clique of size larger than $\omega$.  

    Next we consider {\bf Case 2} or $J=\emptyset$.
    We first consider the case $Y = \emptyset$. 
    Define $q$ to be $q_{i-2}''$ if $|M|=\omega$ and to be $q_{i-2}'$ if $|M|=\omega-1$. We show that $q\in M$, which completes the proof. Suppose for a contradiction that $q\notin M$. Then there exists a vertex $u \in M\cap B_{i-1}$ such that $uq \notin E(G)$.
    By Lemma \ref{lem:component of A_3(i)}, $u \in N_{B_{i-1}}(K)$. If $u$ has non-neighbor $v$ in $B_i \setminus N(K)$, then $q - q_{i-2} - u - w - N_{B_{i+1}}(K) - q_{i+1} - v$ is an induced $P_7$. So $u$ is complete to $B_i \setminus N(K)$. For any $x \in M \cap N_{B_{i-1}}(K)$ with $xq\in E(G)$, we have $N_{B_{i-2}\cup B_i}(u) \subseteq N_{B_{i-2}\cup B_i}(x)$ and so $M \cap N_{B_{i-1}}(K)$ is complete to $B_i$. By Lemma \ref{lem:component of A_3(i)} and $Y=\emptyset$, $M\cap (B_{i-1}\setminus N(K))$ is complete to $B_i$. 
    This implies that  $(M\cap B_{i-1})\cup \{q_i,q'_i,b\}$ is a clique of size $\omega+1$.

    We now consider the case $Y \neq \emptyset$. Let $y \in Y$ be minimal in $Y$. Since $Y$ is anticomplete to $B_i \setminus N(K)$, $y$ is also minimal in $B_{i-1}$. So $N(y)\setminus B_{i-2}$ is a clique and thus $|N_{B_{i-2}}(y)| > \ceil{\frac{1}{4}\omega}$. In particular, $|N_{B_{i-2}}(y)|\ge 2$ and so $q'_{i-2}\in N_{B_{i-2}}(y)$. Therefore, $q_{i-2}, q_{i-2}', a\in M$. If $M \cap Y \neq \emptyset$, then $d \in M$, since $d$ is maximal in $Y$. Assume  that $M \cap Y = \emptyset$. If $M \cap B_{i-2} \subseteq N(d)$, then $M \cup \{d\}$ is a clique, a contradiction. So $(M \cap B_{i-2})\setminus N(d) \neq \emptyset$. By the choice of $e$, we have $e\in M$. This completes the proof of this case.

    \vskip 0.3cm
    \noindent
    {\bf (iv) $B_{i-1}\cup B_i$}
    \vskip 0.2cm

    We first consider \textbf{Case 1}. Since $a, a', q_i$ are universal in $B_{i-1} \cup B_i$, $a, a', q_i\in M$. 
    Suppose first that $B_{i-1}\setminus (N(K)\cup N(J))\neq \emptyset$. 
    Let $v\in B_{i-1}\setminus (N(K)\cup N(J))$ be a vertex with minimal neighborhood in $B_{i-2}\cup B_{i}$. 
    By Lemma \ref{lem:single vertex in A_3(i)} (4), $v$ is a vertex in $B_{i-1}$ with minimal neighborhood in $B_{i-2}\cup B_{i}$. 
    It follows that $N(v)\setminus B_i$ is a clique. 
    Thus, $q_i'$ is complete to $B_{i-1}$ and so $q_i'\in M$. 
    So $B_{i-1}\setminus (N(K)\cup N(J))= \emptyset$. 
    By Lemma \ref{lem:single vertex in A_3(i)} (2), $q_i'$ is complete to $N_{B_{i-1}}(K)\cup N_{B_{i-1}}(J)=B_{i-1}$ and so $q_i'\in M$. 

    We then consider \textbf{Case 2} or $J=\emptyset$.
    We first consider the case $Y = \emptyset$.  Assume that $X \neq \emptyset$. By Lemma \ref{lem:single vertex in A_3(i)} (4) and Lemma \ref{lem:component of A_3(i)}, $M=B_{i-1} \cup B_i$ contains $\{a, q_i, q_i', b\}$. Next we assume that $X = \emptyset$. Note that $a, q_i, q_i'$ are universal in $B_{i-1}\cup B_i$ and so $a, q_i, q_i'\in M$. 
    If $M \cap (B_i \setminus N(K)) = \emptyset$, then $M \cup \{w\}$ is a clique, a contradiction. So $M \cap (B_i \setminus N(K)) \neq \emptyset$ and this implies that $b \in M$.  We now consider the case $Y \neq \emptyset$. Let $y \in Y$ be minimal in $Y$. By $Y$ is anticomplete to $B_i \setminus N(K)$ and $P_4$-structure, $y$ is also minimal in $B_{i-1}$. So $N(y)\setminus B_i$ is a clique and so  $|N(y) \cap N_{B_i}(K)| > \ceil{\frac{1}{4}\omega}$. 
    So $q_i, q_i'$ are universal in $B_{i-1} \cup B_i$. By the choice of $a$, we have $a,q_i,q'_i\in M$. We assume that $X \neq \emptyset$. Recall that $N_{B_{i-1}}(K)$ is complete to $B_i\setminus N(K)$ and so $X \cup N_{B_{i-1}}(K) \cup B_i$ is a clique. If $M \cap Y = \emptyset$, then $M$ contains $\{a, q_i, q_i', b\}$. If $M \cap Y \neq \emptyset$, then $d \in M$ by the choice of $d$. In this case, every maximal clique in $B_{i-1} \cup B_i$ contains one of $\{a, q_i, q_i', b\}$ and $\{a, q_i, q_i', d\}$. 
    So $X=\emptyset$. Since $\{w\}\cup N_{B_{i-1}}(w)\cup N_{B_i}(w)$ is a clique, $M\cap Y\neq \emptyset$ or $M\cap (B_i\setminus N(K))\neq \emptyset$.  If $M\cap (B_i\setminus N(K))\neq \emptyset$, then $b\in M$ by the choice of $b$. If $M\cap Y\neq \emptyset$, $d\in M$ by the choice of $d$.

    \vskip 0.3cm
    \noindent
    {\bf (v) $M\subseteq B_i\cup B_{i+1}$}
    \vskip 0.2cm

    Note that $q_i$ and $q_{i+1}$ are universal in $B_i \cup B_{i+1}$ and so $q_i,q_{i+1}\in M$. Since $B_i\setminus N(K)$ is anticomplete to $N_{B_{i+1}}(K)$, it follows that $M$ contains no vertex from $B_i\setminus N(K)$ or $N_{B_{i+1}}(K)$. 
    Since $q'_i$ is complete to $B_{i+1}$, $q'_i\in M$. If $M$ contains no vertices from $N_{B_{i+1}}(K)$, then $M$ contains a vertex in $B_{i}\setminus N(K)$ (for otherwise $\{c\}\cup M$ is a clique, which contradicts the maximality of $M$). By the choice of $b$, $b\in M$.  If $M$ contains no vertices from $B_i\setminus N(K)$, then $c\in M$ since $c$ is universal in $(B_i\cup B_{i+1})\setminus (B_{i}\setminus N(K))$. So $|M\cap F|=4$. 

    \vskip 0.3cm
    \noindent
    {\bf (vi) $M\subseteq B_{i+1}\cup B_{i+2}$}
    \vskip 0.2cm

    Since $q_{i+1},q_{i+2},c$ are universal in $B_{i+1}\cup B_{i+2}$, $q_{i+1},q_{i+2},c\in M$.  
    Suppose that $|M|=\omega$ and $q'_{i+2}\notin M$.
    It follows that  $M \cap B_{i+2} = \{q_{i+2}\}$ and $M = B_{i+1} \cup \{q_{i+2}\}$. So $|B_{i+1}|=\omega-1$. 
    Since $B_{i+1}$ is complete to $\{q_i, q_i'\}$, $|B_{i+1}\cup \{q_i,q'_i\}|\ge \omega+1$, a contradiction.

    \vskip 0.3cm
    \noindent
    {\bf (vii) $M\subseteq B_{i-2}\cup B_{i+2}$}
    \vskip 0.2cm 
    By Lemma \ref{lem:smv on B_{i+2}+B_{i-2}}, $M$ contains $\{q_{i-2}, q_{i-2}', q_{i+2}, q_{i+2}'\}$.

    \medskip
    Therefore, $F$ is a $(4, 5)$-good subgraph. 
\end{proof}

Finally, we deal with the case that $A_3(i)$ has two components and one of the components has only one neighbor in $B_i$.

\begin{theorem}[Two components with one having one neighbor in $B_i$]
    Let $K$ and $J$ be two non-empty clique components of $A_3(i)$. If $|N_{B_i}(K)| = 1$, then $\chi(G)\le \ceil{\frac{5}{4} \omega}$.
\end{theorem}

\begin{proof}
    Note that $N_{B_i}(K) = \{q_i\}$.
    By Lemma \ref{lem:blowup-threeneighbor}, we may assume that $B_i \setminus N(K)$ is anticomplete to $N_{B_{i+1}}(K)$. By Lemma \ref{lem:smv on B_{i+2}+B_{i-2}}, we have  $|B_{i-2}|,|B_{i+2}|>\ceil{ \frac{\omega}{4}}$. 
    Let $w_K$ and $w_J$ be a maximum vertex in $K$ and $J$, respectively.
    Let $w_K'$ and $w_J'$ be a minimum vertex in $K$ and $J$, respectively. It follows that $|N_{B_{j}}(w_K')|>\ceil{\frac{\omega}{4}}$ and $|N_{B_{j}}(w_J')|>\ceil{\frac{\omega}{4}}$ for $j\in \{i-1,i+1\}$. 
    
    Let $z \in B_i \setminus \{q_i\}$ be a minimal vertex in $B_i \setminus \{q_i\}$. Since $N_{B_i}(J)$ is complete to $N_{B_{i\pm 1}}(J)$ but $B_i \setminus N(J)$ is not, it follows that $z \in B_i \setminus N(J)$. By assumption, $N_{B_{i+1}}(K)$ and $N_{B_{i+1}}(z)$ are disjoint. As $z$ is minimal in $B_i$, we have $N_{B_{i-1}}(z) \cup B_i$ is a clique and so $|N_{B_{i+1}}(z)| > \ceil{\frac{\omega}{4}}$ by Lemma \ref{lem:smv on w}. 
    By Lemma \ref{lem:B_i complete to two large cliques in B_{i-1}}, $N_{B_{i+1}}(z)$ and $N_{B_{i+1}}(K)$ are complete to $B_{i+2}$. 
    This implies that $\omega\ge 6$. 
    If $N_{B_{i+1}}(J)$ and $N_{B_{i+1}}(K) \cup N_{B_{i+1}}(z)$ are disjoint, then $N_{B_{i+1}}(K) \cup N_{B_{i+1}}(J) \cup N_{B_{i+1}}(z) \cup B_{i+2}$ is a clique of size larger than $\omega$ by \ref{lem:B_i complete to two large cliques in B_{i-1}}. Therefore, $N_{B_{i+1}}(J)$ is comparable with $N_{B_{i+1}}(K)$ or $N_{B_{i+1}}(z)$.

    \medskip
    \noindent
    \textbf{Case 1:} $N_{B_{i+1}}(K) \cup N_{B_{i+1}}(z) \subseteq N_{B_{i+1}}(J)$.

    \medskip
    
    Then $N_{B_{i-1}}(J)$ and $N_{B_{i-1}}(K) \cup N_{B_{i-1}}(z)$ are disjoint.
    Since $z \in B_i \setminus N(J)$, $B_i \setminus N(J)$ is anticomplete to $N_{B_{i-1}}(J)$. If $|N_{B_i}(J)| \ge 2$, then there is a vertex $x \in N_{B_i}(J) \setminus N_{B_i}(K)$ that is complete to $N_{B_{i+1}}(K)$, a contradiction. So $N_{B_i}(J) = N_{B_i}(K)=\{q_i\}$.

    Suppose that $B_{i+1}\setminus N(J)\neq \emptyset$. By Lemma \ref{lem:nonedge in A_3(i)}, $B_i \setminus \{q_i\}$ is anticomplete to $B_{i+1}\setminus N(J)$. Let $x\in B_{i+1}\setminus N(J)$ be a minimal vertex in $B_{i+1}\setminus N(J)$. If $B_{i+1} \cup N_{B_{i+2}}(x)$ contains two non-adjacent vertices $u$ and $v$, then $u\in N_{B_{i+1}}(J)\setminus (N(K)\cup N(z))$ and $v\in B_{i+2}$, then $w_J-u-x-v$ contradicts Lemma \ref{lem:generalized P4-structure}.
    Since $N_{B_i}(x) = \{q_i\}$, $d(x)\le \omega$, a contradiction. This show that $N_{B_{i+1}}(J) = B_{i+1}$.

    By Lemma \ref{lem:B_i complete to two large cliques in B_{i-1}}, $N_{B_{i-1}}(K)$ and $N_{B_{i-1}}(J)$ are complete to $B_{i-2}$.
    Let $X_{i-1} = B_{i-1}\setminus (N(K)\cup N(J))$. By Lemma \ref{lem: Y is empty}, every vertex in $X_{i-1}$ has a neighbor in $B_i\setminus \{q_i\}$. 
    By Lemma \ref{lem:component of A_3(i)}, $X_{i-1}$ is complete to $B_{i-2}$ and it follows that $B_{i-1}$ is complete to $B_{i-2}$. 
    
    Recall that $N_{B_{i+1}}(K)$ and $N_{B_{i+1}}(z)$ are complete to $B_{i+2}$. If $y\in B_{i+1}=N_{B_{i+1}}(J)$ has a non-neighbor $y'\in B_{i+2}$, then $y' - q_{i+2} - y - J - N_{B_{i-1}}(J) - N_{B_{i-1}}(K) - K$ is an induced sequence. So $B_{i+1}$ is complete to $B_{i+2}$. By Lemma \ref{lem:B_{i-2} is complete to B_{i+2}}, $B_{i+2}$ is complete to $B_{i-2}$. Since $|B_{i-2}|, |B_{i+2}| < \frac{\omega}{2}$, $B_{i-2} \cup B_{i+2}$ is not a maximum clique.
    By Lemma \ref{lem:component of A_3(i)}, $X_{i-1}$ is complete to $B_i$.

    We now show that $K$ is complete to $N_{B_{i-1}}(K)$. If $x \in N_{B_{i-1}}(K)$ has a non-neighbor $x' \in K$, then $x' - w_K - x - q_{i-2} - q_{i+2} - N_{B_{i+1}}(z) - w_J$ is an induced $P_7$. So $K$ is complete to $N_{B_{i-1}}(K)$ and $K \cup N_{B_{i-1}}(K) \cup \{q_i\}$ is a clique. 

    \medskip

    Suppose that $J \cup N_{B_{i-1}}(J) \cup N_{B_i}(J)$ contains a maximum clique. It follows that $|N_{B_{i-1}}(J)| \ge \frac{\omega-1}{2}$ or $|J| \ge \frac{\omega-1}{2}$. If $|N_{B_{i-1}}(J)| \ge \frac{\omega-1}{2}$, then $B_{i-2} \cup B_{i-1}$ is a clique of size larger than $\omega$. So $|J| \ge \frac{\omega-1}{2}$. As $N_{B_{i+1}}(K) \cup N_{B_{i+1}}(z) \subseteq N_{B_{i+1}}(J)$, $J$ is complete to $N_{B_{i+1}}(K)$ and $N_{B_{i+1}}(z)$. Since $|N_{B_{i+1}}(K)|,|N_{B_{i+1}}(z)|>\ceil{\frac{\omega}{4}}$, we obtain $N_{B_{i+1}}(K)\cup N_{B_{i+1}}(z)\cup J$ is a clique of size larger than $\omega$. This proves that $J \cup N_{B_{i-1}}(J) \cup N_{B_i}(J)$ does not contain a maximum clique. 
    If $X_{i-1} \neq \emptyset$, then $\{x_{i-1}, w_K, q_{i+1}\}$ is a good stable set where $x_{i-1} \in X_{i-1}$ (rely on $J$ is complete to $N_{B_{i+1}}(K)$ and $N_{B_{i+1}}(z)$). 
    So $X_{i-1} = \emptyset$.

    \medskip

    Note that $B_{i-1} \cup \{q_i\}$ and $B_{i+1} \cup \{q_i\}$ are not maximum clique.
    Let $q'_{i-1}\in B_{i-1}$ be a common neighbor of $w_K'$ and $z$. It follows that $q'_{i-1}$ is complete to $B_i$ and $K$. Let $a \in N_{B_{i-1}}(w_J')$. Then $a$ is complete to $J$. Let $b \in N_{B_{i+1}}(w_K')$. Then $b$ is complete to $K$. Let $v_K \in K$ be the second largest vertex in $K$ if $|K|\ge 2$. Let $q'_{j}\in B_{j}$ be the second largest vertex in $B_j$ for $j\in \{i-2,i+2\}$. Let $q'_i,q''_i$ be the second and third largest vertices in $B_i$, respectively. 

    We set $F = (F_1, F_2, F_3, F_4, F_5, O)$ as follows:
    \begin{itemize}
        \item $F_1 = \{q_{i-2}, q_{i-2}'\}$,
        \item $F_2 = \{q'_{i-1}, a\}$,
        \item $F_3 = \{q_i, q_i', q_i''\}$,
        \item $F_4 = \{b, q_{i+1}\}$,
        \item $F_5 = \{q_{i+2}, q_{i+2}'\}$,
        \item $O = \{w_K, v_K, w_J\}$.
    \end{itemize}

    We can assign color $1$ to $q_{i-2}$ and $q_i$, color $2$ to $w_K$, $w_J$, $q_i'$ and $q_{i+2}$, color $3$ to $a$, $q_i''$ and $b$, color $4$ to $v_K$, $q_{i-2}'$ and $q_{i+1}$, and color $5$ to $q_{i-1}'$ and $q_{i+2}'$. So $F$ is 5-colorable. 
    
    Next, we show that every maximal clique of size $\omega-j$ with $j\in \{0,1,2,3\}$ of $G$ contains at least $4-j$ vertices from $F$. Let $M$ be a maximal clique of $G$. We consider the following cases depending on all possible locations of $M$. If $M$ is one of $B_{i-2} \cup B_{i-1}$, $B_{i+1} \cup B_{i+2}$ and $B_{i-2} \cup B_{i+2}$, $M$ contains four vertices of $F$.

    \medskip
    \noindent
    \textbf{(i) $M\subseteq K\cup N_H(K)$}
    
    We first assume that $M$ is contained in $K \cup N_{B_{i-1}}(K) \cup \{q_i\}$. As $K \cup N_{B_{i-1}}(K) \cup \{q_i\}$ is a clique, $M$ contains $\{w_K, q_{i-1}', q_i\}$.
    If $v_K\notin M$, then $|N_{B_{i-1}}(K)| \ge \omega- 2$ and so $B_{i-1}$ is a clique of size larger than $\omega$ as $|N_{B_{i-1}}(J)|\ge 3$. 
    So we may assume that $M \subseteq K \cup N_{B_{i+1}}(K) \cup \{q_i\}$. Note that $q_i, w_K, b$ are universal in $K \cup N_{B_{i+1}}(K)$. So we can assume that $|M|=\omega$.  If $v_K\notin M$, then $|N_{B_{i+1}}(K)| \ge \omega- 2$ and so $B_{i+1} \cup B_{i+2}$ is a clique of size larger than $\omega$ as $|N_{B_{i+1}}(z)|\ge 3$ and $|B_{i+2}|\ge 3$. Therefore, $M$ contains $\{w_K, v_K, q_i, b\}$.
    
    \medskip
    \noindent
    \textbf{(ii) $M\subseteq J\cup N_H(J)$}

    Suppose first that $M\cap B_{i+1}=\emptyset$.
    Since $w_J, q_i, a$ are universal in $J \cup N_{B_{i-1}}(J) \cup \{q_i\}$, $w_J, q_i, a\in M$. Since $J \cup N_{B_{i-1}}(J) \cup \{q_i\}$ does not contain a maximum clique of $G$, we are done in this case. So $M \subseteq J \cup N_{B_{i+1}}(J) \cup \{q_i\}$. As $N_{B_{i+1}}(K) \cup N_{B_{i+1}}(z)$ is complete to $J$, $M$ contains $\{w_J, b, q_i, q_{i+1}\}$.

    \medskip
    \noindent
    \textbf{(iii) $M\subseteq B_{i-1} \cup B_i$} ($M\subseteq B_{i}\cup B_{i+1}$ is symmetric)

    Since $B_{i-1}\cup \{q_i\}$ is not a maximum clique, $M\cap (B_i\setminus \{q_i\})\neq \emptyset$ and so $q_i' \in M$. Then, $M\cap N_{B_{i-1}}(J)=\emptyset$. Since $q_i, q'_{i-1}$ are universal in $B_{i-1}\cup B_i$, $q'_{i-1}, q_i\in M$. So we can assume that $|M|= \omega$.
    If $q''_i$ is not in $M$, then $|M\cap (B_{i-1}\setminus N(J))|\ge \omega-2$. Since $|B_{i-2}| \ge 3$, $B_{i-1}\cup B_{i-2}$ is a clique of size larger than $\omega$. So $M$ contains $\{q_{i-1},q_i,q'_i,q''_i\}$.

    \medskip
    Therefore, $F$ is a $(4,5)$-good subgraph. Hence, we may assume that $N_{B_{i+1}}(J)$ is comparable with exactly one of $N_{B_{i+1}}(K)$ and $N_{B_{i+1}}(z)$, and is disjoint from the other.

    \medskip
    \noindent
    \textbf{Case 2:} $N_{B_{i+1}}(J)$ and $N_{B_{i+1}}(K)$ are comparable.

    \medskip
    Then $N_{B_{i+1}}(J)$ and $N_{B_{i+1}}(z)$ are disjoint. If $|N_{B_i}(J)| \ge 2$, then there is a vertex $x \in N_{B_i}(J) \setminus N_{B_i}(K)$ that is complete to $N_{B_{i+1}}(J) \cap N_{B_{i+1}}(K)$, a contradiction. So we obtain $|N_{B_i}(J)| = 1$ and this implies that $N_{B_i}(J) = N_{B_i}(K) = \{q_i\}$. If $B_i \setminus N(J)$ is anticomplete to $N_{B_{i-1}}(J)$, then $z - q_{i-1} - N_{B_{i-1}}(J) - J - (N_{B_{i+1}}(K) \cap N_{B_{i+1}}(J)) - N_{B_{i+1}}(z) - z$ is an induced cyclic sequence. So $B_i \setminus N(J)$ is anticomplete to $N_{B_{i+1}}(J)$. By Lemmas \ref{lem: Y is empty}, \ref{lem:reduce to four complete pair} and \ref{lem:color four complete pairs}, $\chi(G) \le \ceil{\frac{5}{4}\omega}$. 

    \medskip
    \noindent
    \textbf{Case 3:} $N_{B_{i+1}}(J)$ and $N_{B_{i+1}}(z)$ are comparable.

    \medskip

    Then $N_{B_{i+1}}(J)$ and $N_{B_{i+1}}(K)$ are disjoint and it follows that $N_{B_{i-1}}(J) \cup N_{B_{i-1}}(z) \subseteq N_{B_{i-1}}(K)$. So $|N_{B_{i-1}}(K)| > 2 \cdot \ceil{\frac{\omega}{4}}$. Since $z \in B_i \setminus N(J)$, we have $B_i\setminus N(J)$ is anticomplete to $N_{B_{i-1}}(J)$. If $|N_{B_i}(J)|=1$, then we argue symmetrically to \textbf{Case 1}. So we may assume that $|N_{B_i}(J)| \ge 2$. 
    Let $Y\subseteq B_{i+1}\setminus N(K)$ be the set of vertices that are anticomplete to $B_i\setminus \{q_i\}$. We show that $Y$ is complete to $B_{i+2}$. Suppose not. Choose a vertex $y\in Y$ that has a non-neighbor in $B_{i+2}$ with $N_{B_{i}\cup B_{i+2}}(y)$ is minimum. By Lemma \ref{lem:B_i complete to two large cliques in B_{i-1}}, $N_{B_{i+1}}(K)$ and $N_{B_{i+1}}(J)$ are complete to $B_{i+2}$ and so $y$ is anticomplete to $K\cup J$. 
    By Lemmas \ref{lem:single vertex in A_3(i)} (2) and \ref{lem:B_i complete to two large cliques in B_{i-1}}, $y$ is minimum in $B_{i+1}$. It follows that $d(y)\le \omega$, a contradiction. So $Y$ is complete to $B_{i+2}$. By Lemma \ref{lem:B_i complete to two large cliques in B_{i-1}}, $B_{i+1}$ is complete to $B_{i+2}$. 

    We then show that $N_{B_{i-1}}(K)$ is complete to $B_{i-2}$. By Lemma \ref{lem:B_i complete to two large cliques in B_{i-1}}, $N_{B_{i-1}}(J)$ and $N_{B_{i-1}}(z)$ is complete to $B_{i-2}$. If $x \in N_{B_{i-1}}(K)$ has a non-neighbor $x' \in B_{i-2}$, then $x' - q_{i-2} - x - K - N_{B_{i+1}}(K) - N_{B_{i+1}}(J) - J$ is an induced sequence. So $N_{B_{i-1}}(K)$ is complete to $B_{i-2}$.

    We now show that $K$ is complete to $N_{B_{i-1}}(K)$. As $N_{B_{i-1}}(J) \cup N_{B_{i-1}}(z) \subseteq N_{B_{i-1}}(K)$, $N_{B_{i-1}}(J) \cup N_{B_{i-1}}(z)$ is complete to $K$. If $x \in N_{B_{i-1}}(K)$ has a non-neighbor $x' \in K$, then $x' - w_K - x - q_{i-2} - q_{i+2} - N_{B_{i+1}}(J) - J$ is an induced sequence. So $K$ is complete to $N_{B_{i-1}}(K)$. We then show that $J$ is complete to $N_{B_{i+1}}(J)$. If $y \in N_{B_{i+1}}(J)$ has a non-neighbor $y' \in J$, then $y' - w_J - y - N_{B_{i+1}}(K) - K - N_{B_{i-1}}(z) - q_{i-2}$ is an induced sequence. So $J$ is complete to $N_{B_{i+1}}(J)$. 

    \medskip

    We now show that $K \cup N_{B_{i+1}}(K) \cup \{q_i\}$ does not contain a maximum clique. Let $M$ be a maximum clique in $K \cup N_{B_{i+1}}(K) \cup \{q_i\}$. Since $K \cup N_{B_{i-1}}(K) \cup \{q_i\}$ is a clique and $|N_{B_{i-1}}(K)| > 2 \cdot \ceil{\frac{\omega}{4}}$, we have $|K| <\omega - 2 \cdot \ceil{\frac{\omega}{4}} - 1$. Hence, we obtain $|M \cap N_{B_{i+1}}(K)| > 2 \cdot \ceil{\frac{1}{4}\omega}$. Then, $N_{B_{i+1}}(K) \cup N_{B_{i+1}}(J) \cup B_{i+2}$ is a clique of size larger than $\omega$ and the claim holds.

    \medskip

    By the definition of $z$ and Lemma \ref{lem:B_i complete to two large cliques in B_{i-1}}, there is a vertex in $N_{B_{i-1}}(z)$ complete to $B_{i-2}\cup B_i$. We set this vertex to be $q_{i-1}$. Then $q_{i-1} \in N_{B_{i-1}}(K)$. Let $a \in N_{B_{i-1}}(w_J')$. Then $a$ is complete to $J$. Let $q_i' \in N_{B_i}(J) \setminus \{q_i\}$ be maximal over $N_{B_i}(J) \setminus \{q_i\}$.
    Let $b \in B_i \setminus N(J)$ be a maximal vertex over $B_i \setminus N(J)$. Let $c \in N_{B_{i+1}}(w_K')$. Then $c$ is complete to $K$. Let $q_j' \in B_j$ be the second largest vertex in $B_j$ for $j \in \{i-2, i+2\}$. By Lemma \ref{lem:smv on B_{i+2}+B_{i-2}}, we have $q'_{i+2}$ is complete to $B_{i-2}$ and $q'_{i-2}$ is complete to $B_{i+2}$. If $q'_{i-2}$ has a non-neighbor in $B_{i-1}\setminus N(K)$, then a minimal vertex $x$ in $B_{i-1}\setminus N(K)$ has $N(x)\setminus B_{i-2}$ is a clique and $N_{B_{i-2}}(x)=\{q_{i-2}\}$. So $d(x)\le \omega$, a contradiction. So $q'_{i-2}$ is complete to $B_{i-1}$.

    We set $F = (F_1, F_2, F_3, F_4, F_5, O)$ as follows:
    \begin{itemize}
        \item $F_1 = \{q_{i-2}, q_{i-2}'\}$,
        \item $F_2 = \{q_{i-1}, a\}$,
        \item $F_3 = \{q_i, q_i', b\}$,
        \item $F_4 = \{c, q_{i+1}\}$.
        \item $F_5 = \{q_{i+2}, q_{i+2}'\}$,
        \item $O = \{w_K, w_J\}$.
    \end{itemize}
    
    We can assign color $1$ to $q_i$ and $q_{i+2}$, color $2$ to $q_{i-2}$, $w_K$ and $w_J$, color $3$ to $q_{i-2}'$, $q_i'$ and $c$, color $4$ to $a$, $b$ and $q_{i+2}'$, and color $5$ to $q_{i-1}$ and $q_{i+1}$. So $F$ is $5$-colorable. 
    
   Next, we show that every maximal clique of size $\omega-j$ with $j\in \{0,1,2,3\}$ of $G$ contains at least $4-j$ vertices from $F$. Let $M$ be a maximal clique of $G$. We consider the following cases depending on all possible locations of $M$.  If $M$ is contained in one of $B_{i-2} \cup B_{i-1}$, $B_{i+1} \cup B_{i+2}$ and $B_{i-2} \cup B_{i+2}$, $M$ contains four vertices of $F$.
    
    \medskip
    \noindent
    \textbf{(i) $M\subseteq K \cup N_H(K)$}

    Since $M \subseteq K \cup N_{B_{i-1}}(K) \cup \{q_i\}$ is a clique, $M$ contains $\{w_K, a, q_{i-1}, q_i\}$. Hence, we can assume that $M\subseteq K \cup N_{B_{i+1}}(K) \cup \{q_i\}$.
    Since $K \cup N_{B_{i+1}}(K) \cup \{q_i\}$ does not contain a maximum clique and $w_K, q_i, c$ are universal in $K \cup N_{B_{i+1}}(K)$, we obtain $w_K, q_i, c \in M$.

    \medskip
    \noindent
    \textbf{(ii) $M\subseteq J \cup N_H(J)$}

    Note that $q_i, q_i', w_J$ are universal in $J \cup N_H(J)$. We first assume that $M \subseteq J \cup N_{B_i}(J) \cup N_{B_{i+1}}(J)$. Since $N_{B_{i+1}}(J)$ is complete to $J \cup N_{B_i}(J)$, $q_{i+1}$ is universal in $J \cup N_{B_i}(J) \cup N_{B_{i+1}}(J)$. So $w_J, q_i, q_i', q_{i+1} \in M$. Now, we may assume that $M \subseteq J \cup N_{B_{i-1}}(J) \cup N_{B_i}(J)$. Since $a$ is complete to $J$, we obtain $w_J, a, q_i, q_i' \in M$.

    \medskip
    \noindent
    \textbf{(iii) $M\subseteq B_{i-1} \cup B_i$}

    Note that $q_{i-1}$ and $q_i$ are universal in $B_{i-1} \cup B_i$. If $M = B_{i-1} \cup \{q_i\}$, then $M$ is not a maximum clique and $q_{i-1}, q_i, a \in M$. So we may assume that $M \cap (B_i \setminus \{q_i\}) \neq \emptyset$. We first assume that $M \cap (B_i \setminus N(J)) = \emptyset$ and this implies that $M \cap (N_{B_i}(J) \setminus \{q_i\}) \neq \emptyset$ and so $q_i' \in M$. Since $a$ is universal in $B_{i-1}\cup N_{B_i}(J)$, we have $a \in M$. Hence, we may assume that $M \cap (B_i \setminus N(J)) \neq \emptyset$ and this implies that $b \in M$. 
    Since $q'_i$ has a neighbor in $N_{B_{i-1}}(J)$, $N_{B_{i-1}}(b) \subseteq N_{B_{i-1}}(q'_i)$ and so $q'_i\in M$.

    \medskip
    \noindent
    \textbf{(iv) $M\subseteq B_i \cup B_{i+1}$}

    Note that $q_i$ and $q_{i+1}$ are universal in $B_i \cup B_{i+1}$. If $M = B_{i+1} \cup \{q_i\}$, then $M$ is not a maximum clique and $q_i, q_{i+1}, c \in M$. So we may assume that $M \cap (B_i\setminus \{q_i\}) \neq \emptyset$. We first assume that $M \cap (B_i \setminus N(J)) = \emptyset$ and this implies that $M \cap (N_{B_i}(J) \setminus \{q_i\}) \neq \emptyset$ and $q_i' \in M$. So $M \cap N_{B_{i+1}}(K) = \emptyset$ and we can assume that $|M|=\omega$.
    As $N_{B_i}(J) \cup N_{B_{i-1}}(J) \cup N_{B_{i-1}}(z)$ is a clique and $|N_{B_{i-1}}(J) \cup N_{B_{i-1}}(z)| > 2\cdot \ceil{\frac{\omega}{4}}$, we obtain $|M \cap (B_{i+1}\setminus N(K))| > 2\cdot \ceil{\frac{\omega}{4}}$. Then $B_{i+1} \cup B_{i+2}$ is a clique of size larger than $\omega$. Hence, $M$ is not a maximum clique and $q_i, q_i', q_{i+1} \in M$. Now we may assume that $M \cap (B_i \setminus N(J)) \neq \emptyset$ and so $b \in M$. Since $N_{B_{i-1}}(b) \subseteq N_{B_{i-1}}(q'_i)$, we have $q_i' \in M$. Hence, $M$ contains $\{q_i, q_i', q_{i+1}, b\}$.

    \medskip
    Therefore, $F$ is a $(4, 5)$-good subgraph.  
\end{proof}

\section{Concluding Remarks}\label{sec:conclue}

In this paper, we prove that $f(n)=\ceil{\frac{5}{4}n}$ is the optimal $\chi$-bounding function for the class of ($P_7$, even-hole)-free graphs, which generalizes the result of Karthick and Maffary \cite{KM19} on ($P_6$, even-hole)-free graphs. We feel that $f(n)=2n-1$ may not be the optimal $\chi$-bounding function for even-hole-free graphs, as this comes merely from the existence of $f$-small vertices. Equal-size blowups of $C_5$ suggest that 
$f(n)=\ceil{\frac{5}{4}n}$ may be the optimal $\chi$-bounding function for the class of even-hole-free graphs and we conjecture this should be the case.

\begin{conjecture}
    For every even-hole-free graph $G$, $\chi(G)\le \ceil{\frac{5}{4}\omega(G)}$.
\end{conjecture}

Chen, Xu and Xu \cite{cap} shows that the conjecture is true for (cap, even-hole)-free graphs. 
Theorem \ref{thm:main} shows that the conjecture is true for ($P_7$, even-hole)-free graphs. 
Recently, there has been a series of work on tree-width and tree-independence number of even-hole-free graphs (see \cite{CGHLS24,CT25} for instance). The techniques developed there may offer new insights and tools to attack this conjecture.


\end{document}